\documentclass[12pt,a4paper,twoside]{memoir}

\usepackage[T1]{fontenc}
\usepackage[cp1250]{inputenc}

\usepackage{amssymb,amsmath,amsthm,amsfonts} 
\usepackage{amscd}
\input xy
\xyoption{all}
\usepackage{bm} 

\usepackage{color}
\usepackage[twoside]{geometry}
\usepackage[hyperindex]{hyperref} 

\usepackage{makeidx} 
\makeindex  
\usepackage{showidx} 

\usepackage[nosort,nocompress,space]{cite} 

\newtheorem{theorem}{Theorem}[chapter]

\newtheorem{proposition}[theorem]{Proposition}
\newtheorem{lemma}[theorem]{Lemma}
\newtheorem{corollary}[theorem]{Corollary}

\theoremstyle{definition}

\newtheorem{remark}[theorem]{Remark}
\newtheorem{definition}[theorem]{Definition}

\renewcommand{\bf}{\textbf}

\newcommand{\R}{\mathbb{R}}
\newcommand{\N}{\mathbb{N}}

\newcommand{\eps}{\varepsilon}
\newcommand{\del}{\delta}
\newcommand{\g}{\mathfrak{g}}

\newcommand{\ra}{\rightarrow}
\newcommand{\lra}{\longrightarrow}

\newcommand{\xra}{\xrightarrow} 

\newcommand{\Sec}{\operatorname{Sec}}
\newcommand{\id}{\operatorname{id}}
\newcommand{\Image}{\operatorname{Im}}
\newcommand{\dd}{\mathrm{d}} 
\newcommand{\sT}{\textrm{T}}
\newcommand{\T}{\mathrm{T}} 
\newcommand{\pa}{\partial} 
\newcommand{\p}{\operatorname{p}}
\newcommand{\spann}{\operatorname{span}}
\newcommand{\cl}{\operatorname{cl}}
\newcommand{\pt}{\operatorname{pt}}
\newcommand{\modulo}{\operatorname{mod}}

\newcommand{\wh}{\widehat}
\newcommand{\wt}{\widetilde}

\newcommand{\conv}{\operatorname{conv}}
\newcommand{\X}{\mathcal{X}} 
\newcommand{\K}{\bm{\mathcal{K}}} 
\newcommand{\eS}{\mathcal{S}}
\newcommand{\cH}{\mathcal{ H}}

\newcommand{\B}{\bm}
\newcommand{\m}{\mathfrak{m}}
\newcommand{\BB}{\mathfrak{v}}
\newcommand{\so}{\mathfrak{so}(3)}
\newcommand{\SO}{\mathrm{SO}(3)}
\newcommand{\A}{\mathfrak{w}} 

\newcommand{\Uadm}{\mathcal{U}_{adm}}
\newcommand{\BM}{\mathcal{BM}}
\newcommand{\GG}{\mathcal{G}}
\newcommand{\HH}{\mathcal{H}}
\newcommand{\ol}{\overline}
\newcommand{\ul}{\underline}
\newcommand{\PP}{\operatorname{P}}
\newcommand{\AG}{\mathcal{A}(\mathcal{G})}
\newcommand{\XX}{\mathfrak{X}}
\newcommand{\LL}{\mathcal{L}}
\newcommand{\FF}{\mathcal{F}}
\def\<#1>{\big\langle #1\big\rangle}

\newcommand{\linia}{\rule{\linewidth}{0.5pt}}

\makechapterstyle{mychapterstyle}{%
    \setlength{\afterchapskip}{1.5cm}
    \setlength{\beforechapskip}{0cm}
    \setlength{\midchapskip}{0cm} 

   }

\chapterstyle{mychapterstyle}

\makeevenhead{headings}{\thepage}{}{\small\scshape\leftmark}
\makeoddhead{headings}{\small\scshape\rightmark}{}{\thepage}

\input{maketitle.tex}

\begin{document}
\frontmatter
\author{Michał Jóźwikowski}
\title{Optimal
control theory\\ 
on almost Lie algebroids}
\date{Warsaw, June 2011}


\maketitle

\newpage

\chapter{Acknowledgements}

This work would not be completed without the help of my supervisor---professor Janusz Grabowski. I would like to thank him for many inspiring discussions, useful suggestions (including the choice of this topic) and, especially, for his patience and tolerance.

I would also like to express my gratitude to the whole scientific and non-scientific stuff of the Institute of Mathematics where I spent the last four years.

Finally, I thank my Parents and Olga for their constant support.

\newpage
\tableofcontents
\mainmatter
\numberwithin{equation}{chapter}
\chapter{Introduction}\label{ch:intro}

\subsection{An overview}
This work is rooted in two important areas of mathematics. One of them is the optimal control theory with its central theorem, the celebrated Pontryagin maximum principle (PMP). The second is the theory of Lie algebroids and, in particular, its applications to geometric mechanics. 

Note that the PMP may be regarded as an extension of the calculus of variations to a much bigger class of problems. On the other hand, the language of Lie algebroid theory has proved to be a very fruitful tool in mechanics and variational calculus allowing, for example, to treat standard systems and systems reduced by inner symmetries in a unified way, and to have a deeper insight into the nature of the Lagrange and Hamilton formalisms. Therefore, we may think of the PMP and the algebroidal formulation of the Lagrange formalism as two different extension of the standard calculus of variations---the first by generalising the class of problems, the second by generalising the geometric context. 

Consequently, it is a natural idea to unify these two generalisations and formulate the PMP in the language of algebroids. Some attempts in this direction has already been made (\cite{cortes_martinez, martinez_red_opt_ctr,GG_var_calc}), yet so far there is no satisfactory solution. In this work we tried to give a full solution of the problem in a possibly general context both geometrical and technical. Namely, we formulate our extension of the PMP for optimal control problems (OCPs) on almost Lie algebroids (objects generalising Lie algebroids), we work with bounded measurable controls and absolutely continuous base trajectories, and we consider quite general boundary conditions. Note that, according to \cite{agrachev_pmp_50years}, the PMP was never a subject of any substantial generalisation, apart from the technical ones discussed in the next subsection.  

\subsection{Optimal control theory}

The PMP, proved in 1956 by L. Pontryagin and his collaborators \cite{pontryagin}, was an answer to the problem of finding  solutions of optimal problems of new type which could not be treated with the standard variational methods. Roughly speaking, we are interested in minimising the standard action functional of the calculus of variation (on a manifold $M$), but we restrict our attention to trajectories $x(t)\in M$ whose velocities can be controlled, i.e., they belong to a given subset of $\T M$. The PMP (Theorem \ref{thm:pmp_class}) expresses the necessary conditions for optimality in the language of the canonical symplectic structure of the cotangent bundle $\T^\ast M$. A short account of the result and the historical discussion of its development can be found in \cite{agrachev_pmp_50years}. 

From its appearance, the PMP became an object of intensive studies both on theoretical and applied level. As a technique it is used in a wide range of disciplines which include engineering, aerospace, robotics, medicine, economics, and other (see the references in \cite{barbero_pmp} for more details), and as such is used for solving concrete practical problems.  

The theoretical development concentrated in several directions. 
One of them, initiated by Clarke in the 70s, was devoted to relaxing the assumptions under which the result holds. This research, undertaken mostly by Clarke, Ioffe, Loewen, Mordukhovich, Rockafellar, and Vinter, used the tools provided by the non-smooth analysis and led to generalisations of the PMP among which the most important is \cite{clarke_min_hypotheses}. The monograph \cite{clarke_necessary} discusses this topic in details (see also \cite{clarke_brief} for a brief account of the most important results).    

In the 60s there  was a search for a simple proof of the PMP. The original argument of Boltyanskii (\cite{pontryagin}) is long, and some people believed that a shorter reasoning based on variational methods can be found. Such a proof does not exists so far, even though some simpler versions of the PMP can be proven quite elementary. Essentially, there are two kinds of proofs of the PMP. The first is the original one which uses needle variations---a tool developed by Boltyanskii especially for this purpose. The other argument was given by \cite{gamkrelidze_78} and is based on the concept of generalised controls. The research in this area concentrated mostly on translating the proof expressed originally in the language of differential equations to the language of differential geometry. Recent results \cite{barbero_pmp,agrachev} show a deeper understanding of the geometric origins on the PMP and connection between optimality and accessibility.

The last topic is closely related to the problem of abnormal extremals. These are the solutions of the OCPs which do not depend on the cost function but on the geometry of the considered system only. For a long time, until the discovery  of counterexamples in sub-Riemannian geometry  \cite{montgomery_abnormal}, people believed that such curves cannot be optimal. Since then this area became a subject of a growing interest \cite{agrachev1996abnormal,agrachev1998abnormal, bonnard_abnormal,langerock_phd,langerock2003}.

\subsection{Algebroids and their application to mechanics}

Lie algebroids were introduced by Pradines in the mid 60s as infinitesimal objects associated with Lie groupoids, per analogy to Lie algebras and Lie groups. In a series of short articles \cite{pradines1966,pradines1967,pradines1967a, pradines1968} he announced a very general program of developing the Lie theory for Lie groupoids.  The progress was not very fast until the 80s, when Weinstein introduced the notion of a symplectic groupoid to Poisson geometry \cite{weinstein1987symplectic}. Since then Lie algebroids and Lie groupoids has become objects of great significance in this field (see e.g. \cite{coste_dazord_weinstein,weinstein_coisotropic, courant, xu_weinstein, xu_symplectic, cattaneo_felder, crainic_fernandes_poisson}). The main reason of this is the fact that with every Poisson manifold $P$ one  can naturally associate a Lie algebroid structure on the cotangent bundle $\T^\ast P\lra P$ (on the other hand Lie algebroids are objects dual to linear Poisson structures). Therefore many problems of Poisson geometry can be translated into the language of Lie algebroid theory. Other applications of Lie algebroids appeared in the theory of foliations (e.g. \cite{pradines1966, winkelnkemper,crainic_moerdijk, moerdij_mrcun}) and, for locally trivial Lie algebroids, in the theory of connections (cf. \cite{mackenzie_1987}). 

In all these applications the problem of integrability of Lie algebroids posted already by Pradines plays a central role. For example the existence of a symplectic realisation of a Poisson manifold $P$ is equivalent to the integrability of the associated Lie algebroid $\T^\ast P\lra P$. The integrability problem was attacked by many authors and partial solutions for some special classes of Lie algebroids has been obtained (see \cite{almeria,almeida_kumpera,almeida_molino,mackenzie_1987,cattaneo_felder}) until it was completely solved by \cite{crainic_fernandes}. A detailed discussion of this topic is given in Appendix \ref{sapp:groupoids}.   

We are interested mostly in applications of the algebroid theory to mechanics which was also a Weinstein's idea \cite{weinstein} (see also \cite{libermann}). 
Since then the topic was studied in different contexts
by many authors (\cite{cortes_leon_marrero,cortes_martinez,leon_marrero_martinez, martinez_geom_form,martinez_lagr_mech,martinez_cft,martinez_lie_classs_mech,martinez_var_calc}). It was observed a little bit later, following the approach to
analytical mechanics proposed by Tulczyjew \cite{tulczyjew_ham_lagr,tulczyjew_urbanski_slow}, that
geometrical mechanics, together with the Euler-Lagrange and the
Hamilton equations, constrained dynamics, etc., can be developed based on more
general objects than Lie algebroids (\cite{GGU_geom_mech,GG_var_calc}). They were introduced in \cite{GU_algebroids} under the name \emph{(general)
algebroids}. This generalisation turns out to be of practical use,
as systems of mechanical type with nonholonomic constrains
allow a nice geometrical description in terms of
\emph{skew-algebroids} \cite{grabowski_nonholonomic} which do not have to satisfy the Jacobi identity in general.

\subsection{Reduction in optimal control theory}
As a motivation, before formulating our main results, let us discuss a reduction by inner symmetries of a control or mechanical system.  It is a well-known phenomena in analytical mechanics and control
theory that symmetries of a system lead to reductions of its
degrees of freedom. It is also well-understood that such a
reduction procedure is not purely computational but is associated
with a reduction of the geometrical structures hidden behind.

A typical situation considered in control theory is a control
system $F:P\times U\lra\T P$ on a manifold $P$ (with $U$ being
the set of control parameters) which is equivariant w.r.t. the
action of a Lie group $G$ on $P$ and the induced action on $\T
P$. If this action is free and proper,  we deal in fact with a
$G$-invariant control system on a principal bundle $P\ra P/G$.
Introducing a $G$-invariant cost function $L:P\times U\lra\R$, one
ends up with a $G$-invariant optimal control problem on a
principal bundle $P\ra P/G$.

There are basically two ways of obtaining optimality necessary
conditions for such a problem. In the first, one takes the
PMP for the unreduced system on $P$
and performs the Poisson reduction of the associated Hamiltonian
equations. For the simple case of an invariant system on a Lie
group $P=G$ (see eg. \cite{jurdjevic}) one obtains a system on the Lie
algebra $\g=\T G/G$, and the reduced Hamilton equations are
the Hamilton equations obtained by means of the Lie--Poisson
structure on $\g^\ast$. The best known example of this type is
probably the reduction for the rigid body in analytical mechanics:
from the cotangent bundle $\T^\ast\,\SO$ of the group $\SO$
playing the role of the configuration space to the linear Poisson
structure on $\so^\ast$---the dual of the Lie algebra $\so$.
Similar situation appears for homogeneous spaces \cite{jurdjevic} and
general principal bundles \cite{martinez_red_opt_ctr, martinez_lie_classs_mech}. The reduced system
lives on the bundle $\T P/G$ which is canonically a Lie
algebroid, called the {\em Atiyah algebroid of $P$}, and the
reduced Hamilton equations are associated with the linear
Poisson structure on $\T^\ast P/G$ (equivalent to the presence of
a Lie algebroid structure on $\T P/G$). In this approach one
obtains a version of the PMP, yet the Hamiltonian reduction seems
to be purely computational and a big part of the geometry of the
problem remains hidden.

The second approach, called the \emph{Lagrangian reduction}, was
introduced by Marsden and his collaborators (see for example
\cite{ cendra_holm_marsden}) in the context of analytical mechanics. Here, one
uses the reduced data $f:P/G\times U\lra \T P/G$ and $l:P/G\times
U\lra\R$, and the reduced variations (homotopies) to obtain a
reduced version of the Euler-Lagrange equations. In this approach
it becomes clear that the reduction of the variational principle
is not only the reduction of the data and geometrical structure, but
also a reduction of variations (homotopies)---this is most
clearly stated in \cite{cendra_holm_marsden} for the case of an invariant system
on a Lie group. By means of the Lagrangian reduction one can
obtain various results such as Euler-Poincar{\'e} equations and
Hammel equations. Despite of this advantages, the Lagrangian
reduction seems to be useful rather in mechanics than in control
theory, as one requires the geometry of the set of controls
$U$ and controls itself being very regular ($U$ should be at least an affine subspace of $\R^n$, and controls  differentiable), so accepts no discontinuity, switch-on-switch-off controls, etc.

\subsection{The main result}
The aim of our work is to extend the fundamental theorem of
optimal control---the PMP---to
the setting of almost Lie (AL) algebroids---geometrical objects
generalising Lie algebroids. 

Since Lie algebroids are infinitesimal (reduced) objects of
(local) Lie groupoids (like Lie algebras are for Lie groups), we
are motivated mostly by the Lie groupoid $\GG$---Lie algebroid
$A(\GG)$ reduction. Obviously, a reduction of an invariant control
system on a Lie groupoid should lead to a system on the associated
Lie algebroid. An example of such a situation was discussed in the previous subsection, where an invariant control system on a principal bundle lead to a system on the associated Atiyah algebroid.  

What is more, similarly to the scheme of the Lagrangian reduction, we should also reduce the variations (homotopies) from $\GG$ to $A(\GG)$. This will motivate the abstract definition of the homotopy of admissible paths on an AL
algebroid (algebroid homotopy). Finally, reducing an invariant OCP
on the Lie groupoid $\GG$ would not be complete without reducing
the boundary conditions as well. The idea is to substitute
fixed-end-points boundary conditions on $\GG$ by fixed-homotopy-class
conditions. These two are closely related (see Chapter \ref{ch:OCP} for a detailed discussion) and equivalent if $\GG$ is $\alpha$-simply connected (the homotopy class of a curve on a simply connected manifold is uniquely determined by its end-points).
Now, since homotopies in $\GG$ correspond to algebroid homotopies in $\AG$,
we can express the reduced boundary conditions in $A(\GG)$ as
fixing the algebroid homotopy class of the trajectory of the
reduced control system. A similar construction can be made also for more general boundary conditions.

At the end, for a general AL algebroid
$E$, we can formulate an analog of the OCP which, in the case of
an integrable algebroid $E=A(\GG)$, turns out to be an  invariant
OCP reduced from $\GG$. Let us note that our  understanding of
algebroid homotopies and homotopy classes is closely related to
that of Crainic and Fernandes \cite{crainic_fernandes}, where similar techniques
were used to generalise the Third Theorem of Lie and integrate Lie
algebroids. However, our framework is much more general, as we no
longer remain in the smooth category.

To explain briefly the result, let us note that an AL algebroid is
a vector bundle $\tau:E\lra M$ together with a vector bundle map
$\rho:E\lra\T M$ (\emph{anchor}) and a skew-symmetric bilinear
bracket $[\cdot,\cdot]$ on the space of sections of $E$ which
satisfy certain compatibility conditions. The algebroid structure
on $E$ is equivalent to the presence of a certain linear bi-vector
field $\Pi$ on the dual bundle $E^\ast$. Note that the bivector field $\Pi$ defines the \emph{Hamiltonian vector field} $\X_H$ associated with any $C^1$-function $H$ on
$E^\ast$, defined in the standard way as the contraction
$\X_H=\iota_{\dd H}\Pi$.

Standard examples of
AL (in fact Lie) algebroids are: the tangent bundle $E=\T M\lra M$
with $\rho=\id_{\T M}$ and the Lie bracket of vector fields, and
a finite-dimensional real Lie algebra $E=\g$ with the trivial anchor
map ($M$ is a single point in this case) and the Lie bracket on
$\g$. In the first case, $\Pi$ is the canonical Poisson tensor on
$\T^\ast M$, whereas in the second---the Lie--Poisson structure
on $\g^\ast$. An example of an AL algebroid which is not a Lie
algebroid is given by any real vector bundle with a smooth family
of skew-symmetric bilinear (but not Lie) operations on its fibers.

On the bundle $E$ we can consider \emph{admissible paths}; i.e.,
bounded measurable  maps $a:[t_0,t_1]\lra E$ such that the
projection $x(t)=\tau(a(t))$ of $a(t)$ onto $M$ is absolutely
continuous (AC) and $\dot x(t)=\rho(a(t))$ a.e. On admissible
paths we have an equivalence relation $a\sim b$ interpreted as a
reduction of homotopy equivalence (with fixed end-points). Note
that equivalent paths need not to be defined on the same time
interval. For an admissible path $\sigma$, we denote with
$[\sigma]$ the equivalence class of $\sigma$.

A \emph{control system} is defined by a continuous map $f:M\times
U\lra E$, where $U$ is a topological space of control parameters
such that, for each $u\in U$, the function $f(\cdot,u)$ is a section
of class $C^1$ of the bundle $E$. Every \emph{admissible control},
i.e., a bounded measurable path $u(t)$ in $U$, gives rise to an
absolutely continuous path in $M$ defined by the differential
equation
\begin{equation*}
\dot x(t)=\rho\left(f(x(t),u(t))\right),
\end{equation*}
and to an admissible path $a(t)=f(x(t),u(t))$ covering $x(t)$. We
will call $a(t)$ the \emph{trajectory} of the control system and
the pair $(x(t),u(t))$---the \emph{controlled pair}. An
\emph{optimal control problem} for this control system is
associated with a fixed homotopy class $[\sigma]$ of an admissible
path $\sigma$ and a \emph{cost function} $L:M\times U\ra\R$. The
problem is to find a controlled pair $(x(t),u(t))$ with
$t\in[t_0,t_1]$ (the time interval is to be found as well) such
that
\begin{equation}\tag{P}\label{eqn:PPP} \begin{split}\text{the integral $\int_{t_0}^{t_1}L\big(x(t),u(t)\big)\dd t$ is
minimal among all controlled pairs $(x,u)$ for}\\\text{ which the
$E$-homotopy class of the trajectory $f(x(t),u(t))$ equals
$[\sigma]$.}\end{split}
\end{equation}

Our main result is the following.

\begin{theorem}\label{thm:PPP}
Let $(x(t),u(t))$, with $t\in[t_0,t_1]$, be a controlled pair
solving the optimal control problem \eqref{eqn:PPP}. Then there exists a
curve $\xi:[t_0,t_1]\lra E^\ast$ covering $x(t)$ and a constant
$\ul\xi_0\leq 0$ such that
\begin{itemize}
    \item the curve $\xi(t)$ is a trajectory of the time-dependent family of Hamiltonian vector fields
    $\X_{H_t}$,
    $H_t(x,\xi):=H(x,\xi,u(t))$, where
    $$H(x,\xi,u)=\< f\left(x,u\right), \xi>+\ul\xi_0 L\left(x,u\right);$$
    \item the control $u$ satisfies the ``maximum principle''
    $$H(x(t),\xi(t),u(t))=\sup_{v\in U}H(x(t),\xi(t),v)$$
and $H(x(t),\xi(t),u(t))=0$ at every regular point $t$ of $u$;
    \item if $\ul\xi_0=0$, then the covector $\xi(t)$ is nowhere-vanishing.
\end{itemize}
\end{theorem}

We have also developed a version of this result for general boundary conditions. These can be expressed by means of two smooth algebroid morphisms $\Phi_0:\T S_0\lra E$ and $\Phi_1:\T S_1\lra E$. In the integrable case $E=\AG$ it is convenient to think of $\Phi_0$ and $\Phi_1$ as of two smooth maps $\wt\Phi_0:S_0\lra \GG$ and $\wt\Phi_1:S_1\lra\GG$ reduced to $\AG$. Now we can formulate the \emph{relative OCP} by substituting in the problem \eqref{eqn:PPP} the algebroid homotopy class $[\sigma]$ by the relative algebroid homotopy class $[\sigma]\modulo(\Phi_0,\Phi_1)$. Here the relative class can be understood as a reduction of a homotopy in $\GG$, with end-points in the images $\Image\wt\Phi_0$ and $\Image\wt\Phi_1$, to the algebroid $\AG$.

For a solution of the problem described above we can repeat Theorem \ref{thm:PPP} with additional transversality conditions, namely, that the covectors $\xi(t_0)$ and $\xi(t_1)$ annihilate the images $\Image\Phi_0$ and $\Image\Phi_1$, respectively.

\subsection{Discussion of the main result}

The above result looks quite similar to the standard PMP. Indeed,
in the case $E=\T M$ we obtain the PMP. The only
difference is that the fixed-end-point boundary condition are
substituted by the fixed-homotopy-class condition. However,  this
makes no essential difference, as is discussed in detail in
Chapter \ref{ch:OCP}. For the case of an integrable Lie algebroid
$E=A(\GG)$ our version of the PMP can be understood as a general
reduction scheme for invariant OCPs on Lie groupoids. In
particular, the theorem covers the known results on Hamiltonian
reduction of Jurdjevic \cite{jurdjevic} and Martinez \cite{martinez_red_opt_ctr,martinez_lie_classs_mech},
and Lagrangian reduction \cite{cendra_holm_marsden} (see Chapter
\ref{ch:examples} for details). It is, however, worth mentioning
that in our approach the reduced and the unreduced PMPs are parts of
the same universal formalism. Roughly speaking, we have
generalised the geometrical context in which the PMP can be used.
The technical setting remains quite general---we work with bounded measurable controls and AC base trajectories.
Moreover, since AL algebroids do not come, in general, from
reductions, our result admits a wider spectrum of possible
applications. An attempt in this direction can be found in  the
last example of Chapter \ref{ch:examples}. Finally, note that a version of Theorem \ref{thm:PPP} for general boundary conditions admits arbitrary algebroid morphisms $\Phi_0:\T S_0\lra E$ and $\Phi_1:\T S_1\lra E$. In the integrable case $E=\AG$ these correspond to arbitrary smooth maps $\wt\Phi_0:S_0\lra \GG$ and $\wt\Phi_1:S_1\lra \GG$. On the other hand, in literature, when speaking about general boundary conditions one usually restricts attention to immersions only.

The original contributions of the author includes:
\begin{itemize}
\item  a detailed study of the notion of an algebroid homotopy and algebroid homotopy classes in Chapter \ref{ch:E_htp}:
\begin{itemize}
\item  The definition of the Lie algebroid homotopy appeared in \cite{crainic_fernandes}. It was given it terms of time-dependent algebroid sections and connections and though was not very intuitive. We extended the notion of the algebroid homotopy to almost Lie algebroids, extended it to measurable class, and reformulated the definition to emphasise the similarities with the standard notion of homotopy,

\item We introduced the notion of a relative algebroid homotopy class.
\item We gave a new interpretation of algebroid homotopies in terms of a Stokes-like formula and also extended the well-known interpretation of algebroid homotopies as reduced homotopies of a groupoid to the measurable class (Theorem \ref{thm:int_htp}).
\item We ask a question about existence and uniqueness of algebroid homotopies. A uniqueness result (Lemma \ref{lem:htp_unique}) is a simple consequence of certain results from the theory of differential equations. On the other hand, the existence is strongly connected with the axioms of AL algebroid. We prove Lemma \ref{lem:gen_E_htp} which states that only for AL algebroids every sufficiently regular one-parameter family of admissible paths generates an algebroid homotopy (for a given initial-point algebroid homotopy). This result, which in its infinitesimal and smooth version appeared earlier in \cite{GG_var_calc}, distinguishes AL algebroids from more general objects of similar nature (skew-algebroids or general algebroids). 

\item We study further properties of algebroid homotopies in Section \ref{sec:E_htp_prop}. These include Lemma \ref{lem:reparam}, which shows the behaviour of algebroid homotopy classes under reparametrisation, and Lemma \ref{lem:htp}, which compares algebroid homotopies with and without fixed end-points.
\end{itemize}
\item formulating the OCPs in the language of AL algebroids in Chapter \ref{ch:OCP}:
\begin{itemize}
\item We proposed to express boundary conditions of the OCPs in terms of algebroid homotopies.
\item We gave an interpretation of these new OCPs and, in particular, studied in detail their relation to standard OCPs. 
\item We proposed to express general boundary conditions in the OCP in terms of algebroid morphisms rater than submanifolds.
\end{itemize}
\item formulating a version of the PMP in the language of AL algebroids for fixed end-points and general boundary conditions (Theorems \ref{thm:pmp} and \ref{thm:pmp_rel});
\item proving these theorems in Chapters \ref{ch:needle}--\ref{ch:proof}:
\begin{itemize}
\item The proof, in principle, imitates the argument of Boltyanskii \cite{pontryagin}. There are, however, technical difficulties connected with using the language of AL algebroids. These appeared mostly in two places. In the proof of Theorem \ref{thm:1st_main} we used reparametrisation and composition of algebroid homotopies to study the impact of needle variations on the trajectories of a control system. In Lemma \ref{lem:htp_of_a_r} to prove the existence of an admissible path realising a certain algebroid homotopy class we had to pass through infinite-dimensional Banach spaces. The reason for that is the following: AL algebroids are, in general, not integrable, and hence homotopy classes cannot be represented by points on a finite dimensional manifold (they are just cosets in a big space of curves). Our idea was to semi-parametrise these classes by a finite-dimensional space and reduce the reasoning to a finite dimensional topological problem.   
\item Moreover, the standard proof \cite{pontryagin} like most of the other proofs in literature (perhaps apart from \cite{barbero_pmp}) contains smaller or greater gaps. We put much effort to explain all the details and make the reasoning self-contained. 
\end{itemize}
\end{itemize}
A significant part of this work is based on \cite{grabowski_jozwikowski_pmp}. In this article we concentrated only on OCPs with fixed-end-point boundary conditions. Therefore all parts concerning general boundary conditions, in particular Definition \ref{def:htp_class_rel}, Theorem \ref{thm:pmp_rel} and its proof including Lemma \ref{lem:htp_of_a_r1} and Section \ref{sec:proof_pmp_rel}, and parts concerning an interpretation of relative algebroid homotopies  and an interpretation of the OCP \eqref{eqn:P_rel} (in Chapter \ref{ch:OCP}), has not been published before. Also broad parts of \cite{grabowski_jozwikowski_pmp} were reformulated to make the argument more understandable.

\subsection{Organisation of the manuscript}

The first major part of this work is intended to give all important definitions and motivations which allow to define OCPs on an AL algebroid in Chapter \ref{ch:OCP} and finally state our main results in Chapter \ref{ch:pmp}. We start with a brief introduction of AL algebroids in Chapter \ref{ch:algebroids}. In Chapter \ref{ch:E_htp} we concentrate on algebroid homotopies which are crucial in our work. Much effort was made to give a satisfactory definition in both smooth and measurable setting, and later to motivate this definition, mainly by Lie groupoid---Lie algebroid reduction arguments. We also derive all properties of algebroid homotopies which will be used later in the proof of our main results. Finally, in Chapter \ref{ch:OCP}, we define and motivate algebroid OCPs. Much attention is payed to algebroid homotopies naturally associated with a control system on an AL algebroid. This leads to the notion of a parallel transport.

In Chapter \ref{ch:examples} we derive some known results on reduction in optimal control theory and variational calculus by means of our result. In particular we formulate the version of the PMP for invariant OCPs on principal bundles and use it to study the example of the falling cat problem of \cite{montgomery_isohol}. Some attention is payed to the problems of the calculus of variations on principal bundles. We obtain the results on Lagrangian reduction, Hammel equations and Euler--Poincar{\'e} equations as a special case. We also derive the generalised Euler-Lagrange equations on a general AL algebroid.

The second mayor part, consisting of Chapters \ref{ch:needle}--\ref{ch:proof}, contains the proof of Theorem \ref{thm:PPP}. In Chapter \ref{ch:needle} we define needle variations and a cone $\bm K_\tau^u$ of infinitesimal variations of the trajectory of the control system. The geometry of this cone is studied in detail in Chapter \ref{ch:proof}, using technical results proved in Chapter \ref{ch:main}. Then we can follow \cite{pontryagin} to derive the necessary conditions for optimality from the geometric properties of $\bm K_\tau^u$ along the optimal trajectory. 

Parallel to the proof of Theorem \ref{thm:PPP} we prove its version with general boundary conditions. Usually this requires just a minor modification of the arguments used. We decided to give two proofs in spite of the fact that  Theorem \ref{thm:PPP} is just a special case of the version with general boundary conditions. We believe that in this way the already complicated reasoning is easier to follow. Moreover, this is the typical way the proof of the PMP is presented in literature.

In the main part of this work we assume that the reader is familiar with basics of control theory, geometry of convex sets, topology, theory of ODEs in the sense of Carath{\'e}odory, and basics of the theory of Lie groupoids. However, the reader who is not  confident with these topics can find necessary information in Appendixes \ref{app:dif_geom}--\ref{app:geom_top} (we give the references when necessary). We believe that our presentation is self-contained. The Appendixes contain also some minor technical results which are used in the argument, yet their derivation in the main text would make the presentation less clear.

\chapter{Almost Lie algebroids}\label{ch:algebroids}

This chapter is concerned with some basic definitions and constructions from the theory of algebroids. We begin with the definition of a skew-algebroid and an almost Lie algebroid as a special case. Later we introduce the notion of a Hamiltonian vector field and the complete lift of an algebroid section. The characterisation of skew-algebroids in terms of exterior differential operators is used to define a morphism of algebroids. This, in turn, leads to the notion of an admissible path. We end this chapter with the construction of the product of two algebroids.

Let us note that many aspects of the theory of algebroids are not present in this introductory chapter. The interested reader should confront \cite{mackenzie,mackenzie_1987,GU_algebroids,weinstein_silva}.

\subsection{Differentiable manifolds and vector bundles}

In this work we use the following notation and conventions of differential geometry. By $M$ we denote a smooth $n$-dimensional manifold, by $\tau_M: \T M\lra M$ the tangent vector bundle, and by $\pi_M:
\T^\ast M\lra M$ the cotangent vector bundle of $M$. When passing to a local description we will use a coordinate system $(x^a)$, $a=1,\dots,n$ in $M$. We have the induced (adapted) coordinate systems $(x^a, {\dot x}^b)$ in $\T
M$ and $(x^a, p_b)$ in $\T^\ast M$.\index{local coordinates}
        
More generally, let $\tau: E \lra M$ be a vector bundle, and let $\pi: E^\ast \ra M$ be the dual bundle. Choose $(e_1,\dots,e_m)$ --- a basis of local sections of $\tau: E\ra M$, and let $(e^{1}_*,\dots, e^{m}_*)$ be the dual
basis of local sections of $\pi: E^\ast\lra M$. We have the
induced coordinate systems: $(x^a, y^i),  y^i=\iota(e^{i}_*)$ in $E$, and $(x^a, \xi_i), \xi_i = \iota(e_i)$ in $E^\ast$, where the linear functions  $\iota(e)$ are given by the canonical pairing $\iota(e)(v_x)=\< e(x),v_x>$.
The null section of $\tau:E\lra M$ will be denoted by $\theta$, and $\theta_x$ will stand for the null vector at point $x\in M$.

In this work the summation convention is assumed.

\subsection{Almost Lie algebroids}

\begin{definition} Let $M$ be a manifold and $\tau:E\lra M$ a vector bundle over $M$. A \emph{skew-algebroid structure}\index{skew-algebroid} on $E$ is a vector bundle morphism $\rho:E\lra \T M$ over $M$, called the \emph{anchor map}, and a skew-symmetric bilinear bracket $[\cdot,\cdot]:\Sec(E)\times_M\Sec(E)\lra\Sec(E)$\index{algebroid bracekt} on (local) sections of $\tau$, which satisfies the Leibniz rule\index{Leibniz rule}
\begin{equation}\label{eqn:lieb_rule}
[X,f\cdot Y]=f[X,Y]+\rho(X)(f)Y
\end{equation}
for every $X,Y\in\Sec(E)$ and $f\in C^\infty(M)$.

If, additionally, the anchor map is an algebroid morphism, i.e.,
\begin{equation}\label{eqn:ala}
\rho\left([X,Y]\right)=[\rho(X),\rho(Y)]_{\T M},
\end{equation}
we will speak of an \emph{almost Lie algebroid}\index{almost Lie algebroid} (\emph{AL algebroid}\index{AL algebroid|see{almost Lie algebroid}} briefly).

If, in addition to \eqref{eqn:lieb_rule} and \eqref{eqn:ala}, the
bracket  satisfies the Jacobi identity (in other words, the pair
$(\Sec(E)$,$[\cdot,\cdot])$ is a Lie algebra), we speak of a
\emph{Lie algebroid}\index{Lie algebroid}.
\end{definition}

In local coordinates $(x^a,y^i)$, introduced at the beginning of
this chapter the structure of an algebroid on $E$\index{skew-algebroid!local description} can be
described in terms of local function $\rho^a_i(x)$ and $c^i_{jk}(x)$ on
$M$ given by
$$\rho(e_i)=\rho^a_i(x)\pa_{x^a}\quad \text{and} \quad [e_i,e_j]=c^k_{ij}(x)e_k.$$
The skew-symmetry of the algebroid bracket results in the
skew-symmetry of $c^i_{jk}$ in lower indices, whereas condition
\eqref{eqn:ala} reads as
$$\left(\frac{\pa}{\pa
x^b}\rho^a_k(x)\right)\rho^b_j(x)-\left(\frac{\pa}{\pa
x^b}\rho^a_j(x)\right)\rho^b_k(x)=\rho^a_i(x)c^i_{jk}(x).$$

In the context of mechanics it is convenient to think about an
algebroid as a generalisation of the tangent bundle.
An element $a\in E$ has the interpretation of a generalized
velocity with actual velocity $v\in\sT M$ obtained by applying the
anchor map $v=\rho(a)$. The kernel of the anchor map represents
inner degrees of freedom.

A basic example of a skew-algebroid structure is the tangent bundle $\T
M$ of a manifold $M$ with the standard Lie bracket and $\rho=\id_{\T M}$. We will refer to this structure as to a \emph{tangent algebroid}\index{tangent algebroid}. Another natural example is a finite-dimensional real Lie algebra $\g$ considered as a vector bundle over a single point with its Lie bracket and the trivial anchor. 

Natural examples of skew-algebroids are associated with systems with symmetries. For instance, the Lie algebra $\g$ of a Lie group $G$ can be understood as a
reduction of the tangent bundle $\T G$ by the left (or right)
action of $G$. Similarly, for a principal bundle $G\lra
P\lra M$, the reduced bundle $\T P/G\lra M$ has the structure of
an \emph{Atiyah algebroid}\index{Atiyah algebroid}. The Atiyah algebroid is a common generalisation of $\T M$ and $\g$. This example is discussed in more details in Appendix \ref{sapp:atiyah}. More generally, every Lie groupoid $\GG$ has an associated Lie algebroid $\mathcal{A}(\GG)$\index{Lie algebroid!of a Lie groupoid} which can be interpreted as a reduction of a subbundle of the tangent algebroid $\T \GG$ by the right (or left) action of $\GG$. This example is discussed in Appendix \ref{sapp:groupoids}. 

All the above are examples of Lie algebroids. Natural examples of skew-algebroids which are not Lie can be associated with nonholonomically constrained mechanical systems \cite{grabowski_nonholonomic}. 

\subsection{Hamiltonian vector fields and tangent lifts} Let us now describe some geometric constructions
associated with the structure of a skew-algebroid .

It can be shown (cf. \cite{GU_algebroids,GU_poiss_nijn}) that the presence of the
structure of a skew-algebroid on $E$ is equivalent to the
existence of a linear bivector field $\Pi_{E^\ast}$ on $E^*$\index{skew-algebroid!as a linear bi-vector}. In local
coordinates, $(x^a,\xi_i)$ on $E^\ast$, it is given by
\begin{equation}\label{eqn:poisson}
 \Pi_{E^\ast} =c^k_{ij}(x)\xi_k
\partial _{\xi_i}\wedge \partial _{\xi_j} + \rho^b_i(x) \partial _{\xi_i}
\wedge \partial _{x^b}.
\end{equation}
The linearity of $\Pi_{E^\ast}$ means that the corresponding mapping
$\wt\Pi:\sT^\ast E^\ast\lra\sT E^\ast$ is a morphism of double
vector bundles (cf. \cite{KU_dvb, GR_higher}). The tensor $\Pi_{E^\ast}$ is well
recognised in the standard situations: for the tangent algebroid
structure on $\T M$, it is the canonical Poisson structure on
$\T^\ast M$ dual to the canonical symplectic structure,
whereas for a Lie algebra $\g$, it is the
Lie--Poisson structure $\Pi_{\g^\ast}$ on $\g^\ast$. Actually, $E$ is a Lie
algebroid if and only if $\Pi_{E^\ast}$ is a Poisson tensor.

Now we can introduce the notion of a Hamiltonian vector field on
$E^\ast$. Let, namely, $h:E^\ast\lra\R$ be any $C^1$-function. We
define the \emph{Hamiltonian vector field}\index{Hamiltonian vector field} $\X_h$ in an obvious
way: $\X_h=\iota_{\dd h}\Pi_{E^\ast}$. In local coordinates,
\begin{equation}\label{eqn:ham_vf}
\X_h(x,\xi)=\rho^a_i(x) \frac{\pa h}{\pa
{\xi_i}}(x,\xi)\pa_{x^a}+\left(c^k_{ji}(x)\xi_k\frac{\pa h}
{\partial {\xi_j}}(x,\xi)- \rho^a_i(x)\frac{\pa h}{\pa
{x^a}}(x,\xi)\right) \pa_{\xi_i}.
\end{equation}

Another geometrical construction in the skew-algebroid setting is
the \emph{complete lift of an algebroid section}\index{complete lift} (cf.
\cite{GU_algebroids,GU_poiss_nijn}). For every $C^1$-section $X=f^i(x)e_i\in\Sec(E)$
we can construct canonically a vector field $\dd_\T(X)\in\Sec(\T
E)$ which, in local coordinates, reads as
\begin{equation}\label{eqn:tan_lift}
\dd_\T(X)(x,y) = f^i(x)\rho^a_i(x)\pa_{x^a} + \left( y^i
\rho^a_i(x) \frac{\pa f^k}{\pa x^a}(x) + c^k_{ij}(x)y^if^j(x)
\right) \pa _{y^k}.
\end{equation}
The vector field $\dd_\T(X)$ is linear w.r.t. the vector bundle structure $\T\tau:\T E\ra\T M$ (the above equation is linear w.r.t. $y^i$).

Consider the Hamiltonian vector field $\X_{\iota(X)}$ associated
with a linear function $\iota(X)(\cdot)=\<X,\cdot>_\tau$ on
$E^*$. It turns out that fields $\X_{\iota(X)}$ and $\dd_\sT(X)$
are related by
\begin{equation}\label{eqn:hvf_tgl}
\<\dd_\sT(X),\X_{\iota(X)}>_{\sT\tau}=0,
\end{equation}
where $\<\cdot,\cdot>_{\sT\tau}:\sT E\times_{\sT M}\sT
E^\ast\lra\R$ is the canonical pairing, being the tangent map of
$\<\cdot,\cdot>_\tau:E\times_M E^\ast\lra\R$ (in local coordinates,
$\<(x,y,\dot x,\dot y),(x,\xi,\dot x,\dot\xi)>_{\sT\tau}=\dot
y^j\xi_j+y^j\dot \xi_j$).

\subsection{Cartan Calculus}
The existence of a skew-algebroid structure on $E$ is equivalent
to the existence of exterior differential (de Rham) operators\index{exterior differential}
$$d_E:\Sec(\Lambda^kE^*)\lra\Sec(\Lambda^{k+1}E^*),\quad k=0,1,\dots\,,$$
defined by a straightforward generalisation of the Cartan formula
\begin{equation*}
\begin{split}d_E\omega (a_0,a_1,\hdots,a_k)=\sum_{i=0}^k(-1)^i\rho(a_i)\omega(a_0,\hdots,\check{a}_i,\hdots,a_k)
\\ +\sum_{1\leq i<j\leq
k}(-1)^{i+j}\omega\left([a_i,a_j],a_0,\hdots,\check{a}_i,\hdots,\check{a}_j,\hdots,a_k\right),
\end{split}
\end{equation*}
for $\omega\in\Sec(\Lambda^kE^*)$ and
$a_0,a_1,\hdots,a_k\in\Sec(E)$.

These operators, in general, needs not be cohomological. In fact,
$E$ is a Lie algebroid if and only if $d_E^2=0$. AL algebroids, in
turn, can be characterized by the condition that $d_E^2f=0$ for
every $f\in C^\infty(M)=\Sec(\Lambda^0E^*)$.

\subsection{Morphisms}
The above concept of the de Rham derivative allows one to give
a simple definition of a morphism of skew-algebroids. Namely,
given skew-algebroids $\wt\tau:\wt E\lra\wt M$ and $\tau:E\lra M$,
a bundle map $\Phi:\wt E\lra E$ over $\varphi:\wt M\lra M$ is a
\emph{skew-algebroid morphism}\index{morphism of skew-algebroids} if it is compatible with the exterior derivative:
\begin{equation}\label{eqn:E_morph_forms}
\Phi^*d_E\theta=d_{\wt E}\Phi^*\theta,\quad \text{for every
$\theta\in\Sec(\Lambda^kE^*)$}.
\end{equation}
Note that a vector bundle map $\Phi$ does not, in general, induce
any map on sections of $\wt E$, while the pull-back $\Phi^*$ of
sections of $E^*$ is always well defined.

Introduce local coordinates $(\wt x^\alpha,\wt y^\iota)$ and
$(x^a,y^i)$ and structure functions $\wt\rho^\alpha_\iota(\wt x)$, $\wt
c^\iota_{\kappa\mu}(\wt x)$ and $\rho^a_i(x)$, $c^i_{km}(x)$ on $\wt E$ and
$E$, respectively. The condition that
$\Phi\sim(\Phi^i_\iota,\varphi^a)$ is an algebroid morphism reads as
\begin{equation}\label{eqn:alg_morph}
\begin{split}&\Phi^i_\kappa(\wt x)\rho^a_i(\varphi( x))=\wt \rho^\alpha_\kappa(\wt x)\frac{\pa\varphi^a(\wt x)}{\pa \wt x^\alpha},\\
&\wt\rho^\alpha_\kappa(\wt x)\frac{\pa\Phi^i_\lambda(\wt x)}{\pa \wt
x^\alpha}-\wt\rho^\alpha_\lambda(\wt x)\frac{\pa\Phi^i_\kappa(\wt
x)}{\pa \wt x^\alpha}=c^i_{jk}\left(\varphi(\wt
x)\right)\Phi^j_\kappa(\wt x)\Phi^k_\lambda(\wt x).
\end{split}
\end{equation}

\subsection{Admissible paths} Consider an algebroid morphism $\sT\R|_I\lra E$, where $I=[t_0,t_1]\subset\R$ is an interval.
Every such map is uniquely determined by the image of the
canonical section $(t,\pa_t)$ of $\sT\R$ being a smooth curve
$a(t)$ in $E$ over the base path $x(t)$ in $M$. Condition
\eqref{eqn:E_morph_forms} reads as
\begin{equation}\label{eqn:adm}
\rho\left(a(t)\right)=\dot{x}(t) \quad \text{for every $t\in I$}.
\end{equation}
This means that the anchor map coincides with the tangent
prolongation of the projection $x(t)=\tau\left(a(t)\right)$. The
curves which satisfy \eqref{eqn:adm} will be called
\emph{admissible}\index{admissible path}\index{E-path|see{admissible path}}. In fact, \eqref{eqn:adm} also makes sense for non-smooth maps. From now on, by an \emph{admissible path on $E$} (or briefly \emph{$E$-path}) we shall mean a bounded measurable map $a:I\lra E$ over an absolutely continuous (AC) base
path $x=\tau\circ a:I\lra M$ such that \eqref{eqn:adm} is
satisfied a.e. in $I$. In such a case we will speak
of \emph{measurable $E$-paths}. For more information on measurable
functions see Section \ref{sec:meas}. Observe that from \eqref{eqn:alg_morph} it follows that a morphism of algebroids maps admissible paths into admissible paths.

To explain the meaning of admissible curves, observe that in the case
of the tangent algebroid $\T M$  admissible curves are precisely
the tangent lifts of base curves. We will show later (cf. Theorem
\ref{thm:int_htp}) that if an algebroid $E$ is integrable,
admissible curves come from a reduction of real curves in a Lie
groupoid integrating $E$.

Finally, we can introduce the concept of \emph{composition of
measurable $E$-paths}\index{composition of admissible paths}. Let $a:[t_0,t_1]\lra E$ and $\ol
a:[t_1,t_2]\lra E$ be two measurable $E$-paths with base paths
$x=\tau\circ a$ and $\ol x=\tau\circ\ol a$, respectively. Assume
that $x(t_1)=\ol x(t_1)$ (such paths will be called
\emph{composable}\index{composable admissible paths}). Clearly, the map $\wt a:[t_0,t_2]\lra E$
defined by
$$\wt a(t):=\begin{cases}
a(t) &\text{for $t\leq t_1$},\\
\ol a(t) &\text{for $t>t_1$}
\end{cases}$$
is another measurable $E$-path covering the AC curve
$$\wt x(t):=\begin{cases}
x(t) &\text{for $t\leq t_1$},\\
\ol x(t) &\text{for $t>t_1$}.
\end{cases}$$
This new $E$-path will be called the \emph{composition} of $a$ and
$\ol a$ and will be denoted by $\wt a=a\circ \ol a$.

\subsection{The product of skew-algebroids}
\index{product of skew-algebroids}
Given two skew-algebroids $(\tau_1:E_1\lra
M_1,\rho_1,[\cdot,\cdot]_1)$ and $(\tau_1:E_2\lra
M_2,\rho_2,[\cdot,\cdot]_2)$ we can define a skew-algebroid
structure on the product bundle $\tau=\tau_1\times\tau_2:E_1\times
E_2\lra M_1\times M_2$. The anchor will simply be
$\rho=\rho_1\times\rho_2:E_1\times E_2\lra\T M_1\times\T
M_2\approx \T(M_1\times M_2)$. The bracket can be defined by
equalities
\begin{align*}
\left[\p_1^\ast X_1,\p_1^\ast Y_1\right]&=\p_1^\ast [X_1,Y_1]_1,\\
\left[\p_2^\ast X_2,\p_2^\ast Y_2\right]&=\p_2^\ast [X_2,Y_2]_2,\\
\left[\p_1^\ast X_1,\p_2^\ast Y_2\right]&=\theta,
\end{align*}
where $X_1,Y_1\in\Sec(E_1)$ and $X_2,Y_2\in\Sec(E_2)$ are sections,  $\p_1:E_1\times E_2\lra E_1$ and $\p_2:E_1\times E_2\lra E_2$ are canonical vector bundle projections, and $\theta$ is a null section of $\tau$. The above equalities can be extended to arbitrary
sections by linearity and the Leibniz rule \eqref{eqn:lieb_rule}.
Clearly, the canonical projections $E_1\times E_2\lra E_i$, with
$i=1,2$, are algebroid morphisms, and if $E_1$ and $E_2$ are almost
Lie, then so is their product.

The local coordinate description of the product $E_1\times E_2$ is
very simple. If $(x^a,y^i)$ and $(\wt x^\alpha,\wt y^\iota)$ are
local coordinates on $E_1$ and $E_2$, respectively, we can
introduce natural coordinates $(X^A,Y^I)=(x^a,\wt x^\alpha,
y^i,\wt y^\iota)$ on $E_1\times E_2$. The structure functions
$C^I_{JK}(X)$ and $R^A_I(X)$ in these coordinates are trivial on
mixed-type terms ($R^\alpha_i=C^\iota_{j\kappa}=0$, etc.) and the
same as the structure functions of $E_1$ and $E_2$ on simple-type
terms ($C^i_{jk}(x,\wt x)=c^i_{jk}(x)$, $R^\alpha_\iota(x,\wt
x)=\rho^\alpha_\iota(\wt x)$, etc.).

\chapter{Homotopies of admissible paths}\label{ch:E_htp}

The notion of the homotopy of $E$-paths ($E$-homotopy) is crucial in this work. It will be our main tool to define the boundary conditions for optimal control problem on AL algebroids in Chapter \ref{ch:OCP}. In this chapter we give a definition of $E$-homotopy and study its basic properties. 

In the first section much attention is put on interpreting $E$-homotopy. In particular, if $E= \mathcal{A}(\GG)$ is a Lie algebroid of a Lie groupoid $\GG$, we show an equivalence of $E$-homotopies and true homotopies in $\GG$ (Theorem \ref{thm:int_htp}). We also interpret $E$-homotopies by a Stokes-like formula. Finally, we introduce an important notion of $E$-homotopy classes and $E$-homotopy classes relative to a pair of algebroid morphisms.

In the second section we study properties of $E$-homotopies. We prove an important Lemma \ref{lem:gen_E_htp} which states that AL algebroids are characterised by the property that a one-parameter family of $E$-paths establishes an $E$-homotopy. This result explains why AL algebroids are objects of our primary interest rather than a smaller class of Lie algebroids or a more general class of skew-algebroids. Later we prove Lemma \ref{lem:htp} which compares $E$-homotopies with and without fixed end-points. Finally, the behaviour of $E$-homotopy classes under reparametrisation (Lemma \ref{lem:reparam}) is studied. 

\section{The \texorpdfstring{$E$}{E}-homotopy and its interpretation}

The definition of $E$-homotopy will be given in
two steps. First, we will describe the smooth case and later generalise the
concept to measurable $E$-paths, more suitable in control theory. The rest of this section is concerned with giving a convincing motivation and interpretation of the $E$-homotopy. We give an interpretation in therms of a Stokes-like formula, and later show that $E$-homotopies on an integrable algebroid $\AG$ correspond to true homotopies in the groupoid $\GG$ integrating $\AG$. The last interpretation will be crucial in Chapter \ref{ch:OCP} to motivate the definition of an optimal control problem on an AL algebroid.

\begin{definition} Let $a_0,a_1:I=[t_0,t_1]\lra E$ be two smooth admissible paths.
An \emph{algebroid homotopy in $E$}\index{algebroid homotopy} (or \emph{$E$-homotopy} briefly) between $a_0$ and $a_1$ is a
pair of maps $a,b:[t_0,t_1]\times[0,1]\lra E$, over the same base map
$x:[t_0,t_1]\times[0,1]\lra M$, with $a_0(\cdot)=a(\cdot,0)$ and
$a_1(\cdot)=a(\cdot,1)$, such that
\begin{subequations}
\begin{align}
&&\label{eqn:a_adm} t\mapsto a(t,s) &\quad\text{is admissible for every $s\in[0,1]$}\,, \\
&&\label{eqn:b_adm} s\mapsto b(t,s) &\quad\text{is admissible for
every $t\in [t_0,t_1]$,}
\end{align}
\end{subequations}
and, moreover, $a$ and $b$ satisfy a system of
differential equations given in local coordinates $(x^a,y^i)$ in
$E$ by
\begin{equation}\label{eqn:htp_smooth}
\pa_tb^i(t,s)-\pa_sa^i(t,s)=c^i_{jk}(x(t,s))b^j(t,s)a^k(t,s).
\end{equation}
The $E$-paths $b_0(s):=b(t_0,s)$ and $b_1(s):=b(t_1,s)$ will be
called \emph{initial-point} and \emph{final-point $E$-homotopies}\index{initial-point E-homotopy}\index{final-point E-homotopy},
respectively. We will say that $E$-homotopy $(a,b)$ has \emph{fixed
end-points}\index{algebroid homotopy!with fixed end-points} if $b_0\equiv \theta_{x(t_0)}$ and $b_1\equiv
\theta_{x(t_1)}$.
\end{definition}

Having in mind that admissible paths are in a 1-1 correspondence with algebroid morphisms $A:\T\R|_I\lra E$ (an admissible path is the $A$-image of the canonical section $(t,\pa_t)$ of $\T\R$ --- compare Chapter \ref{ch:algebroids}), we may define $E$-homotopy in an equivalent way. An $E$-homotopy between two smooth alegebroid morphisms $A_0,A_1:\sT\R|_I\lra E$ (corresponding to $E$-paths $a_0$ and $a_1$) is an algebroid morphism 
$$H:\sT\R|_I\times\sT\R|_{[0,1]}\lra E,$$
such that $A_0(\cdot)=H(\cdot,\theta_0)$ and
$A_1(\cdot)=H(\cdot,\theta_1)$, where $\theta_0\in\sT_0\R$ and
$\theta_1\in\sT_1\R$ are null vectors. The equivalence with the previous definition can be seen as follows. The map $H$ is determined by the images of two canonical sections $(\pa_t,\theta_s)$ and $(\theta_t,\pa_s)$. We can define $a(t,s):=H(\pa_t,\theta_s)$ and $b(t,s):=H(\theta_t,\pa_s)$. Now conditions \eqref{eqn:alg_morph} for $H$ to be an algebroid morphism, translated to the language of $a$ and $b$, are precisely \eqref{eqn:a_adm}, \eqref{eqn:b_adm} and \eqref{eqn:htp_smooth}. 

\noindent Note that this alternative formulation agrees with the notion of the
homotopy of Lie algebroid morphisms as introduced by 
\cite{Kubarski}.

The notion of an $E$-homotopy can be also extended to
measurable setting.

\begin{definition}\label{def:htp_meas}
Consider two bounded measurable admissible paths $a_0,a_1:[t_0,t_1]\lra E$. 
An \emph{algebroid homotopy in $E$}\index{algebroid homotopy|main}\index{E-homotopy|see{algebroid homotopy}} (or \emph{$E$-homotopy} briefly) between $a_0$ and $a_1$ is a
pair of bounded measurable maps (w.r.t. both variables
separately) $a,b:[t_0,t_1]\times[0,1]\lra E$, over the same ACB base map
$x:[t_0,t_1]\times[0,1]\lra M$, such that
\begin{itemize}
\item  $a_0(\cdot)=a(\cdot,0)$ and
$a_1(\cdot)=a(\cdot,1)$ are well-defined trace values,
\item \eqref{eqn:a_adm} and \eqref{eqn:b_adm} hold in a measurable sense,
\item the pair $(a,b)$ is a weak solution of \eqref{eqn:htp_smooth} with a well-defined trace (see Appendix \ref{sapp:pde}), that is 
\begin{align}\label{eqn:htp_weak}
\begin{split}
&\iint_{I\times[0,1]}\Big[b^i(t,s)\pa_t\psi_i(t,s)- a^i(t,s)\pa_s\psi_i(t,s)+c^i_{jk}(x(t,s))b^j(t,s)a^k(t,s)\psi_i(t,s)\Big]\dd t\dd s\\
&=\int_{[0,1]}\Big[b^i(0,s)\psi_i(0,s)-b^i(1,s)\psi_i(1,s)\Big]\dd s- \int_I\Big[a^i(t,0)\psi_i(t,0)-a^i(t,1)\psi_i(t,1)\Big]\dd t
\end{split}
\end{align}
holds for every family of functions $\psi_i\in
C^\infty(I\times[0,1];\R)$.
\end{itemize}
Note that considering only W-solutions
of \eqref{eqn:htp_smooth} would not be enough, since otherwise the
boundary paths $a_0(t)$, $a_1(t)$, $b_0(s)$, and $b_1(s)$ would not
be well defined. For more information on W- and WT-solutions see Appendix \ref{sapp:pde}.  The notion of the \emph{initial-point} and the
\emph{final-point $E$-homotopy}, as well as the \emph{$E$-homotopy with fixed
end-points}, also remains valid in this new setting.
\end{definition}

From now on, by an $E$-homotopy we will mean a pair of maps $(a,b)$ in the sense of Definition \ref{def:htp_meas}.

Observe that any two measurable maps
$a,b:I\times[0,1]\lra E$ over the same AC base map
$x:I\times[0,1]\lra M$ define a measurable bundle map
$H:\sT\R|_I\times\sT\R|_{[0,1]}\lra E$ (that is, a measurable map
linear on fibers), where $H(\pa_t,\theta_s)=a(t,s)$ and
$H(\theta_t,\pa_s)=b(t,s)$.


Like in the case of $E$-paths, also 
$E$-homotopies allow a natural notion of \emph{composition}\index{composition of algebroid homotopies}. Let, namely, $a,b:I\times[0,1]\lra E$ over $x$, and
$\ol a,\ol{b}:J\times[0,1]\lra E$ over $\ol{x}$ (where
$I=[t_0,t_1]$ and $J=[t_1,t_2]$) be two $E$-homotopies. Assume
that the final-point $E$-homotopy of the first and the
initial-point $E$-homotopy of the second coincide; i.e.,
$b(t_1,s)=\ol b(t_1,s)$ a.e. (hence $x(t_1,s)=\ol x(t_1,s)$, so
$a(\cdot,s)$ and $\ol{a}(\cdot,s)$ are composable for every
$s\in[0,1]$). The maps $\wt a,\wt b:I\cup J\times[0,1]\lra E$
defined as 
$$\wt a(t,s)=\begin{cases}
a(t,s)& \text{for $t\leq t_1$},\\
\ol a(t,s)& \text{for $t>t_1$}
\end{cases}$$
and
$$\wt b(t,s)=\begin{cases}
b(t,s)& \text{for $t\leq t_1$},\\
\ol b(t,s)& \text{for $t>t_1$}
\end{cases}$$
clearly form an $E$-homotopy joining $a(\cdot,0)\circ \ol
a(\cdot,0)$ and $a(\cdot,1)\circ \ol a(\cdot,1)$. The
initial-point $E$-homotopy is $b(t_0,\cdot)$, while the
final-point $E$-homotopy is $\ol b(t_2,\cdot)$.

\subsection{Uniqueness of \texorpdfstring{$E$}{E}-homotopies}
As a direct consequence of the definition of an
$E$-homotopy and Lemma \ref{lem:wt_uniq} we get the following
result.

\begin{lemma}[uniqueness of $E$-homotopies]\index{algebroid homotopy!uniqueness}\label{lem:htp_unique} Let $a:I\times[0,1]\lra E$ be a bounded
measurable map covering $x:I\times[0,1]\lra M$ such that $t\mapsto
a(t,s)$ is admissible for every $s$. Then there exists at most
one bounded measurable map $b:I\times[0,1]\lra E$ covering $x$
such that $(a,b)$ is an $E$-homotopy with a given initial-point
$E$-homotopy $b(t_0,s)=b_0(s)$.
\end{lemma}

\subsection{The \texorpdfstring{$E$}{E}-homotopy via Stokes theorem} 

We shall now give another, more geometrical, description of an
$E$-homotopy by means of a Stokes-like formula. First, we
will introduce the notion of an integral of an $E$-$k$-form, i.e.,
an element $\omega\in\Sec(\Lambda^kE^\ast)$, over a bundle
morphism $\Phi:\sT N\lra E$. We define
$$\int_{\Phi(N)}\omega:=\int_N\Phi^\ast\omega\,,$$
where the last integral is the standard integration of the
differential $k$-form $\Phi^\ast\omega$ on the manifold $N$. Now,
if $N$ is a manifold with boundary $\pa N$, we define
$$\int_{\pa\Phi(N)}\omega:=\int_{\pa N}\Phi^\ast\omega.$$
Observe that in case $\Phi:\sT N\lra\sT M$ is the tangent lift of
a diffeomorphism $\varphi:N\lra M$, the above definitions coincide
with the standard concept of differential form integration. The
morphism $\Phi$ need not be differentiable. Since, given local
coordinates $N\supset V\overset \psi\lra
V^{'}\subset\R^n\ni(y^1,\hdots,y^n)$ on $N$,
$$\int_V\Phi^*\omega=\int_{V^{'}\subset\R^n}\omega\left(\Phi(\pa_{y^1}),\hdots,\Phi(\pa_{y^n})\right)\dd y^1\cdots\dd y^n,$$
we shall require only that $\Phi$ maps smooth sections of $\sT N$
into bounded measurable sections of $E$.

Now assume that $\Phi:\sT\R|_I\times\sT\R|_{[0,1]}\lra E$ over
$\varphi:\R|_I\times\R|_{[0,1]}\lra M$ is a bundle map defined by
means of measurable maps $a(t,s)$ and $b(t,s)$ as in the
definition of an $E$-homotopy. Assume, moreover, that conditions
\eqref{eqn:a_adm} and \eqref{eqn:b_adm} are satisfied. Take any
$E$-1-form $\alpha\in\Sec(E^\ast)$; in local coordinates,
$\alpha\sim(x^a,\alpha_i(x))$. Now
\begin{align*}
\int_\Phi\dd_E\alpha&=\iint_{I\times[0,1]}\dd_E\alpha\left(\Phi(\pa_t),\Phi(\pa_s)\right)\dd
t\dd s=
\iint_{I\times[0,1]}\dd_E\alpha\left(a(t,s),b(t,s)\right)\dd t\dd s=\\
&=\iint_{I\times[0,1]}\left(\rho^a_i(x)a^i\frac{\pa\alpha_j}{\pa
x^a}b_j-\rho^a_i(x)b^i\frac{\pa\alpha_j}{\pa
x^a}a_j-\alpha_ic^i_{jk}(x)a^jb^k\right)\dd t\dd s.
\end{align*}
Having in mind that $\rho^a_i(x(t,s))a^i=\pa_tx^a(t,s)$ and
$\rho^a_i(x(t,s))b^i=\pa_sx^a(t,s)$, and defining $\wt
\alpha_i(t,s):=\alpha_i(x(t,s))$, we get
$$\int_\Phi\dd_E\alpha=\iint_{I\times[0,1]}\left(b^j\pa_t\wt\alpha_j-a^j\pa_s\wt\alpha_j-\wt\alpha_ic^i_{jk}a^jb^k\right)\dd t\dd s.$$
Similarly,
\begin{eqnarray*}\int_{\pa\Phi}\alpha &=&\int_I\left(\wt\alpha_i(t,0)a^i(t,0)-\wt\alpha_i(t,1)a^i(t,1)\right)\dd t\\&&-
\int_{[0,1]}\left(\wt\alpha_i(0,s)b^i(0,s)-\wt\alpha_i(1,s)a^i(1,s)\right)\dd s.
\end{eqnarray*}
As we see, \eqref{eqn:htp_weak} holds for all $\wt\alpha_i$
if and only if
\begin{equation}\label{eqn:stokes}
\int_\Phi\dd_E\alpha=\int_{\pa\Phi}\alpha,
\end{equation}
which can be understood as a generalized Stokes formula\index{Stokes theorem}.

\begin{remark}
In fact, \eqref{eqn:htp_weak} is more general than
\eqref{eqn:stokes} since $\wt\alpha_i$ being the pull-back of
$\alpha$ via the map $\Phi$ cannot be an arbitrary function of $t$
and $s$. We can, however, easily overcome this drawback by using
the graph of $\Phi$ in $\sT\R|_I\times\sT\R|_{[0,1]}\times E$ and
$(\sT\R|_I\times\sT\R|_{[0,1]}\times E)$-1-forms instead of $\Phi$
and $E$-1-forms.
\end{remark}\medskip

\begin{remark}
The condition \eqref{eqn:a_adm} for $a$ (and analogously
\eqref{eqn:b_adm} for $b$) can be expressed in the Stokes-like way
as well. Consider, namely, the map
$\Phi_s(\cdot):=\Phi(\cdot,\theta_s):\sT\R|_I\lra E$. The
admissibility of $a$ reads as
$$\int_{\Phi_s}\dd_Ef=\int_{\pa\Phi_s}f$$
for every $f\in C^\infty(M)$ and $s\in[0,1]$.
\end{remark}\medskip

\subsection{$E$-homotopies on an integrable algebroid}

Now we show that admissible paths and algebroid homotopies on an integrable algebroid $\AG$ are true paths and true homotopies on an integrating groupoid $\GG$ reduced to $\AG$ by means of the reduction map \eqref{eqn:reduction}. We are using the language of Lie groupoids, so the reader unfamiliar with this topic should consult Appendix \ref{sapp:groupoids}.

\begin{theorem}[integration]\label{thm:int_htp}\index{admissible path!integration}\index{algebroid homotopy!integration}
Let $A(\GG)\ra M$ be a Lie algebroid of a Lie groupoid
$\GG$. Fix $x_0,y_0\in M$ and an element $g_0\in\alpha^{-1}(y_0)\cap\beta^{-1}(x_0)$.

There is a 1-1 correspondence between:
\begin{itemize}
\item bounded measurable
admissible paths $a:[t_0,t_1]\lra \AG$ over an ACB path $x:[t_0,t_1]\lra M$ such that $x(t_0)=x_0$, and
\item ACB paths $g:[t_0,t_1]\lra \GG_{y_0}$ such that $g(t_0)=g_0$ and  $x(t)=\beta(g(t))$. 
\end{itemize} The correspondence is given by means of the reduction map \eqref{eqn:reduction}; i.e., $a(t)=\mathcal{R}(\pa_t g(t))=\T R_{g(t)^{-1}}(\pa_t g(t))$. 

Similarly, there is a 1-1 correspondence between:
\begin{itemize}
\item bounded
measurable algebroid homotopies $a,b:[t_0,t_1]\times[0,1]\lra \AG$ over an ACB map $x:[t_0,t_1]\times[0,1]\lra M$ such that $x(t_0,0)=x_0$, and
\item ACB homotopies $h:[t_0,t_1]\times [0,1]\lra \GG_{y_0}$ (i.e., $h$ is ACB w.r.t. both variables) such that $h(t_0,0)=g_0$ and  $x(t,s)=\beta(h(t,s))$. 
\end{itemize} 
Again, the correspondence is given by means of the reduction map \eqref{eqn:reduction}; i.e., $a(t,s)=\mathcal{R}(\pa_th(t,s))=\T R_{h(t,s)^{-1}}(\pa_t h(t,s))$ and $b(t,s)=\mathcal{R}(\pa_sh(t,s))=\T R_{h(t,s)^{-1}}(\pa_s h(t,s))$.  
\end{theorem}

\begin{proof}
In the smooth case the result can be easily derived from Theorem \ref{thm:Lie1} and Corollary \ref{cor:int}. We know that (cf. considerations on page \pageref{cor:int}) smooth admissible paths and smooth algebroid homotopies in $\AG$ correspond to algebroid morphisms   $\T\R|_{[t_0,t_1]}\lra\AG$ and $\T\R|_{[t_0,t_1]}\times\T\R|_{[0,1]}\lra\AG$. The underlying manifolds are simply connected, hence the morphisms can be lifted to smooth maps $g:[t_0,t_1]\lra\GG_{y_0}$ and $h:[t_0,t_1]\times[0,1]\lra\GG_{y_0}$, which are unique up to the choice of the initial points (cf. Corollary \ref{cor:int}). 

In the measurable case, however, the argument needs a little more attention. 
We will work in local coordinates $(z^i)$ on $\GG_{y_0}$, $(x^a)$
on $M$, and linear coordinates $(x^a,y^i)$ on
$\AG$. We have induced coordinates $(z^i,\dot z^j)$ on $\T\GG_{y_0}$ and
$(x^a,\dot x^b)$ on $\T M$. 

For $g\in\GG_{y_0}$, $\T R_{g^{-1}}$ maps $\T_g\GG_{y_0}=\T^\alpha_g\GG$ isomorphically into $\AG_{\beta(g)}$. In coordinates, $\T R_{g^{-1}}:(z^i,\dot z^j)\mapsto(x^a,y^i)$ can be expressed as 
\begin{align*}
&x^a=\beta^a(z),\\
&y^i=F^i_j(z)\dot z^j,
\end{align*}
where $\beta^a(z)$ and  $F^i_j(z)$ are smooth and $F^i_j(z)$ is invertible. By $f^j_i(z)$ we will denote the inverse matrix of $F^i_j(z)$. The structure functions of the algebroid $\AG$ in these coordinates satisfy
\begin{align*}
&\rho^a_i(\beta(z))F^i_j(z)\dot z^j=\frac{\pa \beta^a(z)}{\pa z^j}\dot z^j,\\
&c^i_{jk}(\beta(z))F^j_m(z)F^k_n(z)\dot z^m\dot
z^n=\left(\frac{\pa F^i_n(z)}{\pa z_m}-\frac{\pa F^i_m(z)}{\pa
z_n}\right)\dot z^m\dot z^n,
\end{align*}
since $\rho$ is the reduced $\T\beta$, and the $\AG$-bracket is the reduced Lie bracket on $\GG$. From the above we get
\begin{align*}
&\rho^a_i(\beta(z))=\frac{\pa \beta^a(z)}{\pa z^j}f^j_i(z),\\
&c^i_{jk}(\beta(z))=\left(\frac{\pa F^i_n(z)}{\pa z_m}-\frac{\pa
F^i_m(z)}{\pa z_n}\right)f^m_j(z)f^n_k(z).
\end{align*}

To prove the first part of the assertion, observe that, if $g:[t_0,t_1]\lra\GG_{y_0}$ is an ACB path over an ACB path $x:[t_0,t_1]\lra M$, then the derivative $\pa_t g(t)\in\T_{g(t)}\GG_{y_0}$ is a bounded measurable
path, and so is  $a(t)=\mathcal{R}(\pa_t g(t))=\T R_{g(t)^{-1}}(\pa_t g(t))$, since $\mathcal{R}$ is smooth. Clearly,  $\dot x(t)=\T\beta(\pa_t g(t))=\rho(a(t))$ (cf. diagram \eqref{eqn:reduction}), so $a(t)$ is
a bounded measurable $\AG$-path.

Conversely, consider a bounded measurable admissible path $a:[t_0,t_1]\lra \AG$
over an ACB path $x:[t_0,t_1]\lra M$. For every $t\in[t_0,t_1]$ and all $g$ satisfying $\beta(g)=x(t)$ we may lift $a(t)\in \AG_{x(t)}$ to a vector $A(t,g):=\T R_g(a(t))\in\T_g\GG_{y_0}$. We would like to define $g(t)$ as a solution of the differential equation in $\GG_{y_0}$
$$\pa_t g(t)=A(g,t)$$
with the initial condition $g(t_0)=g_0$. Then, clearly, $\T R_{g(t)^{-1}}(\pa_t g(t))=\T R_{g(t)^{-1}}A(g(t),t)=a(t)$ as in the assertion. The problem is that, since $A(g,t)$ is defined only on a subset of $\GG_{y_0}$ it is not clear that the solution exists, nor that it is unique.  To overcome this difficulty consider a differential equation on $\GG_{y_0}$ given in local coordinates by
\begin{equation}\label{eqn:int_a}
\dot z^i=f^i_j(z)a^j(t).
\end{equation}
It satisfies the assumptions of Theorem \ref{thm:exist}
for measurable ODEs, so it has an ACB solution $z(t)$, unique up
to the choice of the initial point. In particular, let $z(t)$ be the solution  with $z(t_0)=g_0\in\GG_{y_0}$. The base trajectory $\wt x(t)=\beta(z(t))$ satisfies 
\begin{align*}
\pa_t\wt x^a(t)&=\frac{\pa \beta^a(z(t))}{\pa z^j}\dot z^j(t)=\frac{\pa \beta^a(z(t))}{\pa z^j}f^j_i(z(t))a^i(t)=\rho^a_i(\wt x(t))a^i(t),\\
\wt x(t_0)&=\beta(z(t_0))=\beta(g_0)=x_0.
\end{align*}
On the other hand, by admissibility of $a(t)$, we have $\pa_t x^a(t)=\rho^a_i(x(t))a^i(t)$ and $x(t_0)=x_0$,
hence; clearly, $\wt x(t)=x(t)$. This, in turn, implies that $\dot z^i(t)=f^i_j(t) a^j(t)=A^i(z(t),t)$, i.e., $g(t)=z(t)$ as above is well defined and unique. 
 
Now consider a homotopy $h:[t_0,t_1]\times[0,1]=:K\lra \GG_{y_0}$ over $x:[t_0,t_1]\times[0,1]\lra M$, which is ACB w.r.t. both variables. In local coordinates it is given by $z^i(s,t)$. Repeating the argument from the previous part, we can prove that the maps $t\mapsto a(t,s):=\T R_{h(t,s)^{-1}}\left(\pa_th(t,s)\right)$ and $s\mapsto b(t,s):=\T R_{h(t,s)^{-1}}(\pa_sh(t,s))$ are bounded measurable admissible paths over $t\mapsto x(t,s)$ and $s\mapsto x(t,s)$, respectively. In local coordinates,
\begin{align*}
&a^i(t,s)=F^i_j(z(t,s))A^j(t,s),\\
&b^i(t,s)=F^i_j(z(t,s))B^j(t,s),
\end{align*}
where we denoted $A^i(t,s):=\pa_tz^i(t,s)$ and
$B^i(t,s)=\pa_sz^i(t,s)$.

Since $h$ is a homotopy, we have
$$\iint_K z^i(t,s)\pa_t\pa_s\phi_i(t,s)\dd t\dd s=\iint_Kz^i(t,s)\pa_s\pa_t\phi_i(t,s)\dd t\dd s,$$ 
 for every $\phi_i\in C^\infty(K)$. Integrating the above equality several times by parts, we get that $A^i(t,s)$ and $B^i(t,s)$ satisfy the differential equation
$$\pa_sA^i(t,s)\overset{\text{WT}}=\pa_tB^i(t,s).$$

Now calculating the WT-derivatives of $a^i(t,s)$ and $b^i(t,s)$ we get
\begin{eqnarray*} &\pa_t b^i(t,s)-\pa_s a^i(t,s)\overset {\text{WT}}= \pa_t\left(F^i_j(z(t,s))B^j(t,s)\right)-
\pa_s\left(F^i_j(z(t,s))A^j(t,s)\right)\\
&\overset {\text{WT}}=\left(\frac{\pa F^i_n}{\pa
z^m}(z(t,s))-\frac{\pa F^i_n}{\pa z^m}(z(t,s))
\right)B^n(t,s)A^m(t,s)\\&+F^i_j(z(t,s))\left(\pa_tB^j(t,s)-\pa_sA^j(t,s)\right)\\
&=c^i_{jk}(\beta(z(t,s)))b^j(t,s)a^k(t,s)+0
=c^i_{jk}(x(t,s))b^j(t,s) a^k(t,s).
\end{eqnarray*}
We see that $(a,b)$ is an $\AG$-homotopy.

Conversely, let $a,b:[t_0,t_1]\times[0,1]\lra \AG$ over $x:[t_0,t_1]\times[0,1]\lra
M$ be an algebroid homotopy. By the first part of the assertion we can uniquely integrate the admissible path $s\mapsto b_0(t_0,s)$ to an ACB path $g_0(s)\in\GG_{y_0}$ with $g_0(0)=g_0$. Next we can uniquely integrate each admissible path $t\mapsto a(t,s)$ to an ACB path $g(t,s)\in\GG_{y_0}$ such that $g(t_0,s)=g_0(s)$. In local coordinates $g(t,s)$
is a solution of the differential equation (cf. the previous part of this proof)
\begin{align*} 
\pa_t z^i(t,s)&=f^i_j(z(t,s))a^j(t,s),\\
 z^i(t_0,s)&=z_0^i(s),
\end{align*}
where $z_0(s)=g_0(s)$ is ACB. By Theorem \ref{thm:param}, $g(t,s)$ is ACB w.r.t. both variables. Now $g(t,s)$ is an ACB homotopy in $\GG_{y_0}$ hence, as has already been proved, it reduces to an algebroid homotopy $\wt a,\wt b:[t_0,t_1]\lra\AG$. By construction, $\wt a(t,s)=a(t,s)$ and $\wt b(t_0,s)=b(t_0,s)$. We see that $(a,b)$ and $(a,\wt b)$ are two WT-solutions of \eqref{eqn:htp_smooth} with the same initial-point $\AG$-homotopy $b(t_0,s)$. By Lemma \ref{lem:htp_unique} $b(t,s)= \wt b(t,s)$.
\end{proof}

\begin{remark}
The above theorem is closely related to the ideas of \cite{crainic_fernandes}. The correspondence between algebroid homotopies and homotopies in an integrating groupoid may be used to address the question about integrability of Lie algebroids (see remark on page \pageref{CF}). 

In fact, we can also use it to prove Theorem \ref{thm:Lie1} of Mackenzie and Xu. To sketch the idea, let us concentrate on the case when $\HH=S\times S$ is a pair groupoid, with $S$ simply connected. Consider a morphism of Lie algebroids $\Phi:\T S\lra\AG$ over $f:S\lra M$. Fixing points $x_0\in S$ and $g_0\in\GG_{f(x_0)}\cap\beta^{-1}(f(x_0))$, we can attach to each sufficiently regular curve $\gamma:[0,1]\lra S$, originated at $\gamma(0)=x_0$, a curve $g:[0,1]\ra \GG_{f(x_0)}$, with $g(0)=g_0$, being the lift of an admissible curve $\Phi(\dot\gamma(\cdot)):[0,1]\lra\AG$. Now, if $\gamma_0$ and $\gamma_1$ are two curves such that $\gamma_0(0)=\gamma_1(0)=x_0$ and $\gamma_0(1)=\gamma_1(1)$, then, since $S$ is simply connected, there exists a homotopy $\gamma(t,s)$ in $S$ (with fixed end-points) joining $\gamma(\cdot,0)=\gamma_0(\cdot)$ and $\gamma(\cdot, 1)=\gamma_1(\cdot)$. The lift of the $\AG$-homotopy $\left(\Phi(\pa_t\gamma(t,s)),\Phi(\pa_s\gamma(t,s))\right)$ is a homotopy $g(t,s)$ in $\GG_{f(x_0)}$ (with fixed end-points) joining the lifts of $\gamma_0(\cdot)$ and $\gamma_1(\cdot)$. Consequently, the map $\wt\Phi:\gamma(1)\mapsto g(1)$, $S\lra\GG_{f(x_0)}$ is well defined. One can prove that $\mathcal{R}\circ\T\wt\Phi=\Phi$. The presence of such a map is equivalent to the integrability of $\Phi$ (see Corollary \ref{cor:int}). A similar argument can be used to prove Theorem \ref{thm:Lie1} in full generality.
\end{remark}

\begin{corollary}\label{cor:htp_P}
Theorem \ref{thm:int_htp} establishes the equivalence between
$A(\GG)$-paths/homotopies and standard paths/homotopies in a
single $\alpha$-fibre in the groupoid $\GG$. For the
groupoid $\GG_P=P\times P/G$ and the associated Atiyah algebroid
$\T P/G$, these fibres are canonically isomorphic to $P$, so
the $E$-homotopies are just standard homotopies in $P$ reduced to
$\T P/G$. The two are equivalent up to the choice of the initial point.
\end{corollary}

\subsection{\texorpdfstring{$E$}{E}-homotopy classes} 
\begin{definition}\label{def:E_htp_class}
Two measurable $E$-paths $a_0,a_1:[t_0,t_1]\lra E$ are
\emph{$E$-homotopic}\index{algebroid homotopic paths} iff there exists an
$E$-homotopy $(a,b)$ \underline{with fixed end-points} (i.e., $b(t_0,\cdot)\equiv 0 \equiv b(t_1,\cdot)$) between $a_0$ and
$a_1$. Being $E$-homotopic is an equivalence relation. 

An equivalence class of an element $a$ will be denoted by
$[a]$ (or sometimes $[a(t)]_{t\in[t_0,t_1]}$) and called an
\emph{$E$-homotopy class}\index{algebroid homotopy!class}.
\end{definition}

So far, the above definition does not allow us to compare the $E$-homotopy classes of $E$-paths defined on different time intervals. Therefore, we will add   a natural condition that a composition with a null path not change the equivalence class:  $[a\circ
\theta_{x(t_1)}]=[a]=[\theta_{x(t_0)}\circ a]$.

Observe that, since two $E$-homotopies with fixed end-points are composable (iff the final base point of the first coincides with the inital base point of the second), the composition of $E$-homotopies defines a multiplication of $E$-homotopy classes by a natural formula
$$[a]\cdot[\ol a]:=[a\circ \ol a]\quad \text{ when $a$ and $\ol a$ are composable.}$$

Consider now an algebroid $E$ and two smooth algebroid morphisms $\Phi_0:\T S_0\lra E$ and $\Phi_1:\T S_1\lra E$ over $\phi_0: S_0\lra M$ and $\phi_1:S_1\lra M$, respectively. 

\begin{definition}\label{def:htp_class_rel} We say that measurable $E$-paths $a_0$ and $a_1$ are \emph{$E$-homotopic relative to the morphisms $\Phi_0$ and $\Phi_1$}\index{algebroid homotopic paths!relative to a pair of morphisms} iff there 
exists an
$E$-homotopy $(a,b)$  between $a_0$ and
$a_1$, and AC paths $z:[0,1]\ra S_0$ and $w:[0,1]\ra S_1$, such that $b(t_0,s)=\Phi_0(\pa_s z(s))$ and $b(t_1,s)=\Phi_1(\pa_s w(s))$. In other words,  initial-point and final-point $E$-homotopies lie in $\Image\Phi_0$ and $\Image\Phi_1$, respectively. Note that $b(t_0,\cdot)$ and $b(t_1,\cdot)$ are admissible as images of admissible paths under an algebroid morphism. 

The relation of being relatively $E$-homotopic is again an equivalence relation and we may again speak of the equivalence classes (\emph{relative $E$-homotopy classes})\index{algebroid homotopy! class!relative}. A class of an element $a$ will be denoted by $[a]\modulo(\Phi_0,\Phi_1)$.
\end{definition}

\begin{remark} $E$-paths $a_0$ and $a_1$ are $E$-homotopic iff they are $E$-homotopic relative to a morphism $\iota_{x_0}$, which maps $S_0=\{\pt\}\approx\T S_0$ to a null vector $\theta_{x_0}$, where $x_0=\tau\circ a_0(t_0)=\tau\circ a_1(t_0)$, and a morphism $\iota_{x_1}$ defined analogously for $x_1=\tau\circ a_0(t_1)=\tau\circ a_1(t_1)$.
\end{remark}

\subsection{The interpretation of $E$-homotopy classes}
Tn light of Theorem \ref{thm:int_htp}, an interpretation of the notion of $E$-homotopy classes is clear. If $E=\AG$ is an integrable Lie algebroid and $(a,b):[t_0,t_1]\times[0,1]\lra E$ is an algebroid homotopy with fixed end points, we can lift it to the true homotopy $g(t,s)$ in a single $\alpha$-fibre $\GG_{y_0}$ of $\GG$. The $E$-paths $a(\cdot,0)$ and $a(\cdot,1)$ correspond to $g(\cdot,0)$ and $g(\cdot,1)$, respectively, and null paths $b(t_0,\cdot)$ and $b(t_1,\cdot)$ to constant paths $g(t_0,\cdot)$ and $g(t_1,\cdot)$, respectively. In other words, $g(t,s)$ is a homotopy in $\GG_{y_0}$ between $g(\cdot,0)$ and $g(\cdot,1)$ with fixed end-points. Consequently, we can interpret $E$-homotopy classes as reduced homotopy classes from an $\alpha$-fibre $\GG_{y_0}$ of a Lie groupoid $\GG$ to the associated Lie algebroid $E=\AG$. In particular, if $\GG_P$ is a gauge groupoid of a principal $G$-bundle $G\ra P\ra M$, all $\alpha$-fibres are isomorphic to $P$ (cf. Appendix \ref{sapp:groupoids}), hence algebroid homotopy classes in the Atyiah algebroid $\T P/G$ are the standard homotopy classes in $P$ reduced to $\T P/G$ by the $G$-action. 

For relative $E$-homotopy classes things are a little more complicated.  Assume that $E=\AG$ is an integrable Lie algebroid  and  $a,b:[t_0,t_1]\times[0,1]\lra \AG$ is an algebroid homotopy relative to $(\Phi_0,\Phi_1)$. Let $z(\cdot)\subset S_0$ and $w(\cdot)\subset S_1$ be as in the Definition \ref{def:htp_class_rel}. Assume, in addition, that $\Phi_0$ and $\Phi_1$ are integrable. By Theorem \ref{thm:int_htp} we can lift $(a,b)$ to the homotopy $g(t,s)\in\GG_{y_0}\subset \GG$. By Corollary \ref{cor:int} we can lift $\Phi_0$ to a smooth map $\wt\Phi_0:S_0\lra\GG_{y_0}$ such that $\wt\Phi_0(z(0))=g(t_0,0)$, and we can lift $\Phi_1$ to a smooth map $\wt\Phi_1:S_1\lra\GG_{y_0}$ such that $\wt\Phi_1(w(0))=g(t_1,0)$. Paths $\wt\Phi_0(z(\cdot))$ and $g(t_0,\cdot)$ correspond to the same $\AG$-path $b(t_0,\cdot)$ and have the same initial point, hence are equal. Similarly, $\wt\Phi_1(w(\cdot))=g(t_1,\cdot)$. In other words, $g(t,s)$ is a homotopy in $\GG_{y_0}$ joining $g(\cdot,0)$ and $g(\cdot,1)$ with end-points in the images of $\wt\Phi_0$ and $\wt\Phi_1$.

To sum up, $(a,b)$ is a reduction of a homotopy in $\GG_{y_0}$ with end-points contained in the images of $\wt\Phi_0$ and $\wt\Phi_1$ integrating $\Phi_0$ and $\Phi_1$. 

Note that algebroid morphisms $\Phi_0$ and $\Phi_1$ need not to be integrable.
In such a case the interpretation given above is still valid, but locally. Let namely $U_0\subset S_0$ be an open simply connected neighbourhood of $z(0)$ and let $U_1\subset S_1$ be an open simply connected neighbourhood of $w(0)$. Now, by Corollary \ref{cor:int}, we can lift $\Phi_0|_{U_0}:\T U_0\lra\AG\ $ and $\ \Phi_1|_{U_1}:\T U_0\lra\AG\ $ to $\ \wt{\Phi_0}|_{U_0}:U_0\lra \GG_{y_0}\ $ and $\ \wt{\Phi_1}|_{U_1}:U_1\lra \GG_{y_0}$, respectively. Hence, if $b(t_0,\cdot)\subset \Image\Phi_0|_{U_0}$ and $b(t_1,\cdot)\subset \Image\Phi_1|_{U_1}$ we can still interpret the relative $\AG$-homotopy $(a,b)$ as the homotopy in $\GG_{y_0}$ with end-points in the images of $\wt\Phi_0|_{U_0}$ and $\wt\Phi_1|_{U_1}$, reduced to the algebroid $\AG$.
 
Another, more universal approach is the following. Consider the universal covers
$\pi_0:\left(\wt S_0,\wt z(0)\right)\lra \left(S_0, z(0)\right)$ and $\pi_1:\left(\wt S_1,\wt w(0)\right)\lra \left(S_1, w(0)\right)$. Now take Lie algebroid morphisms  $\Phi_i^{'}:=\Phi_i\circ\pi_i:\T \wt S_i\lra \AG$, for $i=0,1$. They are clearly integrable to $\wt\Phi_i^{'}:\wt S_i\lra\GG_{y_0}$, since $\wt S_0$ and $\wt S_1$ are simply connected. As $b(t_0,\cdot)\subset\Image\Phi^{'}_0$ and $b(t_1,\cdot)\subset\Image\Phi^{'}_1$, we can interpret $(a,b)$ as a reduced homotopy in $\GG_{y_0}$ with end-points in the images of $\wt\Phi_0^{'}$ and $\wt\Phi_1^{'}$.  

 In particular, if $\GG=\GG_P$ is a gauge groupoid of a principal $G$-bundle $G\ra P\ra M$, then $\GG_{y_0}\approx P$. Relative $\T P/G$-homotopies are homotopies  in $P$ with end-points in the images of maps $\wt\Phi^{'}_0:\wt S_0\lra P$ and $\wt\Phi^{'}_1:\wt S_1\lra P$, reduced by the $G$-action.

\section{Fundamental properties of \texorpdfstring{$E$}{E}-homotopies}\label{sec:E_htp_prop}

\subsection{\texorpdfstring{$E$}{E}-homotopies as families of \texorpdfstring{$E$}{E}-paths} 
The following lemma emphasis the role of AL algebroids. Roughly
speaking, it turns out that for AL algebroids one-parameter
families of $E$-paths are $E$-homotopies.

\begin{lemma}[generating $E$-homotopies]\label{lem:gen_E_htp}\index{almost Lie algebroid!characterisation}
Let $E$ be an AL algebroid, and let $a:I\times[0,1]\lra E$ be a
one-parameter family of bounded measurable $E$-paths (that is,
$t\mapsto a(t,s)$ is admissible for every $s$) covering
$x:I\times[0,1]\lra M$. Assume that $a(t,s)$ is ACB w.r.t. $s$; that is, $\pa_sa(t,s)$ is defined a.e. and is bounded
measurable w.r.t. both variables. Let $b_0(s)$ be an
arbitrary bounded measurable $E$-path covering $x(t_0,s)$.

Then there exists an unique $E$-homotopy $a,b:I\times[0,1]\lra E$
such that $b(t_0,s)=b_0(s)$. Moreover, $b(t,s)$ is ACB w.r.t. $t$
(that is, $\pa_tb(t,s)$ is defined a.e. and is bounded measurable
w.r.t. both variables).
\end{lemma}

\begin{proof}
By the definition of an $E$-homotopy, $b(t,s)$ should be a map covering
$x(t,s)$ such that $s\mapsto b(t,s)$ is admissible and 
\eqref{eqn:htp_weak} is WT-satisfied. Observe that,
since $\pa_s a(t,s)$ is well defined a.e., the system of equations
\begin{equation}\label{eqn:gen_htp}
\pa_tb^i(t,s)=\pa_s a^i(t,s)+c^i_{jk}(x(t,s))b^j(t,s)a^k(t,s)
\end{equation}
for $b(t,s)$ satisfies the assumptions of Theorem \ref{thm:param}. Consequently,
 it has a unique Carath{\'e}odory solution $b^i(t,s)$ for a given
initial condition $b^i(t_0,s)= b_0^i(s)$. The solution $b(t,s)$ is
ACB w.r.t. $t$ and, since the parameter-$s$-dependence of both right-hand side
of \eqref{eqn:gen_htp} and the initial condition is bounded
measurable, so is the $s$-dependence of the solution $b(t,s)$.
Consequently, the right-hand side of \eqref{eqn:gen_htp} is bounded and
measurable w.r.t. both $t$ and $s$, and hence so is $\pa_tb(t,s)$---
the left-hand side of \eqref{eqn:gen_htp}. Clearly, $a$ and
$b$ are regular enough to satisfy the assumptions of Theorem
\ref{thm:w_wt}, so the integral condition \eqref{eqn:htp_weak}
holds.

To prove that thus constructed $(a,b)$ is indeed an
$E$-homotopy, it is enough to show that $s\mapsto b(t,s)$ is
admissible for every fixed $t$. Consider a map
$$\chi^a(t,s):=\pa_s x^a(t,s)-\rho^a_i\big(x(t,s)\big)b^i(t,s).$$
We shall show that $\chi^a= 0$ a.e. Observe that, since $a(t,s)$
is a family of admissible paths, we have
\begin{equation}\label{eqn:gen_htp_2}
\pa_tx^a(t,s)=\rho^a_k\left(x(t,s)\right)a^k(t,s) \text{ a.e.}
\end{equation}
The right-hand side of this equation is differentiable with
respect to $s$, and hence so is the left-hand side, and
$$\pa_s\pa_tx^a(t,s)=\frac{\pa\rho^a_k}{\pa x^b}\left(x(t,s)\right)\pa_s x^b(t,s)a^k(t,s)+\rho^a_k\left(x(t,s)\right)\pa_s a^k(t,s).$$

Consequently, as $\pa_t\pa_s
x^a(t,s)\overset{\text{WT}}{=}\pa_s\pa_tx^a(t,s)$ (since $x(t,s)$ is a
true homotopy in $M$),
\begin{eqnarray*}\pa_t \chi^a&=&\pa_t\pa_s
x^a-\pa_t\left(\rho^a_k b^k\right)\overset {\text{WT}}=\pa_s\pa_t
x^a-\pa_t\left(\rho^a_k b^k\right)\\&=&\frac{\pa \rho^a_k}{\pa x^b}\pa_s
x^b a^k+\rho^a_k\pa_s a^k-\frac{\pa \rho^a_i}{\pa x^b}\pa_t x^b
b^i- \rho^a_i\pa_t b^i\,,
\end{eqnarray*}
which, in view of \eqref{eqn:gen_htp_2}
and \eqref{eqn:gen_htp}, equals
\begin{align*}
&\frac{\pa\rho^a_k}{\pa x^b}\pa_s x^b a^k+\rho^a_k\pa_s
a^k-\frac{\pa \rho^a_i}{\pa x^b}\rho^b_k a^k b^i-
\rho^a_i\left(\pa_s a^i+c^i_{jk}b^ja^k\right)=\\
&=\left(\frac{\pa}{\pa
x^b}\rho^a_k\right)a^k\chi^b+\left[\left(\frac{\pa}{\pa
x^b}\rho^a_k\right)\rho^b_j-\left(\frac{\pa}{\pa
x^b}\rho^a_j\right)\rho^b_k-\rho^a_ic^i_{jk}\right]b^ja^k.
\end{align*}
Since $E$ is an AL algebroid, the last term vanishes and we have
$$\pa_t\chi^a\overset{\text{WT}}=\left(\frac{\pa}{\pa x^b}\rho^a_k\right)a^k\chi^b.$$
Thus $\chi^a$ is a WT-solution of a linear differential equation
with measurable r.h.s. and the initial condition $\chi^a(t_0,s)=0$
(since $b_0(s)$ is admissible). Repeating the argument from the proof of Lemma
\ref{lem:wt_uniq} we conclude that $\chi^a=0$ a.e.
\end{proof}

It turned out that, when a skew-algebroid is almost Lie, 
$E$-homotopies are the true homotopies in the space of $E$-paths
(i.e. one-parameter families of $E$-paths). This has already been
observed in \cite[Thm. 3]{GG_var_calc} in a slightly different form.

\subsection{Relation between \texorpdfstring{$E$}{E}-homotopies with and without fixed end-points, reparametrisation}

\begin{lemma}\label{lem:reparam1}\index{admissible path!reparametrisation}
Let $a:[t_0,t_1]\lra E$ be a bounded measurable $E$-path over $x(t)$, and let $h:[0,1]\ra[t_0,t_1]$ be an invertible $C^1$-function. Define
\begin{equation}\label{eqn:reparam}
\begin{split}
a(t,s)&:=\frac{h(s)-t_0}{t_1-t_0} a\left(t_0+\frac{t-t_0}{t_1-t_0}(h(s)-t_0)\right),\\
b(t,s)&:=\frac{t-t_0}{t_1-t_0}\dot h(s) a\left(t_0+\frac{t-t_0}{t_1-t_0}(h(s)-t_0)\right).
\end{split}
\end{equation}
Then the pair $(a,b)$ is an $E$-homotopy over $x(t,s):=x\left(t_0+\frac{t-t_0}{t_1-t_0}(h(s)-t_0)\right)$.
\end{lemma}

\begin{proof}
For notation simplicity assume that $[t_0,t_1]=[0,1]$. Then
\begin{align*}
a(t,s)&=h(s)a(t h(s))\quad \text{and}\\
b(t,s)&=t\dot h(s) a(t h(s)).
\end{align*}
First, note that $t\mapsto a(t,s)$ and $s\mapsto b(t,s)$
are admissible. Indeed, from $\dot{x}(t)=\rho\left(a(t)\right)$ we deduce that 
$$\pa_tx(t, s)=h(s)\dot{x}(th(s))=h(s)\rho\left(a(th(s))\right)=\rho\left(
a(t,s)\right).$$
Similarly, $\pa_sx(t,s)=\rho\left(b(t,s)\right)$. 
Now we will check that
$$\pa_t[t\dot h(s)a^i(th(s))]\overset {\text{W}}=\pa_s[h(s)a^i(th(s))]+c^i_{jk}(x(t,s))[t\dot h(s)a^j(th(s))]\cdot[h(s)sa^k(th(s))].$$
By the skew-symmetry of $c^i_{jk}$, the last term vanishes, so we
have to check if
$$\pa_t[t\dot h(s)a^i(th(s))]\overset {\text{W}}=\pa_s[h(s)a^i(th(s))].$$
The latter is certainly true, as both sides are equal 
$\dot h(s)a^i(th(s))+th(s)\dot h(s) A^i(th(s))$, where $A^i(t)$ is the distributive derivative of $a^i(t)$. 

To finish the proof we shall show that
$a(t,s)$ and $ b(t,s)$ satisfy the regularity conditions \eqref{eqn:traces} and \eqref{eqn:traces1}.
 
Assume that $h(0)=t_0=0$ and $h(1)=t_1=1$ (the case $h(1)=t_0=0$ and $h(0)=t_1=1$ is completely analogous). Now
\begin{align*}\int_0^1\int_0^\eps| a(t,s)- a(t,0)|\frac 1\eps\dd s\dd t=\frac 1\eps\int_0^\eps\int_0^1|h(s) a(th(s))| \dd t \dd s
\leq \frac 1\eps\int_0^\eps h(s)\|a\|\dd s.
\end{align*}
The later converges to 0 as $\eps\to 0$, since $h(s)\overset{s\to 0}\lra h(0)=0$.
Next, 
\begin{align*}\int_0^1\int_{1-\eps}^1| a(t,s)- a(t,1)|\frac 1\eps \dd s\dd t=\frac 1\eps\int_{1-\eps}^1\int_0^1|h(s) a(th(s))-h(1)a(th(1))|\dd t\dd s\\
\leq\frac 1\eps\int_{1-\eps}^1 h(s)\int_0^1|a(th(s))-a(th(1))|\dd t\dd s+\frac 1\eps\int_{1-\eps}^1|h(s)-h(1)|\int_0^1 |a(t h(s))|\dd t \dd s.
\end{align*}
The second factor converges to 0 as $\eps\to 0$ because $h(s)$ is continuous at $s=1$ and $a$ is bounded. By Lemma \ref{lem:reg} the measurable function $g(s):=\int_0^1|a(th(s))-a(th(1))|\dd t$ is regular at $s=0$ and, moreover, $g(s)=0$. We conclude that
$$\frac 1\eps\int_{1-\eps}^1h(s)g(s)\dd s\to 0 \quad\text{as}\quad \eps\to 0.$$
Consequently, $\int_0^1\int_{1-\eps}^1| a(t,s)- a(t,1)|\frac 1\eps \dd s\dd t\overset{\eps\to 0}\lra 0$ and conditions \eqref{eqn:traces} are fulfilled.

Now check \eqref{eqn:traces1}. The first of the two conditions is a matter of a simple estimation:
\begin{align*}
&\int_0^1\int_0^\eps|b(t,s)-b(0,s)|\frac 1\eps \dd t\dd s=\frac 1\eps\int_0^\eps\int_0^1|t\dot h(s) a(th(s))|\dd s\dd t\leq \frac 1\eps\int_0^\eps t\|\dot h\|\cdot\|a\|\dd t\overset{\eps\to 0}\lra 0.
\intertext{For the second we can estimate:}
&\int_0^1\int_{1-\eps}^1|b(t,s)-b(1,s)|\frac 1\eps\dd t\dd s=\frac 1\eps\int_{1-\eps}^1\int_0^1|t\dot h(s)a(t h(s))-\dot h(s)a(h(s))|\dd s\dd t\\
&\leq\frac 1\eps\int_{1-\eps}^1 t\|\dot h\|\int_0^1|a(t h(s))-a(h(s))|\dd s\dd t+\frac 1\eps \int_{1-\eps}^1|1-t|\cdot\|\dot h\|\int_0^1|a(h(s))|\dd s\dd t. 
\end{align*}
The last factor clearly converges to 0 as $\eps\to 0$. Using Lemma \ref{lem:reg} we show that the measurable function $k(t):=\int_0^1|a(t h(s))-a(h(s))|\dd s$ is regular at $t=1$ and, moreover, $k(1)=0$. We conclude that
$$\frac 1\eps\int_{1-\eps}^1t\|\dot h\|k(t)\dd t\overset{\eps\to 0}\lra 0,$$ 
which proves that conditions \eqref{eqn:traces1} are satisfied. By Theorem \ref{thm:w_wt} the pair $(a,b)$ is a WT-solution of \eqref{eqn:htp_smooth}, and hence $E$-homotopy. 
 \end{proof}

\noindent As a corollary we obtain the following fact.
\begin{lemma}[shrinking an $E$-path]\label{lem:van_curv}
Let $a:[0,1]\lra E$ be a measurable $E$-path over $x(t)$. Define $a(t,s):=s a(ts)$ and $b(t,s):=t
a(ts)$ for $t,s\in [0,1]$. The pair $(a,b)$ is
an $E$-homotopy over $ x(t,s)=x(ts)$. Its initial-point
$E$-homotopy is $\theta_{x(0)}$, and the final-point $E$-homotopy
is $a$.

Similarly, consider $\wt{a}(t,s):=(1-s)a(1-(1-t)(1-s))$ and
$\wt{b}(t,s):=(1-t)a(1-(1-t)(1-s))$ where $t,s\in[0,1]$. The pair $(\wt{a},\wt{b})$ is an $E$-homotopy over $\wt x(t,s):=x(1-(1-t)(1-s))$. Its
initial-point $E$-homotopy is $a$, and the final-point $E$-homotopy
is $\theta_{x(1)}$.
\end{lemma}

\begin{proof}  The assertion follows from Lemma \ref{lem:reparam1}. For $(a,b)$ we simply take $a(t)$ and $h(s)=s$. 

\noindent For $(\wt a,\wt b)$ we use Lemma \ref{lem:reparam1} with $\wt a(t):=a(1-t)$ defined on an interval $[t_0=1,t_1=0]$ and $h(s)=s$. 
\end{proof}

\noindent We can now state the following important result.
\begin{lemma}\label{lem:htp}\index{algebroid homotopy!class}
Let $a,b:[t_0,t_1]\times[0,1]\lra E$ be an $E$-homotopy covering
$x:[t_0,t_1]\times[0,1]\lra M$. Then we have the following equality of $E$-homotopy classes:
$$[a(t,0)]_{t\in[t_0,t_1]} [b(t_1,s)]_{s\in[0,1]}=[b(t_0,s)]_{s\in[0,1]} [a(t,1)]_{t\in[t_0,t_1]}.$$
\end{lemma}

\begin{proof} The first part of Lemma \ref{lem:van_curv}, applied to the curve $s\mapsto b(t_0,s)$,
gives us the existence of $E$-homotopy $c,d:[0,1]\times[0,1]\lra
E$ such that $c(t,0)=\theta_{x(t_0)}$, $c(s,1)=b(t_0,s)$,
$d(0,s)=\theta_{x(t_0)}$, and $d(1,s)=b(t_0,s)$. Similarly, using
the second part of Lemma \ref{lem:van_curv} for $s\mapsto
b(t_1,s)$, we obtain $E$-homotopy $e,f:[0,1]\times[0,1]\lra E$
such that $e(s,0)=b(t_1,s)$, $e(t,1)=\theta_{x(t_1)}$,
$f(0,s)=b(t_1,s)$, and $f(1,s)=\theta_{x(t_1)}$.

Clearly, $E$-homotopies $(c,d)$, $(a,b)$, and $(e,f)$ are
composable and their composition is an $E$-homotopy with fixed
end-points which establishes an equivalence of $E$-paths
$\theta_{x(t_0)}\circ a(\cdot,0)\circ b(t_1,\cdot)$ and
$b(t_0,\cdot)\circ a(\cdot,1)\circ\theta_{x(t_1)}$.
\end{proof}

\begin{remark} \label{rem:lem_htp}
The above lemma is very important, as it shows the relation
between $E$-homotopies with and without fixed end-points. If
$a,b:I\times[0,1]\lra E$ is an $E$-homotopy joining $a_0$ and
$a_1$, then the composition of $a_0$ with the final-point
$E$-homotopy $b(t_1,\cdot)$ is equivalent to the composition of
the initial-point $E$-homotopy $b(t_0,\cdot)$ with $a_1$. Thus, if
the initial-point $E$-homotopy $b(t_0,\cdot)$ vanishes, in order
to check whether $[a_0]=[a_1]$, it is enough to check whether
$[b(t_1,\cdot)]=0$. Thus the problem of equivalence of $a_0$ and
$a_1$ can be solved by investigating the final-point $E$-homotopy.
Similarly, we can address the problem of relative $E$-homotopy equivalence $[a_0]=[a_1]\modulo(\Phi_0,\Phi_1)$ by studying the classes $[b(t_0,\cdot)]$ and $[b(t_1,\cdot)]$.   
\end{remark}\medskip

\noindent Finally, as a corollary from Lemmas \ref{lem:reparam1} and \ref{lem:htp} we obtain a  result about reparametrisation of
$E$-paths.

\begin{lemma}[reparametrization]\label{lem:reparam}\index{admissible path!reparametrisation}
Let $a:[t_0,t_1]\lra E$ be a measurable $E$-path, and let $h:[0,1]\ra[t_0,t_1]$ be an invertible $C^1$-function. Define $\wt a(t):=\dot h(t)a(h(t))$ for $t\in[0,1]$. Then
\begin{align}
&[\wt a(t)]_{t\in[0,1]}=[a(t)]_{t\in[t_0,t_1]} \quad\text{if $h(0)=t_0$ and $h(1)=t_1$,}\label{eqn:reparam1}\\
&[a(t)]_{t\in[t_0,t_1]}\cdot[\wt a(t)]_{t\in[0,1]}=0 \quad\text{if $h(0)=t_1$ and $h(1)=t_0$.}\label{eqn:reparam2} 
\end{align}
\end{lemma}

\begin{proof} Consider a homotopy \eqref{eqn:reparam} from Lemma \ref{lem:reparam1}. If $h(0)=t_0$ and $h(1)=t_1$ we have $a(t,0)=0$, a(t,1)=a(t)$, b(t_0,s)=0$, and $b(t_1,s)=\dot h(s)a(h(s))=\wt a(s)$. By Lemma \ref{lem:htp}, 
$$[\theta_{x(t_0)}]\cdot[\wt a(s)]_{s\in[0,1]}=[\theta_{x(t_0)}]\cdot[a(t)]_{t\in [t_0,t_1]}.$$

Analogously, for $h(0)=t_1$ and $h(1)=t_0$ we have  $a(t,0)=a(t)$, a(t,1)=0$, b(t_0,s)=0$, and $b(t_1,s)=\dot h(s)a(h(s))=\wt a(s)$. From Lemma \ref{lem:htp} we deduce that 
$$[a(t)]_{t\in[t_0,t_1]}\cdot[\wt a(s)]_{s\in[0,1]}=[\theta_{x(t_0)}]\cdot[\theta_{x(t_0)}]=0.$$
\end{proof}

\chapter[Optimal control problems]{Optimal control problems on AL algebroids}\label{ch:OCP}


In this chapter we introduce the notion of a control system and an optimal control problem on algebroids. Much attention is payed to motivate these definitions. We show that equivariant control systems and optimal control problems on a Lie groupoid $\GG$ lead naturally to system and problems on the associated Lie algebroid $\AG$. What is more, our definitions coincide with the standard ones for special cases of  tangent algebroid and Atiyah algebroid. At the end, we define the natural notion of algebroid homotopy associated with a control system.
   
From this chapter on, our attention is restricted to AL algebroids
only. This choice is justified by the properties of algebroid homotopies
on AL algebroids discussed in Lemma  \ref{lem:gen_E_htp}.

\subsection{Control systems on AL algebroids} 
\begin{definition}\label{def:con_sys}
A \emph{control system}\index{control system on AL algebroid} on an AL algebroid $E$ is a continuous map
\begin{equation}\label{eqn:def_con_sys}
f:M\times U\lra E
\end{equation}
such that, for every $u\in U$, the map $f(\cdot,u):M\lra E$ is a
$C^1$-section of $E$.
We will assume that $U$ is a subset of some Euclidean space $\R^r$. Moreover, we demand that the maps $f:M\times U\lra E$ and $\T_xf:\T M\times U\lra\T E$ are continuous. In local coordinates, if $f\sim (f^i(x,u), x^a)$, this means that $f^i(x,u)$ is continuous w.r.t. $x$ and $u$, differentiable w.r.t. $x$, and that $\frac{\pa f^i}{\pa x_a}(x,u)$ is continuous w.r.t. $x$ and $u$. 
\end{definition}

Observe that, for the tangent algebroid $E=\T M\ra M$, the above definition coincides with the classical one (cf. Definition \ref{def:cs_class} in Appendix \ref{app:ctr_theory}).
On the other hand, one easily sees (cf. Theorem \ref{thm:int_htp})
that a right-invariant control system on a Lie groupoid $\GG$
reduces to a system of the above form on the associated Lie
algebroid $A(\GG)$. For example, a right-invariant control system
on a gauge groupoid $\GG_P=P\times P/G$ of a principal bundle $G\ra
P\ra M$ is determined by its values on a single leaf of
$\GG^\alpha_P$ canonically isomorphic to $P$. Consequently, it is
equivalent to a $G$-invariant control system on $P$ and reduces to
a control system of the form \eqref{eqn:def_con_sys} on the Atiyah
algebroid $E=\sT P/G$. In particular, for a right-invariant system
on a Lie group $G$, Definition \ref{def:con_sys} coincides
with the reduced control system on its Lie algebra $\g$ as
described in \cite[Ch. 12]{jurdjevic}.

Now, for a given function $u:I\lra U$ (\emph{control}\index{control}), the map
\eqref{eqn:def_con_sys} defines a first-order ODE on $M$,
\begin{equation}\label{eqn:con_sys}
\dot{x}(t)=\rho\left(f(x(t),u(t))\right).
\end{equation}
We will restrict our attention only to functions $u$ of a certain
class (called \emph{admissible controls}\index{admissible controls|main}). In this paper these
are controls which are bounded and measurable, but one can
think of smaller classes: piecewise continuous or piecewise
constant functions. The set of all admissible controls will be
denoted by $\Uadm$.

Clearly, if $u(\cdot)$ is admissible, the map
$g(x,t)=\rho\left(f(x,u(t))\right)$ is differentiable w.r.t. $x$
and measurable w.r.t. $t$, so the assumptions of Theorem
\ref{thm:exist} hold. Consequently, we have the results of local
existence and uniqueness for the solutions of \eqref{eqn:con_sys}.
Observe that if $x(\cdot)$ is a solution of \eqref{eqn:con_sys}
for $u(\cdot)\in\Uadm$, then the path $f(x(\cdot),u(\cdot)):I\lra
E$ is a measurable $E$-path over $x(\cdot)$. This path will be
called a \emph{trajectory}\index{trajectory of a control system|main} of \eqref{eqn:con_sys}, whereas for the
pair $(x(\cdot), u(\cdot))$ we will use the term \emph{controlled
pair}.\index{controlled pair|main}

Observe that if $E=\T M$ is a tangent algebroid, then the  trajectory of the system \eqref{eqn:con_sys} is the tangent lift of the trajectory of a corresponding system \eqref{eqn:cs_class} on $M$. The notions of controlled pairs coincide in both cases.

\subsection{Optimal control problems on AL algebroids}
We introduce now a \emph{cost function}\index{cost function} $L:M\times U\lra \R$. We will assume the same regularity conditions for $L$ as in the case of
$f$, namely, that $L$ is a continuous function on $M\times U$,
which is of class $C^1$ w.r.t. the first variable and that the derivative $\dd_x L:\T M\times U\lra\R$ is continuous. If now
$(x(t), u(t))$, with $t\in I$, is a controlled pair for
\eqref{eqn:con_sys}, we define the \emph{total cost}\index{total cost|main} of this pair
to be $\int_{t_0}^{t_1}L\big(x(t),u(t)\big)\dd t.$ Note that, since
$L$ is continuous, $u(t)$ is bounded measurable, and the interval
$[t_0,t_1]$ is compact, the above integral is finite whenever the
solution $x(t)$ exists. Now we can define
optimal control problems for the data introduced above. These
definitions may seem unnatural at first sight, yet we will
motivate them in the next subsection.

\begin{definition}\label{def:ocp} For a control system \eqref{eqn:def_con_sys} and a cost function $L$ we define
an \emph{optimal control problem}\index{optimal control problem|main}\index{OCP|see{optimal control problem}} (\emph{OCP})  as follows:
\begin{equation}\tag{P}\label{eqn:P}
\begin{split}
\text{minimise} \int_{t_0}^{t_1}L\big(x(t),u(t)\big)\dd t
\text{ over all controlled pairs $(x,u)$ of \eqref{eqn:con_sys} such that}\\
\text{ the $E$-homotopy class of the trajectory $f(x(t),u(t))$
equals $[\sigma]$\,,}
\end{split}
\end{equation}
where $\sigma$ is a fixed $E$-path. The interval $[t_0,t_1]$ is to
be determined as well.

Given two smooth algebroid morphisms $\Phi_0:\T S_0\ra E$ and $\Phi_1:\T S_1\ra E$, we can also define an \emph{optimal control problem relative to $(\Phi_0,\Phi_1)$} as follows:\index{optimal control problem!relative to the pair of morphisms}
\begin{equation}\tag{P\ rel}\label{eqn:P_rel}
\begin{split}
\text{minimise} \int_{t_0}^{t_1}L\big(x(t),u(t)\big)\dd t
\text{ over all controlled pairs $(x,u)$ of \eqref{eqn:con_sys} such that}\\
\text{ the relative $E$-homotopy class of $f(x(t),u(t))$
equals $[\sigma]\modulo (\Phi_0,\Phi_1)$,}
\end{split}
\end{equation}
where $\sigma$ and $[t_0,t_1]$ are as above.
\end{definition}

\subsection{Interpretation of the algebroid OCPs}
Now we shall relate the OCPs \eqref{eqn:P} and \eqref{eqn:P_rel} to the standard OCPs considered in control theory. 
Briefly speaking, the $E$-homotopy restrictions in the OCPs on AL algebroids play a role of boundary condition in standard problems. 

Let us concentrate first on \eqref{eqn:P}. Recall from Chapter \ref{ch:E_htp} that, for an integrable algebroid $\AG$, we have interpreted $\AG$-homotopy classes as the standard homotopy classes reduced from a single $\alpha$-fibre $\GG_{y_0}$ of an integrating groupoid $\GG$ to the associated algebroid $\AG$. Moreover, at the beginning of this chapter we have interpreted  control system \eqref{eqn:def_con_sys} on an integrable algebroid $E=\AG$ as a reduction of a right-invariant control system on the groupoid $\GG$ (or on a single $\alpha$-fibre $\GG_{y_0}$). Consider now a control system on a manifold $N$,
\begin{equation}\label{eqn:cs_test}
\FF:N\times U\lra \T N,
\end{equation}
with the cost function $\LL:N\times U\lra\R$, and let us compare the following two OCPs on $N$: 
\begin{equation}\label{eqn:P1}\tag{$\textrm{P}_1$}
\begin{split}
\text{minimise } \int_{t_0}^{t_1}\LL\big(x(t),u(t)\big)\dd t \text{ over all controlled pairs $(x,u)$}\\
\text{satisfying\ } x(t_0)=x_0,\ x(t_1)=x_1,
\end{split}
\end{equation}
and
\begin{equation}\label{eqn:P2}\tag{$\textrm{P}_2$}
\begin{split}
\text{minimise } \int_{t_0}^{t_1}\LL\big(x(t),u(t)\big)\dd t \text{ over all controlled pairs $(x,u)$}\\
\text{for which the homotopy class of $x(t)$ equals $[\sigma]$},
\end{split}
\end{equation}
where $\sigma$ is a fixed path in $N$ joining $x_0$ and $x_1$ and the time interval $[t_0,t_1]$ is not determined.

Problem \eqref{eqn:P1} is a standard OCP on the manifold $N$. On the other hand, \eqref{eqn:P2} is equivalent to the OCP \eqref{eqn:P} on a tangent algebroid $\T N$. 

We may also think of $N$ as of an $\alpha$-fibre $\GG_{y_0}$ of a groupoid $\GG$ with the control system \eqref{eqn:cs_test} and the cost $\LL$ being $\GG$-equivariant, and such that they reduce to the control system \eqref{eqn:def_con_sys} and the cost function $L:U\times M\lra\R$ on the associated algebroid $E=\AG$. Clearly, in this situation, problem \eqref{eqn:P} on $E=\AG$ is equivalent to problem \eqref{eqn:P2} on $N$.

Now let us compare problems \eqref{eqn:P1} and \eqref{eqn:P2}. First, note that every solution of \eqref{eqn:P1} gives a solution of \eqref{eqn:P2} for some $[\sigma]\in \Pi_1(N,x_0,x_1)$. On the other hand, if we know the solutions of \eqref{eqn:P2} for all possible classes $[\sigma]\in\Pi_1(N,x_0,x_1)$ then one (or more) of these solutions which has a minimal total cost is a solution of \eqref{eqn:P1}. To sum up, problem \eqref{eqn:P2} is more refined  than \eqref{eqn:P1}. 

Observe that candidates for the solutions of \eqref{eqn:P1} are usually indicated by the PMP. In the proof one compares the optimal trajectory with  nearby (and hence homotopic) ones. Consequently, the PMP gives only conditions for local optimality and as such will also indicate all candidates for the solutions of \eqref{eqn:P2} for all possible classes $[\sigma]\in\Pi_1(N,x_0,x_1)$ (if such candidates exist). Then, to solve \eqref{eqn:P1} or \eqref{eqn:P2}, one has to investigate closer these candidates to check whether they are really optimal.

Finally, note that problems \eqref{eqn:P1} and \eqref{eqn:P2} are equivalent if $N$ is simply connected. In fact, we can always lift the control system $\FF:N\times U\lra\T N$ and the const function $\LL:N\times U\lra\R$ to $\wt\FF:\wt N\times U\lra \T\wt N$ and $\wt\LL:\wt N\times U\lra\R$ defined on the universal cover $\wt N$ of $N$. 

The discussion for the OCP \eqref{eqn:P_rel} is quite similar. In the same setting as before consider two smooth maps $\wt\Phi_0:S_0\lra N$ and $\wt\Phi_1:S_1\lra N$. Now compare the following two OCPs on $N$:
\begin{equation}\label{eqn:P3}\tag{$\textrm{P}_3$}
\begin{split}
\text{minimise } \int_{t_0}^{t_1}\LL\big(x(t),u(t)\big)\dd t \text{ over all controlled pairs $(x,u)$}\\
\text{satisfying\ } x(t_0)\in\Image\wt\Phi_0,\ x(t_1)\in\Image\wt\Phi_1,
\end{split}
\end{equation}
and
\begin{equation}\label{eqn:P4}\tag{$\textrm{P}_4$}
\begin{split}
\text{minimise } \int_{t_0}^{t_1}\LL\big(x(t),u(t)\big)\dd t \text{ over all controlled pairs $(x,u)$ for which }\\
\text{the homotopy class of $x(t)$ equals $[\sigma]$ relatively to the images $\Image\wt\Phi_0$ and $\Image\wt\Phi_1$}.
\end{split}
\end{equation}
Here $\sigma$ is a fixed path in $N$, and we say that two paths $\sigma_0$ and $\sigma_1$ are homotopic relatively to the images $\Image\wt\Phi_0$ and $\Image\wt\Phi_1$ if there exists a homotopy joining $\sigma_0$ and $\sigma_1$ with the end-points in  $\Image\wt\Phi_0$ and $\Image\wt\Phi_1$. As before the time interval $[t_0,t_1]$ is not fixed. 

Problem \eqref{eqn:P3} has a form of the standard OCP on the manifold $N$ (one usually assumes that $\wt\Phi_0$ and $\wt\Phi_1$ are immersions). Problem \eqref{eqn:P4}, in turn, is equivalent to the OCP \eqref{eqn:P_rel} for a control system on the tangent algebroid $\T N$ for algebroid morphisms $\Phi_0=\T\wt\Phi_0:\T S_0\lra\T N$ and $\Phi_1=\T\wt\Phi_1:\T S_1\lra\T N$.

Analogously as before, we can also think of $N$ as of an $\alpha$-fibre $\GG_{y_0}$ of a groupoid $\GG$, with the control system \eqref{eqn:cs_test} and the cost function being $\GG$-equivariant and reducing to \eqref{eqn:con_sys} and $L$. If now $\wt\Phi_0$ and $\wt\Phi_1$ are algebroid morphisms $\Phi_0:\T S_0\lra\AG$ and $\Phi_1:\T S_1\lra\AG$ lifted to $\GG_{y_0}=N$ (cf. Corollary \ref{cor:int}), then \eqref{eqn:P4} on $N$ is equivalent to \eqref{eqn:P_rel} on $E=\AG$. 

Relation between \eqref{eqn:P3} and \eqref{eqn:P4} is analogous to the relation of \eqref{eqn:P1} and \eqref{eqn:P2}:
\begin{itemize}
\item Every solution of \eqref{eqn:P3} is a solution of \eqref{eqn:P4} for some class $[\sigma]\modulo(\wt\Phi_0,\wt\Phi_1)$.
\item The solution of \eqref{eqn:P3} is this solution of \eqref{eqn:P4} which has a minimal total cost of all solutions of \eqref{eqn:P4} for all possible classes $[\sigma]\modulo(\wt\Phi_0,\wt\Phi_1)$.
\item The solutions of \eqref{eqn:P3} and \eqref{eqn:P4} are not distinguishable by the PMP.
\end{itemize} 

Problems \eqref{eqn:P3} and \eqref{eqn:P4} are equivalent if $N$ is simply connected. This fact may not be obvious at first. It can be deduced from the following lemma.
\begin{lemma} Let $N$ be a simply connected manifold, and let $S_0,S_1\subset N$ be two path-connected subsets. Choose paths $\sigma_0,\sigma_1:[0,1]\lra N$ such that $\sigma_0(0),\sigma_1(0)\in S_0$ and $\sigma_0(1),\sigma_1(1)\in S_1$. Then there exists a homotopy in $N$ joining $\sigma_0$ with $\sigma_1$ which has its end-points in $S_0$ and $S_1$. 
\end{lemma}
\begin{proof} By path-connectedness of $S_0$ and $S_1$, there exists a path $b_0:[0,1]\lra S_0$ joining $\sigma_0(0)$ and $\sigma_1(0)$, and a path $b_1:[0,1]\lra S_1$ joining $\sigma_0(1)$ and $\sigma_1(1)$. Denote by $\wt b_0(t):=b_0(1-t)$ and $\wt b_1(t):=b_1(1-t)$ the inverse paths of $b_0$ and $b_1$.  

Now path $\sigma_0$ and the composition $b_0\ast\sigma_1\ast \wt b_1$ are homotopic with fixed-end-points in $N$, since they have the same end-points and $N$ is simply connected (by $\ast$ we denote the concatenation of paths). Let $H:[0,1]\times[0,1]\lra N$ be the appropriate homotopy. Consider homotopies $H_0:[0,1]\times[0,1]\lra S_0$ and $H_1;[0,1]\times[0,1]\lra S_1$ defined by the formulae $H_0(t,s):=b_0((1-t)s)$ and $H_1(t,s)=b_1(ts)$. It is straightforward to verify that the composition of homotopies $\wt H:=H_0\ast H\ast H_1$ makes sense, and it is a homotopy joining $\wh\sigma_0:=c_{\sigma_0(0)}\ast\sigma_0\ast c_{\sigma_0(1)}$ with $\wh\sigma_1:=\wt b_0\ast b_0\ast\sigma_1\ast\wt b_1\ast b_1$ (here $c_x$ stands for a constant path equal $x\in N$). Moreover, the initial-point homotopy of $\wt H$ is $H_0(0,s)=b_0(s)\in S_0$ and the final-point homotopy is $H_1(1,s)=b_1(s)\in S_1$, i.e., the paths $\wh\sigma_0$ and $\wh\sigma_1$ are homotopic relative to $S_0$ and $S_1$. 

To finish the proof observe that $\wh\sigma_0$ is homotopic (with fixed end-points) to $\sigma_0$ and $\wh\sigma_1$ to $\sigma_1$. \end{proof}

\subsection{OCPs in terms of the product algebroid \texorpdfstring{$E\times\T\R$}{TEXT}}
For a control system on an AL algebroid $E$, similar to the classical
situation of the tangent algebroid $E=\T M\ra M$, there is an elegant formulation of the OCPs \eqref{eqn:P} and \eqref{eqn:P_rel} in terms of the product algebroid
$E\times\T\R$. The idea is to incorporate the cost function into
the control system \eqref{eqn:con_sys}.

Denote by $\bm{A}$ the product algebroid structure on
$\bm\tau=(\tau_E,\tau_{\T\R}): E\times\T\R\lra M\times\R$ (see Chapter \ref{ch:algebroids}). We
will consequently use bold letters to emphasise objects associated
with $\bm{A}$, whereas objects associated with the $\T\R$-component of
$\bm A$ will be distinguished by underlining. For example,
$\bm{x}=(x,\ul x)\in M\times\R$ and $\bm{a}=(a,\ul a)\in
E\times\T\R=\bm A$.

Introduce now a new variable $\ul x\in\R$ and, for a given admissible control
$u\in\Uadm$, consider the following extension of the differential
equation \eqref{eqn:con_sys}:
\begin{equation}\label{eqn:con_sys_A}
\left\{ \begin{aligned}
\dot{x}(t)&=\rho\left(f(x(t),u(t))\right),\\
\dot{\ul x}(t)&=L\left(x(t),u(t)\right).
\end{aligned}\right.
\end{equation}
Clearly, $\ul x(t_1)-\ul x(t_0)=\int_{t_0}^{t_1}L(x(t),u(t))\dd t$
is the total cost of the controlled pair $(x(t),u(t))$ of
\eqref{eqn:con_sys}. Equation \eqref{eqn:con_sys_A} is a
differential equation associated with the following control system
on $\bm A$:
\begin{equation}
\label{def:con_sys_A} \bm{f}=(\wt f,\ul f):(M\times\R)\times U\lra
E\times\T\R=\bm{A},
\end{equation}
where $\ul f\left((x,\ul x),u\right):=\left(\ul
x,L(x,u)\right)\in\R\times\R\approx\T\R$ and $\wt{f}\left((x,\ul
x),u\right):=f(x,u)\in E$. For a given $u\in\Uadm$ the base
trajectory $\bm x(t)=(x(t),\ul x(t))$ of \eqref{eqn:con_sys_A}
contains information on both the base trajectory $x(t)$ of
\eqref{eqn:con_sys} (for the same control $u$) and the total cost
of the controlled pair $(x(t),u(t))$. Observe that the trajectory
$\bm f(\bm x(t),u(t))$ of \eqref{eqn:con_sys_A} projects onto the
trajectory $f(x(t),u(t))$ of \eqref{eqn:con_sys} under the
canonical algebroid projection $p_E:\bm A=E\times\sT\R\lra E$.
Now the OCP \eqref{eqn:P} can be reformulated in terms of control
system \eqref{def:con_sys_A} as follows:
\begin{align}\label{eqn:P_A}\tag{\textbf{P}}
\begin{split}
&\text{minimise } \ul x(t_1)
\text{ over controlled pairs }\ (\bm x(t),u(t))=\left((x(t),\ul x(t)),u(t)\right)\\
& \text{ of \eqref{eqn:con_sys_A} satisfying the following:}\\
&1.\text{ the $E$-projection $f(x(t),u(t))$ of the trajectory $\bm f(\bm x(t),u(t))$}\\
 &\text{belongs to a fixed $E$-homotopy class $[\sigma]$;}\\
&2.\text{ }\ul x(t_0)=0.
\end{split}
\end{align}

\noindent Similarly, the OCP \eqref{eqn:P_rel} can be expressed in the following way:
\begin{align}\label{eqn:P_A_rel}\tag{\textbf{P rel}}
\begin{split}
&\text{minimise } \ul x(t_1)
\text{ over controlled pairs }\ (\bm x(t),u(t))=\left((x(t),\ul x(t)),u(t)\right)\\
& \text{ of \eqref{eqn:con_sys_A} satisfying the following:}\\
&1.\text{ the $E$-projection $f(x(t),u(t))$ of the trajectory $\bm f(\bm x(t),u(t))$}\\
 &\text{belongs to a fixed relative $E$-homotopy class $[\sigma]\modulo(\Phi_0,\Phi_1)$;}\\
&2.\text{ }\ul x(t_0)=0.
\end{split}
\end{align}

The advantage of these new formulations of the OCPs \eqref{eqn:P} and \eqref{eqn:P_rel} presented here may seem unclear. The main reason is that the unified treatment of cost and controls simplifies some aspects of the proof of the PMP.

\subsection{The algebroid homotopy associated with a control system}

As has been observed in \cite{crainic_fernandes}, algebroid homotopies can be generated
by time-dependent algebroid sections. Since the control system
\eqref{eqn:con_sys} is a family of $E$-sections $f(\cdot,u)$,
fixing an admissible control $u(t)$ gives a time-dependent section
$f(\cdot,u(t))$.  The associated $E$-homotopy can be well
understood in terms of Lemma \ref{lem:gen_E_htp}.

Solving \eqref{eqn:con_sys} for a one-parameter family of
initial conditions $x(t_0,s)=x_0(s)$ produces a one-parameter
family of base paths $x(t,s)$. It follows from Theorem \ref{thm:exist} that, if the solution $x(t,0)$ is defined on $I=[t_0,t_1]$, then so is $x(t,s)$
at least for $x_0(s)$'s close enough to $x_0(0)$. With $x(t,s)$
we can associate a one-parameter family of trajectories
$$ a(t,s):=f\big(x(t,s),u(t)\big).$$
One easily sees that for \eqref{eqn:con_sys} the
assumptions of Theorem \ref{thm:param_dif} are satisfied.
Consequently, the base trajectories $x(t,x_0)$ are continuous
differentiable w.r.t. the initial condition $x_0$ (and ACB in
$t$). As $x(t,s)=x(t,x_0(s))$ if $s\mapsto x_0(s)$ is an ACB map,
we deduce that $x(t,s)$ is ACB w.r.t. the second variable; that
is, $\pa_s x(t,s)$ is a well-defined measurable function of both
variables. Consequently, the derivative $\pa_sa^i(t,s)=\frac{\pa
f^i}{\pa x^a}(x(t,s),u(t))\pa_sx^a(t,s)$ satisfies the assumptions
of Lemma \ref{lem:gen_E_htp}. Thus, the conclusions of Lemma
\ref{lem:gen_E_htp} hold; namely, for a given bounded measurable
$E$-path $b_0(s)$ covering $x_0(s)$, there exists a measurable map
(AC w.r.t. the first variable) $b:I\times[0,1]\lra E$ with
$b(t_0,s)=b_0(s)$ such that $(a,b)$ is an $E$-homotopy. The
$t$-evolution of $b(t,s)$ is given by \eqref{eqn:gen_htp}. Observe that, since
$\pa_sx(t,s)=\rho\left(b(t,s)\right)$, we have
$\pa_sa^i(t,s)=\frac{\pa f^i}{\pa
x^a}\left(x(t,s),u(t)\right)\rho^a_k\left(x(t,s)\right)b^k(t,s)$.
Consequently, $b(t,s)\sim\left(x^a(t,s),b^i(t,s)\right)$ is a
solution of the following differential equation
\begin{equation}\label{eqn:par_trans}
\left\{
\begin{aligned} \pa_tb^i(t,s)&=\frac{\pa
f^i}{\pa x^a}\Big(x(t,s),u(t)\Big)
\rho^a_k\big(x(t,s)\big)b^k(t,s)\\&+c^i_{jk}\big(x(t,s)\big)b^j(t,s)f^k\big(x(t,s),u(t)\big),\\
\pa_tx^a(t,s)&=\rho^a_i\big(x(t,s)\big)f^i\big(x(t,s),u(t)\big),
\end{aligned}\right.\end{equation}
with the initial conditions $b^i(t_0,s)=b_0^i(s)$ and
$x^a(t_0,s)=x^a_0(s)$.

 The above differential equation is well understood  in terms of the tools introduced in Chapter \ref{ch:algebroids}.
 For every $u\in U$, the section $f_u(\cdot):=f(\cdot,u):M\lra E$ gives rise to a linear vector field $\dd_\sT(f_u)$ on $E$.
 Evaluating it on $u(t)$ gives a time-dependent family of vector fields $\dd_\sT(f_{u(t)})$. Equation \eqref{eqn:par_trans}
 is simply the evolution along this family, $\pa_tb(t,s)=\dd_\sT\left(f_{u(t)}\right)(b(t,s))$.
On the other hand, with a time-dependent family of section
$f_u(t)$ we may associate the family of linear functions
$h_t(x,\xi):=\< f(x,u(t)),\xi>_{\tau}$ on $E^\ast$, and the
corresponding family of Hamiltonian vector fields $\X_{h_t}$. In
local coordinates,
$$\X_{h_t}(x,\xi)=\rho^b_j(x)f^j(x,u(t))\pa_{x^b}+\left(c^k_{ij}(x)f^i(x,u(t))\xi_k-
\rho^a_j(x)\frac{\pa f^i}{\pa
x^a}(x,u(t))\xi_i\right)\pa_{\xi_j}.$$ As we have seen in Chapter
\ref{ch:algebroids} (equations \eqref{eqn:ham_vf}--\eqref{eqn:hvf_tgl}),
the fields $\dd_\sT(f_{u(t)})$ and $\X_{h_t}$ give the same base
evolution (given by \eqref{eqn:con_sys}), and are related by
$\<\dd_\sT(f_{u(t)}),\X_{h_t}>_{\sT\tau}=0$.

\begin{definition}\label{def:op_par_tr}
The flows of the fields  $\dd_\sT(f_{u(t)})$ and $\X_{h_t}$ (for a
given $u\in\Uadm$) will be called \emph{operators of parallel
transport}\index{parallel transport} (in $E$ and $E^\ast$ respectively) along the solution $x(t)$ of the system \eqref{eqn:def_con_sys}. We will denote them
with $B_{tt_0}$ and $B^\ast_{tt_0}$, respectively. Analogously we
define operators  $\bm{B}_{tt_0}$ and ${\bm{B^\ast}}_{tt_0}$ for
the control system \eqref{def:con_sys_A}. Note that, by
construction, $B_{tt^{'}}\circ B_{t^{'}t_0}=B_{tt_0}$ and
$B^\ast_{tt^{'}}\circ B^\ast_{t^{'}t_0}=B^\ast_{tt_0}$.
\end{definition}

\begin{remark} \label{rem:B_htp}
Let us see that, by construction, the map
$b(t,s)=B_{tt_0}\left(b_0(s)\right)$ together with $a(t,s)$ forms
an $E$-homotopy. Moreover, $B_{tt_0}(b_0)$ is continuous w.r.t. $b_0$, $t$, and $t_0$. Indeed, $b(t)=B_{tt_0}(b_0)$ is
the solution of \eqref{eqn:par_trans} for $s=0$. The right-hand side is
measurable in $t$ and locally Lipschitz (linear) in $b$, so, by
Theorem \ref{thm:param}, $b(t)$ is AC w.r.t. $t$ and
continuous w.r.t. the initial condition $b_0$.
\end{remark}


\begin{remark} \label{rem:B_paring}
Note also that the operators $B$ and $B^\ast$ have the property of
preserving the parring $\<\cdot,\cdot>_{\tau}$; that is, for every
$a\in E_{x(t^{'})}$ and $\xi\in E^\ast_{x(t^{'})}$ over the same
base point ${x(t^{'})}\in M$,
$$\< B_{tt^{'}}(a),B^\ast_{tt^{'}}(\xi)>_{\tau}=\< a,\xi>_{\tau}\quad \text{ for every $t\in I$}.$$
Indeed, since by definition the pairing
$\<\cdot,\cdot>_{\sT\tau}:\sT E\times_{\sT M}\sT E^\ast\lra\R$ is
the tangent map of $\<\cdot,\cdot>_\tau:E\times_M E^\ast\lra\R$,
we have
\begin{align*}
&\pa_t\<B_{tt^{'}}(a),B^\ast_{tt^{'}}(\xi)>_{\tau}=\<\pa_tB_{tt^{'}}(a),\pa_tB^\ast_{tt^{'}}(\xi)>_{\sT\tau}\\&=
\<\dd_\sT(f_{u(t)})\left(B_{tt^{'}}(a)\right),\X_{h_t}\left(B_{tt^{'}}^\ast(\xi)\right)>_{\sT\tau}=0.
\end{align*}
\end{remark}\medskip

Finally, observe that the evolution of $\bm\xi(t)=\bm
{B^\ast}_{tt_0}(\bm\xi_0)$ for the control system
\eqref{def:con_sys_A} is trivial on the $\sT^\ast\R$-component.
Indeed, the associated linear Hamiltonian
$$\bm H_t(\bm x,\bm\xi)=\<\bm f(\bm x,u(t)),\bm\xi>_{\bm\tau}=\<f(x,u(t)),\xi>_\tau+\ul\xi L(x,u(t))$$
does not depend on the $\R$-component of $\bm x=(x,\ul x)\in
M\times\R$; hence $\bm\xi(t)=\left(\xi(t),\ul\xi(t)\right)\in\bm
{A^\ast}= E^\ast\times\T^\ast\R$ (in local coordinates,
$(\xi,\ul\xi)\sim \left((x^a,\xi_i),(\ul x,\ul\xi)\right)$)
evolves due to equations
\begin{equation}
\label{eqn:par_trans_*A} \left\{ \begin{aligned}
\pa_t\xi_k(t)=&-\rho^a_k\left(x\right)\left(\frac{\pa f^i}{\pa
x^a}\left(x,u(t)\right)\xi_i(t)+\frac{\pa L}{\pa
x^a}\left(x,u(t)\right)\ul\xi(t)\right)\\ &+c^i_{jk}\left(x\right)f^j\left(x,u(t)\right)\xi_i(t),\\
\pa_t{\ul\xi}(t)=&0,\\
\dot{x}(t)=&\rho\left(f(x(t),u(t))\right),\\
\dot{\ul x}(t)=&L\left(x(t),u(t)\right).
\end{aligned}\right.\end{equation}
In other words, $\ul\xi(t)\equiv\ul\xi_0$ is a constant and
$\xi(t)\in E^\ast$ evolves due to a time-dependent family of
Hamiltonian vector fields $\X_{H_t}$ on $E^\ast$, where
$H_t(x,\xi)=\<f(x,u(t)),\xi>_\tau+\ul\xi_0L(x,u(t))$.

\chapter{The Pontryagin Maximum Principle}\label{ch:pmp}

In the previous chapter we have introduced OCPs \eqref{eqn:P} and  \eqref{eqn:P_rel} in the AL algebroid setting. The main difference in comparison
with the classical formulation are, apart from using
algebroid-valued velocities, the fixed-homotopy boundary
conditions. This new formulation of the OCPs was motivated in the previous chapter for problems on an integrable algebroid $E=\AG$. In light of these considerations we may thing of \eqref{eqn:P} and \eqref{eqn:P_rel} as of a general framework which includes the standard OCPs as well as OCPs reduced by general (groupoid) symmetries. It contains also abstract problems on general AL algebroids.
Now we will formulate a generalisation of the Pontryagin maximum principle for the problems \eqref{eqn:P} and \eqref{eqn:P_rel}.

\subsection{Formulation on an AL algebroid \texorpdfstring{$E$}{E}}

\begin{theorem}\label{thm:pmp}\index{Pontryagin maximum principle|main}
Let $(x(t),u(t))$, with $t\in[t_0,t_1]$, be a controlled pair of
\eqref{eqn:con_sys} solving the optimal control problem
\eqref{eqn:P}. Then there exists a curve $\xi:[t_0,t_1]\lra
E^\ast$ covering $x(t)$ and a constant $\ul\xi_0\leq 0$ such that the following holds:
\begin{itemize}
    \item the curve $\xi(t)$ is a trajectory of the time-dependent family of Hamiltonian vector fields
    $\X_{H_t}$ associated with Hamiltonians $H_t(x,\xi):=H(x,\xi,u(t))$, where
    $$H(x,\xi,u)=\< f\left(x,u\right), \xi>_\tau+\ul\xi_0 L\left(x,u\right);$$
    \item the control $u$ satisfies the ``maximum principle''
    $$H(x(t),\xi(t),u(t))=\sup_{v\in U}H(x(t),\xi(t),v)$$
and $H(x(t),\xi(t),u(t))=0$ at every regular point $t$ of $u$;
    \item if $\ul\xi_0=0$, the covector $\xi(t)$ is nowhere-vanishing.
\end{itemize}
\end{theorem}

The above result clearly reduces to the standard PMP (Theorem \ref{thm:pmp_class}) for the case of the tangent algebroid $E=\T M\lra M$. It also covers the known results for system with symmetry on Lie groups and, more generally, principal bundles. A more detailed discussion and examples will be given in Chapter \ref{ch:examples}.

Consider smooth algebroid morphisms $\Phi_0:\T S_0\lra E$ over $\phi_0:S_0\lra M$ and $\Phi_1:\T S_1\lra E$ over $\phi_1:S_1\lra M$. For a relative OCP \eqref{eqn:P_rel} we have the following version of the PMP.

\begin{theorem} \label{thm:pmp_rel} \index{Pontryagin maximum principle! for general boundary conditions|main}
Let $(x(t),u(t))$, with $t\in[t_0,t_1]$, $x(t_0)=\phi_0(z_0)$, and $x(t_1)=\phi_1(w_0)$, be a controlled pair of
\eqref{eqn:con_sys} solving the optimal control problem
\eqref{eqn:P_rel}. Then there exists a curve $\xi:[t_0,t_1]\lra
E^\ast$ covering $x(t)$ and a constant $\ul\xi_0\leq 0$ which satisfy the assertion of Theorem \ref{thm:pmp} and, additionaly
$\xi(t_0)$ annihilates  $\Phi_0\left(\T_{z_0}S_0\right)$ and
$\xi(t_1)$ annihilates  $\Phi_1\left(\T_{w_0}S_1\right)$.
\end{theorem}

It is clear that Theorem \ref{thm:pmp_rel} agrees with the standard PMP for problems with general boundary conditions (Theorem \ref{thm:pmp_class_rel}) for the special case of the tangent algebroid $E=\T M$. We can regard this result as an extension of the standard PMP to systems with general (groupoid) symmetries. In particular, it covers the known results for symmetric systems on Lie groups and principal bundles (see Section \ref{sec:pmp_gpm}). 

Obviously, Theorem \ref{thm:pmp} is a special case of Theorem \ref{thm:pmp_rel} obtained for $\Phi_0$ and $\Phi_1$ trivial.

\subsection{Alternative formulation}
Theorems \ref{thm:pmp} and \ref{thm:pmp_rel} have equivalent formulations in terms of the product algebroid $\bm A=E\times\T\R$.

\begin{theorem}\label{thm:pmp_A} Let $(\bm x(t),u(t))$, with $t\in[t_0,t_1]$, be a controlled pair of\index{Pontryagin maximum principle}
\eqref{eqn:con_sys_A} solving the optimal control problem
\eqref{eqn:P_A}. There exists a nowhere-vanishing curve $\bm
\xi=(\xi,\ul\xi):[t_0,t_1]\lra\bm A^\ast= E^\ast\times\sT^\ast\R$
covering $\bm x(t)$, with $\ul\xi(t_1)\leq 0$, such that the following hold:
\begin{itemize}
    \item the curve $\bm\xi(t)$ is a trajectory of the time-dependent family of Hamiltonian vector fields
    $\X_{\bm H_t}$, for
    $\bm H_t(\bm x,\bm\xi):=\bm H(\bm x,\bm\xi,u(t))$, where
    $$\bm H(\bm x,\bm\xi,u)=\<\bm f\left(\bm x,u\right), \bm \xi>_{\bm\tau};$$
    \item the control $u$ satisfies the ``maximum principle''
    $$\bm H(\bm x(t),\bm \xi(t),u(t))=\sup_{v\in U}\bm H(\bm x(t),\bm\xi(t),v)=0$$
at every regular point $t$ of $u$.
\end{itemize}
\end{theorem}

The equivalence of Theorems \ref{thm:pmp} and \ref{thm:pmp_A} is obvious in light of our previous considerations. Indeed, the covector $\bm\xi(t)$ can
be decomposed as $\left(\xi(t),\ul\xi(t)\right)$, and its evolution
along $\X_{\bm H_t}$ is given by \eqref{eqn:par_trans_*A}.
Consequently, as we have observed at the very end of Chapter
\ref{ch:OCP}, covector $\ul\xi(t)=\ul\xi_0$ is constant and the
evolution of $\xi(t)$ is given by $\X_{H_t}$.  Since
$H(x,\xi,u)+\ul{\xi_0}L(x,u)=\bm H(\bm x,\bm\xi,u)$ for $\bm\xi=(\xi,\ul\xi_0)$, and
$\bm x=(x,\ul x)$, the corresponding statements in Theorems
\ref{thm:pmp} and \ref{thm:pmp_A} are equivalent.

Define now $\eS_0:=\Phi_0(\T_{z_0}S_0)\subset E_{x(t_0)}$ and $\eS_1:=\Phi_1(\T_{w_0}S_1)\subset E_{x(t_1)}$, where $x(t_0)=\phi_0(z_0)$ and $x(t_1)=\phi_1(w_0)$. We have the following reformulation of Theorem \ref{thm:pmp_rel}.

\begin{theorem} \label{thm:pmp_A_rel}\index{Pontryagin maximum principle! for general boundary conditions}
Let $(\bm x(t),u(t))$, with $t\in[t_0,t_1]$  be a controlled pair of
\eqref{eqn:con_sys_A} solving the optimal control problem
\eqref{eqn:P_A_rel}. Then there exists a nowhere-vanishing  curve $\bm\xi:[t_0,t_1]\ra
\bm A^\ast$ covering $\bm x(t)$, which satisfies the assertion of Theorem \ref{thm:pmp_A} and, additionally, $\bm \xi(t_0)$ annihilates  $\eS_0\oplus\theta_{\ul x(t_0)}$ and
$\bm \xi(t_1)$ annihilates $\eS_1\oplus\theta_{\ul x(t_1)}$.
\end{theorem}

\noindent The equivalence of Theorems \ref{thm:pmp_A_rel} and \ref{thm:pmp_A} is clear.

\begin{remark}\label{rem:pmp_versions}
There are many different versions of the PMP --- for autonomous and
non-autonomous systems, with mowing or fixed end-points, with free
or fixed time interval $[t_0,t_1]$, etc. Theorem \ref{thm:pmp_A_rel}, in fact, covers all these cases (under certain regularity conditions). For details see Section \ref{sec:pmp_versions}.
\end{remark}

\chapter{Discussion of the main result}\label{ch:examples}

This part is devoted to the discussion of our main results formulated in the previous chapter. We begin with formulating and proving a version of the PMP on AL algebroids for non-autonomous systems. In section \ref{sec:pmp_gpm} we formulate a version of the PMP on for invariant OCPs on Lie groups and principal bundles. We derive Montgomery's falling cat problem as an illustration. Later we use the PMP to derive the known results form the calculus of variation on Lagrangian reduction, Hammel equations and Euler--Poincar{\'e} equations. We also formulate an algebroid analog of Euler-Lagrange equations. Finally, in section \ref{sec:examples_other} we give a few concrete examples of the usage of our results.

\section{Non-autonomous versions of the PMP}\label{sec:pmp_versions}

In \cite{pontryagin} analogs of the PMP for other versions of the classical OCP \eqref{eqn:P_class} (including fixing the time interval, or changing the setting to the non-autonomous systems) were obtained. These extensions were proved by a clever reformulation of a problem given in order to make it a special case of the already known solution. Now we perform similar derivations for the extensions of the problems \eqref{eqn:P} and \eqref{eqn:P_rel}. 

Consider a non-autonomous version of the control system\index{control system on AL algebroid!non-autonomous} \eqref{eqn:con_sys} on an AL algebroid $E$
\begin{equation}\label{eqn:cs_na}
\dot x(t)=\rho\left(f(x(t),t,u(t))\right),
\end{equation}
where $f:M\times\R\times U\lra E$ is a time-dependent family of $C^1$-sections of $E$. Moreover, we assume that $f$ is continuous w.r.t. all variables, differentiable w.r.t. $x$ and $t$, and that the derivative $\T_{(x,t)}f$ is also continuous w.r.t. all variables. 

Let $L:M\times\R\times U\lra\R$ be a non-autonomous \emph{cost function}\index{cost function!non-autonomous} satisfying the same regularity assumptions as $f$. Per analogy to definitions introduced in Chapter \ref{ch:OCP} we will speak of \emph{trajectories}\index{trajectory of a control system!non-autonomous} $f(x(t),t,u(t))$ and \emph{extended controlled pairs}\index{controlled pair!extended} $(x(t),t,x(t))$ of \eqref{eqn:cs_na}.

 Consider now the following generalisation of the OCP \eqref{eqn:P}\index{optimal control problem!non-autonomous} :
\begin{equation}\tag{$\mathrm{P_{na}}$}\label{eqn:P_na}
\begin{split}
&\text{minimise} \int_{t_0}^{t_1}L\big(x(t),t,u(t)\big)\dd t
\text{ over all extended controlled pairs $(x,t,u)$}\\
&\text{of \eqref{eqn:cs_na} s.t. the $E$-homotopy class of the trajectory $f(x(t),t,u(t))$
equals $[\sigma]$.}
\end{split}
\end{equation}
Here $[\sigma]$ is a fixed $E$-homotopy class. We allow the time interval $[t_0,t_1]$ either to be fixed or to be unspecified. 

We can define also a relative version of the above problem by substituting the fixed-$E$-homotopy class $[\sigma]$ by a fixed relative-$E$-homotopy class $[\sigma]\modulo(\Phi_0,\Phi_1)$, for a pair of smooth algebroid morphisms $\Phi_0:\T S_0\lra E$ and $\Phi_1:\T S_1\lra E$. 

\noindent For the above non-autonomous OCPs we have the following extension of Theorem \ref{thm:pmp}.
\begin{theorem}\label{thm:pmp_na}\index{Pontryagin maximum principle!non-autonomous}
Let $(x(t),t,u(t))$, with $t\in[t_0,t_1]$, be an extended controlled pair of
\eqref{eqn:cs_na} solving the optimal control problem
\eqref{eqn:P_na}. Then there exists a curve $\xi:[t_0,t_1]\lra
E^\ast$ covering $x(t)$ and a constant $\ul\xi_0\leq 0$ such that the following holds:
\begin{itemize}
    \item the curve $\xi(t)$ is a trajectory of the time-dependent family of Hamiltonian vector fields
    $\X^E_{H_t}$ associated with Hamiltonians $H_t(x,\xi):=H(x,t,\xi,u(t))$, where
    $$H(x,t,\xi,u)=\< f\left(x,t,u\right), \xi>_\tau+\ul\xi_0 L\left(x,t,u\right);$$
    \item the control $u$ satisfies the ``maximum principle''
    $$H(x(t),t,\xi(t),u(t))=\sup_{v\in U}H(x(t),t,\xi(t),v)$$
and $H(x(t),t,\xi(t),u(t))-\int_{t_0}^t\frac{\pa H}{\pa s}(x(s),s,\xi(s),u(s))\dd s=\mathrm{const}$ at every regular point $t$ of $u$. If the time interval $[t_0,t_1]$ is unspecified then this constant is 0;
    \item if $\ul\xi_0=0$, the covector $\xi(t)$ is nowhere-vanishing.
\end{itemize}
\end{theorem}
\begin{proof}
The proof requires a simple reformulation of the given data. Consider, namely, the control system $$\wt f=(f,1):M\times\R\times U\lra E\times\T R$$ on the product algebroid $E\times\T R$ with the associated base dynamics 
\begin{align*}
\dot x(t)&=\rho\left(f(x(t),z(t),u(t))\right),\\
\dot z(t)&=1.
\end{align*}
Here $(x,z)\in M\times\R$. The variable $z$ plays a role of time. Indeed, regardless of the chosen admissible control $u(t)$, the solution of the equation $\dot z(t)=1$ with the initial condition $z(t_0)=t_0$ gives $z(t)=t$. 

Now, if the time interval $[t_0,t_1]$ is fixed, OCP \eqref{eqn:P_na} is equivalent to the OCP \eqref{eqn:P} for the control system $\wt f$ with unspecified time interval and the $E\times\T\R$-homotopy class defined by $[\sigma]$ on $E$ and $[t_0,t_1]$ on $\T\R$. Consequently, we can use Theorem \ref{thm:pmp} to obtain the necessary conditions for optimality.

For $(\xi,\alpha)\in E_x^\ast\times\T_z^\ast\R$ we define the Hamiltonian 
\begin{align*}\wt H(x,z,\xi,\alpha,u):&=\<\wt f(x,z,u),(\xi,\alpha)>+\xi_0L(x,z,u)=\\
&=\<f(x,z,u),\xi>_\tau+\xi_0L(x,z,u)+1\cdot\alpha=:H(\xi,z,u)+\alpha.
\end{align*}
Now the evolution of the Pontryagin covector $(\xi(t),\alpha(t))$ reads as
\begin{align*}
&\dot \xi(t)=\X^E_{H(\cdot,z,u(t))}(\xi(t))\\
&\dot \alpha(t)=-\frac{\pa}{\pa z} H(\xi(t),z,u(t)).
\end{align*}
Since $z(t)=t$ and $\wt H(x(t),z(t),\xi(t),\alpha(t),u(t))=0$ at regular $t$, we get
\begin{align*}
&\alpha(t)=c-\int_{t_0}^t\frac{\pa}{\pa s} H(\xi(s),s,u(s))\dd 
\intertext{and} 
&H(\xi(t),t,u(t))+\alpha(t)=0\quad\text{at $t$ regular}.
\end{align*}
The maximum principle for $H$ follows directly from the maximum principle for $\wt H$. Finally, if $\xi_0=0$ and $\xi(t)=0$, we would have $H(\xi(t),t,u(t))=0$, and hence also $\alpha(t)=0$ a.e., which is impossible. 
This proves the assertion. 

The proof for the case of unspecified time-interval $[t_0,t_1]$ is analogous, yet instead of fixed-homotopy class boundary conditions we have to general boundary conditions associated with algebroid morphisms $\Phi_0=(\theta_{x_0},\id):\T\R\lra E\times\T\R$ and $\Phi_1=(\theta_{x_1},\id):\T\R\lra E\times\T\R$. The additional condition $c=0$ now follows from the transversality conditions of Theorem \ref{thm:pmp_rel} for $\alpha(t_0)$ and $\alpha(t_1)$. 
\end{proof}

\section{The known results}\label{sec:pmp_gpm}

\subsection{The PMP on Lie groups and principal bundles}
The already proven results on the Lie groupoid---Lie algebroid
reduction of a control system and homotopy (cf. Theorem
\ref{thm:int_htp}, Corollary \ref{cor:htp_P}, and Chapter
\ref{ch:OCP}) allow us to formulate the following result which can
be understood as a general reduction scheme of the PMP on a
principal $G$-bundle. Consider a principal $G$-bundle $G\ra
P\overset \pi\ra M$.

\begin{theorem}\label{thm:OCP_P}
Let $F:P\times U\lra \sT P$ be a $G$-invariant control system on
$P$ and let $L:P\times U\lra\R$ be a $G$-invariant cost function. Choose a path $\Sigma:[t_0,t_1]\lra P$ joining two fixed points $p_0,p_1\in P$.

Then the OCP \eqref{eqn:P2} on $P$ for a fixed homotopy class $[\Sigma]$ in $P$ is equivalent to the OCP \eqref{eqn:P} for the system
$f:M\times U \lra E=\sT P/G$ where $f(\pi(p),u):=[F(p,u)]$, with
the cost function $l:M\times U\lra\R$ defined by
$l(\pi(p),u):=L(p,u)$, and the $E$-homotopy class $[\sigma]$ being
the reduction  of $[\Sigma]$.\end{theorem}

Now applying Theorem \ref{thm:pmp} to the OCP described above we obtain a general result for equivariant OCP on principal bundles described in therms of the linear Poisson  structure $\Pi_{E^\ast}$ on $E^\ast=\T^\ast P/G$ (cf. Theorem \ref{thm:ham_E1}).

\begin{theorem}\label{thm:pmp_}\index{Pontryagin maximum principle!for principal bundles}
Let $(p(t),u(t))$, with $t\in[t_0,t_1]$, be a controlled pair of
$F$ solving the OCP described above. Denote by $x(t)$ the base projection of $p(t)$. Then there exists a curve $\xi:[t_0,t_1]\lra
\T^\ast P/G$ covering $x(t)$ and a constant $\ul\xi_0\leq 0$ such that the following holds:
\begin{itemize}
    \item the curve $\xi(t)$ is a trajectory of the time-dependent family of Hamiltonian vector fields
    $\X_{h_t}$ associated with the linear Poisson structure $\Pi_{E^\ast}$ on $\T^\ast P/G$ and Hamiltonians $h_t(x,\xi):=h(x,\xi,u(t))$, where
    $$h(x,\xi,u)=\< f\left(x,u\right), \xi>_\tau+\ul\xi_0 l\left(x,u\right);$$
    \item the control $u$ satisfies the ``maximum principle''
    $$h(x(t),\xi(t),u(t))=\sup_{v\in U}h(x(t),\xi(t),v)$$
and $h(x(t),\xi(t),u(t))=0$ at every regular point $t$ of $u$;
    \item if $\ul\xi_0=0$, the covector $\xi(t)$ is nowhere-vanishing.
\end{itemize}
\end{theorem}
\noindent An analogous result (with additional transversality conditions) is valid for system with general boundary conditions.
 
Note that for the case of a Lie group ($P=G$, $M=\{\ast\}$, $\Pi_{E^\ast}=\Pi_{\g^\ast}$) we recover the results of Jurdjevic \cite[Ch.12, Thms 5,6]{jurdjevic}.

\subsection{An application---the falling cat problem}
Now we will reconsider the well-known results of Montgomery
\cite{montgomery_isohol} (see also \cite[ch. 7.1]{bloch} and \cite{cendra_holm_marsden}) on
the isoholonomic problem by means of the PMP in the Atiyah
algebroid setting.

Let $G\ra P\ra M$ be a principal $G$-bundle, let $\cH\subset \sT P$ be
a $G$-invariant horizontal distribution, and let $\mu(\cdot,\cdot)$ be
a $G$-invariant sub-Riemannian metric on $\cH$ ($\mu(\cdot,\cdot)$
can be understood as a base metric lifted to $\cH$ by the
horizontal lift). The problem is now to find a horizontal curve
$q(t)$ with $t\in[0,1]$ joining two fixed points $q_0$, $q_1$ in
$P$ and minimizing the total energy
$$\frac 12\int_0^1\mu(\dot q(t),\dot q(t))\dd t.$$

Clearly, due to the $G$-invariance of the problem, after changing
the fixed-end-point condition into a fixed-homotopy condition (as
discussed in detail in Chapter \ref{ch:OCP}), the above problem
is equivalent to an OCP of the form \eqref{eqn:P} on the Atiyah
algebroid $E=\T P/G$.

With the invariant distribution $\cH$, understood as a principal
connection, we can associate a map $\nabla:\sT M\ra E$ inducing a
splitting $E\simeq_\nabla\sT M\times\g$. Our control system will
be $f:\sT M\lra \sT M\times\g$ given by $f(X)=(X,0)$ (this assures
that the trajectory is horizontal), the cost function $L:\sT
M\lra\R$ reads as $L(X)=\frac 12 \mu(X,X)$, and the fixed
$E$-homotopy class is simply a reduction of a classical homotopy
class in $P$.

Note two differences with the formulation of the OCP
\eqref{eqn:P}. Firstly, our control and cost functions have
arguments in $\sT M$ instead of in $M\times U$. Of course, this makes
no big difference, since locally $\sT M\approx M\times\R^n$.
Secondly, our time interval is fixed. This, in turn, results in
substituting the condition $H(x(t),\xi(t),u(t))=0$ by
$H(x(t),\xi(t),u(t))=\mathrm{const}$ in the assertion of Theorem
\ref{thm:pmp} (cf. Section \ref{sec:pmp_versions}).

Now we can apply Theorem \ref{thm:pmp} with the Hamiltonian
evolution described in Theorem \ref{thm:ham_E1}. The covector
$\xi\in E^\ast$ can be decomposed as $\xi=(p,\zeta)\in\sT^\ast
M\times\g^\ast$, and the corresponding Hamiltonian is
$$H(p,\zeta,X)=\<X,p>+\frac 12\lambda_0\mu(X,X)=:h(p,X),$$
with $\lambda_0\leq 0$. The maximum principle reads as
$p(t)=-\lambda_0\mu(X(t),\cdot)$; hence on the optimal trajectory,
$H(p(t),\zeta(t),X(t))=-\frac 12\lambda_0\mu(X(t),X(t))$ (which is
constant in $t$). The evolution of $p(t)$ and $\zeta(t)$ is given
by
\begin{align*}
&\dot\zeta(t)=0\,,\\
&\dot p(t)=\X^{\sT^\ast M}_{h(p,X)}+\<\xi(t),F_\nabla(X,\cdot)>\,;
\end{align*}
hence $\zeta(t)=\mathrm{const}$. The second equation is equivalent to
$$\lambda_0\nabla_X^\mu X=\<\zeta,F_\nabla(X,\cdot)>^{\#\mu},$$
where $\nabla^\mu$ denotes the Levi-Civita covariant derivative on
$(M,\mu)$ and $A^{\#\mu}$ is the vector dual to $A$ by means of
$\mu$. Indeed, the equation $\dot p(t)=\X^{\sT^\ast M}_{h(p,X)}$,
together with $p(t)=-\lambda_0\mu(X(t),\cdot)$, is the PMP for a
geodesic problem on $(M,\mu)$. Passing to the dual vector
$p(t)^{\#\mu}=\lambda_0X(t)$ we should obtain the geodesic
equation multiplied by the factor $\lambda_0$. The equation
$\dot\zeta(t)=0$ means that the curve $\zeta(t)\in\g^\ast$ is
covariantly constant, hence
$$\nabla_X\zeta=0.$$
We have thus obtained the Wong equations as in \cite{montgomery_isohol}.

The abnormal case $\lambda_0=0$ implies $p(t)=0$ and
$\<\zeta,F_\nabla(X,\cdot)>=0$. This allows us to exclude abnormal
solutions in certain situations. For example, if $P$ is a bundle
of circles over a two-dimensional base and the connection is
non-integrable (i.e., $F_\nabla$ is non-vanishing), we have
$\<\zeta,F_\nabla(X,\cdot)>=0$ if and only if $X=0$ (hence the
solution is trivial) or $\zeta=0$, which can be excluded by the
non-vanishing of the covector in the PMP.

\subsection{Applications to variational problems}

It is a well-known fact that Euler-Lagrange equations can be derived by means of the classical PMP if one considers a trivial control system on a manifold $M$
$$f:M\times\R^n\underset{\text{loc}}\approx \T M\overset{\id}\lra\T M.$$
In this case, since we make no restrictions for velocities, the abnormal case can be excluded.

Similarly, for a trivial control system on a general AL algebroid
$$f:M\times\R^m\underset{\text{loc}}\approx E\overset{\id}\lra E,$$
we can obtain generalised Euler-Lagrange equations studied by many authors (see e.g. \cite{GG_var_calc} and the references therein). 

Indeed, for a system of the above form with a cost function $$L:M\times\R^m\underset{\text{loc}}\approx E\lra\R$$ consider the OCP \eqref{eqn:P} for some fixed homotopy class $[\sigma]$ and fixed time interval $[t_0,t_1]$ (cf. Section \ref{sec:pmp_versions}), and denote by $\gamma:[t_0,t_1]\lra E$ its solution (the controlled pair). In local coordinates $\gamma(t)\sim\left(x^a(t),y^i(t)\right)$. The associated time-dependent Hamiltonian $H:\R\times E^\ast\lra\R$ reads as
$$H(t,\xi)=\<\gamma(t),\xi>_\tau-\xi_0L(\gamma(t)).$$
The maximum principle 
$$H(t,\xi(t))=\sup_{e\in E_{\pi(\xi(t))}}\<e,\xi(t)>+\xi_0L(e)$$
implies that if $\xi_0=0$, then also $\xi(t)=0$, which is forbidden by the PMP.  
Consequently, we may assume that $\xi_0=-1$. Now the maximum condition implies that the Pontryagin covector $\xi(t)$ is a vertical part of the derivative $\dd L$ evaluated on $\gamma(t)$; i.e., $\xi(t)=\T^\ast\pi\left( \dd L(\gamma(t))\right)\sim\left(x^a(t),\frac{\pa L}{\pa y^i}(x(t),y(t))\right)$. 

The evolution equation reads as
\begin{equation}\label{eqn:ham}
\dot\xi(t)=\wt \Pi_{E^\ast}\left(\dd_\xi H(t,\xi(t))\right),
\end{equation}
where $\wt \Pi_{E^\ast}:\T^\ast E^\ast\lra\T E^\ast$ is induced by the linear bi-vector field $\Pi_{E^\ast}$. 

For our purposes it will be more convenient to describe the dynamics via the canonical double vector bundle isomorphism $\mathcal{R}^{-1}:\T^\ast E^\ast\lra T^\ast E$ (see \cite[Sec. 11]{mackenzie}), which in local coordinates reads as
$$\mathcal{R}^{-1}:(x^a,\xi_i,p_b,y^j)\mapsto(x^a,y^i,-p_b,\xi_j).$$
Since $\dd_\xi H(t,\xi)\sim(x^a(t),\xi^i(t),-\frac{\pa L}{\pa x^b}(x(t),y(t)),y^i(t))$, the image $\mathcal{R}^{-1}\left(\dd_\xi H(t,\xi(t))\right)$ is simply the derivative $\dd L$ evaluated at $\gamma(t)$. Equation \eqref{eqn:ham} can be thus expressed as
$$\frac{\dd}{\dd t}\T^\ast\tau(\gamma(t))=\eps\circ \dd L(\gamma(t)),$$
where $\eps:=\wt \Pi_{E^\ast}\circ\mathcal{R}:\T^\ast E\lra \T E^\ast$. This equation considered as an implicit differential equation for $\gamma(t)$ is precisely the generalised \emph{Euler--Lagrange equations}\index{Euler--Lagrange equations} considered in \cite{GG_var_calc,GGU_geom_mech}. In local coordinates it reads as
\begin{equation}\label{eqn:EL}
\begin{split}
&\frac{\dd x^a}{\dd t}=\rho^a_k(x)y^k\\
&\frac{\dd }{\dd t}\left(\frac{\pa L}{\pa y^j}\right)=c^k_{ij}(x)\frac{\pa L}{\pa y^k}+\rho^a_j(x)\frac{\pa L}{\pa x^a}.
\end{split}
\end{equation}

In a special case if $E=\g$ is a Lie algebra we recover the \emph{Euler--Poincar{\'e} equations}\index{Euler--Poincar{\'e} equations}
$$\frac{\dd }{\dd t}\left(\frac{\pa L}{\pa y}\right)=\operatorname{ad}^\ast_y\left(\frac{\pa L}{\pa y}\right)$$

More generally, for the Atiyah algebroid $E=\T P/G$, generalised Euler--Lagrange  equations \eqref{eqn:EL} take a from of \emph{Hammel equations}\index{Hammel equations} (if we use local trivialisation defined by a local section---see \eqref{eqn:bracket1}) and \emph{reduced Euler-Lagrange equations}\index{reduced Euler-Lagrange equations} (in local trivialisation given by a principal connection---see \eqref{eqn:bracket}) The interested reader should confront \cite[Sec. 5]{cendra_holm_marsden}.

\section{Other examples} \label{sec:examples_other}

\subsection{Two-point time OCP on \texorpdfstring{$\so$}{TEXT}}
Consider now a rigid body in $\R^3$ which can rotate with constant
angular velocity along two fixed axes in the body. At every moment
the position of the body is described by an element $q\in SO(3)$.
The rotation axes can be represented by elements of the Lie
algebra $l_+,l_-\in\so$. The rotation along the axis $l_\pm$ is
described by the equation
$$\pa_t q=ql_\pm.$$
It would be suitable to write $l_+=a+b$ and $l_-=a-b$. The above
equation can be regarded as a control system on the Lie group
$SO(3)$ with the control function $F(q,u)=qf(u)$, where
$f(u)=a+ub$ and the set of controls is simply $U=\{-1,1\}$. We
would like to find a control $u(t)$ which moves the body from a
position $q_0\in SO(3)$ to $q_1\in SO(3)$ (or such that the
trajectory belongs to a fixed homotopy class in $SO(3)$) in the
shortest possible time.

It is obvious that the above OCP on the Lie
group reduces to the OCP on the Lie algebra
$\so$ with the control function $f$ and the cost function $L\equiv
1$. Fix a basis $(e_1,e_2,e_3)$ on $\so$, and denote by
$c^\alpha_{\beta\gamma}$ the structure constants of the Lie
algebra in this basis. Let $u(t)$, for $t\in[t_0,t_1]$, be a
solution of the above OCP. It follows from
theorem \ref{thm:pmp} that there exist a number $\lambda_0\leq 0$
and a curve $\zeta(t)\in \so^*$ such that
$$H(\zeta(t),u(t))=\< \zeta(t),a+u(t)b>+\lambda_0=\max_{v=\pm 1}\< \zeta(t),a+vb>+\lambda_0.$$
This implies that $u(t)=\operatorname{sgn}\big(\<
\zeta(t),b>\big)$. Moreover, the evolution of $\zeta(t)$ is given
by the equation
$$\pa_t \zeta_\beta(t)=c^\gamma_{\alpha\beta}(a^\alpha+u(t)b^\alpha)\zeta_\gamma(t).$$
We have obtained the same equation as in (\cite[Sec.
19.4]{agrachev}). We refer the reader to this book for the detailed discussion on
solutions.

\subsection{An application to a nonholonomic system}

In \cite{GGU_geom_mech} and \cite{GG_var_calc} a framework of geometric mechanics on
general algebroids was presented. Roughly speaking, the structure
of an algebroid on a bundle $\tau:E\ra M$ allows one to develop
Lagrangian formalism for a given Lagrangian function $L:E\ra\R$.
Moreover, if $E$ is an AL algebroid, then the associated
Euler-Lagrange equations have a variational interpretation: a
curve $\gamma:[t_0,t_1]\ra E$ satisfies the Euler-Lagrange
equations if and only if it is an extremal of the action
$\mathcal{J}(\gamma):=\int_{t_0}^{t_1}L(\gamma(t))\dd t$
restricted to those $\gamma$'s which are admissible and belong to
a fixed $E$-homotopy class \cite{GG_var_calc}.  Hence, the trajectories of
the Lagrange system should be derivable from our version of the PMP
for the unconstrained control system on $E$ with the cost function
$L$.

In \cite{grabowski_nonholonomic} it has been shown that if $D\subset E$ is a
subbundle and $L$ is of mechanical type (that is, $L(a)=\frac 12
\mu(a,a)-V(\tau(a))$, where $\mu$ is a metric on $E$ and $V$ is an
arbitrary function on the base), then nonholonomically constrained
Euler-Lagrange equations associated with $D$ can be obtained as
unconstrained Euler-Lagrange equations on the skew-algebroid
$\left(D,\rho_E|_D,[\cdot,\cdot]_D:=\PP_D[\cdot,\cdot]_E\right)$,
where $\PP_D:E\ra D$ denotes the projection orthogonal w.r.t.
$\mu$. It follows that if $D$ with the algebroid structure
defined above is AL, then the solutions of the nonholonomically
constrained Euler-Lagrange equations are extremals of the
unconstrained OCP on $D$ with the cost function $L|_D$. On the
other hand, using our version of the PMP on the algebroid $E$ with
controls restricted to $D$ and the cost function $L$, one will
obtain nonholonomically constrained Euler-Lagrange equations
associated with $D$. Note that the algebroid bracket
$[\cdot,\cdot]_D$ need not satisfy Jacobi identity even if
$[\cdot,\cdot]_E$ does. Concluding, the PMP on general (not
necessarily Lie) AL algebroids can be used in the theory of
nonholonomic systems. To our knowledge this point of view is
completely novel.

To give a concrete example we will use PMP to study the Chaplygin
sleigh. It is an example of a nonholonomic system on the Lie
algebra $\mathfrak{se}(2)$ which describes a rigid body sliding on
a plane. The body is supported in three points, two of which
slide freely without friction, while the third point is a knife
edge. This imposes the constraint of no motion orthogonal to this
edge (see \cite{chaplygin,neimark}).

The configuration space before reduction is the Lie group
$G=SE(2)$ of the Euclidean motions of the two-dimensional plane
$\R^2$. Elements of the Lie algebra $\mathfrak{se}(2)$ are of the
form
$$\hat{\xi}=
\begin{pmatrix}
0&\omega&v_1\\
-\omega&0&v_2\\
0&0&0
\end{pmatrix}=v_1E_1+v_2E_2+\omega E_3,
$$
where $[E_3,E_1]=E_2$, $[E_2, E_3]=E_1$, and  $[E_1, E_2]=0$.

The system is described by the purely kinetic Lagrangian function
$L:\mathfrak{se}(2)\ra\R$, which reads as
$$L(v_1, v_2, \omega)=\frac{1}{2}\left[ (J+m(a^2+b^2))\omega^2 + mv_1^2+m v_2^2-2bm\omega v_1-2am\omega v_2\right].$$
Here $m$ and $J$ denote the mass and the moment of inertia of
the sleigh relative to the contact point, while $(a, b)$
represents the position of the centre of mass w.r.t. the
body frame, determined by placing the origin at the contact point
and the first coordinate axis in the direction of the knife axis.
Additionally, the  system is subjected to the nonholonomic
constraint determined by the linear subspace
$$ D=\{(v_1, v_2, \omega)\in \mathfrak{se}(2)\; |\; v_2=0\}\subset\mathfrak{se}(2).$$
Instead of $\{E_1, E_2, E_3\}$ we take another basis of
$\mathfrak{se}(2)$:
$$
e_1=E_3,\quad e_2=E_1,\quad e_3= -ma E_3-mab E_1+(J+ma^2) E_2,$$ adapted
to the decomposition $D\oplus D^\perp$; $D=\hbox{span }\{ e_1,
e_2\}$ and $D^\perp=\hbox{span }\{ e_3\}$. The induced
skew-algebroid structure on $D$ is given by
$$
[e_1, e_2]_{D}=\frac{ma}{J+ma^2} e_1+\frac{mab}{J+ma^2}e_2.$$
 Therefore, the structural constants are ${\mathcal C}^1_{12}=\frac{ma}{J+ma^2}$ and
${\mathcal C}^2_{12}=\frac{mab}{J+ma^2}$. The algebroid $D$ is
almost Lie (in fact, in this simple case it is a Lie algebra).
Next, we will use theorem \ref{thm:pmp} to derive the nonholonomic
equations of motion. Set $U=\R^2\ni(y^1,y^2)$ and the control
function to be a map $f:U\ra D$ given by
$$ f(y^1,y^2)=y^1e_1+y^2e_2\in D.$$
The Lagrangian restricted to $D$ defines the cost function
$L:U\ra\R$,
$$
L(y^1, y^2)=\frac{1}{2}\left[ (J+m(a^2+b^2))(y^1)^2 +
m(y^2)^2-2bmy^1y^2\right].$$ For a curve
$\xi(t)=\xi_1(t)e^\ast_1+\xi_2(t)e^\ast_2\in D^\ast$ the maximum
principle reads
\begin{eqnarray} \label{eqn:max_nh} H(\xi(t),y(t))&=&\xi_1(t)y^1+\xi_2(t)y^2+\ul\xi_0 \cdot L(y^1,y^2)\\&=&
\max_{(v^1,v^2)\in\R^2}(\xi_1(t)v^1+\xi_2(t)v^2+\ul\xi_0 \cdot
L(v^1,v^2))\,.\nonumber\end{eqnarray} 
If $\ul\xi_0=0$, then maximality would
give $\xi(t)=0$, which is impossible. Hence, we may assume that
$\ul\xi_0=-1$. Now from \eqref{eqn:max_nh} we will get
\begin{equation}\label{eqn:ev_nh}
\begin{split}
\xi_1(t)&=\left(J+m(a^2+b^2)\right)y^1-bmy^2\,,\\
\xi_2(t)&=my^2-bmy^2\,.
\end{split}
\end{equation}
Finally, the Hamiltonian evolution \eqref{eqn:par_trans_*A} is
simply
\begin{align*}
\dot \xi_1&=\mathcal{C}^1_{21}y^2\xi_1+\mathcal{C}^2_{21}y^2\xi_2=-\frac{ma}{J+ma^2}y^2(\xi_1+b\xi_2),\\
\dot
\xi_1&=\mathcal{C}^1_{21}y^1\xi_1+\mathcal{C}^2_{12}y^1\xi_2=\frac{ma}{J+ma^2}y^1(\xi_1+b\xi_2).
\end{align*}
In view of \eqref{eqn:ev_nh} and the above equations we conclude
that the equations of motion are
\begin{align*}
(J+m(a^2+b^2))\dot{y}^1 -bm \dot{y}^2&= -ma y^1y^2,\\
m\dot{y}^2-bm\dot{y}^1&=ma(y^1)^2,
\end{align*}
which completely agrees with \cite{grabowski_nonholonomic}.

\chapter{Needle variations}\label{ch:needle}

In order to prove Theorems \ref{thm:pmp_A} and \ref{thm:pmp_A_rel} we shall somehow compare the cost on the optimal trajectory $\bm f(\bm x(t),u(t))$ with costs of nearby trajectories. As our assumptions input on the
set of controls $U$ are very mild, we cannot use the natural
concept of a continuous deformation, as in the standard calculus
of variations ($U$ can be for instance discrete). Instead, we introduce
the notion of \emph{needle variations} after \cite{pontryagin}.
For a given admissible control $u:[t_0,t_1]\lra U$ this variation
will be, roughly speaking, the family of controls $ u_s(t)$
obtained by substituting $u(t)$ by given elements $v_i\in U$ on a
small intervals $I_i=(\tau_i-s\del t_i,\tau_i]\subset[t_0,t_1]$. Our main result in this chapter is Theorem \ref{thm:1st_main}, where we study the $\bm A$-homotopy classes of the trajectories of the system \eqref{eqn:con_sys_A} obtained for controls $u_s(t)$. 
We finish this chapter with the definition of $\bm K_\tau^u$---the set of infinitesimal variations of the trajectory $\bm f(\bm x(t), u(t))$.

\subsection{Needle variation of controls and trajectories}

Throughout this chapter we will work with a fixed admissible control $u:[t_0,t_1]\lra U$ and fixed trajectory $\bm a(t):=\bm f(\bm x(t),u(t))$. 

Choose points $t_0<\tau_1\leq\tau_2\leq\hdots\leq\tau_k\leq\tau<t_1$,
being regular points of $u$. Next, choose non-negative numbers
$\delta t_1,\hdots,\delta t_k$ and an arbitrary real number
$\delta t$. Finally, take (not necessarily different) elements
$v_1,\hdots,v_k\in U$. The whole set of data $(\tau_i,
v_i,\tau,\delta t_i, \delta t)_{i=1,\hdots,k}$ will be denoted by
$\A$ and called a \emph{symbol}. Its role will be to encode the variation of the control $u(t)$. Intuitively, points $\tau_i$ emphasise moments in which we substitute $u(t)$ by a constant control $v_i$ on an interval $I_i=(\tau_i-s\del t_i,\tau_i]$ of length $s\del t_i$, while $s\del t$ is responsible for shortening or lengthening the time for which $u(t)$ is defined. The precise
definition is quite technical, because one should take care to make the intervals $I_i$  pair-wise disjoint.

Take
$$l_i=\begin{cases}
\del t-(\del t_i+\hdots+\del t_k) &\text{when $\tau_i=\tau$;}\\
\phantom{x.}-(\del t_i+\hdots+\del t_k) &\text{when $\tau_i=\tau_k<\tau$;}\\
 \phantom{x.}-(\del t_i+\hdots+\del t_j) &\text{when
$\tau_i=\tau_{i+1}=\hdots=\tau_j<\tau_{j+1}$,}
\end{cases} $$
and define $s$-dependent intervals  $I_i:=(\tau_i+s l_i,\tau_i+s(l_i+\del t_i)]$. As we see, $I_i$ is left-open and right-closed and it has
length $s\cdot\delta t_i$. If $\tau_i<\tau_{i+1}$, or $i=k$ and
$\tau_k<\tau$, the end-point of $I_i$ lies at $\tau_i$. If
$\tau_i=\tau_{i+1}$, then the end-point of $I_i$ coincides with the
initial-point of $I_{i+1}$. If $\tau_k=\tau$, we set the
end-point of $I_k$ at $\tau+s\delta t$. Clearly, for $s$ small
enough, the intervals $I_i$ lie inside $[t_0,t_1]$ and are
pairwise disjoint.


\begin{definition}\label{def:need_var} For a symbol $\A=(\tau_i, v_i,\tau,\delta t_i, \delta t)_{i=1,\hdots,k}$ we introduce a $s$-dependent
family of admissible controls defined on intervals
$[t_0,t_1+s\delta t]$:
\begin{equation}
\label{eqn:need_var} u^\A_s(t)=
\begin{cases}
v_i & \text{for $t\in I_i$},\\
u(t) & \text{for $t\in [t_0,\tau+s\del t]\setminus \bigcup_i
I_i$}\\
u(t-s\del t) &\text{for $t\in(\tau+s\del t,t_1+s\del t]$}.
\end{cases} \end{equation}
We will call $u^\A_s$ a (\emph{needle})
\emph{variation of the control} $u$ \emph{associated with the symbol
$\A$}. 
\end{definition}

Using $u^\A_s(t)$ and an AC path $s\mapsto \bm x_0(s)\in M\times\R$ where $\bm x_0(0)=\bm x_0$ we can define the variation of $\bm a(t)$.

\begin{definition}
The family of trajectories
$$\bm{a}(t,s):=\bm{f}\left(\B{x}(t,s),u^\A_s(t)\right)$$
of the system \eqref{eqn:con_sys_A}, with the initial conditions $\bm x(t_0,s)=\bm
x_0(s)$, where $t\in[t_0,t_1+s\del t]$, will be called a
\emph{variation of the trajectory} $\bm a(t)=\bm f(\bm
x(t),u(t))$ \emph{associated with the symbol $\A$ and the initial base-point variation $\bm x_0(s)$}.
\end{definition}

\begin{remark}\label{rem:var_dt=0}
Observe that, when $\delta t_i=0$, the interval $I_i$ is empty.
It follows that adding a triple $(\tau_i, v_i, \del t_i=0)$ to the symbol $\A$ does not change the variation 
$u^\A_s$ and, consequently, the associated variations $\bm a(t,s)$.
\end{remark}

\subsection{Needle variations and \texorpdfstring{$\bm A$}{\bf{A}}-homotopy classes} 

Our goal now is to compare the $\bm A$-homotopy classes of the trajectory $\bm a(t)$ and its variation $\bm a(t,s)$ introduced above. We need this because OCPs \eqref{eqn:P} and \eqref{eqn:P_rel} are defined in term of algebroid homotopy classes. Having in mind Lemma \ref{lem:gen_E_htp} and the construction of a $\bm A$-homotopy associated with a control system \eqref{eqn:con_sys_A} given in Chapter \ref{ch:OCP}, we may expect that the family of trajectories $\bm a(t,s)$ forms an $\bm A$-homotopy (for some initial-point homotopy $\bm b_0(s)$). Consequently, the description of $\bm A$-homotopy classes of $\bm a(t,s)$ should be possible by meas of Lemma \ref{lem:htp}. This is indeed the case, yet some technical work is needed in order to reparametrise $\bm a(t,s)$ in a suitable way.

\begin{theorem}\label{thm:1st_main}  Let $s\mapsto \bm b_0(s)$ be a  bounded measurable $\bm A$-path covering $s\mapsto \bm x_0(s)$, where $\bm x_0(0)=\bm x_0$. Consider a variation $\bm a(t,s)=\bm f(\bm x(t,s),u^\A_s(t))$ of the trajectory $\bm a(t)=\bm f(\bm x(t),u(t))$ associated with a symbol $\A=(\tau_i, v_i,\tau,\delta t_i, \delta t)_{i=1,\hdots,k}$ and initial base-point variation $\bm x_0(s)$. 

Then there exists a number $\theta>0$ and an $\bm A$-path $s\mapsto \bm d^{\A}(s)$ defined for $0\leq s\leq\theta$ such that
\begin{equation}\label{eqn:htp_1st_lem}
[\bm b_0(s)]_{s\in[0,\eps]}[\bm{a}(t,\eps)]_{t\in[t_0,t_1+\eps\del t]}=[\bm
{a}(t)]_{t\in[t_0,t_1]}[\bm{d}^{\A}(s)]_{s\in[0,\eps]},\end{equation} 
for every $\eps\leq \theta$. 

Moreover, if $(\tau_i,v_i,\tau)$ in $\A$ are fixed, we can choose $\theta>0$ universal for all $(\del t_i,\del t)$ belonging to a fixed compact set.

Finally, if $\bm b_0(s)$ is regular at $s=0$, then  $\bm d^{\A}(s)$, regarded as a function of $s$, $\del t_i$ and $\del t$, is uniformly regular w.r.t. $\del t_i$ and $\del t$ at $s=0$. What is more,
\begin{equation}
\label{eqn:d_at_0} 
\begin{split}\bm{d}^{\A}(0)=&\bm B_{t_1\tau}[\bm{f}(\bm{x}(\tau),u(\tau))]\del t+\bm{B}_{t_1 t_0}(\bm
b_0(0))\\&+\sum_{i=1}^k
\bm{B}_{t_1\tau_i}\Big[\bm{f}(\bm{x}(\tau_i),v_i)-\bm{f}(\bm{x}(\tau_i),u(\tau_i))\Big]\del
t_i\in\bm{A}_{\bm{x}(t_1)}.
\end{split}
\end{equation} 
\end{theorem}

\begin{proof}
The proof is technically complicated, yet conceptually not very difficult. The idea is to decompose
$\bm a(t,s)$ into several parts, which, after a
suitable reparametrisation, form an $\bm A$-homotopy. As one may have expected, these parts correspond to ''switches'' in the needle variation associated with the symbol $\A$. Our argument will be therefore inductive w.r.t. $k$---the number of ''switches'' in $\A$. Formula \eqref{eqn:htp_1st_lem} will be obtained from the repetitive usage of Lemma \ref{lem:htp} for the partial homotopies, and  \eqref{eqn:d_at_0} will follow from the concrete form of these homotopies. The preservation of the uniform regularity will be obtained using the technical results introduced in Appendix \ref{sec:meas}. 

Finally, let us explain the role of the number $\theta$. We know from Theorem \ref{thm:exist} that if a solution of the ODE for a fixed initial condition $\bm x_0$ is defined on an interval $[t_0,t_1]$, then so are the solutions for initial conditions close enough to $\bm x_0$. Since the base variation $\bm x(t,s)$ associated with $u^\A_s(t)$ is obtained as a composition of the solutions of \eqref{eqn:con_sys} with perturbations on intervals of length $s\del t_i$ and $s\del t$, it is clear that, if numbers $\del t_i$ and $\del t$ are bounded and $\tau_i$, $\tau$ and $v_i$ fixed, for a given $\bm x_0(s)$, we can chose $\theta>0$ such that the trajectory $\bm x(t,s)$ will stay close enough to $\bm x(t)$ to be well-defined for all $0\leq s\leq\theta$ and all $[t_0,t_1]$. 

In our inductive reasoning it will be more convenient to assume that all the data depends on an additional parameter $p\in P$ (i.e., we have $\bm x_0(s,p)$ instead of $\bm x_0(s)$, $\bm a(t,s,p)$ instead of $\bm a(t,s)$, etc.). In the assertion we demand that \eqref{eqn:htp_1st_lem} and \eqref{eqn:d_at_0} hold for each fixed $p\in P$. Moreover, for fixed $(\tau_i,v_i,\tau)$ we want $\bm d^{\A,p}(s)$ to be uniformly regular w.r.t. $p$, $\del t_i$, and $\del t$ at $s=0$ if $\bm b_0(s,p)$ is uniformly regular w.r.t. $p$ at $s=0$.

In what follows we will need two technical lemmas. 
\begin{lemma}\label{lem:sub_lem1} Let $t\mapsto\bm a(t,s,\wt p)= \bm f(\bm x(t,s,\wt p),v(t))$, with $t\in[\wt t_0,\wt t_1]$, be a family of bounded measurable admissible paths over $\bm x(t,s,\wt p)$ parameterised by $\wt p\in \wt P$. Let $s\mapsto \wt{\bm b}_0(s,\wt p)$ be a family of  bounded measurable $\bm A$-paths over $\bm x(t_0,s,\wt p)$. There exists a number $\theta>0$ and a family of bounded measurable $\bm A$-paths $s\mapsto\bm d_1^{\wt p}(s)$ defined for $0\leq s\leq\theta$ such that
\begin{equation}\label{eqn:needle1}
[\wt{\bm b}_0(s,\wt p)]_{s\in[0,\eps]}[\bm{f}(\bm x(t,\eps,\wt p),v(t))]_{t\in[\wt t_0,\wt t_1]}=[\bm
{f}(\bm x(t,0,p),v(t))]_{t\in[\wt t_0,\wt t_1]}[\bm{d}_1^{\wt p}(s)]_{s\in[0,\eps]},\end{equation} 
for all $\eps\leq \theta$. 
 
Explicitly, $\bm d_1^{\wt p}(s)=\bm B^v_{t\wt{t_0}}\left[\wt{\bm b}_0(s,\wt p)\right]$, where $\bm B^v_{t\wt{t_0}}$ is a parallel transport operator associated with the control $v(t)$. Moreover, if $\wt{\bm b}_0(s,\wt p)$ is uniformly regular w.r.t. $\wt p$ at $s=0$, then so is $\bm d_1^{\wt p}(s)$.
\end{lemma}
\begin{proof}[Proof of the lemma]
Fix $\wt p\in\wt P$ and consider $\bm b(t,s,\wt p):=\bm B^v_{t\wt t_0}\left[\wt{\bm b}_0(s,\wt p)\right]$. It follows from the definition of the operator of parallel
transport $\bm B^v_{t\wt{t_0}}$ that the pair $(\bm a(t,s,\wt p),\bm b(t,s,\wt p))$ is and $\bm A$-homotopy over $\bm x(t,s,\wt p)$ (see Remark \ref{rem:B_htp}).  Now \eqref{eqn:needle1} follows directly from Lemma \ref{lem:htp}, since $\bm b(\wt t_1,s,p)=\bm B^v_{\wt t_1\wt t_0}\left[\wt{\bm b}_0(s,\wt p)\right]=\bm d_1^{\wt p}(s)$. 

Finally, since by Remark \ref{rem:B_htp} the map $\bm B^v_{\wt t_1\wt t_0}(\cdot)$ is  continuous  for every fixed $\wt t_1$ and $\wt t_0$, in light of Lemma \ref{lem:ur_comp}, it preserves the uniform regularity of $\wt{\bm b}_0(s,\wt p)$ .
\end{proof}

\noindent The second lemma is the following one.

\begin{lemma}\label{lem:sub_lem2} Let $t\mapsto\bm a(t,s,\wt p)= \bm f(\bm x(t,s,\wt p),v(t))$ be a family of bounded measurable admissible
paths over $\bm x(t,s,\wt p)$ parametrised by $\wt p\in P$. Let $s\mapsto \wt{\bm b}_0(s,\wt p)$ be a family of bounded measurable $\bm A$-paths over $\bm x(\wt t_0+sc,s,\wt p)$. Then there exists a number $\theta>0$ and a family of bounded measurable $\bm A$-paths $s\mapsto\bm d_2^{\wt p,c,d}(s)$ defined for $0\leq s\leq\theta$ such that
\begin{equation}\label{eqn:needle3}
[\wt{\bm b}_0(s,\wt p)]_{s\in[0,\eps]}[\bm{f}(\bm x(\wt t_0+t,\eps,\wt p),v(t))]_{t\in[\eps c,s\eps d]}=[\bm{d}_2^{\wt p,c,d}(s)]_{s\in[0,\eps]},\end{equation} 
for every $\eps\leq \theta$. 
 
 Moreover, if $\wt{\bm b}_0(s,\wt p)$ is uniformly regular  w.r.t. $\wt p$ at $s=0$ and $\wt t_0$ is a regular point of the control $v(t)$,  then $\bm d_2^{\wt p,c,d}(s)$ is uniformly regular  w.r.t. $\wt p$, $c$, and $d$ at $s=0$. Finally, 
\begin{equation}\label{eqn:needle3a}
\bm{d}_2^{\wt p,c,d}(0)=\wt{\bm b}_0(s,\wt p)+(d-c)\bm f(\bm x(\wt t_0,0,\wt p),v(\wt t_0)).\end{equation}
\end{lemma}

\begin{proof}[Proof of the lemma]
For notation simplicity let forget about the $\wt p$-dependence. Define
$$\wh {\bm b}_0(s):=\bm B_{\wt t_0(\wt t_0+sc)}\left[\wt{\bm b}_0(s)-c\bm f(\bm x(\wt t_0+cs),v(\wt t_0+cs))\right].$$
Clearly, $\wh{\bm b}_0(s)$ is an admissible paths over $\bm x(\wt t_0,s)$. Now define a pair of maps
\begin{align*}
\bm a(t,s)&=s\bm f\left(\bm x(\wt t_0+ts,s),v(\wt t_0+ts)\right),\\
\bm b(t,s)&=\bm B^v_{(\wt t_0+ts)\wt t_0}\left[\wt{\bm b}_0(s)\right]+t\bm f\left(\bm x(\wt t_0+ts,s),v(\wt t_0+st)\right),
\end{align*}
where $t\in[c,d]$ and $s\in[0,\theta]$. We shall prove that this pair is an $\bm A$-homotopy.

If this is the case, then clearly \eqref{eqn:needle3} follows form Lemma \ref{lem:htp} since the initial-point $\bm A$-homotopy is 
\begin{align*}
\bm b(c,s)=&\bm B^v_{(\wt t_0+cs)\wt t_0}\bm B^v_{\wt t_0(\wt t_0+sc)}\left[\wt{\bm b}_0(s)-c\bm f\left(\bm x(\wt t_0+sc),v(\wt t_0+sc)\right)\right]+\\
\phantom{=}&+c\bm f\left(\bm x(\wt t_0+sc),v(\wt t_0+sc)\right)=\wt {\bm b}_0(s),
\intertext{the final-point $\bm A$-homotopy is}
\bm d_2^{c,d}(s):=&\bm b(d,s)=\bm B^v_{(\wt t_0+ds)\wt t_0}\left[\wt{\bm b}_0(s)-c\bm f\left(\bm x(\wt t_0+sc),v(\wt t_0+sc)\right)\right]+\\
\phantom{=}&+d\bm f\left(\bm x(\wt t_0+sd),v(\wt t_0+sd)\right), 
\end{align*}
and, by Lemma \ref{lem:reparam}, $\left[\bm a(t,s)\right]_{t\in[c,d]}=\left[\bm f(\wt t_0+t,s),v(\wt t_0+t)\right]_{t\in[cs,ds]}$.
Evaluating the formula for $\bm d_2^{c,d}(s)$ at $s=0$ we get \eqref{eqn:needle3a}. 

Finally, $\bm d_2^{c,d}(s)$ is uniformly regular if $\wt{\bm b}_0(s)$ is and $\bm t_0$ is a regular point of $v(t)$. Indeed, we can use the results from Appendix \ref{sec:meas}. The point is to observe that $\bm d_2^{c,d}(s)$ is obtained from measurable maps $v(\wt t_0+s)$ and $\bm b_0(s)$ regular at $s=0$ by operations described in Propositions \ref{prop:ur_cont}--\ref{prop:ur_multiplication} and Lemmas \ref{lem:ur_rescal}--\ref{lem:ur_comp} which preserve the uniform regularity. One has also to use the fact that $\bm B_{t_1 t_0}^v(\bm b_0)$, $f(x,u)$, and $x(t,s)$ are continuous maps (cf. Remark \ref{rem:B_htp}). 

Now it remains to check that $\bm a(t,s)$ and $\bm b(t,s)$ are indeed an $\bm A$-homotopy. Let us calculate the WT-derivatives:
\begin{align*}
\pa_s\bm a^i(t,s)\overset{\text{WT}}=&\bm f^i\left(\bm x(\wt t_0+ts,s),v(\wt t_0+ts)\right)+ts\pa_{\ol t}\bm f^i\left(\bm x(\ol t,s),v(\ol t)\right)|_{\ol t=\wt t_0+ts}+\\
&+s\pa_{s}\bm f^i\left(\bm x(\ol t,s),v(\ol t)\right)|_{\ol t=\wt t_0+ts},
\intertext{and}
\pa_t\bm b^i(t,s)\overset{\text{WT}}=&\pa_{\ol t}\bm B^v_{\ol t\wt t_0}\left[\wh{\bm b}_0(s)\right]^i\Big|_{\ol t=\wt t_0+ts}+\bm f^i\left(\bm x(\wt t_0+ts,s),v(\wt t_0+ts)\right)+\\
&+ts\pa_{\ol t}\bm f^i\left(\bm x(\ol t,s),v(\ol t)\right)|_{\ol t=\wt t_0+ts}.
\end{align*}
Now, since $\bm f\left(\bm x(\ol t,s),v(\ol t)\right)$ and 
$\bm B^v_{\ol 
t\wt t_0}\left[\wh{\bm b}_0(s)\right]$ is an $\bm A$-homotopy (cf. Lemma \ref{lem:sub_lem1}), we have
$$\pa_{\ol t}\bm B^v_{\ol t\wt t_0}\left[\wh{\bm b}_0(s)\right]^i-\pa_{s}\bm f^i\left(\bm x(\ol t,s),v(\ol t)\right)\overset{\text{WT}}=c^i_{jk}(\bm x(\ol t,s))\bm B^v_{\ol t\wt t_0}\left[\wh{\bm b}_0(s)\right]^j\bm f^k\left(\bm x(\ol t,s),v(\ol t)\right).$$
Consequently, 
\begin{align*}
\pa_t \bm b^i(t,s)-\pa_s\bm a(t,s)\overset{\text{WT}}=&s\left[\pa_{\ol t}\bm B^v_{\ol t\wt t_0}\left[\wh{\bm b}_0(s)\right]^i-\pa_{s}\bm f^i\left(\bm x(\ol t,s),v(\ol t)\right)\right]\Big|_{\ol t=\wt t_0+ts}\\\overset{\text{WT}}=&
sc^i_{jk}(\bm x(\ol t,s))\bm B^v_{\ol t\wt t_0}\left[\wh{\bm b}_0(s)\right]^j\bm f^k\left(\bm x(\ol t,s),v(\ol t)\right)\big|_{\ol t=\wt t_0+ts}\\
=&c^i_{jk}(x(\wt t_0+ts,s))\bm b^j(t,s)\bm a^k(t,s).
\end{align*} \end{proof}

No we return to the inductive proof of Theorem \ref{thm:1st_main}. We will prove first that the assertion is true for $t_1=\tau$. Our argument will be inductive w.r.t. $k$---the number of switches in the symbol $\A=(\tau_i, v_i,\tau,\delta t_i, \delta t)_{i=1,\hdots,k}$.

\underline{Step 1, $k=0$.} We start with $k=0$. This means that $u^\A_s(t)=u(t)$. For \underline{$\del t=0$} we simply have $\bm a(t,s)=\bm f(\bm x(t,s,p),u(t))$ with $t\in[t_0,\tau]$, where $\bm x(t_0,s,p)=\bm x_0(s,p)$. Now we can use the Lemma \ref{lem:sub_lem1} taking $\wt t_0=t_0$, $\wt t_1=\tau$, $\wt p=p$, $\wt {\bm b}_0(s,\wt p)=\bm b_0(s,p)$, and $v(t)=u(t)$ to get the assertion. 

If \underline{$\del t\neq 0$}, things are a little more complicated. We have $\bm a(t,s,p)=\bm f(\bm x(t,s,p),u(t))$ where $t\in[t_0,\tau+s\del t]$ and $\bm x(t_0,s,p)=\bm x_0(s,p)$. We can decompose
\begin{equation}\label{eqn:needle2}\begin{split}
&\left[\bm f(\bm x(t,\eps,p),u(t))\right]_{t\in[t_0,\tau+\eps\del t]}\\&=\left[\bm f(\bm x(t,\eps,p),u(t))\right]_{t\in[t_0,\tau]}\cdot\left[\bm f(\bm x(t,\eps,p),u(t))\right]_{t\in[\tau,\tau+\eps\del t]}.\end{split}
\end{equation}
Now using the assertion for $\del t=0$ we get
\begin{equation}\label{eqn:needle4}
[\bm b_0(s, p)]_{s\in[0,\eps]}[\bm{f}(\bm x(t,\eps, p),u(t))]_{t\in[ t_0,\tau]}=[\bm
{f}(\bm x(t,0,p),u(t))]_{t\in[t_0,\tau]}[\bm{d}_1^{ p}(s)]_{s\in[0,\eps]},\end{equation} 
where $\bm d^p_1(s)$ is uniformly regular w.r.t. $p$ at $s=0$, and $\bm d^p_1(0)=\bm B_{\tau t_0}[\bm b_0(0,p)]$. 
Next, using Lemma \ref{lem:sub_lem2} for $\wt t_0=\tau$, $c=0$, $d=\del t$, $\wt p=p$, $\wt {\bm b}_0(s,\wt p)=\bm d_1^p(s)$, and $v(t)=u(t)$, we get
\begin{equation}\label{eqn:needle5}
[\bm d_1^p(s)]_{s\in[0,\eps]}[\bm{f}(\bm x(\tau+t,\eps,\wt p),v(\tau+t))]_{t\in[0,\eps\del t]}=[\bm{d}_2^{ p,\del t}(s)]_{s\in[0,\eps]},\end{equation} 
where $\bm d_2^{p,\del t}(s)$ is uniformly regular w.r.t. $p$, and $\del t$ at $s=0$ and $$\bm d_2^{p,\del t}(0)=\bm d_1^p(0)+\del t \bm f(\bm x(\tau,0,p),u(\tau)).$$ 
Multiplying \eqref{eqn:needle4} by $[\bm{f}(\bm x(\tau+t,\eps,\wt p),v(\tau+t))]_{t\in[0,\eps\del t]}$, using \eqref{eqn:needle2} and \eqref{eqn:needle5}, and taking $\bm d^{\A,p}(s):=\bm d_2^{p,\del t}(s)$, we get the assertion.

\ul{Step 2.} Assume that the assertion holds for all $l<k$. Consider a symbol $\A=(\tau_i,v_i,\tau,\del t_i, \del t)_{i=1,\hdots, k}$. We will distinguish the following two situations:

\noindent\ul{Situation 2.A.} Not all $\tau_i$ are equal. In particular, $$t_0<\tau_1\leq\hdots\leq \tau_l<\tau_{l+1}\leq\hdots\leq\tau_k\leq\tau$$ for some $l<k$. We can now use the inductive assumption for a symbol $\A_1=(\tau_i,v_i,\tau=\tau_l,\del t_i, \del t=0)_{i=1,\hdots, l}$ to get 
\begin{equation}\label{eqn:needle6}\begin{split}
&[\bm b_0(s,p)]_{s\in[0,\eps]}[\bm{f}(\bm x(t,\eps,p),u^{\A_1}_\eps(t))]_{t\in[t_0,\tau_l]}\\&
=[\bm{f}(\bm x(t,0,\eps),u(t))]_{t\in[t_0,\tau_l]}[\bm{d}^{\A_1,p}(s)]_{s\in[0,\eps]},\end{split}\end{equation} 
where $\bm d^{\A_1,p}(s)$ is uniformly regular w.r.t. $p$, $\del t_1,\hdots,\del t_l$ at $s=0$, and
$$\bm d^{\A_1,p}(0)=\bm{B}_{\tau_l t_0}(\bm
b_0(0,p))+\sum_{i=1}^l
\bm{B}_{\tau_l\tau_i}\Big[\bm{f}(\bm{x}(\tau_i),v_i)-\bm{f}(\bm{x}(\tau_i),u(\tau_i))\Big]\del
t_i.$$
Using the inductive assumption for $\A_2=(\tau_i,v_i,\tau,\del t_i, \del t)_{i=l+1,\hdots, k}$ with $t_0=\tau_l$ and $\bm b_0(s, p,\del t_1,\hdots,\del t_l)=\bm d_1^{\A_1,p}(s)$, we get  
\begin{equation}\label{eqn:needle7}\begin{split}
&[\bm d^{\A_1,p}(s)]_{s\in[0,\eps]}[\bm{f}(\bm x(t,\eps,p),u^{\A}_\eps(t))]_{t\in[\tau_l,\tau]}\\
&=[\bm
{f}(\bm x(t,0,\eps),u(t))]_{t\in[\tau_l,\tau]} [\bm{d}^{\A_1,\A_2,p}(s)]_{s\in[0,\eps]},\end{split}\end{equation} 
where $\bm d^{\A_1,\A_2,p}(s)$ is uniformly regular w.r.t. $p$, $\del t_i$ and $\del t$ at $s=0$, and
\begin{align*}\bm d^{\A_1,\A_2, p}(0)&=\bm{B}_{\tau\tau_l}(\bm
d^{\A_1,p}(0))+\sum_{i=l}^k
\bm{B}_{\tau\tau_i}\Big[\bm{f}(\bm{x}(\tau_i),v_i)-\bm{f}(\bm{x}(\tau_i),u(\tau_i))\Big]\del
t_i\\
&=\bm{B}_{\tau t_0}(\bm
b_0(0,p))+\sum_{i=1}^k
\bm{B}_{\tau\tau_i}\Big[\bm{f}(\bm{x}(\tau_i),v_i)-\bm{f}(\bm{x}(\tau_i),u(\tau_i))\Big]\del
t_i.\end{align*}
Multiplying \eqref{eqn:needle6} by $[\bm{f}(\bm x(t,\eps,p),u^{\A}_\eps(t))]_{t\in[\tau_l,\tau]}$, using \eqref{eqn:needle7}, and taking $\bm d^{\A,p}(s):=\bm d^{\A_1,\A_2,p}(s)$, we get the assertion.

\noindent\ul{Situation 2.B.} If all $\tau_i$ are equal then either $$\tau_1=\hdots\tau_k=\tau\quad \text{or}\quad \tau_1=\hdots\tau_k<\tau.$$

\noindent\ul{2.B.1.} In the first case using the assertion for $\A_1=(\tau_i,v_i,\del t_i,\tau, \del t-\del t_k)_{i=1,2,\hdots,k-1}$ we get 
\begin{equation}\label{eqn:przyp_B1}\begin{split}
&[\bm b_0(t,s)]_{s\in[0,\eps]}[\bm f(\bm x(t,\eps,p),u^{\A_1}_\eps(t))]_{t\in[t_0,\tau+(\del t-\del t_k)\eps]}\\&
=[\bm f(\bm x(t,0,p),u(t))]_{t\in[t_0,\tau]}[\bm d_1^{\A_1,p}(s)]_{s\in[0,\eps]},\end{split}
\end{equation} 
where $\bm d_1^{\A_1,p}(s)$ is uniformly regular w.r.t. $p$, $\del t_1,\hdots,\del t_{k-1}$, $\del t-\del t_k$ at $s=0$ and
$$\bm d_1^{\A_1,p}(0)=\bm B_{\tau t_0}(\bm b_0(0,p))+(\del t-\del t_k)\bm f(x(\tau),u(\tau))+\sum_{i=1}^{k+1}\bm B_{\tau\tau_i}\left[\bm f(\bm x(\tau_i),v_i)-\bm f(\bm x(\tau_i),u(\tau_i))\right]\del t_i.$$
Now using Lemma \ref{lem:sub_lem2} for $\wt t_0=\tau$, $c=\del t-\del t_k$, $\wt v(t)=v_k$, $\wt p=(p,\del t_1,\hdots,\del t_{k-1},\del t-\del t_k)$, and $\wt{\bm b}_0(s,\wt p)=\bm d_1^{\A_1,p}(s)$ we get
\begin{equation}\label{eqn:needle8}
\left[\bm d_1^{\A_1,p}(s)\right]_{s\in[0,\eps]}\left[\bm f(\bm x(\tau+t,\eps,p),v_k)\right]_{t\in[\eps(\del t-\del t_{k-1}),\eps\del t]}=\left[\bm d_2^{\A_1,\del t,\del t_k,p}(s)\right]_{s\in[0,\eps]},
\end{equation}
where $\bm d_2^{\A_1,\del t,\del t_k,p}(s)$ is uniformly regular w.r.t. $p$, $\del t_1,\hdots,\del t_k$, $\del t$ at $s=0$ and
$$\bm d_2^{\A_1,\del t,\del t_k,p}(0)=\bm d_1^{\A_1,p}(0)-\del t_k\bm f(\bm x(\tau,0,p),u(\tau)).$$
Again multiplying \eqref{eqn:needle7} by $\left[\bm f(\bm x(\tau+t,\eps,p),v_k)\right]_{t\in[\eps(\del t-\del t_{k-1}),\eps\del t]}$, using \eqref{eqn:needle8} and taking $\bm d^{\A,p}(s)=\bm d_2^{\A_1,\del t,\del t_k,p}(s)$ we get the assertion.

\noindent\ul{2.B.1.} If $\tau_1=\hdots=\tau_k<\tau$ we can use the result from 2.B.1 for a symbol $\A_1=(\tau_i,v_i,\del t_i,\tau=\tau_k,\del t=0)_{i=1,\hdots,k-1}$ and then use the inductive assumption for $[t_0,\tau]=[\tau_k,\tau]$ and $\A_2=(\tau,\del t)$ on $[t_0=\tau_k,\tau]$ in essentially the same way as in the case A. The inductive argument is now complete.

Finally, to obtain the assertion for $t_1$ not $\tau$, one has just to proceed as in the step 1 with $\del t=0$ and  use Lemma \ref{lem:sub_lem1} again, taking $\wt t_0=\tau$, $\wt t_1=t_1$, $\wt{\bm b}_0(s,\wt p)$ to be the final-point $\bm A$-homotopy $\bm d^{\A,p}(s)$ derived for $t_1=\tau$, and the control $v(t)=u(t)=u^{\A}_s(t+s\del t)$. \end{proof}

\subsection{The set of infinitesimal variations \texorpdfstring{$\bm  K^u_\tau$}{TEXT} } 

\begin{remark}\label{rem:interpretation_K}
Observe that choosing $\bm b_0(s)\equiv \theta_{\bm x_0}$ in Theorem \ref{thm:1st_main} we obtain an admissible path  $\bm d^\A(s)$, regular at $s=0$, defined for $0\leq s<\theta$, and satisfying
\begin{equation}
\label{eqn:var_htp}[\bm{a}(t,\eps)]_{t\in[t_0,t_1+\eps\del
t]}=[\bm {a}(t)]_{t\in[t_0,t_1]}[\bm{d}^\A(s)]_{s\in[0,\eps]},
\end{equation}
for $\eps\leq\theta$.

We define the set $\bm K_\tau^u$ consisting of elements of the form $\bm d^\A(0)$, where $\A=(\tau_i, v_i,\tau,\delta t_i, \delta t)_{i=1,\hdots,k}$ are symbols with $\tau$ fixed: 
\begin{align*} 
\bm K^u_\tau&:=\left\{\bm B_{t_1\tau}[\bm{f}(\bm{x}(\tau),u(\tau))]\del
t+\sum_{i=1}^k
\bm{B}_{t_1\tau_i}\Big[\bm{f}(\bm{x}(\tau_i),v_i)\right.\\&-\left.\bm{f}(\B{x}(\tau_i),u(\tau_i))\Big]\del
t_i:(\tau_i, v_i,\tau, \delta t_i, \delta
t)_{i=1,\dots,k}\text{ is a symbol}\right\}\subset\bm A_{\bm x(t_1)}.
\end{align*}
We will call $\bm K_\tau^u$ the
\emph{set of infinitesimal variations of the trajectory $\bm f(\bm
x(t),u(t))$ associated with the regular $\tau\in (t_0,t_1)$}.

The set $\bm{K}^u_\tau$ can be interpreted as the set of all generalised 
directions in $\bm{A}_{\bm{x}(t_1)}$ in which one can move the final base-point $\bm{x}(t_1)$ by performing needle variations, associated with symbols $\A=(\tau_i, v_i,\tau,\delta t_i, \delta t)_{i=1,\hdots,k}$ with fixed $\tau$ and trivial initial base-point variations $\bm x_0(s)\equiv\bm x_0$. 
\end{remark}

The geometry of $\bm{K}^u_\tau$ will be an object of our main interests in Chapter \ref{ch:proof}. Now let us note the following property

\begin{lemma}\label{lem:K_convex}
The set $\bm{K}^u_\tau$ is a convex cone in
$\bm{A}_{\bm{x}(t_1)}$.
\end{lemma}
\begin{proof}  Take symbols $\A=(\tau_i, v_i,\tau, \del t_i, \del
t)_{i=1,\hdots,k}$, $\A^{'}=(\tau^{'}_i, v^{'}_i,\tau, \del
t^{'}_i, \del t^{'})_{i=1,\hdots,k^{'}}$ and numbers
$\nu,\nu^{'}\geq 0$. We have to find a symbol $\mathfrak{v}$ such
that
$$\bm{d}^{\mathfrak{b}}(0)=\nu\bm{d}^{\A}(0)+\nu^{'}\bm{d}^{\A^{'}}(0).$$
Due to Remark \ref{rem:var_dt=0}, we may change the symbol by
adding $(\tau_i,v_i,\delta t_i=0)$ without changing the variation
$u^\A_\eps$. As we see from the form of \eqref{eqn:d_at_0},
such an addition will not change $\bm{d}^\A(0)$. Consequently, we
may assume that $k=k^{'}$, $\tau_i=\tau^{'}_i$, $v_i=v^{'}_i$, and
the symbols $\A$ and $\A^{'}$ differ only by $\del t_i$ and $\del
t$. Now consider the symbol $\mathfrak{v}=(\tau_i, v_i,\tau,
\nu\del t_i+\nu^{'}\del t^{'}_i, \nu\del
t+\nu^{'}\del^{'})_{i=1,\hdots,k}$. The formula
\eqref{eqn:d_at_0} (for $\bm b_0(0,p)=0$) is linear with respect
to $\del t_i$ and $\del t$, hence
$$\bm{d}^{\mathfrak{v}}(0)=\nu\bm{d}^{\A}(0)+\nu^{'}\bm{d}^{\A^{'}}(0).$$
\end{proof}

\noindent At the end of this chapter we define several geometric objects which will play an important role in Chapter \ref{ch:proof}. 

Consider the real line $\R$ with the canonical coordinate $t\in\R$. The tangent space $\T_{\ul x}\R$ is spanned by the canonical vector $\pa_t$. Denote by $\Lambda_{\ul x}$ the ray
\begin{align*}
\Lambda_{\ul x}&:=\R_+\cdot(-\pa_t)\subset\T_{\ul x}\R,
\intertext{and by $\bm \Lambda_{\bm x}$ the ray}
\bm \Lambda_{\bm x}&:=\theta_x\oplus\Lambda_{\ul x}\subset E_x\oplus\T_{\ul x}\R=\bm A_{\bm x},
\end{align*}
where $\bm x=(x,\ul x)$. Finally define 
$$\K_\tau^u:=\conv\left\{\bm B_{t_1t_0}(\eS_0\oplus\theta_{\ul x(t_0)}),\bm K_\tau^u\right\},$$ 
where $\eS_0=\Phi_0(\T_{z_0}S_0)$ was defined in Chapter \ref{ch:pmp}. By Theorem \ref{thm:1st_main}, $\K_\tau^u$ has an interpretation of the set of all generalised directions in $\bm A_{\bm x(t_1)}$ in which one can move the final base-point $\bm x(t_1)$ by performing needle variations, associated with symbols $\A=(\tau_i, v_i,\tau,\delta t_i, \delta t)_{i=1,\hdots,k}$ with fixed $\tau$ and initial base-point variations in the directions of $\eS_0\oplus\theta_{\ul x(t_0)}$. 

\chapter{Technical lemmas}\label{ch:main}

In this chapter we prove two technical results about $E$-homotopies --- Lemmas \ref{lem:htp_of_a_r} and \ref{lem:htp_of_a_r1}, which will be crucial in the proof of Theorems \ref{thm:pmp} and \ref{thm:pmp_A_rel}. To discuss briefly the results, given a family of smooth curves $x_{\vec r}:I\ra\R^m$, parameterized by $\vec r\in B^m(0,1)\subset\R^m$, which emerges from a single point $x_{\vec r}(0)=0$ and points into every direction $\dot x_{\vec r}(0)=\vec r$, it is quite obvious that, for every $t>0$ small enough, there exists a curve $x_{\vec {r_0}}$ from this family which reaches $0$ at time $t$. A similar result holds for families of admissible curves on a skew-algebroid $E$. Any such family which is sufficiently regular and emerges from a single point into every direction in $E$ will realise a zero homotopy class. This is Lemma \ref{lem:htp_of_a_r}. In Lemma \ref{lem:htp_of_a_r1} we prove that two sufficiently regular families of admissible paths in $E$ emerging from a single point must have a nonempty intersection of homotopy classes, provided that the set of their initial (generalized) velocities is rich enough.

These results seem to be quite natural and they are indeed, if such an algebroid is (locally) integrable. In this case $E$-homotopy classes can be represented by points on a finite-dimensional manifold. However, if $E$ is not integrable, $E$-homotopy is just a relation in the space of bounded measurable curves. Therefore to prove the results we have to pass through the Banach space setting. The main idea in the proof is to semi-parametrise the $E$-homotopy classes by a finite dimensional-space and reduce the problem to a finite-dimensional topological problem. By a semi-parametrisation we mean an epimorphism from a finite-dimensional space to the space of $E$-homotopy classes. 

\subsection{Local coordinates} Since we are going to work in a Banach space setting it is convenient to introduce local coordinates on an algebroid $E$. Consider coordinates $(x^a,y^i)\in U\times\R^m\subset\R^n\times\R^m$ trivialising the bundle $\tau:E\ra M$ around a point $p\in M$. We may assume that $p$ corresponds to $0\in U$.  As usual, we will denote
the structural functions of $E$ in these coordinates by
$\rho^a_i(x)$ and $c^i_{jk}(x)$. Since these functions are smooth,
we can assume (after restricting ourselves to a compact
neighborhood  $\ol V\ni 0$ in $\R^n$) that they are bounded (by
numbers $C_\rho$ and $C_c$, respectively) and Lipschitz w.r.t. $x^a$ (with constants $L_\rho$ and $L_c$,
respectively). It will be convenient to think of $V\times\R^m$
with those functions as of a (local) AL algebroid. Observe that
every bounded measurable $E$-path with the base initial-point
$p$ is represented by a pair of paths
$(x(t),a(t))\in\R^n\times\R^m$, where $a(t)$ is bounded measurable
and $x(t)$ is an AC-solution of the ODE
$$\begin{cases}
\dot x^b(t)=\rho^b_i(x(t))a^i(t),\\
x^b(0)=0.
\end{cases}
$$
As we see, $x(t)$ is determined entirely by $a(t)$. We can thus
identify the space $\mathcal{ADM}_p(I,E)$ of bounded
measurable admissible paths originated at $p$ with  the
space $\BM(I,\R^m)$ of bounded measurable maps $a:I\ra\R^m$. We
will consequently speak of algebroid homotopy classes in
$\BM(I,\R^m)$. Note that $\BM(I,\R^m)$, equipped with the
$L^1$-norm, is a Banach space. We will denote this norm simply by $\|\cdot\|$. The same symbol will be also used for $L^1$-norm in $\R^m$.  In our considerations we will understand a product of Banach spaces $(\mathcal{B}_1, \|\cdot\|_1)$ and  $(\mathcal{B}_2, \|\cdot\|_2)$ as a space  $\mathcal{B}_1\times\mathcal{B}_2$ equipped with the norm $\|\cdot\|= \|\cdot\|_1+\|\cdot\|_2$.

\subsection{First lemma}

\begin{lemma}\label{lem:htp_of_a_r}
Let $a_{\vec{r}}(\cdot)\in\BM(I,R^m)$, where $\vec r\in  B^m(0,1)$, be
a family of $E$-paths uniformly regular at $t=0$ w.r.t. $\vec r$
and such that  $a_{\vec{r}}(0)=\vec{r}$. Then there exists a number $\eta>0$ with the following property. For every $0<\eps<\eta$ there exists a 
vector $\vec r_0$ such that the curve
$a_{\vec r_0}(t)$, after restricting to the interval $[0,\eps]$,
is null-$E$-homotopic:
$$\big[a_{\vec{r}_0}(t)\big]_{t\in[0,\eps]}=\big[0\big].$$
\end{lemma}

Let us briefly sketch the strategy of the proof. 
Denote by $c_{\vec r}$ a constant path $c_{\vec r}(s)=\vec r$ in $\R^m$. We will construct
a continuous and invertible (local) map of Banach spaces
$\Phi:\R^m\times \mathcal{B}\lra \BM(I,\R^m)$ (the space
$\mathcal{B}$ will be specified later) which will have an
additional property that the homotopy class of the image is
determined by the first factor only
$$\left[\Phi(\vec r,d)\right]=\left[c_{\vec r}\right].$$

In such a way we will realise our idea from the introduction to this chapter --- $\R^m$ will semi-parametrise all local $E$-homotopy classes of $\mathcal{ADM}_p(I,E)$. Next, using
the map $a:\vec r\mapsto a_{\vec r}$, we will construct a
continuous map of finite-dimensional spaces 
$$\R^m\supset B^m(0,1)\overset
a{\lra}\BM([0,1],\R^m)\overset{\Phi^{-1}}\lra\R^m\times
\mathcal{B}\overset{\operatorname{pr}_1}\lra\R^m.$$ 
A topological argument will prove that $0$ lies in the image of this map, hence 
$$\left[a_{\vec{r_0}}\right]=\left[c_0\right]=\left[0\right]\quad \text{for some $\vec{r_0}$}.$$

\begin{proof}
Consider an $E$-path with a constant $\R^m$-part $a(t)=\vec r$,
where $\vec r$ is a fixed element in $\R^m$, $t\in[0,1]$. The
associated base path $x(t)\in\R^m$ is the solution of
\begin{equation}\label{eqn:a_x}
\left\{ \begin{aligned}
\pa_t x^b(t)&=\rho^b_i(x(t))a^i(t)=\rho^b_i(x(t))r^i,\\
x^b(0)&=0.
\end{aligned}\right.
\end{equation}
Clearly, if $\|\vec r\|$ is small enough, the solution of this
equation exists for $t\in[0,1]$ and is contained entirely in $\ol
V\subset \R^n$. Now for $a(t)$ and $x(t)$ as above and fixed paths
$d\in\BM(I,\R^m)$, $b\in\mathcal{AC}(I,\R^m)$, consider the
following system of differential equations:
\begin{equation}\label{eqn:phi_1}
\left\{ \begin{aligned}
\pa_s a^i(t,s)&=d^i(t)+c^i_{jk}(x(t,s))a^j(t,s)b^k(t),\\
a^i(t,0)&=a^i(t)=r^i,\\
\pa_s x^b(t,s)&=\rho^b_i(x(t,s))b^i(t),\\
x^b(t,0)&=x^b(t).
\end{aligned}\right.
\end{equation}
The existence and regularity of the solutions of \eqref{eqn:phi_1} can be discussed using the theory developed in Appendix \ref{sapp:ode}. Let us concentrate first on the equation for $x(t,s)$. The right hand-side of this equation is AC in $t$ and Lipschitz in $x$, the initial value depends AC on a parameter $t$, and hence, by the standard theory of ODEs, the solution $x(t,s)$ is defined locally and is AC w.r.t. both variables. By shrinking the norm $\|b\|_{\sup}$ we may change the Lipschitz constant in the defining equation. Consequently, for $\|b\|_{\sup}$ (and $\|\vec r\|$) small enough, the solution $x(t,s)$ is defined for all $t,s\in[0,1]$ and entirely contained in $\ol V$.

Now the right hand-side of the first equation in \eqref{eqn:phi_1} is locally Lipschitz w.r.t. $a$ and bounded measurable w.r.t. the parameter $t$. The initial value $a^i(t,0)$ depends continuously on $t$, hence, by Theorem \ref{thm:param}, the solution $a(t,s)$ locally exists, is AC w.r.t. $s$, and is bounded measurable w.r.t. $t$. Again, shrinking $\|b\|_{\sup}$ makes the Lipschitz constant smaller, hence for $\|b\|_{\sup}$ small enough $a(t,s)$ is defined for all $t,s\in[0,1]$.

Now let us consider \eqref{eqn:phi_1} with $b(t)=\int_0^td(s)\dd
s$, where $d$ is chosen in such a way, that $b(0)=b(1)=0$. We have
\begin{equation}\label{eqn:phi_2}
\left\{ \begin{aligned}
\pa_s a^i(t,s)&=\pa_tb^i(t)+c^i_{jk}(x(t,s))a^j(t,s)b^k(t),\\
\pa_s x^b(t,s)&=\rho^b_i(x(t,s))b^i(t).
\end{aligned}\right.
\end{equation}
We recognise equations \eqref{eqn:htp_smooth} for $E$-homotopy.
Indeed, in such a situation $a(t,s)$ and $b(t,s)=b(t)$ form an
$E$-homotopy with fixed end-points (since $b(0)=b(1)=0$).
Consequently, the homotopy classes of $a(t,0)=c_{\vec r}(t)$ and
$a(t,1)$ are equal. Since $\|b\|_{\sup}\leq\|d\|$, for
$\|d\|$ and $\|\vec r\|$ small enough, this homotopy is defined
for all $t,s\in[0,1]$. For $a(t)$ and $d(t)$ as above we define
$$\Phi(\vec r,d)(t):=a(t,1).$$
This is a (local) map of Banach spaces
$$\Phi:\R^m\times \BM_0(I,\R^m)\supset W_0\lra\BM(I,\R^m),$$
where $\BM_0(I,\R^m)=\{d\in \BM(I,\R^m):\int_0^1d(s)\dd s=0\}$
is a Banach subspace of $\BM(I,\R^m)$ and $W_0$ is some open neighbourhood of the point $(0,0)$. We shall now prove the following:
\begin{description}
    \item[(A)]\label{cond:A} $\Phi$ maps $(\vec{r},0)$ into a constant path $c_{\vec{r}}\in \BM(I,\R^m)$.
     \item[(B)]\label{cond:B} $\Phi$ is a continuous map of Banach spaces.
     \item[(C)]\label{cond:C} The $E$-homotopy class of the curve $\Phi(\vec{r},d)$ is determined by $\vec{r}$; that is,
$$\big[\Phi(\vec{r},d)\big]=\big[\Phi(\vec{r},0)\big]\overset{(A)}{=}\big[c_{\vec{r}}\big].$$
     \item[(D)]\label{cond:D} The map $\widetilde{\Phi}(\vec{r},d):=\Phi(\vec{r},d)-(c_{\vec{r}}+d)$ is Lipschitz with constant $\frac{1}{6}$.
     \item[(E)]\label{cond:E} The map $\Phi$ posses a continuous inverse $\Phi^{-1}$ defined on some open neighbourhood $V_0\ni 0$ in $\BM(I,\R^m)$. Moreover, $\Phi^{-1}$ is Lipschitz with constant $6$.
\end{description}

Property (\textbf{C}) is clear from the construction of $\Phi$, as
$a(t,1)=\Phi(\vec r,d)(t)$ and $a(t,0)=c_{\vec r}(t)$ are
$E$-homotopic. 

Property (\textbf{A}) is obvious, since $\Phi(\vec
r,0)$ is the solution (taken  at $s=1$) of the differential equation $\pa_sa(t,s)=0$
with the  initial condition $a(t,0)=\vec r$.

Property (\textbf{B}) will follow from (\textbf{D}). Indeed, if
$\wt\Phi$ is Lipschitz, then $\Phi(\vec
r,d)=\wt\Phi(\vec r,d)+c_{\vec r}+d$ is continuous as a sum of
continuous maps.

Assuming (\textbf{D}) again, we will be able to prove
(\textbf{E}). As one might have expected, the existence and the
Lipschitz condition for $\Phi^{-1}$ will be proven essentially in
the same way as in the standard proof of the inverse function
theorem (cf. \cite{Lang}). First, we will establish a pair
of linear isomorphism between Banach spaces
\begin{align*}
\BM(I,\R^m)&\overset{\alpha}{\lra}\R^m\times \BM_0(I,\R^m)\,,\\
a(t)&\longmapsto \left(\int_0^1a(s)\dd s,\ a(t)-\int_0^1a(s)\dd
s\right)\,,
\intertext{and}
\R^m\times \BM_0(I,\R^m)&\overset{\beta}{\lra}\BM(I,\R^m),\\
(\vec{r},d)&\longmapsto c_{\vec{r}}+d.
\end{align*}
It is straightforward to verify that $\alpha$ and $\beta$ are
continuous inverses of each other and that $\alpha$ is Lipschitz
with constant $3$. The map $\Phi$ is defined on some open
neighbourhood $W_0\ni(0,0)$. Take $R$ such that $B(0,2R)\subset
W_0$. The map $\widetilde{\Phi}$ is Lipschitz with constant
$\frac{1}{6}$ and $\alpha$ is Lipschitz with constant $3$; hence
$\alpha\circ\widetilde{\Phi}$ is Lipschitz with constant
$\frac{1}{2}$ and, since it preserves the origin, it maps the ball
$B(0,2R)$ into the ball $B(0,R)$.

Fix now any $a\in \BM(I,\R^m)$ such that $\|a\|<\frac{R}{3}$. We
shall construct a unique element $(\vec r,d)\in B(0,2R)\subset W_0$
satisfying $\Phi(\vec r,d)=a$. Consider a map
$\Phi_a(\vec{r},d):=\alpha(a-\widetilde{\Phi}(\vec{r},d))$. From the Lipschitzity of $\alpha$ and $\wt \Phi$ we deduce that 
$$\|\Phi_a(\vec r,d)\|\leq 3\|a-\wt\Phi(\vec r,d)\|\leq 3\|a\|+3\|\wt\Phi(r,d)\|\leq 3\cdot\frac R3+\frac 12\|(\vec r,d)\|.$$ 
Consequently, $\Phi_a$ maps
the ball $B(0,2R)$ into $B(0,2R)$. Moreover,
\begin{align*}
\|\Phi_a(\vec r,d)-\Phi_a(\vec r^{'},d^{'})\|&=\|\alpha(a-\wt \Phi(\vec r,d))-\alpha(a-\wt\Phi(\vec r^{'},d^{'}))\|\leq\\
&\leq 3\|\wt\Phi(\vec r,d)-\wt\Phi(\vec r^{'},d^{'})\|\leq 3\cdot\frac 16\|(\vec r,d)-(\vec r^{'},d^{'})\|,
\end{align*}
 hence $\Phi_a$ is a contraction. Now, using the Banach fixed point theorem,
we deduce that $\Phi_a$ has a unique fixed point $(\vec{r},d)\in
B(0,2R)$. Consequently,
$$a-\widetilde{\Phi}(\vec{r},d)=\beta\circ\alpha\left(a-\widetilde{\Phi}(\vec{r},d)\right)=\beta\circ\Phi_a(\vec{r},d)=\beta(\vec r,d)=c_{\vec{r}}+d,$$
and hence $a=\wt\Phi(\vec{r},d)+c_{\vec{r}}+d=\Phi(\vec{r},d).$
We have proven the existence of $\Phi^{-1}$.

Take now $a, a^{'}\in \BM(I,\R^m)$, and let
$\Phi^{-1}(a)=(\vec{r},d)$,
$\Phi^{-1}(a^{'})=(\vec{r}^{'},d^{'})$. Using the Lipschitz
condition for $\alpha$ and $\widetilde{\Phi}$ once more, we get
\begin{align*}
&\|(\vec{r},d)-(\vec{r}^{'},d^{'})\|=\|\Phi_a(\vec{r},d)-\Phi_{a^{'}}(\vec{r}^{'},d^{'})\|=
\|\alpha\Big(a-\widetilde{\Phi}(\vec{r},d)\Big)-\alpha\Big(a^{'}-\widetilde{\Phi}(\vec{r}^{'},d^{'})\Big)\|\leq\\
&\leq
3\|a-a^{'}\|+3\|\widetilde{\Phi}(\vec{r},d)-\widetilde{\Phi}(\vec{r}^{'},d^{'})\|\leq
3\|a-a^{'}\|+3\cdot\frac{1}{6}\|(\vec{r},d)-(\vec{r}^{'},d^{'})\|.
\end{align*}
We finish the proof of property (\textbf{E}) concluding that
$$\|\Phi^{-1}(a)-\Phi^{-1}(a^{'})\|=\|(\vec{r},d)-(\vec{r}^{'},d^{'})\|\leq 6\|a-a^{'}\|.$$

We are now left with the proof of (\textbf{D}). This will be done
by introducing several integral estimations. In our calculations
we will, for simplicity, omit the indices (hence $c$ will stand
for $c^i_{jk}$, $a$ for $a^i$, etc.). Take pairs $(r,d)$ and
$(r^{'},d^{'})$ from $\R^m\times\BM_0(I,\R^m)$. Denote by
$x(t,s)$, $a(t,s)$, $a(t)$, $b(t)$ and  $x^{'}(t,s)$,
$a^{'}(t,s)$, $a^{'}(t)$, $b^{'}(t)$, respectively, the objects
defined as in the construction of $\Phi$ for pairs $(\vec r,d)$ and $(\vec r^{'},d^{'})$.
To begin with, observe that, since $b(t)=\int_0^td(s)\dd s$, we
have $|b(t)|\leq\int_0^1|d(s)|\dd s$; hence
$$\|b\|_{\sup}\leq\|d\|.$$
Similarly, $\|b^{'}\|_{\sup}\leq\|d^{'}\|$ and
$\|b-b^{'}\|_{\sup}\leq\|d-d^{'}\|$.

Let us now estimate the difference $|x(t,s)-x^{'}(t,s)|$. Since,
by \eqref{eqn:phi_2},
$x(t,s)=x(t)+\int_0^s\rho(x(t,\sigma))b(t)\dd\sigma$, we have
\begin{align*}
|x(t,s)-x^{'}(t,s)|&=\left|\int_0^s\Big(\rho(x(t,\sigma))b(t)-\rho(x^{'}(t,\sigma))b^{'}(t)\Big)\dd\sigma\right|\leq\\
&\leq\int_0^1\Big|\rho(x(t,s))b(t)-
\rho(x^{'}(t,s))b^{'}(t)\Big|\dd s
\leq\\
&\leq\int_0^1\Big|\Big(\rho(x(t,s))-\rho(x^{'}(t,s))\Big)b(t)\Big|\dd
s+\\&\phantom{=}+
\int_0^1\Big|\rho(x^{'}(t,s))\Big(b(t)-b^{'}(t)\Big)\Big|\dd s\leq\\
&\leq
L_\rho\sup_{(t,s)}|x(t,s)-x^{'}(t,s)|\|d\|+C_\rho\|d-d^{'}\|.
\end{align*}
It follows that
$$
\sup_{(t,s)}|x(t,s)-x^{'}(t,s)|\Big(1-L_\rho \|d\|\Big)\leq
C_\rho\|d-d^{'}\|.$$ For $\|d\|$ sufficiently small (in other
words, after a possible shrinking of $W_0$) we will obtain
\begin{equation}\label{est:x}
\sup_{(t,s)}|x(t,s)-x^{'}(t,s)|\leq 2C_\rho\|d-d^{'}\|.
\end{equation}
Now introduce
\begin{align*}
&\del a(t,s):= a(t,s)-(a(t)+sd(t)) \quad \text{and}\\
&\del a^{'}(t,s):= a^{'}(t,s)-(a^{'}(t)+sd^{'}(t)).
\end{align*}
Note that $\del a(t,1)=\wt\Phi(\vec r,d)$ and $\del a^{'}(t,1)=\wt\Phi(\vec r^{'},d^{'})$. 
From \eqref{eqn:phi_1} we deduce that
$$\del a(t,s)=\int_0^s c(x(t,\sigma))a(t,\sigma)b(t)\dd\sigma=\int_0^s c(x(t,\sigma))\Big(\del a(t,\sigma)+a(t)+\sigma d(t)\Big)b(t)\dd\sigma.$$
Further, for a fixed $s\in[0,1]$,
\begin{align*}
\|\del a(\cdot,s)\|&=\int_0^1|\del a(t,s)|\dd
t\leq\int_0^1\int_0^s\left|c(x(t,\sigma))\Big(\del
a(t,\sigma)+a(t)+
\sigma d(t)\Big)b(t)\right|\dd\sigma\dd t\leq\\
&\leq \int_0^1\int_0^1C_c\big(|\del a(t,s)|+\|\vec r\|+|d(t)|\big)\|b\|_{\sup}\dd s\dd t\leq\\
&\leq C_c\|b\|_{\sup}\left(\sup_s\|\del a(\cdot,s)\|+\|\vec
r\|+\|d\|\right)\leq\\
&\leq C_c\|d\|\left(\sup_s\|\del
a(\cdot,s)\|+\|\vec r\|+\|d\|\right)
\end{align*}
and we conclude that
$$\sup_s\|\del a(\cdot,s)\|(1-C_c\|d\|)\leq C_c(\|\vec r\|+\|d\|).$$
Hence, for $\|d\|$ small enough (after possible shrinking of
$W_0$), we get
\begin{equation}\label{est:del_a}
\sup_s\|\del a(\cdot,s)\|\leq 2C_c(\|\vec r\|+\|d\|).
\end{equation}
Finally,
\begin{align*}
\|\del a&(\cdot,s)-\del a^{'}(\cdot,s)\|=\int_0^1
|\del a(t,s)-\del a^{'}(t,s)|\dd t\leq \\
\leq&\int_0^1\int_0^s\left|c(x(t,s))\Big(\del
a(t,\sigma)+a(t)+\sigma d(t)\Big)b(t)\right.+\\&- \left.c(x^{'}(t,\sigma))\left(\del
a^{'}(t,\sigma)+a^{'}(t)+
\sigma d^{'}(t)\right)b^{'}(t)\right|\dd\sigma\dd t\leq\\
\leq&\int_0^1\int_0^1\left|c(x(t,\sigma))\del a(t,s)b(t)-c(x^{'}(t,s))\del a^{'}(t,s)b^{'}(t)\right|\dd s\dd t+\\
&+\int_0^1\int_0^1\left|c(x(t,s))a(t)b(t)-c(x^{'}(t,s))a^{'}(t)b^{'}(t)\right|\dd s\dd t+\\
&+\int_0^1\int_0^1\left|c(x(t,s))d(t)b(t)-c(x^{'}(t,s))
d^{'}(t)b^{'}(t)\right|\dd s\dd t=:I_1+I_2+I_3
\end{align*}
Now we estimate
\begin{align*}
I_1&\leq\int_0^1\int_0^1\left|c(x(t,s))-c(x^{'}(t,s))\right||\del a(t,s)||b(t)|\dd t\dd s+\\
&\phantom{=}+\int_0^1\int_0^1\left|c(x^{'}(t,s))\right|\left|\del a(t,s)-\del a^{'}(t,s)\right|\left|b(t)\right|\dd t\dd s+\\
&\phantom{=}+\int_0^1\int_0^1\left|c(x^{'}(t,s))\right|\left|\del a^{'}(t,s)\right|\left|b(t)-b^{'}(t)\right|\dd t\dd s\leq\\
&\leq L_c\sup_{(t,s)}|x(t,s)-x^{'}(t,s)|\sup_s\|\del
a(\cdot,s)\|\cdot\|b\|_{\sup}+
C_c\sup_s\|\del a(\cdot,s)-\del a^{'}(\cdot,s)\|\cdot\|b\|_{\sup}+\\
&\phantom{=}+C_c\sup_s\|\del a^{'}(\cdot,s)\|\cdot\|b-b^{'}\|_{\sup}\,.
\end{align*}
Finally, using \eqref{est:x}, \eqref{est:del_a} and $\|b\|_{\sup}\leq\|d\|$, we get
\begin{align*}
I_1&\leq L_c 2C_\rho\|d-d^{'}\|2C_c(\|\vec r\|+\|d\|)\|d\|+
C_c\sup_s\|\del a(\cdot,s)-\del a^{'}(\cdot,s)\|\cdot\|d\|
+C_c2C_c(\|\vec r^{'}\|\\&+\|d^{'}\|)\|d-d^{'}\|
=C_c\|d\|\sup_s\|\del a(\cdot,s)-\del
a^{'}(\cdot,s)\|+\|d-d^{'}\|\cdot
F_1(\|r\|,\|d\|,\|r^{'}\|,\|d^{'}\|),
\end{align*}
where $F_1$ converges to 0 when its arguments do. Similar
estimations for $I_2$ and $I_3$ will give
\begin{align*}
&I_2\leq \left(\|r-r^{'}\|+\|d-d^{'}\|\right)\cdot F_2(\|r\|,\|d\|,\|r^{'}\|,\|d^{'}\|),\\
&I_3\leq \left(\|r-r^{'}\|+\|d-d^{'}\|\right)\cdot
F_3(\|r\|,\|d\|,\|r^{'}\|,\|d^{'}\|),
\end{align*}
where $F_2$ and $F_3$ behave as $F_1$. Putting together the
partial results, we would get
$$\sup_s\|\del a(\cdot,s)-\del a^{'}(\cdot,s)\|(1-2C_c\|d\|)\leq  \left(\|r-r^{'}\|+\|d-d^{'}\|\right)\cdot F\left(\|r\|,\|d\|,\|r^{'}\|,\|d^{'}\|\right),$$
where $F$ converges to 0 when its arguments do. As
$\widetilde{\Phi}(r,d)=\del a(t,1)$ and
$\widetilde{\Phi}(r^{'},d^{'})=\del a^{'}(t,1)$, for $W_0$ small
enough, $\wt\Phi$ is Lipschitz with constant $\frac 16$. That
proves property (\textbf{D}).

Now using properties (\textbf{A})--(\textbf{E}) of $\Phi$ we will
make the final step of the proof of Lemma \ref{lem:htp_of_a_r}.
The family $a_{\vec r}(t)$ is uniformly regular w.r.t. $\vec r\in
\ol B^m(0,1)$ at $t=0$. The family $c_{\vec r}(t)$ has the same
properties, so $a_{\vec r}(t)-c_{\vec r}(t)$ is also uniformly
regular (cf. Proposition \ref{prop:ur_sum}) and
$$\int_0^t|a_{\vec r}(s)-c_{\vec r}(s)|\dd s=t|a_{\vec r}(0)-c_{\vec r}(0)|+f(t,\vec r)=t\cdot 0+f(t,\vec r),$$
where $\frac 1tf(t,\vec r)$ converges uniformly to $0$ as $t\to
0$. Hence, there exists a number $\eta>0$ such that
$$\int_0^\eps|a_{\vec r}(s)-c_{\vec r}(s)|\dd s<\frac 1{12}\eps,$$
for every $0\leq\eps\leq\eta$ and $\vec{r}\in\ol B^m(0,1)$. Reparametrising the
paths by the rule
$$\wt a_{\vec r}(t):=\eps a_{\vec r}(\eps t),$$
we will obtain another uniformly regular family of paths satisfying $\wt a_{\vec r}(0)=\eps \vec r$.

By the uniform regularity of $\wt a_{\vec r}$, the map $\wt a:\vec r\mapsto \wt a_{\vec r}$ is a continuous map from $B^m(0,1)$ to $\BM(I,\R^m)$ with $L_1$-topology. For $\eta$ small enough $\wt a$ takes values in $\Phi(W_0)$. 
Composing $\wt a$ with
$\Phi^{-1}$ we will obtain a continuous map
$$\ol B^m(0,1)\xra{(\psi,\phi)}\R^m\times \BM_0(I,\R^m).$$

Observe that 
$$\|\wt a_{\vec r}-c_{\eps\vec r}\|=\int_0^1|\eps a_{\vec
r}(t\eps)-\eps c_{\vec r}(t\eps)|\dd t=\int_0^\eps| a_{\vec
r}(t)- c_{\vec r}(t)|\dd t\leq\frac 1{12}\eps.$$
Using this, the Lipschitz condition for  $\Phi^{-1}$ (property (\textbf{E})), and the fact that $\Phi^{-1}(c_{\vec r})=(\vec r,0)$ (property (\textbf{A})), we obtain
$$\|\frac 1\eps\psi(\vec{r})-\vec{r}\|\leq
\frac 1\eps\|\psi(\vec{r})-\eps\vec{r}\|+\frac 1\eps\|\phi(\vec{r})-0\|=
\frac 1\eps\|\Phi^{-1}(\widetilde{a}_{\vec{r}})-\Phi^{-1}(c_{\eps\vec{r}})\|\leq\frac
6\eps\|\widetilde{a}_{\vec{r}}-c_{\eps
\vec{r}}\|\leq\frac{1}{2}.$$
In other words, $\wt\psi:=\frac 1\eps \psi$ maps a ball $B^m(0,1)$ continuously into $\R^m$ in such a way that $\|\wt\psi(\vec r)-\vec r\|\leq \frac 12$. By Lemma \ref{lem:top1_app}, point $0\in\R^m$ lies in the image of $\wt\psi$. However, that means that $\psi(\vec {r_0})=0$ for some $\vec{r_0}$, and hence $\wt a_{\vec{r_0}}=\Phi(0,d)$ for  some $d\in\BM_0(I,\R^m)$. 
By property (\textbf{C}), $\left[\wt a_{\vec{r_0}}\right]_{t\in[0,1]}=0$. Finally, by Lemma \ref{lem:reparam},
$$0=\left[\wt a_{\vec r_0}(t)\right]_{t\in[0,1]}=\left[a_{\vec r_0}(t)\right]_{t\in[0,\eps]},$$
which finishes the proof.
\end{proof}

\subsection{Second lemma}

We will now formulate and prove a result generalising Lemma \ref{lem:htp_of_a_r}. We will work in the same local setting as before. Let us introduce a decomposition $\R^m=\R^{m_1}\oplus\R^{m_2}$ and fix $\vec k_0\in\R^m$.

\begin{lemma}\label{lem:htp_of_a_r1}
Let $a_{\vec{r}}(\cdot)\in\BM(I,R^m)$, where $\vec r\in B^{m_1}(0,1)\oplus\theta_{m_2}$, be
a family of $E$-paths uniformly regular at $t=0$ w.r.t. $\vec r$
and such that  $a_{\vec{r}}(0)=\vec{r}+\vec k_0$. Let $b_{\vec{s}}(\cdot)\in\BM(I,R^m)$, where $\vec s\in \theta_{m_1}\oplus B^{m_2}(0,1)$, be
a family of $E$-paths uniformly regular at $t=0$ w.r.t. $\vec s$
and such that  $b_{\vec{s}}(0)=\vec{s}+\vec k_0$. 

Then there exists a number $\eta>0$ such that, for every $0<\eps<\eta$, there exists  
vectors $\vec r_0$ and $\vec s_0$ ensuring that the curves
$a_{\vec r_0}(t)$ and $b_{\vec s_0}(t)$, after restricting to the interval $[0,\eps]$,
realise the same $E$-homotopy class
$$\big[a_{\vec{r}_0}(t)\big]_{t\in[0,\eps]}=\big[b_{\vec s_0}(t)\big]_{t\in[0,\eps]}.$$
\end{lemma}
\begin{proof}
We will follow the scheme of the final part of the proof of Lemma \ref{lem:htp_of_a_r}, making extensive use of the map $\Phi:\R^m\times\BM_0(I,\R^m)\supset W_0\ra\Phi(W_0)\subset\BM(I,\R^m)$ constructed before.

By the uniform regularity of the families $a_{\vec r}$, $b_{\vec s}$, and $c_{\vec R}$; there exists a number $\eta>0$ such that, for $0<\eps<\eta$,
\begin{align*}
&\int_0^\eps|a_{\vec r}(t)-c_{\vec r+\vec k_0}(t)|\dd t<\frac 1{24}\eps, &&
\text{ for every $\vec{r}\in B^{m_1}(0,1)\oplus\theta_{m_2}$ and}\\
&\int_0^\eps|b_{\vec s}(t)-c_{\vec s+\vec k_0}(t)|\dd t<\frac 1{24}\eps, &&
\text{ for every $\vec{s}\in \theta_{m_1}\oplus B^{m_2}(0,1)$.}
\end{align*}
Now let us reparametrise the paths $a_{\vec r}$ and $b_{\vec s}$ by the rule
\begin{align*}
&\wt a_{\vec r}(t):=\eps a_{\vec r}(\eps t)   \quad\text{and}\\
&\wt b_{\vec s}(t):=\eps b_{\vec s}(\eps t) \quad \text{for $t\in[0,1]$.}
\end{align*} 
We obtained another uniformly regular families of $E$-paths satisfying $\wt a_{\vec r}(0)=\eps(\vec r+\vec k_0)$ and $\wt b_{\vec s}(0)=\eps(\vec s+\vec k_0)$. Moreover, $\|\wt a_{\vec r}-c_{\eps(\vec r+\vec k_0)}\|\leq \frac 1{24}\eps$ and $\|\wt b_{\vec s}-c_{\eps(\vec s+\vec k_0)}\|\leq \frac 1{24}\eps$, and since $\eta$ can be chosen arbitrary small, we may assume that $\wt a_{\vec r}$ and $\wt b_{\vec s}$ belong to $\Phi(W_0)$ for all $\vec r$ and $\vec s$. 

By the uniform regularity of $\wt a_{\vec r}$ and $\wt b_{\vec s}$, the maps $\wt a:\vec r\mapsto\wt a_{\vec r}$ and $\wt b:\vec s\mapsto\wt b_{\vec s}$ are continuous maps form $B^{m_1}(0,1)\oplus\theta_{m_2}$ and $\theta_{m_1}\oplus B^{m_2}(0,1)$, respectively, to $\BM(I,\R^m)$ with $L_1$-topology. Composing them with $\Phi^{-1}$ we will obtain continuous maps
\begin{align*}
B^{m_1}(0,1)\oplus\theta_{m_2}&\xra{(\psi_a,\phi_a)}\R^m\times \BM_0(I,\R^m) \quad\text{and}\\
\theta_{m_1}\oplus B^{m_2}(0,1)&\xra{(\psi_b,\phi_b)}\R^m\times \BM_0(I,\R^m). 
\end{align*}
Now, using the Lipschitz condition for $\Phi^{-1}$ and the fact that $\Phi^{-1}(c_{\vec r})=(\vec r,0)$, we can estimate in a way analogous as in the proof of Lemma \ref{lem:htp_of_a_r} that
\begin{align*}
&\|\frac 1\eps\psi_a(\vec r)-(\vec r+\vec k_0)\|\leq \frac 14 &&\text{for $\vec r\in B^{m_1}(0,1)\oplus\theta_{m_2}$ and}\\
&\|\frac 1\eps\psi_b(\vec s)-(\vec s+\vec k_0)\|\leq \frac 14 &&\text{for $\vec s\in \theta_{m_1}\oplus B^{m_2}(0,1)$.}
\end{align*}
By Lemma \ref{lem:top2_app}, the images of the maps $\frac 1\eps \psi_a$ and $\frac 1\eps \psi_b$ have a nonempty intersection. In other words, there exist $\vec r_0$ and $\vec s_0$, vector $\vec R\in \R^m$, and $d_a,d_b\in \BM_0(I,\R^m)$ such that $\wt a_{\vec r_0}=\Phi(\vec R,d_a)$ and $\vec b_{\vec s_0}=\Phi(\vec R,d_b)$. By property (\textbf{C}), the $E$-homotopy classes of $\wt a_{\vec r_0}$ and $\vec b_{\vec s_0}$ are equal. Consequently,  by Lemma \ref{lem:reparam},
$$\left[a_{\vec r_0}(t)\right]_{t\in[0,\eps]}=\left[\wt a_{\vec r_0}(t)\right]_{t\in[0,1]}=\left[\wt b_{\vec s_0}(t)\right]_{t\in[0,1]}=\left[b_{\vec s_0}(t)\right]_{t\in[0,\eps]},$$
which finishes the proof. 
\end{proof}

\chapter{The proof of the PMP}\label{ch:proof}

In this chapter we will finish the proof of Theorems \ref{thm:pmp_A} and \ref{thm:pmp_A_rel}. In our considerations it is crucial to understand the geometry of the cone $\B K^u_\tau$ of infinitesimal variations along the optimal trajectory $\B f(\B x(t),u(t))$. We interpreted $\B K^u_\tau$ as the set of all
directions in $\bm A_{\bm x(t_1)}$ in which one can move the
point $\bm x(t_1)$ by performing needle variations of the control
$u$ associated with symbols $\A=(\tau_i, v_i,\tau,\delta t_i, \delta t)_{i=1,\hdots,k}$, where $\tau$ is fixed. Consequently, a movement in the direction of the ray 
$$\bm\Lambda_{\bm
x(t_1)}:=\theta_{x(t_1)}\oplus\R_+\cdot(-\pa_t)\subset
E_{x(t_1)}\oplus\sT_{\ul{x}(t_1)}\R=\bm A_{\bm x(t_1)}$$ 
would correspond to a variation which
decreases the total cost of the trajectory without making changes
in the $E$-evolution. Such a behaviour should not be
possible if $\B f(\B x(t),u(t))$ is a solution of the OCP \eqref{eqn:P_A}, so we may expect that the ray $\B\Lambda_{\B x(t_1)}$ can be separated from the cone $\B K^u_\tau$ in such a case. This result is formulated in Theorem \ref{thm:separation_K_Lambda}. In the proof we use technical Lemma \ref{lem:htp_of_a_r} to deduce the existence of $E$-paths realising certain $E$-homotopy classes from the infinitesimal picture expressed in the language of the cone $\B K^u_\tau$ and the ray $\B\Lambda_{\B x(t_1)}$. When Theorem \ref{thm:separation_K_Lambda} is proved, to finish the proof of Theorem \ref{thm:pmp_A} we need only to follow a few rather technical steps from the original proof of Pontryagin and his collaborators \cite{pontryagin}. 

Theorem \ref{thm:pmp_A_rel} is proven analogously with some technical modifications. The main difference is that instead of $\B K^u_\tau$ we use a bigger set $\B\K^u_\tau$ which contains information about both needle variations and initial variations of a given trajectory. Using technical Lemma \ref{lem:htp_of_a_r1} we prove Theorem \ref{thm:separation_K_S} describing the geometry of  $\B\K^u_\tau$. Then, basing on this result, we make a few final steps after \cite{pontryagin}.

\section [The proof of Theorem 5.3]{The proof of Theorem \ref{thm:pmp_A}}\label{sec:proof_pmp}

\subsection{The geometry of the cone \texorpdfstring{$\bm K^u_\tau$}{TEXT} }

Throughout this section we assume that the controlled pair $(\bm x(t),u(t))$ is a solution of the OCP \eqref{eqn:P_A}. All results obtained in this section are valid under this assumption. 

\begin{theorem}\label{thm:separation_K_Lambda} Let $(\B{x}(t), u(t))$, for $t\in[t_0,t_1]$, be a solution of the optimal control
problem \eqref{eqn:P_A}. Then the ray
$\bm{\Lambda}_{\B x(t_1)}$ and the convex cone
$\B K^u_\tau$ can be separated for any $\tau\in(t_0,t_1)$, which is a regular point of $u$.
\end{theorem}

The idea of the proof is the following.  Assuming the contrary we will
construct a family of symbols
$\A(\vec r)$, where $\vec r\in E_{x(t_1)}$, such that the associated infinitesimal variations
$\bm d^{\A(\vec r)}(s)$ are uniformly regular w.r.t. $\vec r$ at $s=0$ and  point into the 
directions $\vec r-\pa_t\in E_{x(t_1)}\oplus\T_{\ul x(t_1)}\R=\B A_{\B x(t_1)}$.
For such a family we will be able to use Lemma \ref{lem:htp_of_a_r} to deduce 
that for some $\vec r_0$ the variation $\B d^{\A(\vec r_0)}(s)$ has special properties. Next we will show that in such a case the pair $(\B x(t),u(t))$ cannot be a solution of the OCP \eqref{eqn:P_A}.

\begin{proof}
Assume the contrary, i.e., that the convex cone $\B K^u_\tau$ and the ray $\bm{\Lambda}_{\B{x}(t_1)}$ cannot be separated. Denote by  $\B \lambda:=\theta_{x(t_1)}-\pa_t\in E_{x(t_1)}\oplus\T_{\ul{x}(t_1)}\R$ a vector spanning $\B \Lambda_{\B x(t_1)}$. It follows from Lemma \ref{lem:separation} that  there exists a basis $\{\B e_1,\hdots,\B e_m\}$ of
$ E_{x(t_1)}\oplus\theta_{\ul{x}(t_1)}\subset
\B{A}_{\B x(t_1)}$ such that vectors
$\B\lambda,\B\lambda+\B e_i,\B\lambda-\B e_i$ lie in
$\B K^u_\tau$ for $i=1,\hdots, m$. 

Denote by $\A$, $\A_i$, and $\BB_i$ some symbols such that elements in $\bm K_\tau^u$ corresponding to these symbols are $\bm d^\A(0)=\B\lambda$, $\bm d^{\A_i}(0)=\B\lambda+\B e_i$, and $\bm d^{\BB_i}(0)=\B\lambda-\B e_i$, respectively. We deal with a finite set of symbols, hence we can assume that they all are of the form $(\tau_i, v_i,\tau, \delta t_i, \delta t)_{i=1,\hdots,k}$, where $\tau_i$, $v_i$, k, and $\tau$ are
fixed, and that they differ by $\del t_i$ and $\del t$ (we can
always add a triple $(\tau_i,v_i,\del t_i=0)$ to a symbol without
changing anything --- cf. Remark \ref{rem:var_dt=0}). For any
$\vec{r}=\sum_{i=1}^mr^i\B e _i\in B^m(0,1)=\{\vec r:\|\vec r\|=\sum_i|r_i|\leq 1\}$  we
may define a new symbol
$$\A(\vec{r})=\left(1-\sum_{i=1}^m|r^i|\right)\A+\sum_{i=1}^m
h^+(r^i)\A_i+\sum_{i=1}^m h^-(r^i)\BB_i,$$ where
$h^+(r)=\max\{r,0\}$ and $h^-(r)=\max\{-r,0\}$ are non-negative, and
the convex combination of symbols is defined using the natural
rule
\begin{align*}&\nu(\tau_i, v_i,\tau, \delta t_i, \delta
t)_{i=1,\hdots,k}+\mu(\tau_i, v_i,\tau, \delta t^{'}_i, \delta
t^{'})_{i=1,\hdots,k}\\&=(\tau_i, v_i,\tau, \nu\delta t_i+\mu\del
t_i^{'}, \nu\delta t+\mu\del t^{'})_{i=1,\hdots,k}\,.
\end{align*} 
We will now study the properties of $\B A$-paths $s\mapsto\B d^{\A(\vec r)}(s)$ corresponding to symbols $\A(\vec r)$ (see Remark \ref{rem:interpretation_K}).  

From \eqref{eqn:d_at_0} it is straightforward to verify that 
$$\B{d}^{\A(\vec{r})}(0)=\B\lambda+\vec{r}.$$

If follows from Theorem \ref{thm:1st_main} that, since the numbers $\del t(\vec r)$ and $\del t_i(\vec r)$ in the symbol $\A(\vec r)$ depend continuously on $\vec r$, which takes values in a compact set, we may choose $\theta>0$ such that $\bm d^{\A(\vec r)}(s)$ is well-defined for $s\in[0,\theta]$ and all $\vec r\in B^m(0,1)$,

Now consider $\ul d^{\A(\vec r)}(s)$ --- the projections of the family of $\B A$-paths $\B d^{\A(\vec r)}(s)$ to the algebroid $\T\R$.
Observe that since $\bm d^{\A(\vec r)}(s)$ are uniformly regular, so are  $\ul d^{\A(\vec r)}(s)$. Since in canonical coordinates on $\T\R$ we have $\ul d^{\A(\vec
r)}(0)=-1$, there exist a number $0<\eta\leq\theta$ such that 
\begin{equation}\label{eqn:cost_less_zero}
\int_0^\eps\ul d^{\A(\vec r)}(s)\dd s<0
\end{equation}  for all
$\eps\leq\eta$ and all $\vec r\in B^m(0,1)$. This property will be
used later.

After projecting $\B d^{\A(\vec{r})}(s)$ from $\B A $ onto $E$,
we obtain a family of bounded measurable admissible paths
$d^{\A(\vec{r})}(s)$, again uniformly
regular at $s=0$ w.r.t. $\vec{r}\in {B}^m(0,1)\subset \R^m$,
and such that
$d^{\A(\vec{r})}(0)=\vec{r}\in\R^m\approx E_{x(t_1)}$. In other words, the paths $d^{\A(\vec r)}(\cdot)$ satisfy the assumptions of Lemma \ref{lem:htp_of_a_r} and, consequently, there exists a vector $\vec r_0\in B^m(0,1)$ and a number $0<\eps< \eta$ such that 
\begin{equation}\label{eqn:htp_zero}
\left[d^{\A(\vec r_0)}(t)\right]_{r\in[0,\eps]}=\left[0\right].
\end{equation}

Properties \eqref{eqn:cost_less_zero} and \eqref{eqn:htp_zero} contradict the optimality of $\B f(\B x(t),u(t))$. Indeed, from \eqref{eqn:var_htp} we know that $\B f(\B x(t,s), u_s(t))$ --- the variation of the trajectory $\B f(\B x(t),u(t))$ associated with a needle variation $u_s(t):=u^{\A(\vec r_0)}_s(t)$ satisfies
$$\left[\B{f}(\B{x}(t,\eps),u_\eps(t))\right]_{t\in[t_0,t_1+\eps\del
t(\vec r_0)]} =\left[\B{f}(\B{x}(t),u(t))\right]_{t\in[t_0,t_1]}\left[\B{d}^{\A(\vec
r_0)}(s)\right]_{s\in[0,\eps]}.$$
Projecting the above equality to the algebroid $E$ and using \eqref{eqn:htp_zero} we get
$$\left[f(x(t,\eps),u_\eps(t))\right]_{t\in[t_0,t_1+\eps\del
t(\vec r_0)]}=\left[f(x(t),u(t))\right]_{t\in[t_0,t_1]},$$
and hence the $E$-homotopy classes agree.

What is more, the $\T\R$-projection gives
$$\left[L(x(t,\eps),u_\eps(t))\right]_{t\in[t_0,t_1+\eps\del
t(\vec r_0)]}=\left[L(x(t),u(t))\right]_{t\in[t_0,t_1]}\left[\ul d^{\A(\vec
r_0)}(s)\right]_{s\in[0,\eps]}.$$ From \eqref{eqn:cost_less_zero} we deduce
that the total costs satisfy the following inequality:
\begin{align*}\ul x(t_1+\eps\del t(\vec r_0),\eps)=&\int_{t_0}^{t_1+\eps\del
t(\vec r_0)}L(x(s,\eps),u_\eps(s))\dd s=\int_{t_0}^{t_1}L(x(s),u(s))\dd
s+\int_0^\eps \ul d^{\A(\vec r_0)}(s)\dd
s\\ \overset{\eqref{eqn:cost_less_zero}}<
&\int_{t_0}^{t_1}L(x(s),u(s))\dd s=\ul x(t_1).
\end{align*}
The above inequality proves that $\B f(\B x(t),u(t))$ cannot be a solution of the OCP \eqref{eqn:P_A}, which stays in a contradiction to our assumptions. 
\end{proof}

To finish the proof of Theorem \ref{thm:pmp_A} we will now follow the steps of the original result of \cite{pontryagin}. All
the important information is contained in Theorem
\ref{thm:1st_main} telling us that the set of infinitesimal
variations $\bm K_\tau^u$ is a convex cone with elements defined by means of a
local one-parameter group $\bm B_{tt_0}$ (see
\eqref{eqn:d_at_0}) and in Theorem \ref{thm:separation_K_Lambda} describing the geometry of this cone. The structure of an AL
algebroid, necessary to prove the above results, will now play no
essential role.

\subsection{The construction of \texorpdfstring{$\bm{\xi}(t)$}{TEXT} and the ``maximum principle''}
Fix an element $\tau\in(t_0,t_1)$ to be a regular point of $u$. In
view of Theorem \ref{thm:separation_K_Lambda} there exists a non-zero covector
$\bm\xi(t_1)\in\bm A^\ast_{\bm x(t_1)}$ separating $\bm K^u_\tau$ and $\bm\Lambda_{\bm x(t_1)}$; that is (confront Remark \ref{rem:separation}),
\begin{equation}
\label{eqn:K_cov} 
\<\bm d,\bm\xi(t_1)>_{\bm{\tau}}\leq 0\leq\<\bm \lambda,\bm \xi(t_1)>_{\bm \tau} \text{
\ for every $\bm d\in\bm K^u_\tau$}.
\end{equation}
Let us define $\bm\xi(t):=\bm B^\ast_{t t_1}(\bm\xi(t_1))\in\bm
A^\ast_{\bm x(t)}$ for $t\in[t_0,t_1]$.

\begin{lemma}\label{lem:HM} For every $t\in[t_0,\tau]$ which is a regular point of the control $u(t)$ the following ``maximum principle'' holds:
$$\bm{H}(\bm x(t),\bm \xi(t),u(t))=\sup_{v\in U}\bm{H}(\bm x(t),\bm\xi(t),v).$$
Moreover, $\bm{H}(\bm x(\tau),\bm\xi(\tau),u(\tau))=0$.
\end{lemma}
\begin{proof}
Choose a regular point $t\in[t_0,\tau]$, take an arbitrary element
$v\in U$ and a number $\del t_1>0$, and consider a symbol
$\A=(\tau_1=t,\delta t_1,v_1=v,\tau,\delta t=0)$. The
corresponding element $\bm{d}^\A(0)\in \bm K^u_\tau$ equals
$\bm{B}_{t_1 t}[\bm f (\bm x (t),v)-\bm f (\bm x (t),u(t))]\delta
t_1$ (cf. \eqref{eqn:d_at_0}). From \eqref{eqn:K_cov} we obtain
\begin{align*}
&\:\,0\geq\<\bm{B}_{t_1
t}[\bm{f}(\bm{x}(t),v)-\bm{f}(\bm{x}(t),u(t))],\bm\xi(t_1)>_{\bm{\tau}}\delta
t_1\\
&\overset{\text{rem. } \ref{rem:B_paring}}{=}
\<\bm{B}_{t t_1}\bm{B}_{t_1
t}[\bm{f}(\bm{x}(t),v)-\bm{f}(\bm{x}(t),u(t))],\bm{B}^\ast_{t t_1}(\bm\xi(t_1))>_{\bm{\tau}}\delta t_1\\
&\phantom{X}=\<\bm f(\bm{x}(t),v)-\bm{f}(\bm{x}(t),u(t)),\bm\xi(t)>_{\bm\tau}\delta t_1\\
&\phantom{X}=\Big(\bm{H}(\bm x(t),\bm{\xi}(t),v)-\bm{H}(\bm
x(t),\bm{\xi}(t),u(t))\Big)\delta t_1.
\end{align*}
Since $\del t_1>0$, we have $\bm{H}(\bm
x(t),\bm{\xi}(t),v)\leq\bm{H}(\bm x(t),\bm{\xi}(t),u(t))$
for arbitrarily chosen $v\in U$.

To prove the second part of the assertion, consider a symbol
$\mathfrak{v}=(\tau,\del t)$. The associated element
$\bm{d}^{\mathfrak{v}}(0)$ is $\bm B_{t_1\tau}\left(\bm f (\bm x (\tau),u(\tau))\right)\del
t\in\bm K^u_\tau$. Consequently, from \eqref{eqn:K_cov}, we obtain
\begin{align*}
&\:\,0\geq\<\bm B_{t_1\tau}\left(\bm f (\bm{x}(\tau),u(\tau))\right),\bm\xi(t_1)>_{\bm{\tau}}\delta t\\
&\overset{\text{rem. } \ref{rem:B_paring}}{=}
\<\bm f (\bm{x}(\tau),u(\tau)),\bm\xi(\tau)>_{\bm{\tau}}\delta t=\bm{H}(\bm x(\tau),\bm{\xi}(\tau),u(\tau))\del t.
\end{align*}
Since $\del t$ can be arbitrary, we deduce that $\bm{H}(\bm
x(\tau),\bm{\xi}(\tau),u(\tau))=0$.
\end{proof}

\subsection{The condition \texorpdfstring{$\bm{H}(\bm{x}(t),\bm{\xi}(t),u(t))=0$}{TEXT} } To finish the proof just two more things are left.
We have to check that the Hamiltonian
$\bm{H}(\bm{x}(t),\bm\xi(t),u(t))$ is constantly 0 along the
optimal trajectory, and we have to extend the ''maximum principle'' to all regular  $t\in[t_0,t_1]$ (so far it holds only on the interval $[t_0,\tau]$, where
$\tau<t_1$ is a fixed regular point).

\begin{lemma}\label{lem:M=0}
For $\bm\xi(t)$ defined as above, the equality
$\bm{H}(\bm{x}(t),\bm\xi(t),u(t))=0$ holds at every regular point
$t\in[t_0,\tau]$ of the control $u$.
\end{lemma}

\begin{proof}
Denote by $P$ the closure of the set $\{u(t):t\in[t_0,\tau]\}$.
Since $u(t)$ is bounded, $P$ is a compact subset of $U$. Define a
new function $\m:\B{A}^*\lra\R$ by the formula
$$\m(\bm x,\bm{\xi}):=\max_{v\in P}\bm{H}(\bm x,\bm{\xi},v).$$
It follows from the previous lemma that
$\m(\bm{x}(t),\bm\xi(t))=\bm{H}(\bm x(t),\bm \xi(t),u(t))$ at
every regular point $t$ of $u$. We shall show that
$\m(\bm{x}(t),\bm\xi(t))$ is constant on $[t_0,\tau]$, and hence
equals $\bm{H}(\bm x(\tau),\bm\xi(\tau),u(\tau))=0$ (confront Lemma
\ref{lem:HM}). Observe that the function $\bm H\left(\bm
x(t),\bm\xi(t),v\right)$ is uniformly (for all $v\in P$) Lipschitz
w.r.t. $t$. Indeed, in local coordinates $(x^a,\ul
x,y^i,\ul y)$ on $\bm A$ and $(x^a,\ul x,\xi_i,\ul \xi)$ on $\bm
A^\ast$ we have $\bm H\left(\bm
x(t),\bm\xi(t),v\right)=\xi_i(t)f^i(x(t),v)+\ul{\xi}L(x(t),v)$.
Note that, by assumption, functions $f^i(x,v)$ and $L(x,v)$ are $C^1$ w.r.t. $x$ and their $x$-derivatives are continuous functions of both variables. Since $x(t)$ is an AC path with bounded derivative, functions
$\frac{\pa f^i}{\pa x^a}(x(t),v)$ and $\frac{\pa L}{\pa
x^a}(x(t),v)$, as well as functions $f^i(x(t),v)$ and $L(x(t),v)$,
are bounded in $[t_0,\tau]\times P$. As the evolution of $\xi(t)$
is governed by \eqref{eqn:par_trans_*A} and $u(t)$ is
bounded on $[t_0,\tau]$, the derivatives $\pa_t\xi^k(t)$ are also
bounded on $[t_0,\tau]$. Consequently, since the path $(\bm
x(t),\bm\xi(t))\in\bm A^\ast$ can be covered by a finite number of
coordinate charts, the $t$-derivative of $\bm H\left(\bm
x(t),\bm\xi(t),v\right)$ is bounded on $[t_0,\tau]\times P$. As a
result there exists a number $C$ such that
$$\left|\bm{H}(\bm{x}(t),\bm\xi(t),v)-\bm{H}(\bm{x}(t^{'}),\bm\xi(t^{'}),v)\right|\leq C|t-t^{'}|$$
for all $t,t^{'}\in [t_0^{'},\tau]$ and for any $v\in P$. Observe
also that
\begin{align*}
&\frac\pa{\pa t}\bm H\left(\bm
x(t),\bm\xi(t),v\right)|_{v=u(t)}=\pa_t\xi_i(t)f^i(x(t),u(t))+\\
&\phantom{==}+
\xi_i(t)\pa_t f^i(x(t),v)|_{v=u(t)}+\ul{\xi}\pa_tL(x(t),v)|_{v=u(t)}\overset{\eqref{eqn:par_trans_*A}}{=}\\
&=
\left[-\rho^a_i\left(x\right)\left(\frac{\pa f^k}{\pa
x^a}\left(x,u(t)\right)\xi_k(t)+\frac{\pa L}{\pa
x^a}\left(x,u(t)\right)\ul\xi(t)\right)+c^k_{ji}\left(x\right)f^j\left(x,u(t)\right)\xi_k(t) \right]f^i(x,u(t))\\
&\phantom{==} +\xi_i(t)\frac{\pa f^i}{\pa x^a}(x(t),u(t))\rho^a_k(x)f^k(x(t),u(t))+\ul{\xi}\frac{\pa L}{\pa x^a}(x(t),u(t))\rho^a_k(x)f^k(x(t),u(t))\\
&\ =\xi_k(t)c^k_{ij}(x)f^i(x(t),u(t))f^j(x(t),u(t))=0
\end{align*}
by the skew-symmetry of $c^k_{ij}(x)$.

Now take any regular points $t,t^{'}\in[t_0,\tau]$. Since
$u(t),u(t^{'})\in P$, we have
\begin{align*}
-C|t-t^{'}|&\leq\bm{H}(\bm{x}(t),\bm\xi(t),u(t^{'}))-\bm{H}(\bm{x}(t^{'}),\bm\xi(t^{'}),u(t^{'}))\\&\leq\m(\bm{x}(t),\bm\xi(t))-
\m(\bm{x}(t^{'}),\bm\xi(t^{'}))\\
&=\bm{H}(\bm{x}(t),\bm\xi(t),u(t))-\bm{H}(\bm{x}(t^{'}),\bm\xi(t^{'}),u(t^{'}))\\
&\leq\bm{H}(\bm{x}(t),\bm\xi(t),u(t))-\bm{H}(\bm{x}(t^{'}),\bm\xi(t^{'}),u(t))\leq C|t-t^{'}|;
\end{align*}
i.e., $\m(\bm x(t),\xi(t))$ satisfies the Lipschitz condition on
the set of regular points (dense in $[t_0,\tau]$). It is also a
continuous map (since $x(t)$ and $\xi(t)$ are AC and the maps $L(x,v)$,
$f(x,v)$ are continuous in both variables), therefore it is
Lipschitz on the whole interval $[t_0,\tau]$. By Rademacher{'}s
theorem, $\m(\bm x(t),\xi(t))$ is almost everywhere differentiable on $[t_0,\tau]$. Now take any point $t$ of
differentiability of $\m(\bm{x}(t),\bm\xi(t))$ which is also a
point of the regularity of the control $u$. We have
$$\m(\bm{x}(t^{'}),\bm\xi(t^{'}))-\m(\bm{x}(t),\bm\xi(t))\geq
\bm{H}(\bm
x(t{'}),\bm{\xi}(t^{'}),u(t))-\B{H}(\B{x}(t),\bm\xi(t),u(t)).$$
For $t^{'}>t$, we get
$$\frac{\m(\B{x}(t^{'}),\bm\xi(t^{'}))-\m(\B{x}(t),\bm\xi(t))}{t^{'}-t}\geq\frac{\B{H}(\B{x}(t^{'}),\bm\xi(t^{'}),u(t))-
\B{H}(\B{x}(t),\bm\xi(t),u(t))} {t^{'}-t}.$$ Consequently,
$$\frac{d}{dt}\m(\B{x}(t),\bm\xi(t))\geq\frac{\partial}{\partial
t^{'}}\Big|_{t^{'}=t}\B{H}(\B{x}(t^{'}),\bm\xi(t^{'}),u(t))=0.$$
Similarly, for $t^{'}<t$, we would get
$\frac{d}{dt}\m(\B{x}(t),\bm\xi(t))\leq 0$. We deduce that
$\frac{d}{dt}\m(\B{x}(t),\bm\xi(t))=0$ a.e. in
$[t_0,\tau]$; hence $\m(\B{x}(t),\bm\xi(t))$ is constant and equals $\bm H(\bm x(\tau),\bm \xi(\tau),u(\tau))=0$.
\end{proof}

\subsection{Extending the ''maximum principle'' to \texorpdfstring{$[t_0,t_1]$}{TEXT}}
\begin{lemma}\label{lem:K_increase} Let $t\in[t_0,\tau]$ be any regular point of
the control $u$. Then 
$$\B{K}^u_t\subset
\cl\left(\B{K}^u_\tau \right).$$
\end{lemma}

\begin{proof}
Consider an element $\bm d\in\B{K}_t^u$ of the
form
$$\B{d}=\bm B_{t_1t}[\B{f}(\B{x}(t),u(t))]\del
t+\sum_{i=1}^s\B{B}_{t_1\tau_i}\Big[\B{f}(\B{x}(\tau_i),v_i)-\B{f}(\B{x}(\tau_i),u(\tau_i))\Big]\del
t_i.$$ 
Since $\cl\left(\bm K_\tau^u\right)$ is a convex cone it is enough to show that $\bm B_{t_1t}\left(\bm{f}(\B{x}(t),u(t))\del
t\right)$ and $\sum_{i=1}^s\B{B}_{t_1\tau_i}\left[\B{f}(\B{x}(\tau_i),v_i)-\B{f}(\B{x}(\tau_i),u(\tau_i))\right]\del
t_i$ belong to $\cl\left(\B{K}^u_\tau\right)$. The later clearly belongs to $\bm K_\tau^u\subset\cl (\bm K_\tau^u)$ since $\tau_i\leq t<\tau$.

Assume now that  $\B{B}_{t_1
t}\left[\B{f}(\B{x}(t),u(t))\del t\right]$ does not belong to $\cl\left(\bm{K}^u_\tau\right)$. Since this set is a closed convex cone, by Theorem \ref{thm:separation}, there exists a covector $\ol{\bm{\xi}}(t_1)\in \bm A^\ast_{\bm x(t_1)}$ strictly separating  $\cl\left(\bm{K}^u_\tau\right)$ from $\left\{\B{B}_{t_1 t}\left(\B{f}(\B{x}(t),u(t))\delta
t\right)\right\}$; i.e., 
$$\<\B{d},\ol{\B{\xi}}(t_1)>_{\bm{\tau}}\leq 0<\<\B{B}_{t_1 t}\left(\B{f}(\B{x}(t),u(t))\del
t\right),\ol{\B{\xi}}(t_1)>_{\bm{\tau}}\quad \text{ for any $\B{d}\in
\cl(\bm{K}^u_\tau$).}$$ 
Define $\ol{\B{\xi}}(t):=\B{B}^\ast_{t t_1}(\ol{\B{\xi}}(t_1))$. Lemmas
\ref{lem:HM} and \ref{lem:M=0} hold for $\ol{\B{\xi}}(t)$ (we needed only $\<\bm d,\ol{\bm\xi}(t_1)>\leq 0$ for $\bm d\in \bm K_\tau^u\subset\cl(\bm K_\tau^u)$ in the proofs), hence,
in particular,
$$\<\B{f}(\B{x}(t),u(t)),\ol{\B{\xi}}(t)>_{\bm{\tau}}=\B{H}(\bm
x(t),\ol{\B{\xi}}(t),u(t))=0,$$ as $t$ is a regular point of $u$.
On the other hand,
$$0<\<\B{B}_{t_1 t}\big(\B{f}(\B{x}(t),u(t))\del
t\big),\ol{\B{\xi}}(t_1)>_{\bm{\tau}}=\<\B{f}(\B{x}(t),u(t)),\ol{\B{\xi}}(t)>_{\bm{\tau}}\del
t=\bm H(\bm x(t),\ol{\bm\xi}(t),u(t))\del t,$$
and hence $\B{H}(\bm
x(t),\ol{\B{\xi}}(t),u(t))\neq 0$. This contradiction finishes the proof.
\end{proof}

Note that, so far, we could define the set $\B{K}^u_\tau$ only for
a regular point $\tau<t_1$. With the help of the above lemma we can
also define $\B{K}^u_{t_1}$ as the direct limit of the increasing
family of sets $\cl\left(\B{K}^u_\tau\right)$,
$$\B{K}^u_{t_1}:=\bigcup_{\tau<t_1, \text{ $\tau$ regular}}\cl\left(\B{K}^u_\tau\right).$$
It is clear that
$\B{K}^u_{t_1}$  is a convex cone in
$\B{A}_{\B{x}(t_1)}$.
It has geometric properties analogous to the
properties of $\B{K}^u_\tau$ described in Theorem \ref{thm:separation_K_Lambda}.

\begin{lemma}\label{lem:K_at_t1}
The ray $\bm{\Lambda}_{\B{x}(t_1)}$ and the convex cone
$\B{K}^u_{t_1}$ are separable.
\end{lemma}
\begin{proof}
Assume the contrary. By Lemma \ref{lem:separation} (we use it for $V=\bm A_{\bm x(t_1)}=W\oplus\R=E_{x(t_1)}\oplus\T_{\ul x(t_1)}\R$, $K_1=\bm K^u_{t_1}$, $S=\{0\}\subset W$, and $K_2=\bm \Lambda_{\bm x(t_1)}$), there exists a vector $k\in \bm K^u_{t_1}\cap\bm\Lambda_{\bm x(t_1)}$ and vectors $e_1,\hdots,e_m\in E_{x(t_1)}$ such that \eqref{cond:sep_A} and \eqref{cond:sep_B} hold. 
Since $\bm K^u_{t_1}$ is a limit of an increasing family of sets, there exists a regular $t<t_1$ such that \eqref{cond:sep_A} and \eqref{cond:sep_B} hold for $K_1=\cl\left(\bm K_t^u\right)$. In other words, $\bm\Lambda_{\bm x(t_1)}$ and $\cl\left(\bm K_t^u\right)$ are not separable. By Lemma \ref{lem:separation_closure} also $\bm\Lambda_{\bm x(t_1)}$ and $\bm K_t^u$ are not separable. This contradicts
Theorem \ref{thm:separation_K_Lambda}.
\end{proof}

Now choose a non-zero covector $\B{\xi}(t_1)\in\B{A}^\ast_{\B{x}(t_1)}$
separating $\bm K_{t_1}^u$ and $\bm\Lambda_{\bm x(t_1)}$ and define $\bm\xi(t):=\bm B^\ast_{tt_1}(\bm \xi(t_1))$.  We have 
$$\<\B{d},\B{\xi}(t_1)>_{\bm{\tau}}\leq 0\leq\<\bm\lambda,\B{\xi}(t_1)>_{\bm{\tau}}\quad\text{for every $\B{d}\in \B{K}^u_{t_1}$}.$$ 
Since, by construction, $\bm K^u_\tau\subset \bm K_{t_1}^u$, the covector $\bm\xi(t_1)$ separates also 
$\bm K^u_\tau$ and $\bm\Lambda_{\bm x(t_1)}$. This is enough for Lemmas \ref{lem:HM} and \ref{lem:M=0} to hold for $\bm \xi(t)$.
As a consequence, for every regular point of the control $u$, we have
$$
\bm{H}(\bm x(t),\bm \xi(t),u(t))=\sup_{v\in U}\bm{H}(\bm
x(t),\bm\xi(t),v)=0\,.
$$
  
Finally, since $\<\bm \lambda, \bm \xi(t_1)>\geq 0$, we have 
 $\ul\xi(t)\equiv\ul\xi(t_1)\leq 0$. This finishes the proof
of Theorem \ref{thm:pmp_A}. \hfill$\qed$

\section [The proof of Theorem 5.4]{The proof of Theorem \ref{thm:pmp_A_rel}} \label{sec:proof_pmp_rel}
In this section we assume that the controlled pair $(\bm x(t), u(t))$, with $t\in [t_0,t_1]$, solves the OCP \eqref{eqn:P_A_rel}. Recall that we assume that $\bm x(t_0)=(x(t_0),\ul x(t_0))=(\phi_0(z_0),0)$ and $\bm x(t_1)=(x(t_1),\ul x(t_1))=(\phi_1(w_0),\ul x(t_1))$, where $\phi_i$ are base projections of algebroid morphisms $\Phi_i:\T S_i\ra E$ for $i=0,1$. We use the following notation $\eS_0:=\Phi_0(\T_{z_0}S_0)$, $\eS_1:=\Phi_1(\T_{w_0}S_1)$, $\Lambda_{\ul x}:=\R_+(-\pa_t)\subset\T_{\ul x}\R$, and $\K_\tau^u:=\conv\left\{\bm B_{t_1t_0}(\eS_0\oplus\theta_{\ul{x}(t_0)}),\bm K_\tau^u\right\}$.

\begin{theorem}\label{thm:separation_K_S} Let $(\B{x}(t), u(t))$, for $t\in[t_0,t_1]$, be a solution of the optimal control
problem \eqref{eqn:P_A_rel}. Then the  convex cones $\K_\tau^u$ and $\eS_1\oplus\Lambda_{\ul x(t_1)}$ can be separated for any $\tau\in(t_0,t_1)$ which is a regular point
of $u$.
\end{theorem}
\begin{proof}
The argument is very similar to that from Theorem \ref{thm:separation_K_Lambda}. 

Assume that $(\bm x(t),u(t))$ is a solution of \eqref{eqn:P_A_rel} but the cones  $\K_\tau^u$ and $\eS_1\oplus\Lambda_{\ul x(t_1)}$ are not separable.
First, construct a family of admissible paths $s\mapsto b_0(s,p)\in E$ parametrised by $p\in\eS_0$ such that the following conditions hold:
\begin{itemize}
\item the family is uniformly regular w.r.t $p$ at $s=0$,
\item $b_0(0,p)=p$,
\item $b_0(s,p)$ lies in the image of $\Phi_0$.
\end{itemize} 

A family with the desired properties can be build as follows. First, choose a linear subspace $V\subset \T_{z_0}S_0$ such that $\Phi_0|_V:V\ra\eS_0=\Phi_0(\T_{z_0}S_0)$ is a linear isomorphism ($\Phi_0$ at $z_0$ is a linear map from $\T_{z_0}S_0$ to $\eS_0\subset E_{x(t_0)}$). For $p\in\eS_0$ take $v(p)=\left(\Phi_0|_V\right)^{-1}(p)\in V$ and consider a curve $s\mapsto z(s,p)=\exp_{z_0}(s\cdot v(p))$ in $S_0$, where $\exp_{z_0}(\cdot)$ is an exponential map
around $z_0$ defined for some metric on $S_0$. We define $b_0(s,p):=\Phi_0(\pa_sz(s,p))$. 

Let us check that $b_0(s,p)$ has the desired properties. Clearly, it lies in the image of $\Phi_0$. The properties of the exponential map imply that $s\mapsto \wt z(s,v)=\pa_s\left(\exp_{z_0}(s\cdot v)\right)$ is a family of paths uniformly regular at $s=0$ w.r.t $v$. These paths are admissible in the tangent algebroid $\T S_0$. Now, since $b_0(s,p)$ is obtained as a composition of $\wt z(s,v)$ with a continuous map $p\mapsto v(p)$ and a smooth map $\Phi_0$, the uniform regularity is preserved. Since $\Phi_0$ is an algebroid morphism, also admissibility is preserved. Finally, observe that $b_0(0,p)=\Phi_0\left[\pa_s|_{0}\exp_{z_0}(s\cdot v(p))\right]=\Phi_0(v(p))=\Phi_0\left(\left(\Phi_0|_V\right)^{-1}\right)(p)=p$.

Similarly, we construct a family of admissible paths $s\mapsto b_1(s,q)\in E$ parametrised by $q\in\eS_1$ and such that the following holds:
\begin{itemize}
\item the family is uniformly regular w.r.t $q$ at $s=0$,
\item $b_1(0,q)=q$,
\item $b_1(s,q)$ lies in the image of $\Phi_1$.
\end{itemize} 

If the cones $\K_\tau^u$ and $\eS_1\oplus\Lambda_{\ul x(t_1)}$ are not separable, by Lemma \ref{lem:separation}, there exists a vector $\bm k\in\K_\tau^u\cap\left( \eS_1\oplus\Lambda_{\ul x(t_1)}\right)$ and vectors $ e_1,\hdots  e_{m_1}\in E_{x(t_1)}\oplus\theta_{\ul x(t_1)}$ such that 
\begin{itemize}
\item $\spann\{e_1,\hdots,e_{m_1},\eS_1\}=E_{x(t_1)}$,
\item $\bm k\pm e_i\in\K_\tau^u$ for $i=1,\hdots,m_1$.
\end{itemize}
In the last formula we understand $\bm k+e_i$ as $\bm k+(e_i+\theta_{\ul x(t_1)})$, where $e_i+\theta_{\ul x(t_1)}\in E_{x(t_1)}\oplus\T_{\ul x(t_1)}=\bm A_{\bm x(t_1)}$. In view the first property, we can decompose $E_{x(t_1)}\approx\R^m=\R^{m_1}\oplus\R^{m_2}=\spann\{e_1,\hdots,e_{m_1}\}\oplus\eS_1$ and choose a basis $(\wt e_1,\hdots,\wt e_{m_2})$ of $\eS_1$. 

Without the loss of generality (cones are invariant under rescaling) we may assume that $\bm k$ projects to $-\pa_t$ under $\operatorname{p}_{\T\R}:\bm A=E\times\T\R\ra\T\R$. Observe that the $E$-projection $\vec k_0:=\operatorname{p}_E(\bm k)$ belongs to $\eS_1$.

Introduce $\bm b_0(s,p):=(b_0(s,p),0)\in E\times\T\R=\bm A$.
We can consider variations $\bm f(\bm x(t,s,p),u^\A_s(t))$ associated with symbols $\A=(\tau_i, v_i,\tau,\delta t_i, \delta t)_{i=1,\hdots,k}$ and initial base-point variations $\bm x_0(s,p)=\bm \tau(\bm b_0(s,p))$ as in Theorem \ref{thm:1st_main}. From \eqref{eqn:htp_1st_lem}, there exists a family of $\bm A$-paths $s\mapsto\bm d^{\A,p}(s)$ defined for $0\leq s\leq\theta$ such that
\begin{equation}\label{eqn:var_htp_1}\begin{split}
&[\bm b_0(s,p)]_{s\in[0,\eps]}[\bm{f}(\bm x(t,\eps,p)),u^\A_\eps(t)]_{t\in[t_0,t_1+\eps\del t]}\\
&=[\bm{f}(\bm x(t),u(t)))_{t\in[t_0,t_1]}[\bm{d}^{\A,p}(s)]_{s\in[0,\eps]},
\end{split}\end{equation} 
Observe that, due to \eqref{eqn:d_at_0}, $\bm d^{\A,p}(0)\in\K^u_\tau$ and, moreover, all elements of $\K^u_\tau$ can be obtained in this way.

Choose symbols $\A$, $\A_i$, and $\BB_i$ and elements $p,p_i,\wt p_i\in\eS_0$ such that  $\bm d^{\A,p}(0)=\bm k$, $\bm d^{\A_i,p_i}(0)=\bm k+\bm e_i$, and $\bm d^{\BB_i,\wt p_i}(0)=\bm k-\bm e_i$. Since we deal with a finite set of symbols, we can assume that they all are of the form $(\tau_i, v_i,\tau, \delta t_i, \delta t)_{i=1,\hdots,k}$, where $\tau_i$, $v_i$, $\tau$, and $k$ are
fixed, and that they differ by $\del t_i$ and $\del t$ (see Remark \ref{rem:var_dt=0}). Define for any
$\vec{r}=\sum_{i=1}^{m_1}r^i e _i\in B^{m_1}(0,1)\oplus\theta_{m_2}=\{\vec r:\|\vec r\|=\sum_i|r_i|\leq 1\}\subset E_{x(t_1)}$ a new symbol
\begin{align*} \A(\vec{r})&=\left(1-\sum_{i=1}^m|r^i|\right)\A+\sum_{i=1}^m
h^+(r^i)\A_i+\sum_{i=1}^m h^-(r^i)\BB_i,
\intertext{and a new element of $\eS_1$}
p(\vec{r})&=\left(1-\sum_{i=1}^{m_1}|r^i|\right)p+\sum_{i=1}^{m_1}
h^+(r^i)p_i+\sum_{i=1}^{m_1} h^-(r^i)\wt p_i,
\end{align*}
 where
$h^+(r)=\max\{r,0\}$ and $h^-(r)=\max\{-r,0\}$ are non-negative. 
From \eqref{eqn:d_at_0} we get
$$\bm{d}^{\A(\vec{r}),p(\vec r)}(0)=\bm k+\vec{r}.$$
Now form Theorem \ref{thm:1st_main} there exists $\theta>0$ such that $\bm{d}^{\A(\vec{r}),p(\vec r)}(s)$ is well-defined for $0\leq s\leq\theta$ and all $\vec r\in B^{m_1}(0,1)\oplus\theta_{m_2}$ and, moreover, it is uniformly regular w.r.t. $\vec r$ at $s=0$. 

Repeating the reasoning from the proof of Theorem \ref{thm:separation_K_Lambda} we can show that the $\T\R$-projections  $\ul d^{\A(\vec r),p(\vec r)}(s)$ satisfy 
\begin{equation}\label{eqn:cost_less_zero1}
\int_0^\eps\ul d^{\A(\vec r)}(s)\dd s<0
\end{equation}  for all
$\eps\leq\eta$ and all $\vec r\in B^{m_1}(0,1)$, where $0<\eta\leq\theta$ is a fixed number. 

Projecting $\B d^{\A(\vec{r}),p(\vec r)}(s)$ to $E$
we obtain a family of bounded measurable admissible paths
$d^{\A(\vec{r}),p(\vec r)}(s)$ uniformly
regular at $s=0$ w.r.t. $\vec{r}\in {B}^{m_1}(0,1)\oplus\theta_{m_2}\subset \R^{m}$, and such that
$d^{\A(\vec{r}),p(\vec r)}(0)=\vec{r}+\vec k_0\in\R^m\approx E_{x(\tau)}$. 

For 
$$\vec s\in\theta_{m_1}\oplus B^{m_2}(0,1)=\{\vec s=\sum_{i=1}^{m_2}\wt e_i s^i\in\eS_1:\|\vec s\|=\sum_i|s^i|\leq 1\}$$ 
we have a family of $E$-paths $s\mapsto b_1(s,\vec s+\vec k_0)$. This family is uniformly regular w.r.t. $\vec s$ at $s=0$ and, moreover, $b_1(0,\vec s+\vec k_0)=\vec s+\vec k_0$. 

The families $a_{\vec r}(s):=d^{\A(\vec r), p(\vec r)}(s)$ and $b_{\vec s}(s):=b_1(s,\vec s+\vec k_0)$ satisfy the assumptions of Lemma \ref{lem:htp_of_a_r1}, and hence there exist vectors $\vec r_0$, $\vec s_0$ and a number $0<\eps\leq\eta$ such that 
$$[d^{\A(\vec r_0),p(\vec r_0)}(s)]_{s\in[0,\eps]}=[b_1(s,\vec s_0+\vec k_0)]_{s\in[0,\eps]}.$$
Projecting equality \eqref{eqn:var_htp_1} (for $\A=\A(\vec r_0)$ and $p=p(\vec r_0)$) to $E$ and using the above equality, we get
\begin{align*}
[b_0(s,p(\vec r_0))]_{s\in[0,\eps]}\left[f( x(t,\eps,p(\vec r_0)),u^{\A(\vec r_0)}_\eps(t))\right]_{t\in[t_0,t_1+\eps\del t(\vec r_0)]}=\\
\left[{f}(x(t),u(t))\right]_{t\in[t_0,t_1]}[b_1(s,\vec s_0+\vec k_0)]_{s\in[0,\eps]}.
\end{align*}
Since $b_0(s,p(\vec r_0))$ lies in $\Image\Phi_0$ and $b_1(s,\vec s_0+\vec k_0)$ in $\Image\Phi_1$, the trajectories  $f(x(t),u(t))$, with $t\in[t_0,t_1]$, and $f( x(t,\eps,p(\vec r_0)),u^{\A(\vec r_0)}_\eps(t))$, with $t\in[t_0,t_1+\eps\del t(\vec r_0)]$, are $E$-homotopic relative to $(\Phi_0,\Phi_1)$. On the other hand, from \eqref{eqn:cost_less_zero1} we deduce (in the same way as in the proof of Theorem \ref{thm:separation_K_Lambda}) that the cost on the first of these trajectories is smaller. This contradicts the optimality of $(\bm x(t),u(t))$.
\end{proof}

\noindent Now, as a simple consequence of Lemma \ref{lem:K_increase}, we obtain the following result.
\begin{lemma}\label{lem:KK_incerase}
 Let $t\in[t_0,\tau]$ be any regular point of
the control $u$. Then 
$$\K^u_t\subset
\cl\left(\K^u_\tau\right).$$
\end{lemma}
\begin{proof} 
By Lemma \ref{lem:K_increase}, $\bm K^u_t\subset\cl(\bm K^u_\tau)$, and hence
\begin{align*}
\K^u_t= &\conv\{\bm B_{t_1t_0}(\eS_0\oplus\theta_{\ul x(t_0)}),\bm K^u_t\}\subset  \conv\{ \cl(\bm B_{t_1t_0}(\eS_0\oplus\theta_{\ul x(t_0)})),\cl( \bm K^u_\tau)\}\\
=&\cl\left(\conv\{\bm B_{t_1t_0}(\eS_0\oplus\theta_{\ul x(t_0)}),\bm K^u_\tau\}\right)=\cl\left(\K^u_\tau\right).
\end{align*}
\end{proof}
\noindent The above result allows us to define a convex cone
$$\K^u_{t_1}:=\bigcup_{\tau<t_1, \text{ $\tau$ regular}}\cl\left(\K^u_\tau\right)$$
in $\bm A_{\bm x(t_1)}$. Similarly as in Section \ref{sec:proof_pmp} we have the following result.
\begin{lemma}\label{lem:KK_at_t1}
The the convex cones
$\K^u_{t_1}$ and $\eS_1\oplus\Lambda_{\ul x(t_1)}$ are separable.
\end{lemma}
\begin{proof}
If $\K^u_{t_1}$ and $\eS_1\oplus\Lambda_{\ul x(t_1)}$ were not separable, then, by repeating the argument from the proof of Lemma \ref{lem:K_at_t1}, the cones $\cl(\K^u_t)$ and  $\eS_1\oplus\Lambda_{\ul x(t_1)}$ would not be separable for some regular $t<t_1$. Consequently, also $\K^u_t$ and  $\eS_1\oplus\Lambda_{\ul x(t_1)}$ would not be separable (Lemma \ref{lem:separation_closure}), which contradicts Theorem \ref{thm:separation_K_S}. 
\end{proof}
Now we can finish the proof of Theorem \ref{thm:pmp_A_rel}. From the previous lemma we can deduce that there exists a covector $\bm \xi(t_1)\in\bm A^\ast_{\bm x(t_1)}$ separating $\K^u_{t_1}$ and $\eS_1\oplus\Lambda_{\ul x(t_1)}$; i.e.,
$$\<\bm d,\bm \xi(t_1)>\leq 0\leq\<\bm s,\bm \xi(t_1)>\quad\text{for any $\bm d\in\K^u_{t_1}$ and $\bm s\in\eS_1\oplus\Lambda_{\ul x(t_1)}$}.$$

We define $\bm\xi(t):=\bm B^\ast_{t t_1}\bm\xi(t_1)$. Now, since $\bm K_{t_1}^u\subset\K_{t_1}^u$ and $\bm\Lambda_{\bm x(t_1)}=\theta_{x(t_1)}\oplus\Lambda_{\ul x(t_1)}\subset\eS_1\oplus\Lambda_{\ul x(t_1)}$, the covector $\bm\xi(t_1)$ separates also $\bm K^u_{t_1}$ and $\bm\Lambda_{\bm x(t_1)}$. Consequently, $\bm\xi(t)$ satisfies the assertion of Theorem \ref{thm:pmp_A}. On the other hand, the separating covector $\bm\xi(t_1)$ must vanish on the linear subspaces $\bm B_{t_1t}^u(\eS_0\oplus\theta_{\ul x(t_0)})\subset\K^u_{t_1}$ and $\eS_1\oplus\theta_{\ul x(t_1)}\subset \eS_1\oplus\Lambda_{\ul x(t_1)}$,
which gives additional conditions from Theorem \ref{thm:pmp_A_rel}. \hfill$\qed$

\appendix
\renewcommand{\chaptername}{Appendix}

\chapter{Differnetial geometry}\label{app:dif_geom}

\section{Lie groupoids}\label{sapp:groupoids}

In this section we give the definition of a Lie groupoid, study some fundamental examples and recall the construction of a Lie algebroid of a Lie groupoid. Later we address some questions about integrability of Lie algebroids. The discussion is based mostly on \cite{mackenzie}.
\subsection{Lie groupoids}

\begin{definition} A \emph{Lie groupoid}\index{Lie groupoid} consists of two manifolds: $\GG$ (a \emph{groupoid}) and $M$ (a \emph{base}), together with two surjective submersions $\alpha,\beta:\GG\lra M$ called the \emph{source}\index{source map} and the \emph{target maps}\index{target map}, a smooth map $\iota:x\mapsto\iota_x$, $M\lra\GG$, called the \emph{inclusion map}\index{inclusion map}, and a smooth partial multiplication $(h,g)\mapsto hg$, $\GG\ast\GG=\{(h,g)\in\GG\times\GG: \alpha(h)=\beta(g)\}\lra\GG$, subject to the following conditions:
\begin{itemize}
\item $\alpha(hg)=\alpha(g)$ and $\beta(hg)=\beta(h)$ for all $(h,g)\in \GG\ast\GG$;
\item the partial multiplication is associative; i.e., $j(hg)=(jh)g$ for all $j$, $h$, $g$ such that $(j,h),(h,g)\in\GG\ast\GG$;
\item $\alpha(\iota_x)=\beta(\iota_x)=x$ for all $x\in M$;
\item $\iota$ is a two-sided identity; i.e., $g\iota_{\alpha(g)}=\iota_{\beta(g)}g=g$ for all $g\in\GG$;
\item each $g\in\GG$ has a two-sided inverse $g^{-1}\in\GG$ such that $\alpha(g^{-1})=\beta(g)$, $\beta(g^{-1})=\alpha(g)$, $g^{-1}g=\iota_{\alpha(g)}$, and $gg^{-1}=\iota_{\beta(g)}$.
\end{itemize}
A \emph{morphism of Lie groupoids}\index{morphism of Lie groupoids} $\GG$ and $\wt\GG$ is a pair of smooth maps $F:\GG\lra\wt\GG$ and $f:M\lra\wt M$ which preserves the source and target maps; i.e., $ f\circ \alpha=\wt\alpha\circ F$ and $ f\circ\beta= \wt\beta\circ F$, and preserves the multiplication; that is, $F(h)F(g)=F(hg)$ for all $(h,g)\in\GG\ast\GG$.  

\end{definition}

From the point of view of category theory, a groupoid may be regarded as a small category with the set of objects equal to $M$ and the set of arrows equal to $\GG$, and such that all the arrows are invertible. The word ''Lie'' refers to the fact that the groupoid structure is compatible with the smooth structures on $M$ and $\GG$ (similarly as a Lie group is a group with a smooth structure compatible with the multiplication). It is an easy exercise to show that the assumption that $\alpha$ and $\beta$ are surjective submersions imply that $\GG\ast\GG=(\alpha\times\beta)^{-1}(\Delta_M)$ has a smooth structure of a submanifold  of $\GG\times\GG$. Moreover, it follows from the smoothness of the partial multiplication and the properties of $\alpha$ and $\beta$ that the inverse mapping $g\mapsto g^{-1}$ is a diffeomorphism. Details can be found in \cite{mackenzie}.

Lie groupoids appear naturally in many situations. A basic example is a Lie group $G$, with
$M=\{\textrm{pt}\}$ being a single point, trivial $\alpha$ and $\beta$, and  group multiplication. Another standard
example is a \emph{pair groupoid}\index{pair groupoid} $\GG=M\times M$ of a manifold $M$, with the source $\alpha(x,y)=y$, the target $\beta(x,y)=x$, and the multiplication $(x,y)(y,z)=(x,z)$.

For a (right) principal $G$-bundle $G\lra P\overset\pi\lra M$ we can construct an  important example of a \emph{gauge groupoid
$\GG_P=P\times P/G$}\index{gauge groupoid} over $M$.  In $\GG_P$ we identify pairs $(p,q)$ and $(pg,q g)$ for all $p,q\in P$ and $g\in G$. The source and target maps are simply
$\alpha[(p,q)]=\pi(q)$ and $\beta[(p,q)]=\pi(p)$, and the
multiplication reads as $[(p,q)][(q,r)]=[(p,r)]$. For the two extreme cases: ($P=M$, $G=\{1\}$) and ($P=G$, $M=\{\textrm{pt}\}$), $\GG_P$ is the pair groupoid $M\times M$ and the Lie group $G$, respectively. Note that $\GG_P$ can be regarded as a pair groupoid $P\times P$ divided by the action of $G$ (all the groupoid data for $P\times P$ is $G$-equivariant).

For a Lie groupoid $\GG$ and $x\in M$ we may define an \emph{$\alpha$-fibre of $\GG$ over $x$}\index{$\alpha$-fibre}
$$\GG_x:=\{g\in\GG:\alpha(g)=x\}\subset\GG.$$
Note that $\GG_x$ is a closed embedded submanifold of $\GG$. The groupoid $\GG$ is called \emph{$\alpha$-simply connected}\index{Lie groupoid!$\alpha$-simply connected} if each of it $\alpha$-fibres is simply connected. 

Take now an element $g\in\GG$. A \emph{right translation}\index{right translation} corresponding to $g$ is $R_{g}:\GG_{\beta(g)}\lra\GG_{\alpha(g)}$ defined simply as $h\mapsto hg$.

\subsection{A Lie algebroid of a Lie groupoid}

In this part we describe the construction of a Lie algebroid of a Lie groupoid basing mostly on \cite{mackenzie,weinstein_silva}. The procedure follows closely
the standard construction of a Lie group--Lie algebra reduction.

Consider a Lie groupoid $\GG$ over $M$ and the right action of $\GG$ on itself.  Since the right translation $R_g:\GG_{\beta(g)}\ra\GG_{\alpha(g)}$ is a diffeomorphism of $\alpha$-fibres (not the whole $\GG$) there is a sense speaking of right-invariant vector fields on $\GG$ only for fields tangent to $\alpha$-fibres. 

Denote by $\T^\alpha \GG:=\ker\T\alpha\subset\T\GG$ the distribution tangent to  the foliation of $\GG$ by $\alpha$-fibres $\GG^\alpha=\{\alpha^{-1}(x):x\in M\}$. A vector field $X\in\Sec(\T^\alpha\GG)$ is said to be \emph{right-invariant}\index{right-invariant vector field} if $(\T R_g)(X(h))=X(hg)$ for all $(h,g)\in\GG\ast\GG$. The set of right-invariant vector fields will be denoted by  $\XX_R(\GG)$. Observe that a right-invariant vector field is uniquely determined by its value along the identity section $\iota(M)$. Indeed, we have $X(g)= (\T R_g)X(\iota_{\beta (g)})$. Consequently, we can identify the space  $\XX_R(\GG)$ with the space of sections of the bundle $\AG:=\T^\alpha\GG|_{\iota(M)}\lra M$
$$\XX_R(\GG)\approx\Sec(\AG, M).$$

Since $\AG\lra M$ is a pullback of $\T^\alpha\GG\lra\GG$ via $\iota:M\lra G$,
$$\xymatrix{ \AG\ar[d]\ar[rr]  && 
\T^\alpha\GG\ar[d] \\  M\ar[rr]^{\iota} && \GG },$$
it has a smooth vector bundle structure induced from $\T^\alpha\GG$. What is more, since the Lie bracket of right-invariant vector fields on $\GG$ is again right-invariant, the Lie bracket on $\XX_R(\GG)$ induces a natural skew-symmetric bilinear bracket $[\cdot,\cdot]$ on sections of $\AG$. This bracket satisfies the Liebniz rule \eqref{eqn:lieb_rule} for $\rho=\T\beta|_{\AG}$. The bundle $\AG\ra M$, together with $[\cdot,\cdot]$ and $\rho$, is a Lie algebroid called \emph{a Lie algebroid of a Lie groupoid $\GG$}\index{Lie algebroid!of a Lie groupoid|main}. 
Lie algebroids which come from some Lie groupoid by the construction described above are called \emph{integrable}.\index{Lie algebroid!integrable}

Note that the maps $\T R_{g^{-1}}:\T^\alpha_g\GG\ra\T_{\iota_{\beta(g)}}^\alpha\GG=\AG_{\beta(g)}$ defined point-wise for all $g\in\GG$ give rise to a vector bundle map (\emph{reduction map})\index{reduction map}
\begin{equation}\label{eqn:reduction}
\xymatrix{ {\T ^\alpha\GG}\ar[d]_{\T\beta}\ar[rr]^{\mathcal{R}}  &&
\AG\ar[d]_{\rho} \\ \T M\ar[rr]^{\id_{\T M}} && \T M },
\end{equation}
which is a fibre-wise isomorphism.

The construction of $\AG$ can be repeated also for the left action of $\GG$ and left-invariant vector fields. It is a matter of convention which construction we use. 

Natural examples of Lie algebroids are, in fact, associated with Lie groupoids. A Lie algebroid associated with a Lie group $G$ is its Lie algebra $\g\approx\T_eG$. For the pair groupoid $M\times M$, $\alpha$-fibres are simply $\alpha^{-1}(x)=M\times\{x\}$, and hence $\T^\alpha_x\GG\approx\T_x M$. A Lie algebroid associated with this groupoid is the \emph{tangent algebroid}\index{tangent algebroid|main} on $\T M\lra M$. A Lie algebroid associated with the gauge groupoid $\GG_P$ is called an \emph{Atyiah algebroid}. It will be described in detail in Section \ref{sapp:atiyah}.

We have shown above that to every Lie groupoid $\GG$ corresponds a Lie algebroid $\AG$. In fact, also every Lie groupoid morphism $F:\GG\lra\wt\GG$ over $f:M\lra\wt M$ induces a natural morphism of Lie algebroids $\mathcal{A}(F):\AG\lra\mathcal{A}(\wt G)$ over $f:M\lra\wt M$ which can be described as follows. Since $F$ preserves the source map, the vector bundle morphism $\T F:\T\GG\lra\T\wt\GG$ restricts to 
$$
\xymatrix{ {\T ^\alpha\GG}\ar[d]\ar[rr]^{\T^\alpha F}  &&
\T^\alpha\wt\GG\ar[d] \\ \GG\ar[rr]^{F} && \wt\GG }.
$$
Now, since $\wt\iota\circ F=f\circ \iota$, the derivative $\T^\alpha F$ induces a map of pullbacks $\mathcal{A}(F):\AG=\iota^\ast\T^\alpha\GG\lra\mathcal{A}(\wt\GG)=\wt\iota^\ast\T^\alpha\wt\GG$ over $f:M\lra\wt M$. One can check that this is a morphism of Lie algebroids. Lie algebroid morphisms of the form $\mathcal{A}(F)$ are called \emph{integrable}.\index{integrable Lie algebroid morphism}  In fact the association of a Lie algebroid $\AG$ to a Lie groupoid $\GG$, and a Lie algebroid morphism $\mathcal{A}(F)$ to a morphism of Lie groupoids $F$ is a functor form the category of Lie groupoids to the category of Lie algebroids. In literature it is known as a \emph{Lie functor}\index{Lie functor}.

\subsection{Lie theory}\index{Lie theory for Lie algebroids}

Since a Lie groupoid--Lie algebroid reduction can be considered as a generalisation of the Lie group--Lie algebra reduction, it is natural to ask a question about possible extension of Lie integrability theorems to this new context. This topic is extensively treated in \cite{mackenzie} and solved by \cite{crainic_fernandes}. In this work we are interested in two problems:
\begin{itemize}
\item Integration of Lie algebroids---does every Lie algebroid is a Lie algebroid of some Lie groupoid?
\item Integration of Lie algebroids morphism---can we lift a morphism of two integrable Lie algebroids to a morphism of the corresponding groupoids?
\end{itemize}

\label{CF}The answer to the first problem is in general negative. A complete solution was given by \cite{crainic_fernandes} (see also \cite{almeria,almeida_kumpera,almeida_molino,mackenzie_1987,cattaneo_felder}). The idea goes back to covering theory. Recall that the universal cover $\wt X$ of a topological space $X$ can be constructed as a space of homotopy classes of paths emerging from a fixed point $x_0\in X$. If $X=G$ is a Lie group, each sufficiently regular path in $G$ can be reduced to a path in $\g$---its Lie algebra. Now it turns out that the homotopies of paths in $G$ can be reduced to homotopies in $\g$, which are expressed entirely in terms of the Lie algebra structure, without referring to the structure of the underling Lie group. Consequently, the universal cover $\wt G$ of $G$ can be defined as the space of paths in $\g$ divided by the equivalence relation coming from homotopy. The group structure on $\wt G$ is given by the composition of paths \cite{duistermaat_kolk}. The same construction can be repeated for Lie algebroids, yet we have to use admissible paths and algebroid homotopies (see Chapter \ref{ch:E_htp}). The quotient space, with the multiplication defined by the composition of admissible paths, has a structure of an $\alpha$-simply connected topological groupoid. Unlike to the case of a Lie algebra, there may be some obstructions to introduce a smooth structure on this groupoid. These are described in \cite{crainic_fernandes}. Theorem \ref{thm:int_htp} is closely related with the ideas sketched above. Note, however, that we work in a measurable category, whereas Crainic and Fernandes use the smooth data.

The second integrability problem has a positive solution under mild topological assumptions.
\begin{theorem}[\cite{mackenzie_xu}]\label{thm:Lie1}
Consider Lie groupoids $\mathcal{H}$ over $S$ and $\GG$ over $M$, and suppose that $\Phi:\mathcal{A(H)}\ra\AG$ over $f:S\lra M$ is an algebroid morphism. If $\mathcal{H}$ is $\alpha$-simply connected, there is an unique morphism of Lie groupoids $F:\mathcal{H}\ra\GG$ over $f:S\lra M$ such that $\mathcal{A}(F)=\Phi$.
\end{theorem}
 
We shall now study groupoid morphisms $\mathcal{H}\lra\GG$ in a special situation when $\mathcal{H}=S\times S$ is a pair groupoid. From Theorem \ref{thm:Lie1} we can easily derive the following result.

\begin{lemma}\label{lem:unique_mfd_lifts} Let $\Phi:\T S\lra \AG$ over $f:S\lra M$ be a morphism of Lie algebroids. Then, if $S$ is connected, there is at most one morphism of Lie groupoids $F:S\times S\lra\GG$ integrating $\Phi$. If $S$ is simply connected then such a morphism $F$ exists. 
\end{lemma}
\begin{proof} If $S$ is simply connected, the existence of $F$ follows immediately from Theorem \ref{thm:Lie1}, as the groupoid $S\times S$ is $\alpha$-simply connected ($\alpha$-fibres are of the form $S\times\{x_0\}$). 

Assume now that $S$ is an arbitrary connected manifold and let $F:S\times S\lra\GG$ be a morphism integrating $\Phi$. Consider the universal cover $\pi:(\wt S,\wt{x_0})\lra(S,x_0)$. Clearly, $\wt S\times\wt S\overset{\pi\times\pi}\lra S\times S$ is a groupoid morphism, hence also the composition $\wt S\times\wt S\overset{\pi\times\pi}\lra S\times S\overset{F}\lra\GG$ is a groupoid morphism. Moreover, it integrates the Lie algebroid morphism $\T\wt S\overset{\T \pi}\lra\T S\overset{\Phi}\lra \AG$. We will prove that uniqueness of $F\circ(\pi\times\pi)$ (following from Theorem \ref{thm:Lie1})  implies the uniqueness of $F$. To see this, observe that if $F\circ(\pi\times\pi)=F^{'}\circ(\pi\times\pi)$ then the set of points in $S\times S$ on which $F$ and $F^{'}$ coincide would be: closed, since $F$ and $F^{'}$ are continuous; nonempty, since $F(x_0,x_0)=F^{'}(x_0,x_0)=\iota_{f(x_0)}$; and open, since $\pi\times\pi$ is a covering. We deduce that it is the whole $S\times S$. 
\end{proof}

Observe now that, if $F:S\times S\lra\GG$, with $S$ connected, is a groupoid morphism integrating $\Phi:\T S\lra\AG$ over $f:S\lra M$, then, for each $x_0\in S$, $\wt\Phi:x\mapsto F(x,x_0)$ is a map from $S$ to the $\alpha$-fibre $\GG_{f(x_0)}$ of $\GG$ such that the following diagram of vector bundle morphisms commutes 
\begin{equation}\label{eqn:int_mfd}\xymatrix{
\T S\ar[r]^(.35){\T\wt\Phi} \ar[rrd]^{\Phi}&\T\left(\GG_{f(x_0)}\right)\ar@{^{(}->}[r]&\T^\alpha\GG\ar[d]^{\mathcal{R}} \\
&&\AG}.
\end{equation}
Note that, if $\wt \Phi:S\lra\GG_{f(x_0)}$ is a smooth map such that \eqref{eqn:int_mfd} is satisfied, then $F_{\wt\Phi}(x,y):=\wt\Phi(x)^{-1}\wt\Phi(y)$ is a groupoid morphism $F_{\wt\Phi}:S\times S\lra \GG$ integrating $\Phi$.

By Lemma \ref{lem:unique_mfd_lifts}, $F_{\wt\Phi}$ is unique. Consequently, if $\wt\Phi,\wt\Phi^{'}:S\lra\GG_{f(x_0)}$ are two maps satisfying \eqref{eqn:int_mfd}, they are related by $\wt\Phi(x)^{-1}\wt\Phi(y)=\wt\Phi^{'}(x)^{-1}\wt\Phi^{'}(y)$, and hence $\wt\Phi^{'}(y)=\wt\Phi^{'}(x)\wt\Phi(x)^{-1}\wt\Phi(y)$.

\begin{corollary}\label{cor:int} Let $S$ be a connected manifold.
If a Lie algebroid morphism $\Phi:\T S\ra\AG$ over $f:S\ra M$ is integrable then, for each $x_0,y_0\in S$ and $g\in \alpha^{-}(f(x_0))\cap\beta^{-1}(f(y_0))$, there exists a unique smooth map $\wt\Phi:S\lra\GG_{f(x_0)}=\alpha^{-1}(f(x_0))$ such that $\mathcal{R}\circ\T\wt\Phi=\Phi$ and $\wt\Phi(y_0)=g$.

Conversely, if $\wt\Phi$ as above exists (for some $x_0$, $y_0$ and $g$), then the Lie algebroid morphism $\Phi$ is integrable. In particular, such $\wt\Phi$ exists if $S$ is simply connected.
\end{corollary}

\section{The Atiyah algebroid}\label{sapp:atiyah}

In this section we describe the Atiyah algebroid---a Lie algebroid canonically associated with a principal bundle. In particular we study its Lie bracket and give a description of the associated linear Poisson structure. Our discussion is based mostly on \cite{mackenzie}.

\subsection{Invariant vector fields}
\begin{definition} A \emph{principal bundle}\index{principal bundle} $G\lra P\overset\pi\lra M$ is a locally trivial fibre bundle $\pi:P\lra M$, equipped with a (right) free action of the Lie group $G$ on $P$; $(p,g)\mapsto pg$; $P\times G\lra P$ such that its orbits coincide with the fibres of $\pi$. 
\end{definition}

Observe that if $G\lra P\overset\pi\lra M$ is a principal bundle, the action of $G$ on $P$ induces the action on the tangent bundle $\T P$. Denote this action by $R_g$. The quotient space $\T P/G$ has a natural structure of a vector bundle over $M$. We define the addition simply by
\begin{align*}
\left[X_p\right]+\left[Y_{pg}\right]&=\left[R_gX_p+Y_{pg}\right]
\intertext{and the base projection by}
\left[X_p\right]&\mapsto \pi(p).
\end{align*}
It is straightforward to verify that the above constructions are well-defined.

Observe that   section of the quotient bundle $E:=\T P/G\ra M$ can be canonically identified with $G$-invariant vector fields on $P$. Note that, since 
$G$-invariant vector fields on $P$ are closed under the Lie bracket $[\cdot,\cdot]_{\T P}$, we have an induced bracket $[\cdot,\cdot]_E$ on the space of sections of $E$. Clearly,  $[\cdot,\cdot]_{E}$ inherits the skew-symmetry and the Jacobi identity form $[\cdot,\cdot]_{\T P}$, and hence $(\Sec(E),[\cdot,\cdot]_E)$ is a Lie algebra. Moreover, since for any $G$-invariant vector fields $X,Y\in \XX^G(P)$, and every base function $f\in C^\infty(M)$ we have
$$\left[X,\pi^\ast fY\right]_{\T P}=\pi^\ast f\left[X,Y\right]_{\T P}+(\T\pi)(X)(f)Y,$$
and the derivative $(\T\pi)(X)(f)$ depends only on the class of $X$,
the bracket $[\cdot,\cdot]_E$ satisfies
the Leibniz rule 
$$\left[[X],f[Y]\right]_E=f\left[[X],[Y]\right]_E+\rho([X])(f)[Y],$$
where the anchor map $\rho:E\ra\T M$ is defined by  $\rho\left([X]\right)=\T\pi(X)$.
Clearly, $\rho$ satisfies also the compatibility condition \eqref{eqn:ala}.

To sum up, the bundle $E=\T P/G\lra M$, together with the bracket
$(E,[\cdot,\cdot]_E,\rho)$, and the anchor map $\rho:E\lra\T M$ is a Lie algebroid. 

\begin{definition} The Lie algebroid structure on $\pi:\T P/G\lra M$ described above is called an \emph{Atiyah algebroid}\index{Atiyah algebroid|main} of the principal bundle $G\lra P\overset{\pi}\lra M$. 
\end{definition}

As has been already mentioned the Atiyah algebroid can be also described as a Lie algebroid associated with the gauge groupoid $\GG_P$. Now, we shall investigate this structure in detail.

\subsection{The Atiyah sequence}

Observe that, since $G$ acts on the fibres of $\pi$, the action $R_g$ restricts to the space $VP\subset \sT P$ of vertical vectors (i.e., vectors tangent to the fibres of $\pi$).
Since $VP$ is spanned by the fundamental vector fields of the
$G$-action on $P$, we have a canonical isomorphism $VP\simeq
P\times\g$, where $\g$ is the Lie algebra of $G$. The action $R_g$
in this identification reads $R_g(p,a)=(pg,Ad_{g^{-1}}a)$, so 
$K:=VP/G\simeq P\times_G\g$, where the right action of $G$ on $\g$
is $g\mapsto Ad_{g^{-1}}$. What is more, the bracket of two
$G$-invariant vertical vector fields on $P$ corresponds, in this
identification, to the canonical (right) Lie bracket $[\cdot,\cdot]_\g$ on
$\g$. That is, if $x\mapsto (x, a(x))$ and $x\mapsto (x,b(x))$ are
two $G$-invariant sections of $P\times\g\simeq VP$, then
$$[(x,a(x)),(x,b(x))]_{\sT P}=(x,[a(x),b(x)]_\g)\in P\times\g.$$
This shows that the bundle $K$ is a Lie algebroid with the trivial
anchor and the Lie algebra structure in fibres isomorphic to $\g$. Alternatively, we may argue that the sections of $K$ can be identified with $G$-invariant vertical  vector fields on $P$, which are closed under the Lie bracket. The Lie algebroid structure on $K$ is thus the restriction of the Lie algebroid structure on $\T P/G$ to vertical vector fields. 
Hence, we get the following (exact) sequence of Lie algebroid morphisms
called the \emph{Atiyah sequence}\index{Atiyah sequence}:
$$ 0\to K:=P\times_G{\g}\to E\overset{\rho}{\ra}TM\to 0.$$

\subsection{Local description}
Introduce now a local trivialisation $\phi_s:V\times G\ra P|_V$
obtained from a local section $s:V\ra P$ by the formula
$\phi_s(x,g)=s(x)g$. Clearly, we may identify $E|_V$ with $\sT
P|_{s(V)}$, thus $E|_V\simeq \sT M|_V\times\g$. For two local
sections $x\mapsto (X(x),a(x))$ and $x\mapsto (Y(x),b(x))$ in this
trivialisation the Lie bracket reads as
\begin{equation}\label{eqn:bracket1}
\left[(X,a),(Y,b)\right]_E=\left([X,Y]_{\sT
M},[a,b]_\g+X(b)-Y(a)\right)
\end{equation}
and the anchor is $\rho\left((X,a)\right)=X$. In fact, this is just
the product of the Lie algebroids $\sT M|_V\lra V$ and $\g$ (cf.
the last paragraph of Chapter \ref{ch:algebroids}). We have a similar
description globally if the principal bundle $P$ is trivial. Note,
however, that, if we will work with principal bundles over a
neighbourhood of a path in $M$, we can always assume that our
principal bundle is trivial.

In some applications one has to work with a \emph{principal connection}\index{principal connection}
on $P$. It corresponds to a $G$-invariant horizontal distribution
in $\T P$ and is represented by a splitting $E=\T M\oplus_M K$
given by a bundle embedding $\nabla:\T M\ra E$ such that
$\rho\circ\nabla=\id_{\T M}$. If the bundle $K$ is trivial, then
we get another trivialisation $E\simeq_\nabla\sT M\times\g$,
associated with the connection $\nabla$, in which the Lie bracket reads as
\begin{equation}\label{eqn:bracket}
\left[(X,a),(Y,b)\right]_{E}=\left([X,Y]_{\sT M},
F_\nabla(X,Y)+[a,b]_\g+X(b)-Y(a)\right);
\end{equation}
where $F_\nabla$ is the \emph{curvature} of the connection
$\nabla$, i.e.,
$$F_\nabla(X,Y)=[\nabla(X),\nabla(Y)]_E-\nabla\left([X,Y]_{\sT M}\right)\,.$$
The anchor map is still simply $\rho((X,a))=X$. 

Observe that formula \eqref{eqn:bracket1} is a special case of \eqref{eqn:bracket}. Indeed, we may regard a local section $s:V\lra P$ as a $G$-invariant horizontal local distribution on $P$ obtained by spanning $\T s\subset\T P$ by the $G$-action. Clearly the curvature of this distribution vanishes, and hence \eqref{eqn:bracket} and \eqref{eqn:bracket1} coincide in this special case.

\subsection{The linear Poisson structure}

Now we shall describe the linear Poisson structure $\Pi_E$ on the dual bundle $E^\ast=\T^\ast P/G\lra M$ canonically associated with the Lie algebroid structure on $\T P/G$. 

Suppose for simplicity that the vertical subbundle $K=VP/G$ is
trivial (e.g. $P$ is trivial), $K=M\times\g$, and consider a
splitting $E\simeq_\nabla\sT M\oplus_MK$ induced by a principal
connection $\nabla:\sT M\ra E$, so that we get an identification
$E\simeq_\nabla \sT M\times\g$. Let $E^\ast\simeq_\nabla\sT^\ast
M\times\g^\ast$ be the corresponding identification of the dual
bundle.

\begin{theorem}\label{thm:ham_E1}
The Poisson tensor $\Pi_E$ associated with the Lie algebroid
structure on $E$ in the identification
$E^\ast\simeq_\nabla\sT^\ast M\times\g^\ast$ reads as
 $$\Pi_E(p_x,\zeta)=\left(\Pi_{\sT^\ast M}(p_x)+V_{\<\zeta,F_{\nabla}(x)(\cdot,\cdot)>}\right)\times \Pi_{\g^\ast}(\zeta),$$
where $p_x\in T^\ast_x M$, $\xi\in\g^\ast$, $\Pi_{\sT^\ast M}$ and
$\Pi_{\g^\ast}$ are the standard Poisson tensors, and
$V_{\<\zeta,F_{\nabla}(x)(\cdot,\cdot)>}$ is the two-form
$\<\zeta,F_{\nabla}(x)(\cdot,\cdot)>$ associated with the
curvature $F_\nabla(x):\bigwedge^2\sT_x M\ra \g$ understood as a
vertical tensor on $\sT^\ast M$.

Consequently, the Hamiltonian vector field defined by means of $\Pi_E$ and a
Hamiltonian $h:\sT^\ast M\times\g^\ast\lra\R$ reads as
$$\X^E_h(p_x,\zeta)=\left(\X^{\sT^\ast M}_{h(\cdot,\zeta)}(p_x)+V_{\<\zeta,F_{\nabla}(x)(\frac{\pa h}{\pa p}(p_x,\zeta),\cdot)>},
\X^{\g^\ast}_{h(p_x,\cdot)}(\zeta)\right),$$
where with $\X^{\T^\ast M}_h$ and $\X^{\g^\ast}_h$ we denoted the Hamiltonian vector fields associated with Poisson structures  $\Pi_{\sT^\ast M}$ and
$\Pi_{\g^\ast}$.

 In local coordinates,
$p\sim(x^a,p_b)$ and $\zeta\sim(\zeta_\alpha)$,
\begin{align*}\X^E_h(x,p,\zeta)=&\frac{\pa h}{\pa p_a}(x,p,\zeta)\pa_{x^a}+\left(
\zeta_\alpha F^\alpha_{ab}(x)\frac{\pa h}{\pa
p_a}(x,p,\zeta)-\frac{\pa h}{\pa x^b}(x,p,\zeta)\right)\pa_{p_b}\\&+
\zeta_\gamma C^\gamma_{\alpha\beta}\frac{\pa
h}{\pa\zeta_\alpha}\pa_{\zeta_\beta}\,,
\end{align*}
where $F^\alpha_{ab}(x)$
are the coefficients of the curvature $F_\nabla$ and
$C^\alpha_{\beta\gamma}$ are the structure constants of $\g$.
\end{theorem}
\begin{proof}
 The proof is straightforward by an explicit coordinate calculation using the local description of the bracket and the anchor map \eqref{eqn:bracket} and  formula \eqref{eqn:poisson} relating the bracket and anchor with the coefficient of the linear Poisson tensor.\end{proof}

\chapter{Analysis}\label{app:analysis}


\section{Mesuaralbe maps, regular points}\label{sec:meas}

In this section we briefly recall some basic properties of
measurable and absolutely continuous maps. Later we introduce a notion of \emph{uniform regularity} and study its basic properties. Ii is a quite important technical tool in our considerations.

\subsection{Basic facts} 
When speaking about measure we will always have in mind \emph{Lebesgue measure}\index{Lebesgue measure} in $\R^n$ or subsets of $\R^n$. This measure will be denoted by $\mu_L(\cdot)$. 

Recall that a map $f:V\supset\R^n\ra\R^k$, defined on a subset $V\subset\R^n$,
is \emph{measurable}\index{measurable map} if the inverse image of every open set is
Lebesgue measurable in $V$. The measurable map $f$ will be called
\emph{bounded}\index{measurable map!bounded} if the closure of its image is a compact set.
Observe that every bounded (or locally bounded) measurable
function is locally integrable. 

\noindent Measurable maps can be characterised as follows.

\begin{theorem}[Luzin]\label{thm:luzin}\index{Luzin theorem}
The map $f:\R^n\supset V\ra\R$ defined on a measurable set $V$ is measurable iff, for every $\eps>0$, there exists a closed subset $F\subset V$ such that
the restriction $f|_F$ is continuous and $\mu_L(V\setminus F)<\eps$.
\end{theorem}
\noindent For the proof see \cite{lojasiewicz}.

\noindent In our considerations much attention will be payed to regular points of measurable maps.
\begin{definition}
Let $f:[a,b]\ra\R^m$ be a measurable map. A point $x\in [a,b]$ is called a
\emph{regular point} (also: \emph{Lebesgue} or \emph{density point})\index{regular point} of $f$, iff  
$$\lim_{t\to
0}\frac1{|t|}\int_0^t|f(x+s)-f(x)|\dd s=0.$$
\end{definition}

\noindent For bounded measurable (or more generally integrable) maps we have the following result.

\begin{theorem}[Lebesgue]\index{Lebesgue theorem}
For an integrable map $f:[a,b]\ra\R^m$ almost every point in $[a,b]$ is a regular point of $f$.
\end{theorem}
\noindent For the proof see \cite{lojasiewicz}

A map $x:[a,b]\ra\R^k$ is called \emph{absolutely
continuous}\index{absolutely continuous map}\index{AC map|see{absolutely continuous map}}  (AC) if it can be written in the form
$$x(t)=x(a)+\int_{a}^{t}v(\tau)\dd \tau,$$
where $v(\cdot)$ is an integrable map. As we see,
an AC map $x(t)$ is differentiable at all the regular points $t$
of $v$ (hence, by Lebesgue Theorem, differentiable a.e.). Its
derivative at such a point is simply $v(t)$. In this work we concentrate our attention mostly on \emph{absolutely 
continuous maps with bounded derivative}\index{absolutely continuous map!with bounded derivative}\index{ACB map|see{absolutely continuous map with bounded derivative}}  (ACB maps) i.e. maps for which $v(\cdot)$ is bounded measurable.

In our considerations we will use the following lemma.
\begin{lemma}\label{lem:reg}
Let $a:[0,1]\ra\R$ be a bounded measurable map, let $h:[0,1]\ra[0,1]$ be a continuous function, and let $g:[0,1]\ra[0,1]$ be a $C^1$--map with a non-vanishing derivative. Then the map
$$s\longmapsto G(s):=\int_0^1\left|a(g(t)h(s))-a(g(t)h(s_0))\right|\dd t$$
is regular (in fact continuous) at every $s_0$ such that $h(s_0)\neq 0$.
\end{lemma}
\begin{proof}
Let $c$ be a number such that $0<c\leq\left|g^{'}(x)\right|$ for every $x\in[0,1]$. Now if $A\subset[0,1]$ is a measurable subset then
$\mu_L\left(g^{-1}( A)\right)\leq \frac 1c\mu_L(A)$. 

Choose $\eps>0$. By Luzin Theorem \ref{thm:luzin} there exists a closed set $F\subset[0,1]$ such that $a(\cdot)$ is continuous on $F$ and $\mu_L([0,1]\setminus F)<\eps$. Now $a$ is uniformly continuous on $F$, $g$ is bounded and $h$ continuous, hence there exists $\delta>0$ such that $\left|a(g(t)h(s))-a(g(t)h(s_0))\right|<\eps$ if only $|s-s_0|<\delta$ and $t$ and $s$ are such that $g(t)h(s)\in F$ and $g(t)h(s_0)\in F$. As a consequence for $|s-s_0|<\delta$, we can estimate
\begin{align*}
&\int_0^1|a(g(t)h(s))-a(g(t)h(s_0))|\dd t \leq
\int_{\{t:g(t)h(s)\notin F\}}2\|a\|\dd t+\int_{\{t:g(t)h(s_0)\notin F\}}2\|a\|\dd t+\\
&+\int_{\{t:g(t)h(s)\in F\}\cap\{t:g(t)h(s_0)\in F\}}|a(g(t)h(s))-a(g(t)h(s_0))|\dd t\\
&\leq 2\|a\|\cdot\mu_L\left(g^{-1}\left(\frac 1{h(s)}([0,1]\setminus F)\right)\right)+2\|a\|\cdot\mu_L\left(g^{-1}\left(\frac 1{h(s_0)}([0,1]\setminus F)\right)\right)+\int_0^1\eps\dd t\\
&\leq\eps\left(2\|a\|\frac 1 c\left(\frac 1{ h(s)}+\frac 1{h(s_0)}\right)+1\right)
\end{align*}
Since $h(s_0)\neq 0$, if $|s-s_0|$ is sufficiently small, the values s of $G(s)$ are arbitrarily close to $0=G(s_0)$, which finishes the proof. \end{proof}

\subsection{Uniform regularity} 
Regular points play an important role in our considerations, since the behaviour of a measurable map at a regular point is similar to the behaviour of a continuous map. To study behaviour of  the families of measurable maps we introduce a notion of \emph{uniform regularity}.

\begin{definition}\label{def:ur}
Let $P$ be a topological space and consider a map $f:[a,b]\times P\ra\R^m$ such that $t\mapsto f(t,p)$ is a measurable  for every $p\in P$. We call $f$ \emph{uniformly regular with respect to $p\in P$ at $x\in[a,b]$}\index{uniformly regular map} iff the following conditions are satisfied:
\begin{align}
&\frac 1{|t|}\int_0^t\left|f(x+s,p)-f(x,p)\right|\dd s\underset{t\to 0}\lra0\quad\text{locally uniformly w.r.t. $p$,}\label{eqn:ur1}\\
&\text{the map}\quad p\mapsto f(x,p) \quad\text{is continuous},\label{eqn:ur2}
 \intertext{and for every compact set $K\subset P$ there exists a number $t_0>0$ such that}
 & p\mapsto\Big([0,t_0]\ni s\mapsto f(x+s,p)\Big)\quad \text{is a continuous map from $K$ to $L^1([0,t_0],\R^m)$}.\label{eqn:ur3}
\end{align} 
\end{definition}

Usually in mathematics the word ''uniform'' means ''in the same way for all parameters''. In the context of regularity this can be expressed by the condition \eqref{eqn:ur1} itself. Therefore Definition \ref{def:ur} is more specific then what one could expect under the name ''uniform regularity''. The sense of this definition is, however, to abstract several technical properties of measurable  maps which are important from the point of view of this work. Since, according to our knowledge, the notion of uniform regularity is not a well established term, we hope that Definition \ref{def:ur} would not be confusing.

Let us now investigate some simple properties of uniformly regular maps. In what follows we will consider only uniform regularity at point $0\in\R$ and restrict our attention to parameter spaces $P$ which are metric (we can think of $P$  as of  a subset of $\R^m$). 

A basic example of a uniformly regular map is just a continuous map.\index{uniformly regular map!properties}
\begin{proposition}\label{prop:ur_cont}
Let $F:\R\times P\lra\R$ be a continuous map. Then $F$ is uniformly regular w.r.t. $p\in P$ at $s=0$. 
\end{proposition}
\begin{proof} Condition \eqref{eqn:ur2} is obvious. Fix now a compact set $K_P\subset P$ and restrict $s$ to a fixed interval $[0,t_0]$. Since $F$ is uniformly continuous on $[0,t_0]\times K_P$, for every $\eps>0$ there exists $\delta>0$ such that $\left|F(s,p)-F(0,p)\right|<\eps$ if $|s|<\del$ and for all $p\in K_P$. Consequently, 
$$\int_0^t\left|F(s,p)-F(0,p)\right|\leq|t|\eps$$
for $|t|<\delta$ and all $p\in K_P$. 
This proves \eqref{eqn:ur2}, i.e., $\frac 1{|t|}\int_0^t\left|F(s,p)-F(0,p)\right|\underset{t\to 0}\lra 0$ uniformly w.r.t. $p\in K_P$. 

To check \eqref{eqn:ur3} observe that, by the uniform continuity of $F$ on $[0,t_0]\times K_P$, for every $\eps>0$ there exists $\del>0$ such that $\left|F(s,p)-F(s,p^{'})\right|\leq\eps$ for every $s\in[0,t_0]$ and all $p,p^{'}\in K_P$ such that $|p-p^{'}|<\del$. Consequently,
$$\int_0^{t_0}\left|F(s,p)-F(s,p^{'})\right|\dd s\leq\left|t_0\right|\eps$$
for all $p,p^{'}\in K_P$ such that $|p-p^{'}|<\del$. This proves \eqref{eqn:ur3}.
\end{proof}

\noindent Another simple example is the following.

\begin{proposition}\label{prop:ur_trivial}
Let $f:\R\ra\R^m$ be a measurable map regular at $s=0$. For $p\in P$ define $\wt f(s,p):=f(s)$. Then the map $\wt f$ is uniformly regular w.r.t. $p\in P$ at $s=0$.  
\end{proposition}
\begin{proof}
Conditions \eqref{eqn:ur1}--\eqref{eqn:ur3} are trivially satisfied. 
\end{proof}

\noindent Below we discuss several ways of generating uniformly regular maps from given ones. 

\begin{proposition}\label{prop:ur_sum}
Let $f,g:\R\times P\ra\R^m$ be two maps uniformly regular w.r.t. $p\in P$ at $s=0$. Then the sum $f+g$ is also uniformly regular w.r.t $p\in P$ at $s=0$.
\end{proposition}
\begin{proof}
Property \eqref{eqn:ur1} is clear since 
$$|f(s,p)+g(s,p)-f(0,p)-g(0,p)|\leq|f(s,p)-f(0,p)|+|g(s,p)-g(0,p)|.$$
Property \eqref{eqn:ur2} is obvious as the sum of continuous maps is continuous. 

\noindent To prove \eqref{eqn:ur3} fix a compact set $K\subset P$ and assume that $K\ni p\mapsto\left([0,t_0]\ni s\mapsto f(s,p)\right)$ and $K\ni p\mapsto\left([0,\wt t_0]\ni s\mapsto g(s,p)\right)$ are continuous. Without loss of generality $t_0\leq \wt t_0$. Now the restriction  $K\ni p\mapsto\left([0,t_0]\ni s\mapsto g(s,p)\right)$ is also continuous since
$$\|a(\cdot)\|_{L^1([0,t_0],\R^m)}\leq \|a(\cdot)\|_{L^1([0,\wt t_0],\R^m)}.$$
Consequently, $K\ni p\mapsto f(\cdot,p)+g(\cdot,p)\in L^1([0,t_0],\R^m)$ is continuous as a sum of two continuous maps.
\end{proof}

\begin{proposition}\label{prop:ur_multiplication}
Let $f:\R\times P\ra\R^m$ be bounded and uniformly regular w.r.t. $p\in P$ at $s=0$. Let $h:P\lra\R$ be a continuous map. Then the map $\wt f(s,p)=h(p)f(s,p)$ is uniformly regular w.r.t. $p$ at $s=0$.
\end{proposition}
\begin{proof}
Condition \eqref{eqn:ur2} is obvious. 
Choose now a compact set $K\subset P$. For $p\in K$ we have
$$\frac 1{|t|}\int_0^t\left|h(p)f(s,p)-h(p)f(0,p)\right|\dd s\leq \sup_{p\in K}|h(p)|\cdot \frac 1{|t|}\int_0^t\left|f(s,p)-f(0,p)\right|\dd s\underset{t\to 0}\lra 0$$
uniformly w.r.t. $p\in K$, and hence \eqref{eqn:ur1} is satisfied.

Finally, note that for $p,p^{'}\in K$ we have
\begin{align*}
&\int_0^{t_0}\left|h(p)f(s,p)-h(p^{'})f(s,p^{'})\right|\dd s\\
&\leq\left|h(p)\right|\int_0^{t_0}\left|f(s,p)-f(s,p^{'})\right|\dd s+\left|h(p)-h(p^{'})\right|\int_0^{t_0}\left|f(s,p^{'})\right|\dd s\\
&\leq\sup_{p\in K}|h(p)|\int_0^{t_0}\left|f(s,p)-f(s,p^{'})\right|\dd s+\left|h(p)-h(p^{'})\right|\left\|f(\cdot,p^{'})\right\|_{L^1}\underset{p\to p^{'}}\lra 0+0;
\end{align*}
that is, \eqref{eqn:ur3} is satisfied.
\end{proof}

\begin{lemma}\label{lem:ur_rescal}
Let $f:\R\times P\ra\R^m$ be bounded and uniformly regular w.r.t. $p\in P$ at $s=0$. Consider $\wt f(s,p,c):=f(sc,p)$ where $c\in\R$. Then $\wt f:\R\times P\times\R\lra\R^m$ is uniformly regular w.r.t. $p\in P$ and $c\in\R$ at $s=0$. 
\end{lemma}
\begin{proof}
Since $\wt f(0,p,c)=f(0,p)$, condition \eqref{eqn:ur2} is obvious.

Consider now compact sets $K_C\subset\R$ and $K_P\subset P$. For 
 $c\in K_C$ and $p\in K_P$ we have
\begin{align*}
&\frac 1{|t|}\int_0^t\left|\wt f(s,p,c)-\wt f(0,p,c)\right|\dd s=\frac 1{|t|}\int_0^t\left|f(sc,p)-f(0,p)\right|\dd s\\
&=\frac 1{|t|c}\int_0^{tc}\left|f(s^{'},p)-f(0,p)\right|\dd s^{'}
\underset{t\to 0}\lra 0.
\end{align*}
Since $K_C$ is bounded and $\frac 1{|t^{'}|}\int_0^{t^{'}}\left|f(s,p)-f(0,p)\right|\dd s
\underset{t^{'}\to 0}\lra 0$ uniformly w.r.t. $p\in K_P$, the above convergence is uniform w.r.t. $p\in K_P$ and $c\in K_C$. 

We are left with the proof of property \eqref{eqn:ur3}.  We will check that $(p,c)\mapsto\wt f(\cdot,p,c)$; $K_P\times K_C\lra L^1\left([0,\wt t_0],\R^m\right)$ is continuous separately w.r.t. $p$ and w.r.t. $c$ for a suitably chosen $\wt t_0$. 

Let $t_0>0$ be a number from the property \eqref{eqn:ur3} for $f(s,p)$ and $K=K_P$. To prove the continuity w.r.t. $p$ fix $c\in K_C$. If $c\neq 0$, then
\begin{align*}\int_0^t\left|\wt f(s,p,c)-\wt f(s,p^{'},c)\right|\dd s&=\int_0^t\left|f(sc,p)-f(sc,p^{'})\right|\dd s\\
&=\frac 1{c}\int_0^{tc}\left|f(s^{'},p)-f(s^{'},p^{'})\right|\dd s^{'}\underset{p\to p^{'}, }\lra 0
\end{align*}
if only $tc\leq t_0$. 

For $c=0$ we have 
$$\int_0^t\left|\wt f(s,p,c)-\wt f(s,p^{'},c)\right|\dd s=\int_0^t\left|f(0,p)-f(0,p^{'})\right|\dd s=|t|\left|f(0,p)-f(0,p^{'})\right|\underset{p\to p^{'}}\lra 0$$
for every $t$. In particular, we proved continuity w.r.t. $p$  for $\wt t_0:=\frac{t_0}{\sup_{c\in K_C}|c|}$.

Now fix $p\in P$, fix $c^{'}\in K_C$ , choose $\eps>0$ and consider $c\in K_C$. 
If $c^{'}=0$, then
\begin{align*}
&\int_0^{\wt t_0}\left|\wt f(s,p,c)-\wt f(c,p,c^{'})\right|\dd s\\
&=\int_0^{\wt t_0}\left| f(sc,p)-f(0,p)\right|\dd s=
\left|\wt t_0\right|\frac 1{\left|\wt t_0\right|c}\int_0^{\wt t_0 c}\left|f(s^{'},p)-f(0,p)\right|\dd s^{'}\underset{c\to 0}\lra 0.
\end{align*}
If $c^{'}\neq 0$ consider a closed set $F\subset[0,\wt t_0]$ such that $\mu_L\left([0,\wt t_0]\setminus F\right)<\eps \left|\wt t_0\right|$ and $\wt f(\cdot,p,c^{'})$ is continuous on $F$ (note that $p$ and $c^{'}$ are fixed). Such a set exists by Luzin Theorem \ref{thm:luzin}. 

Since $\wt f(\cdot,p,c^{'})$ is uniformly continuous on $F$, there exists a number $\del>0$ such that $\left|\wt f(s,p,c^{'})-\wt f(s^{'},p,c^{'})\right|<\eps$ if $\left|s-s^{'}\right|\leq\delta$ and $s,s^{'}\in F$. Now $\wt f(s,p,c)=\wt f(\frac c{c^{'}}s,p,c^{'})$ and $\left|s-\frac c{c^{'}}s\right|\leq\left|\frac{c^{'}-c}{c^{'}}\right|\left|\wt t_0\right|$ for $s\in[0,\wt t_0]$, so we have 
 $$\left|\wt f(s,p,c^{'})-\wt f(s,p,c)\right|<\eps\quad\text{if}\quad |c^{'}-c|\leq\frac{c^{'}}{\left|\wt t_0\right|}\delta\quad \text{and}\quad s\in F\cap\frac{c^{'}}cF.$$
Note that 
\begin{align*}
&\mu_L\left([0,\wt t_0]\setminus F\cap\frac {c^{'}}c F\right)\leq\mu_L\left([0,\wt t_0]\setminus F\right)+\mu_L\left([0,\wt t_0]\setminus\frac {c^{'}}c F\right)\\
&\leq \eps\left|\wt t_0\right|+\left(\left|1-\frac{c^{'}}c\right|+\frac{c^{'}}c\eps\right)\left|\wt t_0\right|\leq 4\eps\left|\wt t_0\right|
\end{align*} if $|c-c^{'}|$ is small enough. Consequently,
\begin{align*}
&\int_0^{\wt t_0}\left|\wt f(s, p,c)-\wt f(s,p,c^{'})\right|\dd s\\
&\leq\int_{F\cap\frac {c^{'}}c F}\left|\wt f(s, p,c)-\wt f(s,p,c^{'})\right|\dd s+\int_{[0,\wt t_0]\setminus F\cap\frac {c^{'}}c F}\left|\wt f(s, p,c)-\wt f(s,p,c^{'})\right|\dd s\\
&\leq\eps\left|\wt t_0\right|+\mu_L\left([0,\wt t_0]\setminus F\cap\frac {c^{'}}c F\right)\cdot 2\left\|f\right\|\leq \eps\left(\left|\wt t_0\right|+8\left|\wt t_0\right|\left\|f\right\|\right)
\end{align*}
if $|c-c^{'}|$ is small enough. Since $\eps$ is an arbitrary positive number, this proves the continuity of $(p,c)\mapsto\wt f(\cdot,p,c)$ w.r.t. $c$.\end{proof}

\begin{lemma}\label{lem:ur_comp}
Let $f:\R\times P\ra\R^m$ be bounded and uniformly regular w.r.t. $p\in P$ at $s=0$, and let $G:\R^m\times\R\times P\times Q\lra\R^m$ be a continuous map w.r.t all variables. Then the composition $G(f(s,p),s,p,q)$ is uniformly regular w.r.t. $p\in P$ and $q\in Q$ at $s=0$.
\end{lemma}
\begin{proof}
We will prove the assertion for $G$ trivially depending on $q\in Q$. This will suffice, since we can denote $G(f(s,p),s,p,q)$ as $G(\wt f(s,p,q),s,p,q)=G(\wt f(s,\wt p),s,\wt p)$, where $\wt f(s,p,q):=f(s,p)$ and $\wt p:=(p,q)\in P\times Q$. Clearly, $\wt f(s,\wt p)$ is uniformly regular w.r.t. $\wt p=(p,q)\in P\times Q$ at $s=0$ (cf. Proposition \ref{prop:ur_trivial}) and the investigated composition has a desired simpler form $G(\wt f(s,\wt p),s,\wt p)$.  

Property \eqref{eqn:ur2} is obvious. To prove \eqref{eqn:ur1} estimate
\begin{align*}
&\left|G(f(s,p),s,p)-G(f(0,p),0,p)\right|\\
&\leq\left|G(f(s,p),s,p)-G(f(0,p),s,p)\right|+\left|G(f(0,p),s,p)-G(f(0,p),0,p)\right|.
\end{align*}
Since $\wt G(s,p):=G(f(0,p),s,p)$ is continuous w.r.t. $p$ and $s$, it satisfies \eqref{eqn:ur1}. Consequently, it is enough to check if
$$\frac 1{|t|}\int_0^t\left|G(f(s,p),s,p)-G(f(0,p),s,p)\right|\dd s\underset{t\to 0}\lra 0$$
locally uniformly w.r.t. $p$. To prove it consider a compact set $K_P\subset P$ and restrict $s$ to the interval $[0,t_0]\subset\R$. Since $f$ is bounded, its image $\Image f$ is contained in a compact subset $K\subset\R^m$. Fix $\eps>0$. The map $G$ is uniformly continuous on $K\times[0,t_0]\times K_P$, so there exists a number $\delta>0$ such that
$$\left|G(x,s,p)-G(y,s,p)\right|<\eps$$
if $|x-y|<\delta$ and $x,y\in K$, $s\in[0,t_0]$, and $p\in K_P$.  
 Since $f(s,p)$ is uniformly regular there exists a number $0<\wt t_0\leq t_0$ such that 
\begin{equation}\label{eqn:ur4}
\int_0^t\left|f(s,p)-f(0,p)\right|\dd s<|t|\cdot\eps\cdot\delta
\end{equation}
for every $|t|\leq \wt t_0$ and each $p\in K_P$. Define now $A_p:=\{s\in[0,t_0]:\left|f(s,p)-f(0,p)|>\delta\right|\}$. From \eqref{eqn:ur4} we have
$$\delta\cdot\mu_L\left([0,t]\cap A_p\right)\leq\int_{[0,t]\cap A_p}\left|f(s,p)-f(s,p)\right|\dd s\leq\int_0^t\left|f(s,p)-f(s,p)\right|\dd s\leq|t|\cdot\eps\cdot\delta,$$
hence $\mu_L\left([0,t]\cap A_p\right)\leq\eps\cdot|t|$. Consequently, for $|t|<\wt t_0$, we have
\begin{align*}
&\int_0^t\left|G(f(s,p),s,p)-G(f(0,p),s,p)\right|\dd s\leq 2\cdot\sup_{K\times[0,t_0]\times K_P}|G|\int_{[0,t]\cap A_p}1\dd s+\int_{[0,t]\setminus A_p}\eps \dd s\\
&\leq 2\cdot\sup_{K\times[0,t_0]\times K_P}|G|\cdot\mu_L\left([0,t]\cap A_p\right)+\eps|t|=\eps|t|\left( 2\cdot\sup_{K\times[0,t_0]\times K_P}|G|+1\right).
\end{align*}
Since $\eps>0$ was arbitrary, this proves \eqref{eqn:ur1}.

To prove \eqref{eqn:ur3} we proceed similarly. Again we restrict our attention to $K\times[0,t_0]\times K_P$ and fix $\eps>0$. Let $\delta>0$ be such that, for $x,y\in K$, $s\in[0,t_0]$ and $p,p^{'}\in K_P$
$$\left|G(x,s,p)-G(y,s,p^{'})\right|<\eps$$
if $|x-y|<\delta$ and $|p-p^{'}|<\delta$. 

From the uniform regularity of $f(s,p)$, there exists a number $\wt\delta>0$ such that 
$$\int_0^{t_0}\left|f(s,p)-f(s,p^{'})\right|\dd s\leq\eps\cdot \delta$$
if $p,p^{'}\in K_p$ are such that $|p-p^{'}|<\wt\delta$. From that we deduce that the set $B_{pp^{'}}:=\{s\in[0,t_0]:\left|f(s,p)-f(s,p^{'})\right|>\delta\}$ has measure smaller than $\eps$ if $|p-p^{'}|<\wt \delta$. Indeed, we can estimate
$$\delta\cdot\mu_L\left(B_{pp^{'}}\right)\leq\int_{B_{pp^{'}}}\left|f(s,p)-f(s,p^{'})\right|\dd s\leq\int_0^{t_0}\left|f(s,p)-f(s,p^{'})\right|\dd s\leq \delta\cdot\eps.$$

Therefore for $|p-p^{'}|\leq\min\{\delta,\wt\delta\}$ we can estimate
\begin{align*}
&\int_0^{t_0}\left|G(f(s,p),s,p)-G(f(s,p^{'}),s,p^{'})\right|\dd s\leq\int_{B_{pp^{'}}}\left|G(f(s,p),s,p)-G(f(s,p^{'}),s,p^{'})\right|\dd s\\
&\phantom{=}+\int_{[0,t_0]\setminus B_{pp^{'}}}\left|G(f(s,p),s,p)-G(f(s,p^{'}),s,p^{'})\right|\dd s\\
&\leq 2\cdot\sup_{K\times[0,t_0]\times K_P}|G|\cdot\mu_L\left(B_{pp^{'}}\right)+\int_0^{t_0}\eps\dd s\leq\eps\left(2\cdot\sup_{K\times[0,t_0]\times K_P}|G|+|t_0|\right).
\end{align*}
This proves \eqref{eqn:ur3}. \end{proof}

\section{Ordinary differential equations}\label{sapp:ode}
This section contains a revision of  the theory of ordinary differential equations in a measurable setting. We formulate standard theorems about existence, uniqueness and regularity of solutions. We state these results after \cite{bressan} and give sketches of the proofs.

\subsection{Carath{\'e}odory solutions}
Consider an ordinary differential equation associated with a map
$g:\R^n\times\R\ra\R^n$,
\begin{equation}\label{eqn:ode}
\dot x(t)=g(x(t),t).
\end{equation}
By a \emph{(Carath{\'e}odory) solution}\index{Carath{\'e}odory solution}\index{measurable solution of ODE|see{Carath{\'e}odory solution}} of \eqref{eqn:ode} on an
interval $I=[t_0,t_1]$ we shall mean an AC map $t\mapsto x(t)$
which satisfies \eqref{eqn:ode} a.e. For the solutions in the
above sense one can develop the standard theory of existence,
uniqueness, and parameter dependence, as done in
\cite{bressan}. Let us recall the most important results of this
theory. Assume the following:
\begin{equation}\label{ass:A}\tag{A}
\text{$t\mapsto g(x,t)$ is measurable for
every $x$,} \text{ and } \text{$x\mapsto g(x,t)$ is continuous for
every $t$;}
\end{equation}
\begin{equation}\label{ass:B}\tag{B}
\text{$g(x,t)$ is locally bounded and locally Lipschitz w.r.t. $x$;}
\end{equation}
that is, for every compact set $K\subset\R^n\times\R$ there exist
constants $C_K$ and $L_K$ such that $|g(x,t)|\leq C_K$ and
$|g(x,t)-g(y,t)|\leq L_K|x-y|$ for every $(x,t),(y,t)\in K$.

\begin{theorem}[existence and uniqueness of solutions]\label{thm:exist}\index{Carath{\'e}odory solution!existence and uniqueness}
Assuming that \eqref{ass:A} and \eqref{ass:B} hold, for every
$x_0\in\R^n$ there exists a unique solution $x(t,x_0)$ of
\eqref{eqn:ode} with the initial condition $x(t_0)=x_0$, defined
on some interval $[t_0,t_0+\eps]$. If $g$ is globally bounded and
globally Lipschitz (so that the constants $C_K$ and $L_K$ in
\eqref{ass:B} can be chosen universally for all $K$'s), then the
solution is also defined globally. Moreover, if $x(t,x_0)$ is defined on the interval $[t_0,t_1]$ then so are the solutions $x(t,x_0^{'})$ for $x_0^{'}$ close enough to $x_0$. 
\end{theorem}
\begin{proof}[Sketch of the proof]
The proof uses the standard Picard's Method. One constructs a contracting map 
$$A_{x_0}:x(t)\longmapsto x_0+\int_{t_0}^tg(x(\tau),\tau)\dd\tau$$
and uses it to define inductively a sequence of functions $x^0(t,x_0)=x_0$, $x^{n+1}(t,x_0)=A_{x_0}(x^n(t,x_0))$ which converges uniformly in $t$ to the solution $x(t,x_0)$. The length of the interval $[t_0,t_1]$ on which the solution is well-defined depends on the Lipschitz bound of $g(x,t)$. The details can be found in \cite[Thm. 2.1.1]{bressan}. 
\end{proof}

\subsection{Parameter dependence} 
Consider now differential equation \eqref{eqn:ode} with an additional parameter dependence
\begin{equation}\label{eqn:ode1}
\dot{x}(t)=g(x(t),t,s),
\end{equation}
where $g:\R^n\times\R\times\R\lra\R^n$. Assume the following:
\begin{equation}\label{ass:A1}\tag{$A^{'}$}
\text{$t\mapsto g(x,t,s)\,,\ s\mapsto g(x,t,s)$
are measurable,} \text{ and } \text{$x\mapsto g(x,t,s)$ is
continuous;}
\end{equation}
\begin{equation}\label{ass:B1}\tag{$B^{'}$}
\text{$g(x,t,s)$ is locally bounded and locally Lipschitz w.r.t. $x$;}
\end{equation}
that is, for every compact set $K\subset\R^n\times\R\times\R$
there exist constants $C_K$, $L_K$ such that $|g(x,t,s)|\leq C_K$
and $|g(x,t,s)-g(y,t,s)|\leq L_K|x-y|$ for every
$(x,t,s),(y,t,s)\in K$.

\begin{theorem}[parameter dependence]\label{thm:param}\index{Carath{\'e}odory solution!regularity}
Assume that \eqref{ass:A1} and \eqref{ass:B1} hold, and denote by
$x(t,x_0,s)$ the solution of \eqref{eqn:ode1} for a fixed
parameter $s$ and the initial condition $x(t_0,x_0,s)=x_0$ (we
know that such solutions locally exist by Theorem
\ref{thm:exist}). Then the dependence $x_0\mapsto x(t,x_0,s)$ is
continuous, whereas, for any bounded measurable map $s\mapsto
x_0(s)$, the map $s\mapsto x(t,x_0(s),s)$ is also bounded and
measurable for every $t$.
\end{theorem}
\begin{proof}[Sketch of the proof]
As before one constructs a sequence $x^n(t,x_0,s)$ defined by means of the contracting map
$$A_{x_0,s}:x(t)\longmapsto x_0+\int_{t_0}^t g(x(\tau),\tau,s)\dd\tau.$$
The sequence converges to the solution $x(t,x_0,s)$ uniformly w.r.t. $t$ and $x_0$, which implies continuity of the solution w.r.t. the initial value. If $x_0(s)$ is measurable w.r.t. $s$, so is the sequence $x^n(t,x_0(s),s)$. The limit $x(t,x_0(s),s)$ is measurable as a point-wise limit of measurable functions. Moreover, since $A_{x_0,s}$ is a contraction, $\|x(t,x_0(s),s)\|$ is bounded by a constant times $\|x_0(s)\|$. Details can be found in \cite{bressan}.\end{proof}

\noindent Assuming higher regularity of  $g(x,t,s)$, one can
prove a stronger result.

\begin{theorem}[differentiability w.r.t. the initial value]\label{thm:param_dif}\index{Carath{\'e}odory solution!regularity}

Assume that the function $g(x,t,s)$ satisfies \eqref{ass:A1} and
\eqref{ass:B1}, it is differentiable w.r.t. $x$, and the derivative
$\frac{\pa g}{\pa x}(x,t,s)$ satisfies \eqref{ass:A1} and is
locally bounded. Then the solution $x(t,x_0,s)$ of
\eqref{eqn:ode1} is differentiable w.r.t. the initial
condition $x_0$. Moreover, the derivative $\frac{\pa x}{\pa
x_0}(t,x_0,s)$ is continuous in $x_0$, AC in $t$, and measurable
in $s$.
\end{theorem}
\begin{proof}[Sketch of the proof]
We proceed again according to the standard method paying more attention to measurability. Consider a variation of \eqref{eqn:ode1}
\begin{align*}
\dot x(t)&=g(x(t),t,s),\\
\dot y(t)&=\frac{\pa g}{\pa x}(x(t),t,s)
\end{align*}
with the initial conditions $x(t_0)=x_0$ and $y(t_0)=\id$. The above equations satisfy the assumptions of Theorem \ref{thm:param}, hence the solution $x(t,x_0,s)$ and $y(t,x_0,s)$ is a uniform (in $t$ and $x_0$) limit of the Picard's sequence $x^n(t,x_0,s)$ and $y^n(t,x_0,s)$. We observe that $\frac{\pa x^n}{\pa x_0}(t,x_0,s)=y^n(t,x_0,s)$, hence also $\frac{\pa x}{\pa x_0}(t,x_0,s)=y(t,x_0,s)$. The derivative $y(t,x_0,s)$ satisfies the regularity conditions by Theorem \ref{thm:param}. 

Again a detailed proof (the only difference is the absence of
the parameter $s$) can be found in \cite[Thm.
2.3.2]{bressan}. 
\end{proof}

\subsection{Gronwall Inequality}

At the end of this section we will recall the following classical result.

\begin{theorem}[integral Gronwall Inequality]\label{thm:gronwall}\index{Gronwall Inequality}
Assume that $b:[0,T]\ra\R$ is non-negative and integrable, and for almost every $t\in[0,T]$ we have
$$b(t)\leq C\cdot\int_0^tb(s)\dd s,$$
where $C>0$ is a constant. 
Then $b(t)=0$ a.e.
\end{theorem}
For the proof see for instance \cite[App. B]{evans}.

\section{\texorpdfstring{$E$}{E}-homotopy type equations}\label{sapp:pde}

In this section we consider linear PDEs of a special kind which are important in the notion of $E$-homotopy. We study their solutions in a weak sense and address a question about the existence of the trace.  

\subsection{\texorpdfstring{$E$}{E}-homotopy type equations}\index{algebroid homotopy}
Consider two functions $a(t,s)$ and $b(t,s)$ defined on a
rectangle $K=[t_0,t_1]\times[0,1]\ni(t,s)$. We will treat $a$ and
$b$ as $\R$-valued, yet all results remain valid for $\R^n$-valued
maps. Assume that $a$ and $b$ are bounded and measurable w.r.t. both variables
separately. Let $c(t,s)$ be a fixed continuous
function on $K$. We will say that the pair $(a,b)$ is a \emph{weak
solution} (\emph{W-solution})\index{weak solution}\index{W-solution|see{weak solution}} of the differential equation
\begin{equation}\label{eqn:basic}
\pa_t b(t,s)=\pa_sa(t,a)+c(t,s)b(t,s)a(t,s)
\end{equation}
if for every function $\varphi\in C^\infty_0(K)$ the following
equality holds:
\begin{equation}\label{eqn:basic_weak}
-\iint_K\Big[b(t,s)\pa_t\varphi(t,s)-
a(t,s)\pa_s\varphi(t,s)+c(t,s)b(t,s)a(t,s)\varphi(t,s)\Big]\dd
t\dd s=0.
\end{equation}
Observe that, since we assumed only measurability of $a$ and $b$,
the boundary values on $\pa K$ are, in general, not well-defined.

\begin{definition}\index{weak solution!with well defined trace|see{WT-solution}}\index{WT-solution}
We say that a W-solution $(a,b)$ of \eqref{eqn:basic} has a \emph{well-defined trace} if there exist bounded measurable maps $a_0,a_1:[t_0,t_1]\ra\R$ and $b_0,b_1:[0,1]\ra\R$ such that, for every $\psi\in C^\infty(K)$, we have
\begin{equation}\label{eqn:basic_wt}
\begin{split}
&-\iint_K\Big[b(t,s)\pa_t\psi(t,s)- a(t,s)\pa_s\psi(t,s)+c(t,s)b(t,s)a(t,s)\psi(t,s)\Big]\dd t\dd s\\
&=\int_0^1b_1(s)\psi(t_1,s)-b_0(s)\psi(t_0,s)\dd
s+\int_{t_0}^{t_1}a_0(t)\psi(t,0)-a_1(t)\psi(t,1)\dd t.
\end{split}
\end{equation}
In such a case we will call $(a,b)$ a \emph{WT-solution} of \eqref{eqn:basic}. The maps $a_0$, $a_1$ and $b_0$, $b_1$ will be called \emph{traces} of $a$ and $b$, respectively.\index{trace}
\end{definition}

\begin{remark}Since the values of measurable functions are defined a.e. only, for a WT-solution $(a,b)$ of \eqref{eqn:basic} we will assume that the traces agree with the boundary values of $a$ and $b$, i.e.  $a_0(t)=a(t,0)$, $a_1(t)=a(t,1)$, $b_0(s)=b(t_0,s)$, and $b_1(t)=b(t_1,s)$.
\end{remark}

\subsection{Existence of the trace and properties of WT-solutions}

Under certain regularity conditions, W-solutions of \eqref{eqn:basic} are, in fact, WT-solutions.
\begin{theorem}\label{thm:w_wt}\index{WT-solution!existence}
Let $a$ and $b$ be a bounded W-solution of \eqref{eqn:basic}.
Assume in addition that $a$ and $b$ satisfy the following regularity conditions:
\begin{equation}\label{eqn:traces}
\int_{t_0}^{t_1}\int_0^\eps|a(t,s)-a(t,0)|\frac 1\eps\dd s\dd
t\underset{\eps\to 0}{\to}0\,,\
\int_{t_0}^{t_1}\int_{1-\eps}^1|a(t,s)-a(t,1)|\frac 1\eps\dd s\dd t\underset{\eps\to 0}{\to}0\,,
\end{equation}
\begin{equation}\int_{0}^{1}\int_{t_0}^{t_0+\eps}|b(t,s)-b(t_0,s)|\frac 1\eps\dd
t\dd s\underset{\eps\to 0}{\to}0\,,\
\int_{0}^{1}\int_{t_1-\eps}^{t_1}|b(t,s)-b(t_1,s)|\frac 1\eps\dd
t\dd s\underset{\eps\to 0}{\to}0.\label{eqn:traces1}
\end{equation}
Then $(a,b)$ is a  WT-solution of \eqref{eqn:basic} and the traces $a(t,0)$, $a(t,1)$, $b(t_0,s)$, and $b(t_1,s)$ are well-defined.
\end{theorem}
\begin{proof}
Fix an element $\psi\in C^\infty(K)$ and choose $\eps>0$. The idea of the proof is standard: we will approximate $\psi$ by another function $\varphi\in C^\infty_0(K)$ and, using \eqref{eqn:basic_weak} for $\varphi$ and the regularity conditions, show that \eqref{eqn:basic_wt} holds with $\eps$-accuracy. 

Define a rectangle $K_\eps:=[t_0+\eps,t_1-\eps]\times[\eps,1-\eps]\subset K$, and choose a smooth ''hat function'' $\chi_{[a,b]}:[a,b]\ra\R$ which satisfies the following conditions:
\begin{align*}
&\chi_{[a,b]}(a)=0=\chi_{[a,b]}(b)=0,& \chi_{[a,b]}(t)=1\text{\ for $t\in[a+\eps,b-\eps]$}& &\text{and }&& \|D\chi_{[a,b]}\|<\frac 2\eps .
\end{align*}
Now define $\chi(t,s):=\chi_{[t_0,t_1]}(t)\cdot\chi_{[0,1]}(s)$. Obviously, $\chi\equiv 1$ on $K_\eps$, $\chi\in C^\infty_0(K)$ and  $\|D\chi_{[a,b]}\|<\frac 4\eps$. Moreover, $\pa_t\chi(t,s)=0$ for $t\in[t_0+\eps,t_1-\eps]$, and $\pa_s\chi(t,s)=0$ for $s\in[\eps,1-\eps]$.

Now define $\varphi:=\chi\cdot\psi\in C^\infty_0(K)$ and $\wt\psi:=(1-\chi)\psi$. Decomposing $\psi=\varphi+\wt\psi$, we get 
\begin{align*}
&-\iint_K\Big[b\pa_t\psi- a\pa_s\psi+cba\psi\Big]\dd t\dd s\\
&=-\iint_K\Big[b\pa_t\varphi- a\pa_s\varphi+cba\varphi\Big]\dd t\dd s-\iint_K\Big[b\pa_t\wt\psi- a\pa_s\wt\psi+cba\wt\psi\Big]\dd t\dd s
\overset{\eqref{eqn:basic_weak}}{=} \\
&=0-\iint_Kb\pa_t\wt\psi\dd t\dd s +\iint_K a\pa_s\wt\psi\dd t\dd s-\iint_K cba\wt\psi\dd t\dd s=:I_1+I_2+I_3.
\end{align*}
We will now concentrate on the tree last summands. Observe that 
\begin{align*}
|I_3|=\left|\iint_K cba\wt\psi\dd t\dd s\right|=\left|\iint_{K-K_\eps} cba\wt\psi\dd t\dd s\right|\leq\mu(K\setminus K_\eps)\|cba\wt\psi\|\leq \eps\cdot C_3,
\end{align*}
where $C_3$ is a constant depending on $\|a\|$, $\|b\|$, $\|\psi\|$, and $\|c\|$.

Now 
\begin{align*}
I_2&=\iint_Ka\pa_s\wt\psi\dd t\dd s=\iint_{K\setminus K_\eps}a\pa_s\wt\psi\dd t\dd s\\
&=\iint_{K\setminus K_\eps}a\psi\pa_s(1-\chi)\dd t\dd s+\iint_{K\setminus K_\eps}a(1-\chi)\pa_s\psi\dd t\dd s.
\end{align*}
The last summand can be estimated by $\eps\cdot C_2$ in the same way as $I_3$ (with $C_2$ depending additionally on $\|D\psi\|$).
Now, since $\pa_s\chi(t,s)=0$ for $s\in[\eps,1-\eps]$,
\begin{align*}
\iint_{K\setminus K_\eps}a\psi\pa_s(1-\chi)\dd t\dd s=-\int_{t_0}^{t_1}\int_0^\eps a\psi\pa_s\chi\dd s\dd t-\int_{t_0}^{t_1}\int_{1-\eps}^1 a\psi\pa_s\chi\dd s\dd t=:I_4+I_5.
\end{align*}
We can write $I_4$ as 
\begin{align*}
I_4=&\int_{t_0}^{t_1}\left[a(t,s)\psi(t,s)-a(t,0)\psi(t,0)\right]\pa_s\chi\dd s\dd t+\\
&+\int_{t_0}^{t_1}a(t,0)\psi(t,0)\left[\int_0^\eps \pa_s\chi\dd s\right] \dd t=I_6+\int_{t_0}^{t_1}a(t,0)\psi(t,0)\dd t. 
\end{align*}
Now 
\begin{align*}
I_6=&\int_{t_0}^{t_1}\int_0^\eps\left(a(t,s)-a(t,0)\right)\psi(t,0)\pa_s\chi\dd s\dd t+\\
 &+\int_{t_0}^{t_1}\int_0^\eps a(t,s)\left(\psi(t,s)-\psi(t,0)\right)\pa_s\chi\dd s\dd t=:I_7+I_8.
\end{align*}
Clearly,
$$|I_7|\leq\int_{t_0}^{t_1}\int_0^\eps|a(t,s)-a(t,0)|\cdot\|\psi\|\cdot\|D\chi\| \dd s\dd t\leq\int_{t_0}^{t_1}\int_0^\eps|a(t,s)-a(t,0)|\cdot\|\psi\|\frac 4\eps \dd s\dd t,$$
so by assumptions it converges to 0 as $\eps \to 0$. Finally, using $|\psi(t,s)-\psi(t,0)|\leq s\cdot\|D\psi\|$, we get
$$|I_8|\leq\int_{t_0}^{t_1}\int_0^\eps\|a\|s\|D\psi\|\cdot\|D\chi\|\dd s\dd t\leq\int_{t_0}^{t_1}\int_0^\eps\|a\|\eps\|D\psi\|\frac 4\eps\dd s\dd t\leq \eps\cdot C_8.$$
As a consequence, we get 
$$I_4\to \int_{t_0}^{t_1}a(t,0)\psi(t,0)\dd t\quad\text{as $\eps\to 0$}.$$
Analogous estimations can be done for $I_5$. As a result we get that
$$\left|I_2-\int_{t_0}^{t_1}\left[a(t,0)\psi(t,0)-a(t,1)\psi(t,1)\right]\dd t\right|\underset{\eps\to 0}{\to} 0.$$

\noindent We can repeat the above considerations for $I_1$ to prove that
$$\left|I_1+\int_{0}^{1}\left[b(t_0,s)\psi(t_0,s)-b(t_1,s)\psi(t_1,s)\right]\dd s\right|\underset{\eps\to 0}{\to} 0.$$

The estimations for $I_1$, $I_2$ and $I_3$ show that, for a fixed $\psi$, the equality \eqref{eqn:basic_wt} is satisfied with an accuracy converging to 0 as $\eps\to 0$. \end{proof}

For WT-solutions we can formulate an uniqueness result. 

\begin{lemma}[uniqueness of WT-solutions]\label{lem:wt_uniq}\index{WT-solution!uniqueness}
Let $a:K\ra\R$ be a bounded measurable map (w.r.t. both variables
separately), and let $b_0:[0,1]\ra \R$ be any bounded measurable
map. Then there exists at most one bounded measurable map
$b:K\ra\R$ such that $(a,b)$ is a WT-solution of \eqref{eqn:basic},
and $b(t_0,s)=b_0(s)$. Moreover, the trace $b(t_1,s)$ is
 determined uniquely.

\end{lemma}
\begin{proof}
Assume that $b(t,s)$ and $\wt b(t,s)$ are two such solutions for a
fixed $a(t,s)$. The difference $\del b(t,s):=b(t,s)-\wt b(t,s)$ is
a bounded measurable map which is a WT-solution of the linear
equation
\begin{equation}\label{eqn:1}
\pa_t\del b(t,s)=c(t,s)\del b(t,s)a(t,s)
\end{equation}
such that $\del b(t_0,s)=0$. Let us define
$B(\tau,s):=0+\int_{t_0}^\tau c(t,s)\del b(t,s)a(t,s)\dd t.$
Clearly, $B(t,s)$ is ACB w.r.t. $t$ and measurable w.r.t. $s$.
Moreover, we have $\pa_t B(t,s)=c(t,s)\del
b(t,s)a(t,s)$ in the sense of Carath{\'e}odory and, since $B(t,s)$ is continuous w.r.t. $t$, also WT.
Consequently, $(B-\del b)$ satisfies $\pa_t (B-\del
b)\overset{\text{WT}}{=}0$ and, since $B(t_0,s)=\del b(t_0,s)=0$,
we have
\begin{equation}\label{eqn:2}
-\iint_K(B-\del b)(t,s)\pa_t\psi(t,s) \dd t\dd s= \int_0^1 (B-\del
b)(t_1,s)\psi(t_1,s)\dd s
\end{equation}
for every $\psi\in C^\infty(K)$. Taking $\psi(t,s)=\phi(s)$, where
$\phi\in C^\infty(I)$, we get that $\int_0^1(B-\del
b)(t_1,s)\phi(s)\dd s=0$, thus $B(t_1,s)=\del b(t_1,s)$ a.e. In
the light of this observation \eqref{eqn:2} reads as
$$\iint_K(B-\del b)(t,s)\pa_t\psi(t,s) \dd t\dd s=0,$$
for every $\psi\in C^\infty(K)$. Since $\pa_t\psi$ can be an
arbitrary smooth function, we conclude that $B(t,s)=\del b(t,s)$
a.e. Consequently, $\del b$ is a Carath{\'e}odory solution of
\eqref{eqn:1}. Now observe that
\begin{align*}
|\del b(\tau,s)|=|\int_{t_0}^\tau\pa_t\del b(t,s)\dd
t|\leq\int_{t_0}^\tau|\pa_t\del b(t,s)|\dd t=\int_{t_0}^\tau
|c\cdot a||\del b(t,s)|\dd t,
\end{align*}
which, in view of the integral Gronwall Inequality \ref{thm:gronwall}, implies $\del b=0$.
\end{proof}

\chapter{Control theory}\label{app:ctr_theory}

In this part we recall basic definitions from control theory. Later we formulate the Pontryagin maximum principle in its classical form. 

\begin{definition}\label{def:cs_class} A \emph{control system}\index{control system} on a manifold $M$ is a map
\begin{equation}\label{eqn:cs_class}
f:M\times U\lra\T M,
\end{equation}
such that, for every fixed $u\in U$, the map $f(\cdot,u):M\lra\T M$ is a $C^1$-vector field. We assume that $U$ is a subset of some Euclidean space $\R^r$ and that the maps $f:M\times U\lra\T M$ and $T_xf:\T M\times U\lra \T\T M$ are continuous.
\end{definition}

Choose now an \emph{admissible control}\index{admissible controls}, i.e., a bounded measurable function $u:[t_0,t_1]\ra U$. We can consider a time-dependent differential equation on $M$
$$\dot x(t)=f(x(t),u(t)),$$
with a fixed initial condition $x(t_0)=x_0$. The solution $x(t)$ of the above is called a \emph{trajectory}\index{trajectory of a control system} of a control system \eqref{eqn:cs_class} associated with the control $u(t)$, and the pair $\left(x(t),u(t)\right)$ is called a \emph{controlled pair}\index{controlled pair}.

Let us now introduce the \emph{total cost}\index{total cost} of the controlled pair $\left(x(t),u(t)\right)$
$$\mathcal{J}(x(\cdot),u(\cdot))=\int_{t_0}^{t_1}L\left(x(t),u(t)\right)\dd t,$$
where on the integrand $L:M\times U\lra \T M$ (the \emph{cost function})\index[cost function] we put  the same regularity assumptions as on $f$, namely, $L:M\times U\lra\R$ and $\T_xL:\T M\times U\lra\R$ are continuous maps.

Given two points $x_0$ and $x_1$ we can introduce an \emph{optimal control problem}\index{optimal control problem}:
\begin{equation}
\label{eqn:P_class}\tag{$\wt{\text{P}}$}
\begin{split}
&\text{minimise the total cost $\mathcal{J}\left(x(\cdot),y(\cdot)\right)$ over all controlled pairs $(x(t),u(t))$}\\
&\text{ (with all possible time intervals $t\in[t_0,t_1]$) s.t. $x(t_0)=x_0$ and $x(t_1)=x_1$.} \\
\end{split}
\end{equation}

Let now $\iota_0:S_0\hookrightarrow M$ and $\iota_1:S_1\hookrightarrow M$ be two immersed submanifolds of $M$. We define the following OCP with \emph{general boundary conditions}\index{optimal control problem!with general boundary conditions}:
\begin{equation}
\label{eqn:P_class_rel}\tag{$\wh{\text{P}}$}
\begin{split}
&\text{minimise the total cost $\mathcal{J}\left(x(\cdot),y(\cdot)\right)$ over all controlled pairs $(x(t),u(t))$}\\
&\text{ (with all possible time intervals $t\in[t_0,t_1]$) such that $x(t_0)\in S_0$ and $x(t_1)\in S_1$.} \\
\end{split}
\end{equation}

Necessary optimality conditions for the problem \eqref{eqn:P_class} are the following.
\begin{theorem}[the PMP]\label{thm:pmp_class}\index{Pontryagin maximum principle}\index{PMP|see{Pontryagin maximum principle}}
Let $(x(t),u(t))$, with $t\in[t_0,t_1]$, be a controlled pair of
\eqref{eqn:cs_class} solving the optimal control problem
\eqref{eqn:P_class}. Then there exists a curve $p:[t_0,t_1]\lra
\T^\ast M$ covering $x(t)$ and a constant $p_0\leq 0$ such that the following holds:
\begin{itemize}
    \item the curve $p(t)$ is a trajectory of the time-dependent family of Hamiltonian vector fields
    $\X_{H_t}$ for the canonical symplectic structure on $\T^\ast M$ and Hamiltonians $H_t(x,p):=H(x,p,u(t))$, where
    $$H(x,p,u)=\< f\left(x,u\right), p>+p_0 L\left(x,u\right);$$
    \item the control $u$ satisfies the ``maximum principle''
    $$H(x(t),p(t),u(t))=\sup_{v\in U}H(x(t),p(t),v)$$
and $H(x(t),p(t),u(t))=0$ at every regular point $t$ of $u$;
    \item if $p_0=0$ the covector $p(t)$ is nowhere-vanishing.
\end{itemize}
\end{theorem}
For the problem \eqref{eqn:P_class_rel} we have more specific conditions.
\begin{theorem}[the PMP for general boundary conditions] \label{thm:pmp_class_rel}\index{Pontryagin maximum principle! for general boundary conditions}
Let $(x(t),u(t))$, with $t\in[t_0,t_1]$, be a controlled pair of
\eqref{eqn:cs_class} solving the optimal control problem
\eqref{eqn:P_class_rel}. Then there exists a curve $p:[t_0,t_1]\lra
\T^\ast M$ covering $x(t)$ and a constant $p_0\leq 0$ which satisfy the assertion of Theorem \ref{thm:pmp_class} and, additionally, $p(t)$ satisfies the following transversality conditions: $p(t_0)$ annihilates $\T_{x(t_0)}S_0$ and $p(t_1)$ annihilates $\T_{x(t_1)}S_1$. 
\end{theorem}
For the original proof of the above theorems we refer the reader to \cite{pontryagin}. Recent references are \cite{agrachev} and \cite{barbero_pmp}. 

\chapter{Geometry and Topology}\label{app:geom_top}

\section{Separation of convex cones}

Geometrically, Pontryagin maximum principle describes the separation of certain cones associated with the optimal control problem, which live in the fibres of the algebroid $E\oplus\T\R$, along an optimal trajectory. Therefore we need some technical results concerning the separation of convex cones.

\begin{definition} Two convex sets $K_1$ and $K_2$ in a vector space $V$ are \emph{separable}\index{separation of convex sets} iff there exists a non-zero covector $\varphi\in V^\ast$ such that
\begin{equation}\label{eqn:separation}
\<k_1,\varphi>\geq \<k_2,\varphi> \quad \text{for every $k_1\in K_1$ and $k_2\in K_2$.}
\end{equation} 
We say that $K_1$ and $K_2$ are \emph{strictly separable}\index{separation of convex sets!strict} iff there exists a non-zero covector $\varphi\in V^\ast$ and numbers $a,b\in\R$ such that
$$\<k_1,\varphi>\geq a>b\geq \<k_2,\varphi> \quad \text{for every $k_1\in K_1$ and $k_2\in K_2$.}$$
\end{definition} 

\noindent A basic fact from the theory of convex sets in  a finite dimensional space is the following

\begin{theorem}[separation]\label{thm:separation}\index{separation theorem}
Two disjoint convex sets in a finite dimensional vector space are separable.
  If, in addition, these sets are closed, and one of them is compact, they are strictly separable.
\end{theorem}
\noindent For the proof see \cite{giannessi}. We will also need the following fact.
\begin{lemma}\label{lem:separation_closure} The convex sets $K_1$ and $K_2$ in a finite-dimensional space $V$ are separable if and only if $\cl(K_1)$ and $\cl(K_2)$ are separable.
\end{lemma}
\begin{proof} If \eqref{eqn:separation} holds for every $k_1\in K_1$ and $k_2\in K_2$ then, since the weak inequality is preserved under taking limits, it also holds for every $k_1\in\cl(K_1)$ and $k_2\in\cl(K_2)$. 

Conversely, if \eqref{eqn:separation} holds for every  $k_1\in\cl(K_1)$ and $k_2\in\cl(K_2)$, it is also true on smaller sets $K_1\subset\cl(K_1)$ and $K_2\subset \cl(K_2)$. 
\end{proof}

\begin{definition} By a \emph{cone}\index{cone} in a vector space $V$ we will mean a set $K$ which is invariant under positive homotheties, i.e.,
$$t\cdot k\in K\quad \text{ whenever $k\in K$ and $t>0$.}$$
\end{definition}

\begin{remark}\label{rem:separation}\index{separation of convex sets!cones}
If two convex sets $K_1$ and $K_2$ in $V$ are separable, and one of them, say $K_1$, is a cone, then the separating covector $\varphi$ satisfies
\begin{equation}\label{eqn:separation_cones}
\<k_1,\varphi>\geq 0\geq \<k_2,\varphi> \quad \text{for every $k_1\in K_1$ and $k_2\in K_2$.}
\end{equation} 
Indeed, since $K_1$ is invariant under homotheties, the image $\varphi(K_1)$ contains numbers arbitrary close to $0$, hence from \eqref{eqn:separation} it satisfies $0\geq\<k_2,\varphi>$ for all $k_2\in K_2$. On the other hand, if $\<k_1,\varphi><0$ for some $k_1\in K_1$, then the image $\<K_1,\varphi>$ would contain arbitrarily big negative numbers, and hence \eqref{eqn:separation} would not hold.

Above observation has two simple but important consequences. First of all, the separating covector $\varphi$ vanishes on the intersection $K_1\cap K_2$. Secondly, if one of the sets $K_i$ contains an affine subspace $k_i+S$, where $S\subset V$ is a linear subspace, then $\varphi$ vanishes on $S$.  
\end{remark}

 Now we prove a geometric characterisation of non-separability in a certain geometric setting. 

\begin{lemma}\label{lem:separation}
Consider a vector space $V=W\oplus\R$, denote by $\Lambda$ a ray in $\R$ spanned by a vector $\lambda$, i.e., $\Lambda=\R_+\cdot\lambda\subset\R$, and let $S\subset W$ be a linear subspace. Let $K_1$ be a convex cone in $V$, and denote by $K_2$ the convex cone $S\oplus\Lambda$. The cones $K_1$ and $K_2$ are not separable iff there exists a vector $k\in K_1\cap K_2$ and vectors $e_1,e_2,\hdots,e_m\in W$ such that
\begin{align}
\label{cond:sep_A}
& {\operatorname{span}\{e_1,e_2,\hdots,e_m,S\}=W}, \\
& \text{vectors $k\pm e_1, \hdots, k\pm e_m$ belong to $K_1$,} \label{cond:sep_B}
\end{align}
where we identify $e_i\in W$ with $e_i+\theta\in W\oplus\R=V$.
\end{lemma}
\begin{proof}
Assume that vectors $k,e_1,\hdots,e_m$ satisfy conditions \eqref{cond:sep_A} and \eqref{cond:sep_B}. If $\varphi$ is a non-zero covector separating $K_1$ and $K_2$ then, due to Remark \ref{rem:separation}, $\varphi$ vanishes on $k$ and $S$. Moreover, $0\leq\<k\pm e_i,\varphi>=\pm \<e_i,\varphi>$, hence $\<e_i,\varphi>=0$. Since $\operatorname{span}\{k,e_1,\hdots,e_n,S\}=V$, the covector $\varphi$ is null, which gives a contradiction.

The opposite implication is harder to prove. We will make an inductive argument with respect to $\dim S$. 

\ul{Assume that $\dim S=0$}; i.e., $K_2=0\oplus\Lambda=\R_+\cdot \lambda$. If $K_1$ and $K_2$ are not separable then
$$\spann\{K_1,\lambda\}=V\quad\text{and}\quad \lambda\in K_1.$$
Indeed, if $V^{'}=\spann\{K_1,\lambda\}\subsetneq V$, then a non-zero covector vanishing on $V^{'}$ will separate $K_1$ and $K_2$. Secondly, if $\lambda\notin K_1$ then, by Theorem \ref{thm:separation}, convex sets $K_1$ and $\{\lambda\}$ can be separated by a covector $\varphi$. Since $K_1$ is a cone we have
$\<k_1,\varphi>\geq 0\geq \<\lambda,\varphi>$ for $k_1\in K_1$ (compare Remark \ref{rem:separation}), hence $\varphi$ separates also $K_1$ and $\R_+\cdot\lambda=K_2$.   

We deduce that $\lambda\in K_1$, and that there exist vectors $\ol {e_1},\ol{e_2},\hdots,\ol{e_m}\in K_2$ such that $\{\lambda,\ol {e_1},\hdots,\ol{e_m}\}$ is a basis of $V$. Obviously each $\ol{e_i}$ is of the form $\ol{e_i}=a_i\cdot\lambda+\wh{e_i}$, where $a_i$ are numbers and $\{\wh{e_1},\hdots,\wh{e_m}\}$ is a basis of $W$. 

Now fix $i$ and consider a vector $\lambda-\frac 1N\wh{e_i}$, where $N\in\N$. If $\lambda-\frac 1N\wh{e_i}\notin K_1$ for all $N\in\N$, then, by Theorem \ref{thm:separation} and Remark \ref{rem:separation}, there exist covectors $\varphi_N$ such that 
$$\<k_1,\varphi_N>\geq 0\quad\text{for all $k_1\in K_1$ and}\quad \<\lambda-\frac 1N \wh{e_i},\varphi_N>\leq 0.$$
We may assume that all $\varphi_N$ are normalised to 1 and choose a subsequence converging to $\varphi_0\in V^\ast$. Clearly, $\varphi_0$ is non-zero (it is normalised to 1) and  
$$\<k_1,\varphi_0>\geq 0\quad\text{for all $k_1\in K_1$ and}\quad \<\lambda,\varphi_0>\leq 0,$$
that is, $\varphi_0$ separates $K_1$ and $K_2=\R_+\cdot\lambda$ against the assumptions. 

To sum up, we proved that $K_1$ contains elements $\lambda$, $a_i\lambda+\wh{e_i}$ and $\lambda-\frac 1{N_i} \wh{e_i}$. It is clear that some convex combinations of these vectors, after rescaling, are of the form $\lambda+e_i$ and $\lambda-e_i$, where $e_i$ is parallel to $\wh{e_i}$.

\ul{Now consider $\dim S>0$}. We can split $S=S^{'}\oplus\wh S$, where $S^{'}=\R\cdot s$ is one-dimensional. By Remark \ref{rem:separation}, if $K_1$ and $K_2$ are separable, then the separating covector $\varphi$ vanishes on $S^{'}$. It follows that  $K_1$ and $K_2$ are separable in $V$ if and only if $K^{'}_1:=K_1/S^{'}$ and $K^{'}_2:=K_2/S^{'}=\wh S\oplus\Lambda$ are separable in $V^{'}:=V/S^{'}=W/S^{'}\oplus\R$. By the inductive assumption there exists vectors $k\in K_1/S^{'}\cap K_2/S^{'}$ and $\wh{e_1},\hdots\wh{e_m}\in W/S^{'}$ such that conditions \eqref{cond:sep_A} and \eqref{cond:sep_B} are satisfied for $W^{'}=W/S^{'}$ and $K^{'}_1$. In other words, there exists a vector $k\in S\oplus \Lambda$, vectors $\wh{e_1},\hdots,\wh{e_m}\in W$, and numbers $a,a_i,b_i$ such that vectors  $k+a\cdot s$, $k+a_i\cdot s+\wh{e_i}$ and $k+b_i\cdot s-\wh{e_i}$ belong to $K_1$ and $\spann\{\wh{e_1},\hdots,\wh{e_m}, S\}=\spann\{\wh{e_1},\hdots,\wh{e_m}, \wh S,s\}=\spann\{W/S^{'},s\}=W$. 

From convexity of $K_1$ we deduce that $k+\frac{a_i+b_i}2\cdot s\in K_1$. Now, either all numbers $\frac{a_i+b_i}2$ are equal $a$, and then vectors $k+a\cdot s$ and $e_i=\wh{e_i}+(a_i-a)\cdot s$ satisfy the assertion, or the interval $k+[c,d]\cdot s$, where $a,\frac{a_i+b_i} 2\in[c,d]$, is entirely contained in $K_1$. In the second case it is quite clear that some convex combinations of vectors $k+c\cdot s$, $k+d\cdot s$, $k+a_i\cdot s+\wh{e_i}$ and $k+b_i\cdot s-\wh{e_i}$ are of the form $k+\wt{a_i}\cdot s+\wt {e_i}$ and  $k+\wt{b_i}\cdot s-\wt {e_i}$, where $\frac{\wt{a_i}+\wt{b_i}}2=\frac{c+d}2=\wt{a}$ and $\wt{e_i}$ is parallel to $\wh{e_i}$. Therefore we are in the first case again.

The inductive argument is now complete.
\end{proof}

\section{Simple topological lemmas}

\begin{lemma}\label{lem:top1_app}
Every continuous map $\Psi:B^m(0,1)\ra \R^m$ which satisfies the inequality
$$\left\|\Psi(\vec r)-\vec r\right\|\leq \frac 12 \quad\text{for all $\vec r\in B^m(0,1)$},$$
contains point $0\in\R^m$ in its image.
\end{lemma}
\begin{proof}
The map
$$H(\vec R,t):=(1-t)\vec R+t\Psi(\vec R),$$
where $t\in[0,1]$ and $\vec R\in\pa B^m(0,1)$ is a homotopy between $\pa B^m(0,1)$ and $\Psi(\pa B^m(0,1))$. Observe that $\|H(\vec R,t)-\vec R\|=t\|\Psi(\vec R)-\vec R\|\leq \frac 12$, hence
$$\left\|H(\vec R,t)\right\|\geq\left\|\vec R\right\|-\left\|H(\vec R,t)-\vec R\right\|\geq 1-\frac 12\geq\frac 12,$$
and consequently $H$ takes values in $\R^m\setminus\{0\}$.

Assume that $0\notin \Image\Psi$. Then the map 
$$S(\vec R,t):=\Psi((1-t)\vec R)$$ 
is a homotopy between  $\Psi(\pa B^m(0,1))$ and $\Psi(0)$, which takes values in $\Image \Phi\subset\R^m\setminus\{0\}$.

The composition 
$$H\circ S(\vec R,t)=
\begin{cases}
H(\vec R,2t) &\text{for $t<\frac 12$},\\
S(\vec R,2t-1) &\text{for $t\geq \frac 12$}
\end{cases}$$
is a contraction of $\pa B^m(0,1)$ to the point $\Psi(0)$, which takes values in $\R^m\setminus\{0\}$. On the other hand, the sphere $\pa B^m(0,1)$ is not contractible in $\R^m\setminus\{0\}$.
\end{proof}

\begin{lemma}\label{lem:top2_app}
Let $\Psi_1:B^{m_1}(0,1)\oplus\theta_{m_2}\ra \R^{m_1}\oplus\R^{m_2}$ and $\Psi_2:\theta_{m_1}\oplus B^{m_2}(0,1)\ra \R^{m_1}\oplus\R^{m_2}$ be two continuous maps satisfying, for a fixed vector $\vec k_0\in\R^{m_1}\oplus\R^{m_2}$, the following inequalities:
\begin{align*}
&\|\Psi_1(\vec r)-(\vec r+\vec k_0)\|\leq \frac 14 \quad\text{for $\vec r\in B^{m_1}(0,1)\oplus\theta_{m_2}$ and }\\
&\|\Psi_2(\vec s)-(\vec s+\vec k_0)\|\leq \frac 14 \quad\text{for $\vec s\in \theta_{m_1}\oplus B^{m_2}(0,1)$}.
\end{align*}
Then the images $\Image\Psi_1$ and $\Image\Psi_2$ have a non-empty intersection.
\end{lemma}

\begin{proof}
Consider a continuous map
$$B^{m_1+m_2}(0,1)=B^{m_1}(0,1)\oplus B^{m_2}(0,1)\ni(\vec r+\vec s)\overset{\Psi}{\longmapsto}\Psi_1(\vec r)-\Psi_2(-\vec s)\in\R^{m_1}\oplus\R^{m_2}.$$
Now
\begin{align*}
\|\Psi(\vec r,\vec s)-(\vec r+\vec s)\|&=\left \|\Psi_1(\vec r)-(\vec r+\vec k_0)-\left(\Psi_2(-\vec s)-\left(-\vec s+\vec k_0\right)\right)\right\|\leq\\
&\leq\left\|\Psi_1(\vec r)-\left(\vec r+\vec k_0\right)\right\|+\left\|\Psi_2(-\vec s)-\left(-\vec s+\vec k_0\right)\right\|\leq \frac 14+\frac 14=\frac 12.
\end{align*}
By Lemma \ref{lem:top1_app} point $0$ lies in the image of $\Psi$. Consequently, $\Psi_1(\vec r_0)=\Psi_2(-\vec s_0)$ for some $\vec r_0$ and $\vec s_0$.
\end{proof}

\backmatter
\bibliographystyle{nowy}
\bibliography{books,articles}

\begin{thebibliography}{}

\bibitem[{Agrachev \& Gamkrelidze, }2006]{agrachev_pmp_50years}
Agrachev, A.A., \& Gamkrelidze, R.V. 2006.
\newblock The Pontryagin Maximum Principle 50 years later.
\newblock {\em Proceedings of the Steklov Institute of Mathematics}, \textbf{
  253}, 4--12.

\bibitem[{Agrachev \& Sachkov, }2004]{agrachev}
Agrachev, A.A., \& Sachkov, Y.L. 2004.
\newblock {\em Control theory from the geometric viewpoint}.
\newblock Encyclopaedia of mathematical sciences, no. ~2.
\newblock Springer.

\bibitem[{Agrachev \& Sarychev, }1996]{agrachev1996abnormal}
Agrachev, A.A., \& Sarychev, A.V. 1996.
\newblock Abnormal sub-Riemannian geodesics: Morse index and rigidity.
\newblock {\em Annales de l'Institut Henri Poincar{\'e}. Analyse non
  lin{\'e}aire}, \textbf{ 13}(6), 635--690.

\bibitem[{Agrachev \& Sarychev, }1998]{agrachev1998abnormal}
Agrachev, A.A., \& Sarychev, A.V. 1998.
\newblock On abnormal extremals for Lagrange variational problems.
\newblock {\em Journal of Mathematical Systems Estimation and Control},
  \textbf{ 8}, 87--118.

\bibitem[{Almeida, }1980]{almeria}
Almeida, R. 1980.
\newblock {\em Teoris de Lie para os groupoides diferenciaries}.
\newblock Ph.D. thesis, Sao Paulo.

\bibitem[{Almeida \& Kumpera, }1981]{almeida_kumpera}
Almeida, R., \& Kumpera, A. 1981.
\newblock Structure produit dans la cat{\'e}gorie des alg{\'e}bro{\"\i}des de
  Lie.
\newblock {\em Anais da Academia Brasileira de Ci{\^e}ncias}, \textbf{ 53},
  247--250.

\bibitem[{Almeida \& Molino, }1985]{almeida_molino}
Almeida, R., \& Molino, P. 1985.
\newblock Suites d’Atiyah et feuilletages transversalement complets.
\newblock {\em Comptes Rendus de l'Acad{\'e}mie des Sciences, S{\'e}rie A},
  \textbf{ 300}, 13--15.

\bibitem[{Barbero-Li{\~n}{\'a}n \& Mu{\~n}oz-Lecanda, }2009]{barbero_pmp}
Barbero-Li{\~n}{\'a}n, M., \& Mu{\~n}oz-Lecanda, M.C. 2009.
\newblock Geometric approach to Pontryagin’s maximum principle.
\newblock {\em Acta applicandae mathematicae}, \textbf{ 108}(2), 429--485.

\bibitem[{Bloch, }2003]{bloch}
Bloch, A. 2003.
\newblock {\em Nonholonomic mechanics and control}.
\newblock Interdisciplinary applied mathematics: Systems and control.
\newblock Springer.

\bibitem[{Bonnard \& Tr{\'e}lat, }2001]{bonnard_abnormal}
Bonnard, B., \& Tr{\'e}lat, E. 2001.
\newblock On the role of abnormal minimizers in sub-Riemannian geometry.
\newblock {\em Annales de la Facult\'{e} des Sciences de Toulouse
  Math\'{e}matiques}, \textbf{ 10}(3), 405--491.

\bibitem[{Bressan \& Piccoli, }2007]{bressan}
Bressan, A., \& Piccoli, B. 2007.
\newblock {\em Introduction to the mathematical theory of control}.
\newblock AIMS series on applied mathematics.
\newblock American Institute of Mathematical Sciences.

\bibitem[{Cattaneo \& Felder, }2004]{cattaneo_felder}
Cattaneo, A.S., \& Felder, G. 2004.
\newblock Coisotropic submanifolds in Poisson geometry and branes in the
  Poisson sigma model.
\newblock {\em Letters in Mathematical Physics}, \textbf{ 69}(1), 157--175.

\bibitem[{Cendra {\em et~al.}\relax, }1998]{cendra_holm_marsden}
Cendra, H., Holm, D.D., Marsden, J.E., \& Ratiu, T.S. 1998.
\newblock {Lagrangian reduction, the Euler-Poincar{\'e} equations, and
  semidirect products}.
\newblock {\em Pages  1--25 of:} {\em Geometry of Differential Equations}.
\newblock American Mathematical Society Translations, vol. 186.
\newblock American Mathematical Society.

\bibitem[{Chaplygin, }1911]{chaplygin}
Chaplygin, SA. 1911.
\newblock {On the theory of the motion of nonholonomic systems. Theorem on the
  reducing multiplier}.
\newblock {\em Mat. Sbornik}, \textbf{ 28}(2), 303--314.

\bibitem[{Clarke, }1976]{clarke_min_hypotheses}
Clarke, F. 1976.
\newblock The maximum principle under minimal hypotheses.
\newblock {\em SIAM Journal on Control and Optimization}, \textbf{ 14}, 1078.

\bibitem[{Clarke, }2005a]{clarke_brief}
Clarke, F. 2005a.
\newblock The maximum principle in optimal control, then and now.
\newblock {\em Control and Cybernetics}, \textbf{ 34}(3), 709.

\bibitem[{Clarke, }2005b]{clarke_necessary}
Clarke, F. 2005b.
\newblock {\em Necessary conditions in dynamic optimization}.
\newblock Memoirs of the American Mathematical Society, no.  816.
\newblock American Mathematical Society.

\bibitem[{Cort{\'e}s \& Mart{\'\i}nez, }2004]{cortes_martinez}
Cort{\'e}s, J., \& Mart{\'\i}nez, E. 2004.
\newblock {Mechanical control systems on Lie algebroids}.
\newblock {\em IMA Journal of Mathematical Control and Information}, \textbf{
  21}(4), 457.

\bibitem[{Cort{\'e}s {\em et~al.}\relax, }2006]{cortes_leon_marrero}
Cort{\'e}s, J., De~Leon, M., Marrero, J.C., De~Diego, D.M., \& Martinez, E.
  2006.
\newblock {A survey of Lagrangian mechanics and control on Lie algebroids and
  groupoids}.
\newblock {\em International Journal of Geometric Methods in Modern Physics},
  \textbf{ 3}(3), 509--558.

\bibitem[{Coste {\em et~al.}\relax, }1987]{coste_dazord_weinstein}
Coste, A., Dazord, P., \& Weinstein, A. 1987.
\newblock {Groupo\"{i}des symplectiques}.
\newblock {\em Pages  1--62 of:} {\em Publicatins D{\'e}partement de
  Math{\'e}matiques}.
\newblock Universit{\'e} Claude Bernard-Lyon I.

\bibitem[{Courant, }1990]{courant}
Courant, T.J. 1990.
\newblock Dirac manifolds.
\newblock {\em Transactions of the American Mathematical Society}, \textbf{
  319}(2), 631--661.

\bibitem[{Crainic \& Fernandes, }2003]{crainic_fernandes}
Crainic, M., \& Fernandes, R.L. 2003.
\newblock {Integrability of Lie brackets}.
\newblock {\em The Annals of Mathematics}, \textbf{ 157}(2), 575--620.

\bibitem[{Crainic \& Fernandes, }2004]{crainic_fernandes_poisson}
Crainic, M., \& Fernandes, R.L. 2004.
\newblock Integrability of Poisson brackets.
\newblock {\em Journal of Differential Geometry}, \textbf{ 66}(1), 71--137.

\bibitem[{Duistermaat \& Kolk, }2000]{duistermaat_kolk}
Duistermaat, J.J., \& Kolk, J.A.C. 2000.
\newblock {\em Lie groups}.
\newblock Universitext.
\newblock Springer.

\bibitem[{Evans, }2010]{evans}
Evans, L.C. 2010.
\newblock {\em Partial differential equations}.
\newblock Graduate studies in mathematics.
\newblock American Mathematical Society.

\bibitem[{Gamkrelidze {\em et~al.}\relax, }1978]{gamkrelidze_78}
Gamkrelidze, R.V., Makowski, K., \& Berkovitz, L. 1978.
\newblock {\em Principles of optimal control theory}.
\newblock Plenum Press.

\bibitem[{Giannessi, }2005]{giannessi}
Giannessi, F. 2005.
\newblock {\em {Separation of sets and optimality conditions}}.
\newblock Constrained optimization and image space analysis.
\newblock Springer.

\bibitem[{Grabowska \& Grabowski, }2008]{GG_var_calc}
Grabowska, K., \& Grabowski, J. 2008.
\newblock {Variational calculus with constraints on general algebroids}.
\newblock {\em Journal of Physics A: Mathematical and Theoretical}, \textbf{
  41}, 175204.

\bibitem[{Grabowska {\em et~al.}\relax, }2006]{GGU_geom_mech}
Grabowska, K., Grabowski, J., \& Urba{\'n}ski, P. 2006.
\newblock {Geometrical mechanics on algebroids}.
\newblock {\em International Journal of Geometric Methods in Modern Physics},
  \textbf{ 3}(3), 559--575.

\bibitem[{Grabowski \& J{\'o}{\'z}wikowski, }2011]{grabowski_jozwikowski_pmp}
Grabowski, J., \& J{\'o}{\'z}wikowski, M. 2011.
\newblock Pontryagin maximum principle on almost Lie algebroids.
\newblock {\em SIAM Journal on Control and Optimization}, \textbf{ 49}(3),
  1306--1357.

\bibitem[{Grabowski \& Rotkiewicz, }2009]{GR_higher}
Grabowski, J., \& Rotkiewicz, M. 2009.
\newblock {Higher vector bundles and multi-graded symplectic manifolds}.
\newblock {\em Journal of Geometry and Physics}, \textbf{ 59}(9), 1285--1305.

\bibitem[{Grabowski \& Urba{\'n}ski, }1997]{GU_poiss_nijn}
Grabowski, J., \& Urba{\'n}ski, P. 1997.
\newblock {Lie algebroids and Poisson--Nijenhuis structures}.
\newblock {\em Reports on Mathematical Physics}, \textbf{ 40}(2), 195--208.

\bibitem[{Grabowski \& Urba{\'n}ski, }1999]{GU_algebroids}
Grabowski, J., \& Urba{\'n}ski, P. 1999.
\newblock {Algebroids --- general differential calculi on vector bundles}.
\newblock {\em Journal of Geometry and Physics}, \textbf{ 31}(2-3), 111--141.

\bibitem[{Grabowski {\em et~al.}\relax, }2009]{grabowski_nonholonomic}
Grabowski, J., De~Leon, M., Marrero, J.C., \& De~Diego, D.M. 2009.
\newblock {Nonholonomic constraints: A new viewpoint}.
\newblock {\em Journal of Mathematical Physics}, \textbf{ {50}}, 013520.

\bibitem[{Jurdjevic, }1997]{jurdjevic}
Jurdjevic, V. 1997.
\newblock {\em Geometric control theory}.
\newblock Cambridge studies in advanced mathematics.
\newblock Cambridge University Press.

\bibitem[{Konieczna \& Urba{\'n}ski, }1999]{KU_dvb}
Konieczna, K., \& Urba{\'n}ski, P. 1999.
\newblock {Double vector bundles and duality}.
\newblock {\em Archivum Mathematicum}, \textbf{ 35}(1), 59--95.

\bibitem[{Kubarski, }1994]{Kubarski}
Kubarski, J. 1994.
\newblock {Invariant cohomology of regular Lie algebroids}.
\newblock {\em Pages  26--30 of:} {\em Proceedings of the VIIth International
  Colloquium on Differential Geometry, (July 1994, Spain)}.

\bibitem[{Lang, }1985]{Lang}
Lang, S. 1985.
\newblock {\em Differential manifolds}.
\newblock Springer-Verlag.

\bibitem[{Langerock, }2003a]{langerock_phd}
Langerock, B. 2003a.
\newblock {\em Generalised connections and applications to control theory}.
\newblock Ph.D. thesis, Ghent University.

\bibitem[{Langerock, }2003b]{langerock2003}
Langerock, B. 2003b.
\newblock Geometric aspects of the maximum principle and lifts over a bundle
  map.
\newblock {\em Acta Applicandae Mathematicae}, \textbf{ 77}(1), 71--104.

\bibitem[{Le{\'o}n {\em et~al.}\relax, }2005]{leon_marrero_martinez}
Le{\'o}n, M., Marrero, J.C., \& Mart{\'\i}nez, E. 2005.
\newblock {Lagrangian submanifolds and dynamics on Lie algebroids}.
\newblock {\em Journal of Physics A: Mathematical and General}, \textbf{ 38},
  R241.

\bibitem[{Libermann, }1996]{libermann}
Libermann, P. 1996.
\newblock {Lie algebroids and mechanics}.
\newblock {\em Arch. Math. (Brno)}, \textbf{ 32}, 147--162.

\bibitem[{{\L}ojasiewicz \& Ferreira, }1988]{lojasiewicz}
{\L}ojasiewicz, S., \& Ferreira, A.V. 1988.
\newblock {\em {An introduction to the theory of real functions}}.
\newblock A Wiley interscience publication.
\newblock John Wiley \& Sons Inc.

\bibitem[{Mackenzie, }1987]{mackenzie_1987}
Mackenzie, K. 1987.
\newblock {\em Lie groupoids and Lie algebroids in differential geometry}.
\newblock London Mathematical Society lecture note series.
\newblock Cambridge University Press.

\bibitem[{Mackenzie, }2005]{mackenzie}
Mackenzie, K. 2005.
\newblock {\em {General theory of lie groupoids and lie algebroids}}.
\newblock London Mathematical Society lecture note series.
\newblock Cambridge University Press.

\bibitem[{Mackenzie \& Xu, }2000]{mackenzie_xu}
Mackenzie, K.C.H., \& Xu, P. 2000.
\newblock Integration of Lie bialgebroids.
\newblock {\em Topology}, \textbf{ 39}(3), 445--467.

\bibitem[{Mart{\'\i}nez, }2001a]{martinez_geom_form}
Mart{\'\i}nez, E. 2001a.
\newblock {Geometric formulation of mechanics on Lie algebroids}.
\newblock {\em Pages  209--222 of:} {\em Proceedings of the VIII Fall Workshop
  on Geometry and Physics (1999, Medina del Campo), Publicaciones de la RSME},
  vol. 2.

\bibitem[{Mart{\'\i}nez, }2001b]{martinez_lagr_mech}
Mart{\'\i}nez, E. 2001b.
\newblock {Lagrangian mechanics on Lie algebroids}.
\newblock {\em Acta Applicandae Mathematicae}, \textbf{ 67}(3), 295--320.

\bibitem[{Mart{\'\i}nez, }2004]{martinez_red_opt_ctr}
Mart{\'\i}nez, E. 2004.
\newblock {Reduction in optimal control theory}.
\newblock {\em Reports on Mathematical Physics}, \textbf{ 53}(1), 79--90.

\bibitem[{Mart{\'\i}nez, }2005]{martinez_cft}
Mart{\'\i}nez, E. 2005.
\newblock {Classical field theory on Lie algebroids: variational aspects}.
\newblock {\em Journal of Physics A: Mathematical and General}, \textbf{ 38},
  7145.

\bibitem[{Mart{\'\i}nez, }2007]{martinez_lie_classs_mech}
Mart{\'\i}nez, E. 2007.
\newblock {Lie Algebroids in Classical Mechanics and Optimal Control}.
\newblock {\em Symmetry, Integrability and Geometry: Methods and Applications},
  \textbf{ 3}, 050.

\bibitem[{Mart{\'\i}nez, }2008]{martinez_var_calc}
Mart{\'\i}nez, E. 2008.
\newblock Variational calculus on Lie algebroids.
\newblock {\em ESAIM: Control, Optimisation and Calculus of Variations},
  \textbf{ 14}(2), 356--380.

\bibitem[{Moerdijk \& Crainic, }2001]{crainic_moerdijk}
Moerdijk, I., \& Crainic, M. 2001.
\newblock Foliation groupoids and their cyclic homology.
\newblock {\em Advances in Mathematics}, \textbf{ 157}(2), 177--197.

\bibitem[{Moerdijk \& Mr\v{c}un, }2003]{moerdij_mrcun}
Moerdijk, I., \& Mr\v{c}un, J. 2003.
\newblock {\em Introduction to foliations and Lie groupoids}.
\newblock Cambridge Studies in Advanced Mathematics.
\newblock Cambridge University Press.

\bibitem[{Montgomery, }1990]{montgomery_isohol}
Montgomery, R. 1990.
\newblock {Isoholonomic problems and some applications}.
\newblock {\em Communications in Mathematical Physics}, \textbf{ 128}(3),
  565--592.

\bibitem[{Montgomery, }1994]{montgomery_abnormal}
Montgomery, R. 1994.
\newblock Abnormal minimizers.
\newblock {\em SIAM Journal on control and optimization}, \textbf{ 32}(6),
  1605--1620.

\bibitem[{Neimark \& Fufaev, }1972]{neimark}
Neimark, J.I., \& Fufaev, N.A. 1972.
\newblock {\em {Dynamics of Nonholonomic Systems}}.
\newblock Translations of mathematical monographs.
\newblock American Mathematical Society.

\bibitem[{Pontryagin {\em et~al.}\relax, }1962]{pontryagin}
Pontryagin, L.~S., Boltjanskij, V.~G., Gamkrelidze, R.~V., \& Miscenko, E.~F.
  1962.
\newblock {\em The mathematical theory of optimal processes}.
\newblock Interscience Publishers.

\bibitem[{Pradines, }1966]{pradines1966}
Pradines, J. 1966.
\newblock Th{\'e}orie de Lie pour les groupoides diff{\'e}rentiable.
\newblock {\em Comptes Rendus de l'Acad{\'e}mie des Sciences, S{\'e}rie A},
  \textbf{ 263}, 907--910.

\bibitem[{Pradines, }1967a]{pradines1967a}
Pradines, J. 1967a.
\newblock G{\'e}ometri{\'e} diff{\'e}rentielle au-dessus d'un groupo\"{i}de.
\newblock {\em Comptes Rendus de l'Acad{\'e}mie des Sciences, S{\'e}rie A},
  \textbf{ 266}, 1194--1196.

\bibitem[{Pradines, }1967b]{pradines1967}
Pradines, J. 1967b.
\newblock Th{\'e}orie de Lie pour les groupo\"{i}des diff{\'e}rentiables.
  Calcul diff{\'e}rentiel dans la cat{\'e}gorie des groupo{\i}des
  infinit{\'e}simaux.
\newblock {\em Comptes Rendus de l'Acad{\'e}mie des Sciences, S{\'e}rie A},
  \textbf{ 264}, 245--248.

\bibitem[{Pradines, }1968]{pradines1968}
Pradines, J. 1968.
\newblock Troisieme th{\'e}oreme de Lie pour les groupo\"{i}des
  diff{\'e}rentiables.
\newblock {\em Comptes Rendus de l'Acad{\'e}mie des Sciences, S{\'e}rie A},
  \textbf{ 267}, 21--23.

\bibitem[{Silva \& Weinstein, }1999]{weinstein_silva}
Silva, A.C., \& Weinstein, A. 1999.
\newblock {\em Geometric models for noncommutative algebras}.
\newblock Berkeley mathematics lecture notes.
\newblock American Mathematical Society.

\bibitem[{Tulczyjew, }1974]{tulczyjew_ham_lagr}
Tulczyjew, W.M. 1974.
\newblock {Hamiltonian systems, Lagrangian systems and the Legendre
  transformation}.
\newblock {\em Symposia Matematica}, \textbf{ 14}, 247--258.

\bibitem[{Tulczyjew \& Urbanski, }1999]{tulczyjew_urbanski_slow}
Tulczyjew, W.M., \& Urbanski, P. 1999.
\newblock {A slow and careful Legendre transformation for singular Lagrangians,
  The Infeld Centennial Meeting (Warsaw, 1998)}.
\newblock {\em Acta Physica Polonica B}, \textbf{ 30}, 2909--2978.

\bibitem[{Weinstein, }1987]{weinstein1987symplectic}
Weinstein, A. 1987.
\newblock Symplectic groupoids and Poisson manifolds.
\newblock {\em Bulletin of the American Mathematical Society}, \textbf{ 16}(1),
  101--104.

\bibitem[{Weinstein, }1988]{weinstein_coisotropic}
Weinstein, A. 1988.
\newblock Coisotropic calculus and Poisson groupoids.
\newblock {\em Journal of the Mathematical Society of Japan}, \textbf{ 40}(4),
  705--727.

\bibitem[{Weinstein, }1996]{weinstein}
Weinstein, A. 1996.
\newblock {Lagrangian mechanics and groupoids}.
\newblock {\em Fields Institute Communications}, \textbf{ 7}, 207--232.

\bibitem[{Weinstein \& Xu, }1991]{xu_weinstein}
Weinstein, A., \& Xu, P. 1991.
\newblock Extensions of symplectic groupoids and quantization.
\newblock {\em Journal f\"{u}r die reine und angewandte Mathematik (Crelles
  Journal)}, \textbf{ 417}, 159--190.

\bibitem[{Winkelnkemper, }1983]{winkelnkemper}
Winkelnkemper, H.E. 1983.
\newblock The graph of a foliation.
\newblock {\em Annals of Global Analysis and Geometry}, \textbf{ 1}(3), 51--75.

\bibitem[{Xu, }1992]{xu_symplectic}
Xu, P. 1992.
\newblock Symplectic groupoids of reduced Poisson spaces.
\newblock {\em Comptes Rendus de l'Acad{\'e}mie des Sciences, S{\'e}rie 1},
  \textbf{ 314}(6), 457--461.

\end{thebibliography}
\printindex
\end{document}